\documentclass[12pt]{amsart}
\usepackage{amsmath}
\usepackage{amssymb}
\usepackage{xcolor}
\usepackage{mathrsfs}
\usepackage{stmaryrd}
\usepackage{longtable}
\usepackage{multirow}
\usepackage{makecell}
\usepackage{comment}
\usepackage{tablefootnote}
\usepackage{enumerate}
\usepackage{tikz-cd}
\usepackage{adjustbox}
\usepackage[lowtilde]{url}
\usepackage{hyperref}
\usepackage{bbm}
\usepackage{soul}
\usepackage[smalltableaux]{ytableau}

\newcommand{\C}{\mathbb{C}}
\newcommand{\A}{\mathcal{A}}
\newcommand{\B}{\mathcal{B}}
\newcommand{\U}{\mathcal{U}}
\newcommand{\I}{\mathcal{I}}
\newcommand{\J}{\mathcal{J}}
\newcommand{\M}{\mathcal{M}}
\newcommand{\g}{\mathfrak{g}}
\newcommand{\Orb}{\mathbb{O}}
\newcommand{\Walg}{\mathcal{W}}
\newcommand{\Z}{\mathbb{Z}}
\newcommand{\gr}{\operatorname{gr}}
\newcommand{\kf}{\mathfrak{k}}
\newcommand{\tf}{\mathfrak{t}}
\newcommand{\hf}{\mathfrak{h}}

\newcommand{\gl}{\mathfrak{gl}}

\newcommand{\HC}{\operatorname{HC}}
\newcommand{\q}{\mathfrak{q}}
\newcommand{\Pb}{\mathbb{P}}
\newcommand{\SL}{\operatorname{SL}}
\newcommand{\SO}{\operatorname{SO}}
\newcommand{\Sp}{\operatorname{Sp}}
\newcommand{\SU}{\operatorname{SU}}

\newcommand{\Str}{\mathcal{O}}
\newcommand{\jet}{J^\infty}
\newcommand{\Spec}{\operatorname{Spec}}
\newcommand{\send}{\mathcal{E}nd}
\newcommand{\Id}{\operatorname{Id}}
\newcommand{\pr}{\operatorname{pr}}

\newcommand{\R}{\mathbb{R}}
\newcommand{\T}{\mathcal{T}}
\newcommand{\F}{\mathcal{F}}
\newcommand{\TT}{\mathscr{T}}
\newcommand{\LL}{\mathcal{L}}

\newcommand{\tatp}{\TT^+(\Q, Y)}
\newcommand{\tat}{\TT(\Q,Y)}
\newcommand{\pic}{\mathscr{PA}}

\newcommand{\quan}{\mathcal{Q}}
\newcommand{\per}{\operatorname{Per}}
\newcommand{\cm}{\C^\times}
\newcommand{\OO}{\mathcal{O}}
\newcommand{\Q}{\OO_\hbar}

\newcommand{\rf}{\mathfrak{r}}
\newcommand{\GR}{G_\R}

\newcommand{\KR}{K_\R}

\newcommand{\wt}{\widetilde}

\newcommand{\spec}{\operatorname{Spec}}
\newcommand{\EE}{\mathcal{E}}

\newcommand{\wHC}{\operatorname{wHC}}
\newcommand{\Weyl}{\mathbb{A}}
\newcommand{\slf}{\mathfrak{sl}}
\newcommand{\Dcal}{\mathcal{D}}
\newcommand{\Hom}{\operatorname{Hom}}
\newcommand{\param}{\mathfrak{P}}
\newcommand{\Ecal}{\mathcal{E}}
\newcommand{\GL}{\operatorname{GL}}
\newcommand{\loc}{\operatorname{Loc}}
\newcommand{\coh}{\operatorname{Coh}}

\newcommand{\der}{{\operatorname{Der}}}
\newcommand{\aut}{{\operatorname{Aut}}}
\newcommand{\autd}{\aut(\Dcal)}
\newcommand{\derd}{\der(\Dcal)}

\newtheorem{Thm}{Theorem}[subsection]
\newtheorem{Prop}[Thm]{Proposition}
\newtheorem{Cor}[Thm]{Corollary}
\newtheorem{Lem}[Thm]{Lemma}
\theoremstyle{definition}
\newtheorem{Ex}[Thm]{Example}
\newtheorem{defi}[Thm]{Definition}
\newtheorem{Rem}[Thm]{Remark}

\numberwithin{equation}{section}
\title{On Harish-Chandra modules over quantizations of nilpotent orbits}
\author{Ivan Losev and Shilin Yu}
\dedicatory{Dedicated to David Vogan on the occasion of his 70th birthday}
\begin{document}
\begin{abstract}
Let $G$ be a semisimple algebraic group over the complex numbers and $K$ be a connected reductive group mapping to $G$
so that the Lie algebra of $K$ gets identified with a symmetric subalgebra of $\g$. So we can talk about Harish-Chandra
$(\g,K)$-modules, where $\g$ is the Lie algebra of $G$.
The goal of this paper is to give a geometric classification of irreducible Harish-Chandra modules with full support
over the filtered quantizations of the algebras of the form $\C[\Orb]$, where $\Orb$ is a nilpotent orbit in $\g$ with
codimension of the boundary at least $4$. Namely, we embed the set of isomorphism classes of irreducible Harish-Chandra modules into the set of
isomorphism classes of irreducible $K$-equivariant suitably twisted local systems on $\Orb\cap \mathfrak{k}^\perp$. We show that
under certain conditions, for example when $K\subset G$ or when $\g\cong \mathfrak{so}_n,\mathfrak{sp}_{2n}$, this embedding is
in fact a bijection. On the other hand, for $\g=\mathfrak{sl}_n$ and $K=\operatorname{Spin}_n$, the embedding is not
bijective and we give a description of the image. Finally, we perform a partial classification for exceptional Lie algebras.
\end{abstract}

\maketitle
\tableofcontents
\section{Introduction}
\subsection{Harish-Chandra modules}
Our base field is $\C$. Let $G$ be a semisimple simply connected algebraic group,
$\g$ be its Lie algebra and $\U$
be the universal enveloping algebra of $\g$. Let $K$ be a connected\footnote{In fact, our framework works for any symmetric pair $(\g, K)$ of Harish-Chandra class (i.e., which corresponds to real reductive Lie group of Harish-Chandra class), where $K$ is allowed to be disconnected.} algebraic group equipped with
a homomorphism to $G$ such that the induced homomorphism of Lie algebras is injective and
identifies the Lie algebra $\kf$ of $K$ with the symmetric subalgebra $\g^{\sigma}$ for some involution
$\sigma$. Note that we do not require that $K$ is a subgroup of $G$.

By a {\it Harish-Chandra} (shortly, HC) {$(\g,K)$-module} one means a finite length $\U$-module such that the action of $\kf$ integrates to that of $K$. A primary reason to be interested in
HC modules is that they are related to representations of real forms of $G$ and their nonlinear covers.
Namely, when $K\hookrightarrow G$, the relevant real group is the real form of $G$ corresponding
to the involution $\sigma$, while, in general, we need to consider a cover $\tilde{G}_{\R}$ of $G_\mathbb{R}$ with $\ker[\tilde{G}_{\R}\rightarrow G_\R]\cong \ker[K\rightarrow G]$.

A basic question motivating the present work is: given a primitive ideal $\J$ (=the annihilator of
an irreducible module) in $\U$ classify all irreducible HC $(\g,K)$-modules whose annihilator
coincides with $\J$.

In this paper we will only deal with very special (but also very interesting) ideals $\J$.

\subsection{Quantizations of nilpotent orbits}
Let $\Orb$ be a nilpotent orbit in $\g^*$, $\overline{\Orb}$ be its closure and $\partial\Orb$
be the boundary.
The algebra of regular functions $\C[\Orb]$ is a graded Poisson algebra. In this case
it makes sense to speak about filtered quantizations of $\C[\Orb]$. These are filtered
algebras $\A$ with an isomorphism $\gr\A\xrightarrow{\sim} \C[\Orb]$ of graded Poisson algebras.

 Until the end of Introduction (unless indicated otherwise), we assume that
$$\operatorname{codim}_{\overline{\Orb}}\partial\Orb\geqslant 4.$$
We have the following properties.
\begin{itemize}
\item The quantizations are parameterized by $H^2(\Orb,\C)$. This is a special case
of \cite[Theorem 3.4]{orbit}.
\item The action of $G$ lifts from $\C[\Orb]$ to $\A$. The lift possesses a unique quantum comoment map, i.e., a $G$-equivariant algebra homomorphism
$\Phi:U(\g)\rightarrow \A$ lifting the classical comoment map $\g\rightarrow \C[\Orb]$.
See \cite[Section 5.3]{orbit}.
\end{itemize}

In this paper we consider the ideals $\J$ that arise as kernels of the quantum comoment maps
$\Phi:U(\g)\rightarrow \A$.
Here is a reason why one cares about such ideals. Consider the {\it canonical} (in the terminology
from \cite{LMM}) quantization $\A_0$, its parameter is $0\in H^2(\Orb,\C)$. The corresponding
kernel $\J$ is a {\it unipotent ideal} in the terminology of \cite{LMM}. All such ideals are maximal,
\cite[Theorem 8.5.1]{LMM} and \cite[Theorem 1.0.2]{MM}. As argued in \cite[Section 6.3.4]{LMM}, all HC $(\g,K)$-modules annihilated by $\J$
should be considered unipotent. Such modules play a special role: one hopes that they are unitary and the general unitary
modules are obtained from unipotent ones by operations that are relatively easy to understand.

On the other hand, assume that $H^2(\Orb,\C)=\{0\}$. All orbits $\Orb$ satisfying this condition
(and the codimension $4$ condition above) are birationally rigid (in the sense of
\cite[Definition 1.2]{orbit}).
As a partial converse, 
\begin{itemize}
\item 
all rigid orbits $\Orb$
satisfy $H^2(\Orb,\C)=\{0\}$, 
\item 
and all but six rigid orbits (in exceptional types) satisfy
the codimension 4 condition, see \cite[Lemma 2.4]{reg}. 
\end{itemize}
Unipotent modules
for rigid orbits play a special role because they cannot be induced.

\subsection{Main results}
Now we state the main results of the paper. While we do get the classification for all quantizations, here we concentrate on the case when
the quantization of $\C[\Orb]$ is canonical, hence $\J$ is the unipotent ideal associated
with $\Orb$. Let $M$ be an irreducible HC $(\g,K)$-module
annihilated by $\J$.
A result of Vogan, \cite[Theorem 1.2]{Vogan}, says that the associated variety of $M$ is irreducible
hence is the closure of a single $K$-orbit in $\Orb\cap \kf^\perp$, where we write $\kf^\perp$
for the annihilator in $\g^*$ of (the image of) $\kf$. So, let $\Orb_K$ denote the open
orbit in the associated variety. We need to classify all irreducible HC modules $M$ annihilated by $\J$ with associated variety
$\overline{\Orb}_K$.

Pick a good filtration on $M$. One can show the restriction of $\gr M$ to $\Orb_K$ is a (strongly) $K$-equivariant twisted local system with half-canonical twist, see \cite[Theorem 8.7]{Vogan} for a related result, we will suitably generalize it in our setting, see Section \ref{subsec:W-algebra_nilpotent}.  We will also see that the twisted local system is irreducible and non-isomorphic irreducible
modules give rise to non-isomorphic twisted local systems.

Here is the first main result of the paper (Theorem \ref{Thm:main1}).

\begin{Thm}\label{Thm:omnibus}
Suppose that one of the following conditions hold:
\begin{enumerate}
\item $K\hookrightarrow G$,
\item $\g$ is of type $B,C$ or $D$,
\item or $\g=\slf_n$  and $\kf\neq \mathfrak{so}_n$.
\end{enumerate}
Then there is a bijection between
\begin{itemize}
\item[(i)]
the irreducible HC $(\g,K)$-modules annihilated by $\J$
with associated variety $\overline{\Orb}_K$,
\item[(ii)] And irreducible equivariant twisted (by half-canonical twist) local systems
on $\Orb_K$.
\end{itemize}
\end{Thm}

We remark that one can describe (ii) in elementary representation theoretic terms.
We will do so in the main body of the paper.

In the case when $\g=\slf_n$ and $K=\operatorname{Spin}_n$ the map from (i) to (ii)
is not surjective (it is known, see, e.g. \cite[Example 12.4]{Vogan}, that in the case when $n=3$ there are four irreducible twisted local
systems and three irreducible HC modules, this case  is of great importance for the paper).
Our second main result is a description of the image, see Theorem \ref{Thm:spin_HC_classif_general}.

We will also compute some examples in exceptional types when $K\rightarrow G$
is not injective. In these cases, it is also a general phenomenon that some of the irreducible twisted local systems do not correspond to any irreducible HC modules. We will also classify the irreducible Harish-Chandra modules over the canonical quantizations of the six rigid orbits $\Orb$ in the exceptional types such that $\operatorname{codim}_{\overline{\Orb}}\partial\Orb=2$, this is done in Appendix \ref{sec:appendix_A}.

\begin{Rem}
An easy case is when $\operatorname{codim}_{\overline{\Orb}_K}\partial\Orb_K\geqslant 3$.
Here we also have a bijection between (i) and (ii) in the theorem. A bit weaker statement was obtained
in \cite{LY}. In this paper we will explain an alternative approach based on the
previous results of the first named author. On the other hand, the results in \cite{LY} can be used to treat some cases when $\operatorname{codim}_{\overline{\Orb}}\partial\Orb  =2$ but the affinization $X = \spec(\C[\Orb])$ is smooth in codimenison 3. See  Appendix \ref{sec:appendix_A}.
\end{Rem}

\begin{Rem}
We now compare the results of this paper to classification results
from \cite{BMSZ1,BMSZ2,BMSZ3,BMSZ4}. Compared to these papers, our description is more geometric and uniform, we do not require the unipotent representations we classify to be special. Also, in our approach, conceptually
non-linear groups are not more difficult than linear ones, and exceptional groups are not more
difficult than classical ones (with a few exceptions; also technically non-linear/exceptional groups are somewhat
more complicated). On the other hand, the four papers mentioned above avoid the restrictive codimension $4$ condition
that is crucial for us. Also these papers prove the unitarity of all the special unipotent representations of real classical Lie groups. 
\end{Rem}

\begin{Rem}
	Recently the preprint \cite{DM} has proven the unitarity of the unipotent representations we classify in the current paper using the theory of mixed Hodge modules.
\end{Rem}

\subsection{Description of approach}\label{SS_intro_approach}
Recall that for the time being we assume that $\J$ is a unipotent, hence maximal, ideal.

First, let us explain a representation theoretic description of the equivariant
half-canonical twisted local systems on $\Orb_K$. Pick a point $\chi\in \Orb_K$.
Let $K_Q$ denote the reductive part of the stabilizer of $\chi$ in $K$. This is a
reductive group, in general, it is disconnected. Let $\omega_{\Orb_K}$ denote the canonical
bundle on $\Orb_K$, it is $K$-equivariant. So $K_Q$ acts on the fiber $(\omega_{\Orb_K})_\chi$
by a character. Let $\rho_{\omega}\in (\kf_Q^*)^{K_Q}$ be one half of this character.
Let $(K_Q,\rho_\omega)\operatorname{-mod}$ denote the full subcategory in the category
of rational representations of $K_Q$ consisting of all modules, where $\kf_Q$ acts via
$\rho_\omega$ (for the action of $\kf_Q$ obtained by differentiating the $K_Q$-action).
This is a semisimple category. It is equivalent to the category of local systems of interest:
an equivalence is given by taking the fiber of a twisted local system at $\chi$. Below in this section
we will identify the category of twisted local systems of interest with $(K_Q,\rho_\omega)\operatorname{-mod}$. Note that a $K_Q$-representation $V$ in $(K_Q,\rho_\omega)\operatorname{-mod}$ is called \emph{admissible} (with respect to $K$) in \cite[Definition 7.13]{Vogan} and the pair $(\chi, V)$ is called a \emph{(nilpotent) admissible $K$-orbit datum}. We also say that the twisted local system corresponding to $V$ is an \emph{admissible $K$-equivariant vector bundle} over $\Orb_K$. If there exists at least one admissible $K$-equivariant vector bundle over $\Orb_K$, we say that $\OO_K$ is \emph{$K$-admissible} or \emph{admissible for $K$} or \emph{admissible for $G_\R$}, the corresponding real group.

Set $\A:=\U/\J$, in fact, since $\J$ is maximal, the algebra $\A$ admits a filtration making it a filtered quantization of $\C[\Orb]$. Let $\HC(\A,K)$ denote the category of all HC $(\g,K)$-modules annihilated by $\J$\footnote{In the general case -- when $\A$ is an arbitrary quantization of $\C[\Orb]$ -- we only have an embedding $\U/\J\hookrightarrow \A$; still
there is a natural bijection between irreducible HC modules over $\U$ with annihilator $\J$ and certain irreducible HC modules over $\A$,
see Proposition \ref{Prop:HC_bijection}}. Let $\HC_{\overline{\Orb}_K}(\A,K)$  denote the full subcategory of $\HC(\A,K)$ consisting of all modules whose associated variety is contained in $\overline{\Orb}_K$.

The first step in our construction is to produce a fully faithful embedding
$\bullet_{\dagger,\chi}: \HC_{\overline{\Orb}_K}(\A,K)\xrightarrow{\sim} (K_Q,\rho_\omega)\operatorname{-mod}$, see
Proposition \ref{Prop:adj_properties} for a more general statement.
The construction of this functor was sketched in \cite[Section 6.1]{Wdim}
and the claim that it is a full embedding is proved as the similar claim for
Harish-Chandra bimodules, \cite[Theorem 1.3.1]{HC}. On the level of objects, the functor sends a HC module
$M$ to the twisted local system $(\gr M)|_{\Orb_K}$. This implies that if $M$ is irreducible,
then so is the corresponding twisted local system. Moreover, non-isomorphic HC modules
give rise to non-isomorphic twisted local systems. What remains is to describe
the essential image of $\bullet_{\dagger,\chi}$.

Here our inspiration comes from \cite{HC_symp}. This part is technical so we will describe
it here in a somewhat informal fashion. Any equivariant twisted local system
on $\Orb_K$ can be uniquely quantized to a sheaf of modules over the microlocalization of
$\A$ to $\Orb$. An irreducible twisted local system $\mathcal{E}$
lies in the image of $\bullet_{\dagger,\chi}$
if and only if the corresponding quantizations has nonzero global sections. To check whether
the quantization has global sections, we first push it to the locus in $\overline{\Orb}_K$
obtained by taking the union of $\Orb_K$ and all codimension $2$ orbits. The property
necessary for $\mathcal{E}\in \operatorname{Im}(\bullet_{\dagger,\chi})$ is that this
pushforward is still a {\it coherent} module, see \cite[Section 2.5]{HC_symp}. It is also close to being sufficient so for the purposes
of this section we will treat it as necessary and sufficient.

A key observation of \cite{HC_symp} (in the context of HC bimodules over quantizations of
conical symplectic singularities)  is that the condition that the pushforward of the quantization of
$\mathcal{E}$ is coherent can be tested on slices to codimension $2$ orbit. Namely, let $\Orb'_K\subset \overline{\Orb}_K$ be a codimension $2$ orbit. Let $S'$ be a slice in $\overline{\Orb}$ to a point of $\Orb'$ (this definition works as stated when $\overline{\Orb}$
is normal at the points of $\Orb'$, in general, in requires a modification that we will make in the main body of the paper).
We can consider the restriction $\mathcal{L}|_{S'}$, this is a twisted local system
on $S'\cap \Orb_K$. The claim that the pushforward of the quantization of $\mathcal{E}$ is coherent
is equivalent to the following condition holding for every codimension
$2$ orbit in $\overline{\Orb}_K$:
\begin{itemize}
\item[(*)] The pushforward of the quantization of $\mathcal{E}|_{S'}$
from $S'\cap \Orb_K$ to $S'\cap \overline{\Orb}_K$ is coherent.
\end{itemize}
All slices $S'$ are known, see \cite{KP1,KP2} for the classical groups, and \cite{FJLS1} for
the exceptional groups. In the setting of Theorem  \ref{Thm:omnibus} we check that (*) is always
satisfied (there is one orbit in type $E_8$, where the slice is too complicated to analyze (*) due to the limit of the paper, but
we can use other methods -- \texttt{atlas} -- in that case. A theoretic proof will be provided in a forthcoming paper \cite{quant_cover}). For $\g=\mathfrak{sl}_n$ and $K=\operatorname{Spin}_n$
and in the number of cases in exceptional types, (*) reduces (modulo some caveats)
to understanding the case of $\g=\mathfrak{sl}_3, K=\operatorname{SL}_2$ and the minimal orbit $\Orb$.

{\bf Acknowledgements}: We would like to thank Jeffrey Adams, Dougal Davis, Baohua Fu, Jia-Jun Ma, Lucas Mason-Brown and Binyong Sun for helpful discussions. The second named author wants to thank Binyong Sun for his hospitality and inspiring discussions during his visits to Institute for Advanced Study in Mathematics of Zhejiang University. The work of I.L. has been partially supported by the NSF under grant DMS-2001139. The work of S.Y. has been partially supported by China NSFC grants (Grant No. 12001453 and 12131018) and Natural Science Foundation of Fujian Province (Grant No. 2022J06005).




\section{Harish-Chandra modules over filtered algebras}
The goal of this section is to establish a general framework for working with Harish-Chandra modules.
\subsection{HC $(\A,\theta)$-modules}\label{SS_HC_mod}
Our setting is as follows. The base field is $\C$. Let
\begin{itemize}
\item $d$ be a positive integer,
\item $\A=\bigcup_{i\geqslant 0}\A_{\leqslant i}$ be a filtered associative
algebra,
\item $\epsilon$ be a primitive $d$th root of $1$,
\item and $\zeta$ be a filtration preserving automorphism of $\A$ with $\zeta^d=1$
\end{itemize}
satisfying:
\begin{enumerate}
\item $[\A_{\leqslant i},\A_{\leqslant j}]\subset \A_{\leqslant i+j-d}$,
\item $\gr\A$ (a commutative algebra) is  finitely generated.
\item $\zeta$ acts on the graded component $(\gr\A)_i$ by $\epsilon^i$.
\end{enumerate}

Now let $\theta$ be an anti-involution of $\A$ preserving the filtration and commuting with $\zeta$. Following \cite[Section 6.1]{Wdim},
we define a Harish-Chandra $(\A,\theta)$-module as follows.

\begin{defi}\label{defi:HC} By a Harish-Chandra $(\A,\theta,\zeta)$-module (or simply a HC $(\A,\theta)$-module) we mean a finitely generated
$\A$-module $M$ with an automorphism $\zeta$ compatible with the eponymous automorphism
of $\A$ that admits a filtration (called {\it good} for $\theta$),
$M=\bigcup_{j\geqslant 0}M_{\leqslant j}$
with the following properties:
\begin{itemize}
\item[(i)] $\gr M$ is finitely generated over $\gr\A$,
\item[(ii)] For all $a\in \A_{\leqslant i}$ and $m\in M_{\leqslant j}$,
we have $(\theta(a)-a)m\subset M_{\leqslant i+j-d}$.
\item[(iii)] $\zeta$ acts on $(\gr M)_i$ by $\epsilon^i$.
\end{itemize}
\end{defi}

HC $(\A,\theta)$-modules form a full subcategory in the category of $(\A,\Z/d\Z)$-modules, it will be denoted by
$\HC(\A,\theta,\zeta)$.

Here is a classical example.
\begin{Ex}\label{Ex:HC_classical}
Let $\g$ be a semisimple Lie algebra and let
$\sigma$ be an involution of $\g$. We write $\U$ for the universal
enveloping algebra of $\g$ equipped with the PBW filtration (and $d=1$, so $\zeta=1$). It comes with
the unique anti-involution $\theta$ extending $-\sigma:\g\rightarrow \g$.
Let $\A=\U$ and $\kf=\g^\sigma$. Then, as was pointed out in
\cite[Section 6.1]{Wdim}, a Harish-Chandra $(\U,\theta)$-module is the same
things as a Harish-Chandra $(\g,\kf)$-module, a classical object of
study in Lie representation theory.
\end{Ex}

\begin{Rem}\label{Rem:HC_bimodules}
Let $\A,\zeta$ be as above. Form the algebra $\A\otimes \A^{opp}$ and consider its anti-involution
$\theta$ given by $\theta(x\otimes y)=y\otimes x$. The HC $(\A\otimes \A^{opp},\theta,\zeta)$-modules will
be called {\it HC $\A$-bimodules}.
\end{Rem}

\subsection{Equivariance}
Now suppose that $K$ is a (not necessarily connected) reductive algebraic group that acts on
$\A$ rationally by algebra automorphisms and  with quantum comoment map $\Phi: \mathfrak{k}\rightarrow \A_{\leqslant d}$.
Suppose that this Hamiltonian action is compatible with
the anti-involution $\theta$ in the following way:
\begin{equation}\label{eq:action_compatible}
\theta(ka)=k^{-1}\theta(a), \quad \theta \Phi(\xi)=-\Phi(\xi), \quad \forall 
\, k\in K, a\in \A, \xi\in \mathfrak{k}.
\end{equation}
Further, suppose that the $K$-action commutes with $\zeta$ and $\zeta.\Phi(\xi)=\Phi(\xi)$
for all $\xi\in \kf$.

Pick $\kappa\in (\kf^*)^K$.
We can speak about $(K,\kappa)$-equivariant $\A$-modules -- those are $\A$-modules $M$
equipped with a rational $K$-action making the structure map $\A\otimes M\rightarrow M$
equivariant and whose differential is $\xi\mapsto \Phi(\xi)-\langle\kappa,\xi\rangle$.

We can also talk about $(K,\kappa)$-equivariant HC modules.

\begin{defi}\label{eq:HC_equiv}
By a  $(K,\kappa)$-equivariant (sometimes to be called twisted equivariant) HC $(\A,\theta,\zeta)$-module we mean a $(K,\kappa)$-equivariant $(\A,\theta,\zeta)$-module that admits a $K$-stable good filtration.
\end{defi}

The corresponding category will be
denoted by $\HC(\A,\theta,\zeta)^{K,\kappa}$ or simply
$\HC(\A,\theta,\zeta)^K$ if $\kappa=0$ (the morphisms in this category are $K$-equivariant $\A$-linear maps; note that the $\A$-linearity implies the $K$-equivariance if $K$ is connected). We note that if $\kappa_0$
is a character of $K$, then there is a natural category
equivalence between $\HC(\A,\theta,\zeta)^{K,\kappa}$
and $\HC(\A,\theta,\zeta)^{K,\kappa+\kappa_0}$ given by tensoring
with the one-dimensional $K$-representation, where $K$
acts via $\kappa_0$.

For example, when $\A=\U$, $\theta$ is as in Example \ref{Ex:HC_classical},
and $K=(G^\sigma)^\circ$, then a $K$-equivariant HC $(\A,\theta)$-module
is the same thing as a classical HC $(\g,K)$-module.

Now we return to the general situation in the beginning of this section.

\begin{Lem}\label{Lem:twisted_equivar}
Suppose that the derived subgroup $(K,K)$ is simply connected. Then
every simple object in $\HC(\A,\theta,\zeta)$ admits a $K$-action making it
an object of $\HC(\A,\theta,\zeta)^{K,\kappa}$ for some $\kappa$.
\end{Lem}
\begin{proof}
Let $M\in \operatorname{Irr}(\HC(\A,\theta,\zeta))$, where $\operatorname{Irr}(-)$ stands for the set of isomorphism classes of irreducible objects in an abelian category in question. By the definition of a HC module,
every vector in $M$ lies in a finite dimensional $\A_{\leqslant d}$-stable
subspace.  The Lie algebra $\kf$
acts on $M$ locally finitely because $\Phi(\kf)\subset \A_{\leqslant d}$. We can decompose $M$ into the
sum $\bigoplus M_{\mu}$, where $\mu$ runs over the set of characters of $\kf$
and $M_{\mu}$ stands for the generalized $\mu$-eigenspace for the center $\mathfrak{z}(\kf)$.
Since $M$ is irreducible it is generated by a highest vector for $\kf$ (and an eigenvector for $\zeta$), say $v$. In particular,
the action of $\mathfrak{z}(\kf)$  on $M$ is diagonalizable. We can shift the
action of $\kf$ by a character, say $\kappa$, of $\kf$ and assume that the weight of
$v$ is a weight for $K$. This gives rise to a $K$-action on $M$.
This action turns $M$ into  a twisted equivariant object in
$\HC(\A,\theta,\zeta)$.
\end{proof}

\subsection{Weakly HC modules}\label{SS_weakly_HC}
Let $\Dcal_\hbar$ be an associative $\C[\hbar]$-algebra with unit. We assume that
$\Dcal_\hbar$ is flat over $\C[\hbar]$ and that $\Dcal_0 = \Dcal_\hbar/ \hbar \Dcal_\hbar$ is a finitely
generated commutative $\C$-algebra.
Let $\theta$ be a $\C[\hbar]$-linear anti-involution on $\Dcal_\hbar$.

Suppose $\A,\zeta,\theta$ are as in Section \ref{SS_HC_mod}.
Define the (modified) Rees algebra for $(\A,\zeta)$ as follows.
For $j\in \Z/d\Z$, let $\A_j$ denote the eigenspace for $\zeta$
with eigenvalue $\epsilon^j$. Then set $R_\hbar(\A):=\bigoplus_{i=0}^\infty
(\A_{\leqslant i}\cap \A_{i+d\Z})\hbar^{i/d}$. This is a graded $\C[\hbar]$-algebra
with $\hbar$ in degree $d$. We can take $\Dcal_\hbar$ to be either $R_\hbar(\A)$ or its
$\hbar$-adic completion. In both cases the anti-involution $\theta$
of $\A$ gives rise to a $\C[\hbar]$-linear anti-involution of $\Dcal_\hbar$
also denoted by $\theta$.


We write $\Dcal_\hbar^{-\theta}$ for the $-1$ eigenspace of $\theta$.
	We have $\Dcal_\hbar=\Dcal_\hbar^{-\theta}\oplus \Dcal_\hbar^\theta$.
	Note that the following inclusions hold:
	\begin{equation}\label{eq:invol_decomp_properties}
	\begin{split}
	&[\Dcal_\hbar^{-\theta},\Dcal_\hbar^{-\theta}]\subset \Dcal_\hbar^{-\theta}, \\
	&a\in \Dcal_\hbar^{-\theta}, b\in \Dcal_\hbar^\theta\Rightarrow ab+ba\in \Dcal_\hbar^{-\theta}.
	\end{split}
	\end{equation}

Set $\J_\hbar:=\hbar \mathcal{D}_\hbar+ \Dcal_\hbar\Dcal_\hbar^{-\theta}$. This is a two-sided ideal of $\Dcal_\hbar$ and is a Lie algebra under taking the bracket defined by
$[a,b]_\hbar:=\frac{1}{\hbar}(ab-ba)$.

\begin{defi}\label{defi:wHC}
	A {\it weakly HC $(\Dcal_\hbar, \theta)$-module}, is a finitely generated
	$\Dcal_\hbar$-module $M_\hbar$ equipped with
	$\C[\hbar]$-bilinear Lie algebra representation
	$\J_\hbar \times M_\hbar\rightarrow M_\hbar, (a,m)\mapsto a.m$, such that
	\begin{enumerate}[(i)]
		\item $am=\hbar(a.m)$, $\forall\, a \in \J_\hbar$.
		\item $(ba).m = b(a.m)$, $\forall\, a \in \J_\hbar, b \in \Dcal_\hbar$.
		\item $a.(bm)=[a,b]_\hbar m+b(a.m), \forall a\in \J_\hbar, b\in \Dcal_\hbar$.
	\end{enumerate}
A weakly HC $(\Dcal_\hbar,\theta)$-module is called HC if it is flat over $\C[\hbar]$.
\end{defi}

\begin{Rem}\label{Rem:flat_wHC}
Note that for a $\C[\hbar]$-flat $\Dcal_\hbar$-module being HC is not a structure but a condition:
$bm\in \hbar M_\hbar$ for all $b\in \Dcal_{\hbar}^{-\theta},m\in M_\hbar$. Then we can define
$b.m$ as $\hbar^{-1}(bm)$.
\end{Rem}

The category of weakly HC $(\Dcal_\hbar, \theta)$-modules will be denoted by $\wHC(\Dcal_\hbar,\theta)$.

Now suppose that we have a $\C^\times$-action on $\Dcal_\hbar$ by $\C$-algebra automorphisms
such that $t.\hbar=t^d\hbar$ for $t \in \C^\times$. Further suppose that 
\begin{itemize}\item either the action is rational or
\item $\Dcal_\hbar$
is complete and separated in the $\hbar$-adic topology and the action is pro-rational.
\end{itemize}
The Rees algebra $R_\hbar(\A)$ satisfies the first condition, while its $\hbar$-adic completion
satisfies the second condition.

\begin{defi}\label{defi:wHC_graded}
A weakly HC $(\Dcal_\hbar,\theta)$-module $M_\hbar$ is {\it graded} if it is equipped with
a (rational/pro-rational) $\C^\times$-action (rational, if the action on $\Dcal_\hbar$ is rational
and pro-rational, if the action on $\Dcal_\hbar$ is pro-rational) such that
the multiplication map $\Dcal_\hbar\otimes M_\hbar\rightarrow M_\hbar$ is
$\C^\times$-equivariant and the Lie algebra action $\J_\hbar
\otimes M_\hbar\rightarrow M_\hbar$ is of degree $-d$.
\end{defi}

The abelian category of graded weakly HC modules will be denoted by $\wHC^{gr}(\Dcal_\hbar,\theta)$.

Now suppose $\Dcal_\hbar=R_\hbar(\A)$ for a filtered algebra $\A$ as in the beginning of this
section.
A connection between $\wHC^{gr}(\Dcal_\hbar,\theta)$ and $\HC(\A,\theta,\zeta)$ is as follows.
Note that we have the functor $\wHC^{gr}(\Dcal_\hbar,\theta)\rightarrow \HC(\A,\theta,\zeta)$
given by $M_\hbar\mapsto M_\hbar/(\hbar-1)M_\hbar$. It is straightforward to check that $M_\hbar/(\hbar-1) M_\hbar$
is indeed a HC $(\A,\theta,\zeta)$-module: a good filtration on $M_\hbar/(\hbar-1) M_\hbar$
comes from the grading on $M_\hbar$ and the automorphism $\zeta$ is induced by $\epsilon\in \C^\times$ acting via the $\cm$-action on $M_\hbar$.

\begin{Lem}\label{Lem:essential_surjectivity}
The functor $\wHC^{gr}(\Dcal_\hbar,\theta)\rightarrow \HC(\A,\theta,\zeta)$, $M_\hbar\mapsto M_\hbar/(\hbar-1)M_\hbar$, is essentially surjective
and it kills precisely the $\hbar$-torsion modules.
\end{Lem}
\begin{proof}
The description of the kernel is clear.
Let $M \in \HC(\A,\theta,\zeta)$. Pick a good filtration on $M$ as in Definition
\ref{defi:HC}. We can form the (modified) Rees module $R_\hbar(M)$
similarly to what was done for $\A$. By Remark \ref{Rem:flat_wHC}, $R_\hbar(M)$ is an object of
$\wHC^{gr}(\Dcal_\hbar,\theta)$. Of course, $R_\hbar(M)/(\hbar-1)\cong M$.
\end{proof}

We also have a $(K,\kappa)$-equivariant version of weakly HC modules. Let $K$ be an algebraic group which acts rationally or pro-rationally on $\Dcal_\hbar$ by $\C[\hbar]$-linear algebra automorphisms, so that any element $\xi\in \kf$ gives rise to a $\C[\hbar]$-linear endomorphism $\xi_{\Dcal_\hbar}$ of $\Dcal_\hbar$.
Let $\Phi: \kf\rightarrow \Dcal_\hbar$ be a quantum comoment map
meaning that this map is $K$-equivariant and $\hbar^{-1}[\Phi(\xi),\cdot]=\xi_{\Dcal_\hbar}$.
Assume additionally that $\theta$ is $K$-equivariant and $\operatorname{Im} \Phi \subset \J\hbar$.

\begin{defi}\label{defi_wHC_equiv}
Let $M_\hbar\in \wHC(\Dcal_\hbar,\theta)$.
Let $K$ act rationally or pro-rationally on $M_\hbar$ by $\C[\hbar]$-linear automorphisms (so that every element
$\xi\in \kf$ gives rise to a $\C[\hbar]$-linear endomorphism $\xi_{M_\hbar}$ of $M_\hbar$).  We say that
$M_\hbar$ is $(K,\kappa)$-equivariant if
\begin{itemize}
  	\item the multiplication map $\Dcal_\hbar\times M_\hbar\rightarrow M_\hbar$
  and the Lie algebra action map $\J\hbar \times M_\hbar\rightarrow M_\hbar$
  are $K$-equivariant,
  \item $\Phi(\xi).m-\langle\kappa,\xi\rangle=\xi_{M_\hbar}m, \forall \xi\in \kf, \, m \in M_\hbar$.
\end{itemize}
\end{defi}

We denote the corresponding category by $\wHC(\Dcal_\hbar,\theta)^{K,\kappa}$.
Similarly, we can talk about $(K,\kappa)$-equivariant objects in $\wHC^{gr}(\Dcal_\hbar,\theta)$.
Here we assume that the $K$-action commutes with the $\C^\times$-action and that
$\operatorname{Im}\Phi$ lies in the subspace of degree $d$ homogeneous elements with respect to the $\C^\times$-action.

\subsection{Weakly HC modules over formal Weyl algebras}
Here we consider an important special case of weakly HC modules.
Let $V$ be a symplectic vector space. Consider the homogenized Weyl algebra
$\Weyl_\hbar(V)=T(V)[\hbar]/(u\otimes v-v\otimes u-\hbar\omega(u,v))$.

Let $\mathfrak{m}\subset \Weyl_\hbar(V)$ denote the preimage of the maximal ideal of $0$ in $S(V)$
under the epimorphism $\Weyl_\hbar(V)\twoheadrightarrow S(V)$. Let $\Dcal_\hbar$ denote the
$\mathfrak{m}$-adic completion of $\Weyl_\hbar(V)$.

Fix an anti-symplectic involution $\theta$ of $V$. It defines the decomposition
$V=L\oplus L'$, where $L,L'$ are Lagrangian. Namely, $L=V^{-\theta}$ and $L'=V^{\theta}$.
We note that $\C[L][\hbar]$ is a HC $(\Weyl_\hbar, \theta)$-module. Similarly,
$\C[[L,\hbar]]$ is a HC $\Dcal_\hbar$-module.

For a weakly HC $(\Dcal_\hbar,\theta)$-module $M$ we write $M^L$
for the annihilator of $L$ in $M$ (for the Lie algebra action).

\begin{Lem}\label{Lem:wHC_formal_Weyl}
The functors $\bullet^L$ and $\C[[L,\hbar]]\otimes_{\C[[\hbar]]}\bullet$
are mutually quasi-inverse equivalences between $\wHC(\Dcal_\hbar,\theta)$
and $\C[[\hbar]]\operatorname{-mod}$.
\end{Lem}
\begin{proof}
It is easy to see that $(\C[[L,\hbar]]\otimes_{\C[[\hbar]]}\bullet)^L$
is the identity endo-functor of $\C[[\hbar]]\operatorname{-mod}$.

To prove that the other composition is isomorphic to the identity one argues as follows.
Note that $M$ is complete in the $\mathfrak{m}$-adic topology.
Let $p_1,\ldots,p_n$ be a basis in $L'$ and $q_1,\ldots,q_n$ be the
dual basis in $L$. For $m\in M$, we have $q_i.(p_jm)=m+p_j(q_i.m)$.
Using this identity, we finish the proof of this lemma similarly to the analogous claim in
the proof of \cite[Proposition 3.3.1]{HC}.
\end{proof}

\section{Harish-Chandra modules over quantizations of nilpotent orbits}
\subsection{Quantizations of nilpotent orbits and their HC modules}\label{SS_quant_nilp}
Let $\Orb\subset \g^*$ be a nilpotent orbit. This is a symplectic algebraic variety.
The group $G$ naturally acts on $\Orb$ and the action is Hamiltonian.

The algebra
$\C[\Orb]$ of regular functions on $\Orb$ is a graded Poisson algebra
that comes with a Hamiltonian action of $G$, the corresponding moment
map $\operatorname{Spec}(\C[\Orb])\rightarrow \g^*$ is induced by the inclusion
$\Orb\hookrightarrow \g^*$. The bracket on
$\C[\Orb]$ has degree $-1$. 

\begin{defi}\label{defi:quantization}
Let $A$ be a $\Z_{\geqslant 0}$-graded Poisson algebra with bracket of degree $-d$.
By a {\it filtered quantization} of $A$ we mean a pair $(\Dcal,\iota)$ of
\begin{itemize}
\item
a filtered
associative algebra $\Dcal$ such that $\deg [\cdot,\cdot]\leqslant -d$
\item and a graded Poisson algebra isomorphism $\iota:\gr \Dcal\xrightarrow{\sim} A$.
\end{itemize}
\end{defi}

In what follows in the present section we will place an additional restriction on
$\Orb$. Set $\partial \Orb:=\overline{\Orb}\setminus \Orb$.
We will require that
\begin{equation}\label{eq:codim_condition}
\operatorname{codim}_{\overline{\Orb}}\partial\Orb\geqslant 4.
\end{equation}

The following classical lemma describes the second cohomology $H^2(\Orb,\C)$. Let $e\in \Orb$
and $Q$ be the reductive part of the centralizer of $e$ in $G$.

\begin{Lem}\label{Lem:G_orbit_cohomology}
We have an isomorphism $(\q^*)^Q\xrightarrow{\sim} H^2(\Orb,\C)$ that sends a character of $Q$
to the 1st Chern class of the corresponding line bundle on $\Orb$.
\end{Lem}

To fix the normalization, to a character $\chi:Q\rightarrow \C^\times$ (equivalently,
a character of $Z_G(e)$) we assign the $G$-equivariant line bundle $\mathcal{L}$ on $\Orb$
such that $Z_G(e)$ acts on the fiber of $\mathcal{L}$ at $e\in \Orb$ by $\chi$.

The variety $\operatorname{Spec}(\C[\Orb])$ is singular symplectic, \cite{Beauville}.
In what follows we take a singular symplectic variety $X$ with $\operatorname{codim}_X X^{sing}\geqslant 4$. We also assume that $X$ is conical: meaning that $\C[X]$ admits a $\Z_{\geqslant 0}$-algebra grading
with $\C[X]_0=\{0\}$ and the degree of the Poisson bracket is $-d$ for $d\in \Z_{>0}$.

Filtered quantizations of the algebra $\C[X]$ were classified in \cite[Theorem 3.4]{orbit}
in the general setting. Under the assumption $\operatorname{codim}_X X^{sing}\geqslant 4$, they are parameterized by the points of $\param:=H^2(X^{reg},\C)$ via taking the so called {\it period} of the microlocalization of
the quantization to $X^{reg}$, \cite[Corollary 3.2(1)]{orbit}. We will discuss periods in more detail in
Section \ref{subsec:quan}. Moreover, we can take the universal
quantization $\A_{\param}$ of $\C[X]$, \cite[Corollary 3.2(2)]{orbit}.
This is a filtered $\C[\param]$-algebra with
$\deg \param^*=d$ such that for $\lambda\in \param$, the specialization of $\A_{\param}$
to $\lambda$ is the filtered quantization $\A_\lambda$ of $\C[\Orb]$ corresponding to $\lambda$.
Note that $\A_{\param}^{opp}$ is identified with $\A_{\param}$, where the identification is
$-1$ on $\param^*$. This can be deduced from \cite[Proposition 3.2]{BPW}.
The quantization $\A_\lambda$ of $\C[X]$ corresponding to $\lambda\in \param$
is obtained from $\A_{\param}$ by specializing the parameter to $\lambda$.

Fix a Poisson anti-involution $\theta$ of
$\C[X]$ that preserves the grading. Thanks to the universality property
of $\A_{\param}$ (\cite[Proposition 3.5]{orbit}), $\theta$ lifts to an anti-involution
of $\A_{\param}$ also denoted by $\theta$, compare to the beginning of
\cite[Section 3.7]{orbit}. The action of $\theta$ on $\param^*\subset \A_\param$
comes from the action on $\param=H^2(X^{reg},\C)$.

In particular, we can talk about the category $\HC(\A_{\param},\theta,\zeta)$
of HC $(\A_{\param},\theta)$-modules (see Section \ref{SS_HC_mod}). For $\lambda\in \param$ we write $\HC(\A_\lambda,\theta,\zeta)$
for the full subcategory of $\HC(\A_{\param},\theta,\zeta)$ consisting of all modules, where
the action of $\A_{\param}$ factors through $\A_\lambda$. If $\lambda$ is fixed by $\theta$,
then $\theta$ induces an anti-involution of $\A_\lambda$ (again denoted by $\theta$), then
$\HC(\A_\lambda,\theta,\zeta)$ is indeed the category of HC $(\A_\lambda,\theta)$-modules.

For any $M\in \HC(\A_\lambda,\theta,\zeta)$ we can consider its associated variety, to be denoted
by $\operatorname{V}(M)$. By definition, it is the support of the associated graded of
$M$ with respect to any good filtration.

Recall that $X$ has finitely many symplectic leaves, \cite{Kaledin_symplectic}. Consider the fixed point locus $X^\theta$ that we view as a subvariety in $X$.  

\begin{Lem}\label{Lem:lag_antiinvolution}
The following are true:
\begin{enumerate}
\item For any symplectic leaf $\mathcal{L}\subset X$, the intersection $X^\theta\cap \mathcal{L}$
is Lagrangian in $\mathcal{L}$. In particular, $X^\theta$ is an isotropic subvariety in the sense of
\cite{B_ineq}. 
\item For any $M\in \HC(\A_\lambda,\theta,\zeta)$, we have $\operatorname{V}(M)\subset X^\theta$. In particular,  $M$ is holonomic in the sense of \cite{B_ineq}.
\end{enumerate}
\end{Lem}
\begin{proof}
It is a classical and easy fact that the fixed point locus of an anti-symplectic involution of a
smooth symplectic variety is a Lagrangian subvariety. If $X^\theta\cap \mathcal{L}\neq \varnothing$,
then $\theta$ preserves $\mathcal{L}$ and induces an anti-symplectic involution of $\mathcal{L}$.
So $X^\theta\cap \mathcal{L}$ is Lagrangian, in particular, isotropic. So $X^\theta$ is isotropic
in the sense of \cite{B_ineq}. This proves (1).

The claim that $\operatorname{V}(M)\subset X^\theta$ is an easy consequence of the definition of
a good filtration on a module. Together with (1), this proves (2).
%
\end{proof}

\begin{defi}\label{defi:HC_full_support}
We consider the category of HC $\A_\lambda$-modules
{\it with full support},  $\overline{\HC}(\A_\lambda,\theta,\zeta)$, defined as the quotient of $\HC(\A_\lambda,\theta,\zeta)$ by the full
subcategory of all modules $M$ such that $\operatorname{V}(M)\cap X^{reg}=\varnothing$.
\end{defi}

We now get back to the situation when $X=\operatorname{Spec}(\C[\Orb])$, where $\Orb$
satisfies (\ref{eq:codim_condition}).
Note that the action of $G$ on $\A_{\param}$ is Hamiltonian, this is similar to Step 1 of the proof
of \cite[Proposition 4.5]{orbit} as well as the discussion in the beginning of
\cite[Section 5.3]{orbit}. Let $\Phi:\g\rightarrow \A_{\param}$
denote the quantum comoment map. We also write $\Phi$ for the induced homomorphism $\U\rightarrow
\A_{\param}$.

Let $\sigma$ be an involution of $\g$ and $\kf:=\g^\sigma$. Let $\kf^\perp$
denote the annihilator of $\kf$ in $\g^*$. Suppose $\Orb\cap\kf^\perp\neq \varnothing$. 
It follows that $\Orb$ is stable under $\sigma$ and hence under $\theta=-\sigma$ as well.
So $\theta$ induces an anti-Poisson involution of $\C[\Orb]$.
By the above discussion in this section, we can talk about
HC $(\A_{\param},\theta)$-modules and, for $\lambda\in \param$, about HC $(\A_\lambda,\theta)$-modules.

Note that $X^\theta$ is the preimage of $\kf^\perp$ in $X$.
In particular, thanks to Lemma \ref{Lem:lag_antiinvolution},
for any $K$-orbit $\Orb_K\subset \Orb\cap \kf^\perp$, we have

\begin{equation}\label{eq:codim_condition_K}
\operatorname{codim}_{\overline{\Orb}_K}\partial\Orb_K\geqslant 2,
\end{equation}
where we write $\partial\Orb_K$ for $\overline{\Orb}_K\setminus \Orb_K$.

\subsection{HC $\A_{\param}$- and $\U$-modules}
For $\lambda\in \param$, let $\J_\lambda$ denote the kernel of $\U\rightarrow \A_\lambda$. The goal of this section is to prove the following
proposition.

\begin{Prop}\label{Prop:HC_bijection}
There is a natural bijection between
\begin{itemize}
\item irreducible $(K,\kappa)$-equivariant HC $(\A_\lambda,\theta)$-modules of GK dimension
$\frac{1}{2}\dim \Orb$,
\item and irreducible $(K,\kappa)$-equivariant $\U/\J_\lambda$-modules
of GK dimension $\frac{1}{2}\dim \Orb$.
\end{itemize}
\end{Prop}

The proposition follows from the next two lemmas.

\begin{Lem}\label{Lem:HC_bijection1}
There is a natural bijection between
\begin{itemize}
\item irreducible  $(K,\kappa)$-equivariant $\A_\lambda$-modules of GK dimension
$\frac{1}{2}\dim \Orb$,
\item and irreducible  $(K,\kappa)$-equivariant $\U/\J_\lambda$-modules
of GK dimension $\frac{1}{2}\dim \Orb$.
\end{itemize}
\end{Lem}

\begin{Lem}\label{Lem:HC_bijection2}
Every irreducible $(K,\kappa)$-equivariant $\A_\lambda$-module of GK dimension $\frac{1}{2}\dim \Orb$
admits a good filtration (with respect to $\theta$).
\end{Lem}

\begin{proof}[Proof of Lemma \ref{Lem:HC_bijection1}]
The proof is in several steps.

{\it Step 1}.
Take a simple object  $M\in\A_\lambda\operatorname{-mod}^{K,\kappa}$.
It is a finitely generated $(K,\kappa)$-equivariant $\U/\J_\lambda$-module
because $\A_\lambda$ is a finitely generated $\U$-module.
Note that $\J_\lambda$
intersects the center of $\U$ at a maximal ideal, see the discussion in
the beginning of \cite[Section 5.3]{orbit}. Let $\mathcal{I}$ denote
the (two-sided) ideal in $\U$ generated by this intersection. So $\U/\mathcal{I}$
is a filtered quantization of the nilpotent cone  $\mathcal{N}\subset \g^*$.

{\it Step 2}.
Every finitely generated $(K,\kappa)$-equivariant $\U/\mathcal{I}$-module, $M'$, is classically known to have finite length (\cite[Corollary 3.4.7 and Theorem 4.2.1]{Wallach}). Note that $M'$ is supported on
$\kf^\perp\cap \mathcal{N}$. The intersection of $\kf^\perp$ with every $G$-orbit in $\mathcal{N}$ is a Lagrangian subvariety.
Also, by \cite[Theorem 9.11]{KL}, we have $\operatorname{dim} V(M')\geqslant \frac{1}{2}\operatorname{dim} (\U/\operatorname{Ann}_{\U}M')$.
So the following conditions are equivalent:
\begin{itemize}
\item $\operatorname{Ann}_\U(M')\supsetneq \mathcal{J}_\lambda$,
\item and $V(M')\cap \Orb=\varnothing$.
\item $\dim V(M')<\frac{1}{2}\dim \Orb$.
\end{itemize}


{\it Step 3}.
Let $M^0$ be an irreducible $\U$-submodule of $M$. Then $\A_\lambda M^0$
is a nonzero $\A_\lambda$-submodule hence $\A_\lambda M^0=M$.
Recall that $\A_\lambda$
is finite over $\U/\J_\lambda$. It follows that
$$\operatorname{GK-dim}M^0=\operatorname{GK-dim}M,$$
where we write $\operatorname{GK-dim}$ for the Gelfand-Kirillov dimension of $\U$-modules.

Also note that the multiplicity of $\A_\lambda$ on $\Orb$ is $1$, while the multiplicity of
$\U/\J_\lambda$ is at least $1$.  Hence $\operatorname{GK-dim}(\A_\lambda/ (\U/\J_\lambda))<\dim \Orb$.
In particular, $\A_\lambda/(\U/\J_\lambda)$ is annihilated by an ideal $\J$
that strictly contains $\J_\lambda$. Since $M$ and hence $M/M^0$ are holonomic $\U/\mathcal{I}$-modules, it follows that
$$\operatorname{GK-dim}(M/M^0)<\frac{1}{2}\dim \Orb.$$

{\it Step 4}.
So if $M$ has GK dimension $\frac{1}{2}\dim \Orb$, then $M^0$ is the unique irreducible $\U/\J_\lambda$-submodule of $M$. This gives rise to a
map between the sets of irreducibles with GK dimension $\frac{1}{2}\dim \Orb$:
 $$\operatorname{Irr}_{\frac{1}{2}\dim \Orb}(\A_\lambda\operatorname{-mod}^{K,\kappa})\rightarrow \operatorname{Irr}_{\frac{1}{2}\dim \Orb}(\U/\J_\lambda\operatorname{-mod}^{K,\kappa}), \quad M\mapsto M^0.$$

Let us prove that this map is a bijection. Take $N\in \operatorname{Irr}_{\frac{1}{2}\dim \Orb}(\U/\J_\lambda\operatorname{-mod}^{K,\kappa})$.
Consider the $\A_\lambda$-module $\A_\lambda\otimes_\U N$.
It is $(K,\kappa)$-equivariant because the action of $K$ on $\A_\lambda$ is Hamiltonian.
We have the map
$N\rightarrow \A_\lambda\otimes_\U N$ induced by $\U/\J_\lambda\hookrightarrow \A_\lambda$.
Similarly to Step 3 in the proof, the cokernel of this map is annihilated by an ideal strictly containing $\J_\lambda$
and is holonomic. It follows that the map $N\rightarrow \A_\lambda\otimes_\U N$ is nonzero, hence an embedding. As noted above in  Step 2, this means that the cokernel has GK
dimension less than $\frac{1}{2}\dim \Orb$.

Let $N'$ be the quotient of $\A_\lambda\otimes_\U N$
by the maximal $\A_\lambda$-submodule with GK dimension $<\frac{1}{2}\dim \Orb$.
The embedding $N\hookrightarrow \A_\lambda\otimes_\U N$ gives rise
to a $\U/\J_\lambda$-module embedding $N\rightarrow N'$.
Therefore $N'=\A_\lambda N$. If an $\A_\lambda$-submodule of $N'$ intersects trivially with $N$ then it has
GK-dimension $<\frac{1}{2}\dim \Orb$. But $N'$ has no such submodules. So
every $\A_\lambda$-submodule of $N'$ must contain $N$. Therefore $N'$ is simple.

It is easy to see that the maps $M\mapsto M^0$ and $N\mapsto N'$ are mutually inverse.
\end{proof}

\begin{proof}[Proof of Lemma \ref{Lem:HC_bijection2}]
Let $M$ be an irreducible $(K,\kappa)$-equivariant $\A_\lambda$-module.
Let $M=\bigcup_{j}M_{\leqslant j}$ be a filtration such that $\gr M$
is a finitely generated $\C[\Orb]$-module. We claim that this filtration
is good (for $\theta$) if and only if
\begin{itemize}
\item[(*)] $a$ acts by zero on $\gr M$ for every $a\in \C[\Orb]^{-\theta}$,
the $-1$ eigenspace for the action of $\theta$ on $\C[\Orb]$.
\end{itemize}
It is clear that every good filtration for $\theta$ has the desired property.
On the other hand, $\param^*\subset \gr\A_{\param}$ acts by
$0$ on $\gr M$ and the image of $(\gr \A_{\param})^{-\theta}$
in $\C[\Orb]=(\gr \A_{\param})/(\param^*)$ coincides with
$\C[\Orb]^{-\theta}$. So (*) implies the filtration is good
for $\theta$ (recall that in this case $d=1$).

Let $\A_{\lambda,\hbar}$, $\U_\hbar$ denote the Rees algebras of $\A_{\lambda}$, $\U$
and let $\A_{\lambda,\hbar}^{loc}$, $\U_\hbar^{loc}$ denote their microlocalizations
to $\g^*$.  Note that the natural homomorphism $\U_\hbar^{loc}\rightarrow \A_{\lambda,\hbar}^{loc}$
is an epimorphism over $\Orb$. Pick a $K$-stable filtration on $M$ such that
$\gr M$ is finitely generated over $\C[\kf^\perp]$. Note that $\gr M$ is set
theoretically supported on  $\kf^\perp\cap \overline{\Orb}$ and the
restriction to $\Orb$ is, in fact, scheme theoretically supported on
the Lagrangian subvariety $\Orb\cap \kf^\perp$.
The ideals $\langle \C[\Orb]^{-\theta} \rangle$ generated by $\C[\Orb]^{-\theta}$ and the ideal $\C[\Orb]\kf$ in $\C[\Orb]$ have the same localizations
to $\Orb$, as the localizations to $\Orb$ of the quotients by these ideals are reduced. It follows that $\C[\Orb]^{-\theta}$ acts by
$0$ on $(\gr M)|_{\Orb}$. Since $\operatorname{codim}_{\overline{\Orb}}\partial \Orb\geqslant 4$,
we have
\begin{equation}\label{eq:codim_K}
\operatorname{codim}_{\overline{\Orb}\cap \kf^\perp} (\partial\Orb\cap \kf^\perp)\geqslant 2.
\end{equation}
Note that $(\gr M)|_{\Orb}$ a $K$-equivariant coherent sheaf on $\Orb\cap \kf^{\perp}$ hence a vector bundle. Therefore, thanks to (\ref{eq:codim_K})
$\Gamma((\gr M)|_{\Orb})$ is finitely generated over $\C[\overline{\Orb}\cap \kf^\perp]$.
It follows that $\Gamma((\gr M)|_{\Orb})$ is finitely generated over $\C[\Orb]$.
Let $M_\hbar^{loc}$ denote the microlocalization of $M_\hbar$ to $\g^*$. It makes to sense
to restrict $M_\hbar^{loc}$ to $\Orb$.
The global sections $\Gamma(M_\hbar^{loc}|_{\Orb})$
is finitely generated over the $\hbar$-adic completion of $\A_{\lambda,\hbar}$.
Let $\tilde{M}_\hbar:=\Gamma(M_\hbar^{loc}|_{\Orb})^{fin}$ denote the subspace of $\C^\times$-finite
elements and let $\tilde{M}$ be the specialization of this module to $\hbar=1$.
The latter is a filtered module such that $\gr \tilde{M}$ admits a natural embedding into
$\Gamma((\gr M)|_{\Orb})$. Therefore, $\C[\Orb]^{-\theta}$ acts by $0$
on $\gr\tilde{M}$.

We have a natural filtered homomorphism $M\rightarrow \tilde{M}$. The composition of the associated graded of this homomorphism with the embedding
$\gr\tilde{M}\hookrightarrow \Gamma((\gr M)|_{\Orb})$ is the natural
map $\gr M\rightarrow \Gamma((\gr M)|_{\Orb})$. It follows that $M\rightarrow \tilde{M}$
is an inclusion such that $\tilde{M}/M$ is supported on $\partial\Orb$. In particular,
$M$ inherits a good filtration from $\tilde{M}$. We denote the associated graded module of $M$ with respect to this new filtration by $\gr' M$. By the construction,
$\gr' M\hookrightarrow \gr\tilde{M}$. In particular,
$\C[\Orb]^{-\theta}$ acts by $0$ on $\gr' M$. We have constructed the required
filtration on $M$ finishing the proof of the lemma.
\end{proof}

\section{Quasi-classical limits and quantizations}\label{S_qclass_quant}

\subsection{Picard Lie algebroids and  twisted local systems }

\subsubsection{Picard algebroids and TDOs} \label{subsec:pic_defn}

Let $X$ be a smooth variety/scheme and $\T_X$ be its tangent sheaf. We first review the notion of Picard algebroid, its basic properties and relations with twisted differential operators and twisted $\mathscr{D}$-modules. We will see in Section \ref{subsec:RH} that modules over Picard algebroids are equivalent to twisted local systems. A standard reference is \cite[Section 2]{BB}. For the defintion and basic properties of general Lie algebroids, see \cite[Section 1.2]{BB}.

\begin{defi} \label{defn:pic}
	A {\it Picard Lie algebroid}, or simply a {\it Picard algebroid} on $X$ is a Lie algebroid $\TT$ with the ($\OO_X$-linear) anchor map $\eta: \TT \to\T_X$ and a central section $1_{\TT}$ of $\ker \eta$, such that $\ker \eta = \OO_X \cdot 1_{\TT}$, i.e., the sequence of sheaves of $\OO_X$-modules,
	\[  0 \to \OO_X \xrightarrow{\iota_\TT} \TT \xrightarrow{\eta} \T_X \to 0\]
	is exact.
\end{defi}

A morphism of Picard algebroids is a morphism of Lie algebroids that preserves the $1_{\TT}$'s. It follows that any morphism of Picard algebroids is in fact an isomorphism, i.e., the category of all Picard algebroids over $X$ forms a groupoid. We write $\pic(X)$ for the set of isomorphsim classes of Picard algebroids over $X$. The Baer sum construction defines  a natural structure of complex vector space on $\pic(X)$. The zero vector  is given by the trivial Picard algebroid $\TT_{\operatorname{triv}} = \OO_X \oplus \T_X$.

For any line bundle $L$ over $X$, we write $\TT(L)$ for the associated Picard algebroid, i.e., the sheaf  over $X$ of local algebraic differential operators  with coefficients in $L$ of order less or equal than $1$. Let $\operatorname{Pic}(X)$ be the Picard group of $X$, then $L \mapsto \TT(L)$ defines a homomorphism of $\TT(\cdot): \operatorname{Pic}(X) \to \pic(X)$ of abelian groups which transfrom tensor product to the Baer sum. We abuse the notation and write $L^r := r \TT(L)$ for $r \in \C$, so that when $r$ is an integer, $L^r$ stand for both the $r$-tensor power of $L$ as a line bundle and its associated Picard algebroid.

Now we recall the classification of Picard algebroids. Let $\TT$ be a Picard algebroid over $X$. Let $H^\bullet_F(X)$ denote the hypercohomology of the first piece $F^1\Omega_X^\bullet$ of the de Rham complex $\Omega_X^\bullet$ with respect to the {\it filtration bête}. In other words,
\[ F^1\Omega_X^\bullet := (0 \to \Omega_X^1 \to \Omega_X^2 \to \cdots) \]
is the algebraic de Rham complex of $X$ with $\Omega_X^0 = \OO_X$ replaced by $0$ (in degree $0$). The natural inclusion $F^1\Omega_X^\bullet \hookrightarrow \Omega_X^\bullet$ of complexes induces a natural long exact sequnce of hypercohomology
\begin{equation}\label{exsq:DR}
	\cdots \to H_F^\bullet(X) \to H^\bullet_{DR}(X) \to H^\bullet(X, \OO_X) \to H_F^{\bullet + 1}(X) \to \cdots.
\end{equation}
One can naturally associate a cohomology class $c_1(\TT) \in H^2_F(X)$ to $\TT$.  
The map $c_1(\cdot): \pic(X) \to H^2_F(X)$ is an isomorphism of vector spaces (\cite[Lemma 2.1.6]{BB}). In particular, we have $c_1(\TT_{\operatorname{triv}}) = 0 \in H^2_F(X)$. The composition of $c_1: \pic(X) \xrightarrow{\sim} H^2_F(X)$ with the natural map $H^2_F(X) \to H^2_{DR}(X)=H^2(X,\C)$ is the {\it topological first Chern class} map, denoted by $c^{\operatorname{top}}_1 : \pic(X) \to H^2(X,\C)$. For any algebraic line bundle $L$, the class $c^{\operatorname{top}}_1(\TT(L))$ is just the usual topological first Chern class of $L$.

Now we discuss a connection of Picard algebroids with sheaves of twisted differential operators.
Given a Picard algebroid $\TT$,
consider the universal enveloping sheaf $\mathscr{U}(\TT)$ of algebras.
Let $\mathscr{D}_\TT$ denote the quotient of $\mathscr{U}(\TT)$ modulo the ideal generated by the central section $1 - 1_\TT$, where $1$ the  identity section of $\mathscr{U}(\TT)$ and $1_\TT$ is the central section of $\TT$. Then $\mathscr{D}(\TT)$ is a {\it sheaf  of twisted differential operators (TDO)} in the sense of \cite[Definition 2.1.1]{BB}. In particular, it is equipped with
\begin{itemize}
	\item
	a morphism of algebras $\iota_{\mathscr{D}}: \OO_X \to \mathscr{D}_{\TT}$, and
	\item
	an increasing exhaustive algebra filtration $\mathscr{D}_{\TT, \leqslant i}$, $i \geqslant0$.
\end{itemize}
Both data are inherited from the similar structures on $\mathscr{U}(\TT)$. The morphism $\iota_{\mathscr{D}}$ identifies $\OO_X$ with $\mathscr{D}_{\TT, \leqslant 0}$ and the obvious morphism from the symmetric algebra $S(\mathscr{D}_{\TT, \leqslant 1} / \mathscr{D}_{\TT, \leqslant 0})$ into $\operatorname{gr} \mathscr{D}_{\TT}$ is an isomorphism of  commutative graded $\OO_X$-algebras.

In particular, taking the Lie bracket makes $\mathscr{D}_{\TT, \leqslant 1}$ into a Picard algebroid and the composite map $\TT \hookrightarrow \mathscr{U}(\TT) \to \mathscr{D}_{\TT}$ identifies $\TT$ with $\mathscr{D}_{\TT, \leqslant 1}$ as Picard algebroids. It follows that $\TT \mapsto \mathscr{D}_\TT$ gives an equivalence between the category of Picard algebroids and the category of sheaves of TDOs. Under this equivalence, the trivial Picard algebroid $\TT_{\operatorname{triv}}$ corresponds to the sheaf $\mathscr{D}_X$ of usual (untwisted) differential operators on $X$.

We proceed to modules.
A $\TT$-module structure on a quasi-coherent $\OO_X$-module $\EE$ is equivalent to a (left) $\mathscr{D}_\TT$-module structure on $\EE$, such that the $\OO_X$-action on $\EE$ via $\iota_{\mathscr{D}}: \OO_X \to \mathscr{D}_{\TT}$ agrees with the $\OO_X$-module structure on $\EE$. Let $\coh(\TT)$ denote the abelian category of all $\TT$-modules over $X$ which are $\OO_X$-coherent and let $\coh(\mathscr{D}_\TT)$ denote the abelian category of all $\mathscr{D}_\TT$-modules over $X$ which are $\OO_X$-coherent. Then we have a canonical equivalence of categories $\coh(\TT) \simeq \coh(\mathscr{D}_\TT)$. It is well-known that a $\mathscr{D}_\TT$-module or a $\TT$-module $\EE$ is coherent as an $\OO_X$-module if and only if it is locally free (\cite[Lemma 2.3.1 (i)]{BB}).

Assume now that $\EE$ is an algebraic vector bundle of rank $rk(\EE)$ over $X$. A $\TT$-module structure on  $\EE$ induces an isomorphism $\TT \xrightarrow{\sim} \frac{1}{rk(\EE)} \TT(\det(\EE))$ of Picard algebroids, where $\det(\EE) = \bigwedge^{\operatorname{top}} \EE$ is the determinant line bundle of $\EE$ (\cite[Lemma 2.3.1 (ii)]{BB}). Moreover, the $\TT$-module structure induces a projective flat connection on $\EE$, i.e, a flat connection on the projectivization $\mathbb{P}(\EE)$ of $\EE$. Conversely, by \cite[Lemma 2.3.2]{BB} and \cite[Section 2]{BC}, an isomorphism $\TT \xrightarrow{\sim} \frac{1}{rk(\EE)} \TT(\det(\EE))$ and a projective connection on $\EE$ determine the $\TT$-module structure on $\EE$.

This discussion proves the following claim.

\begin{Lem}\label{Lem:algebroid_w_coh_module} If $\TT$ admits a nonzero coherent module, then
	$\TT$ is of the form $r \TT(L)$, where $L$ is an algebraic line bundle and $r \in \mathbb{Q}$.
\end{Lem}

To finish this part we discuss pullbacks.
Let $\varphi:Y \to X$ be a morphism of smooth varieties. For a Picard algebroid, or more generally, a Lie algebroid $\TT$ over $X$, we define the pullback Picard algebroid $\varphi^{\circ} \TT$ as in \cite[Section 1.4.6 and 2.2]{BB}. The pullback map $\TT \mapsto \varphi^{\circ}\TT$ gives a linear map $\varphi^{\circ}: \pic(X) \to \pic(Y)$. Under the identifications $\pic(\cdot) \simeq H^2_F(\cdot)$, $\varphi^{\circ}$ is nothing else but the map $\varphi^*: H^2_F(X) \to H^2_F(Y)$ induced by the morphism $\varphi^* : \varphi^{-1} \Omega^{\geqslant 1}_X \to \Omega^{\geqslant 1}_Y$ of complexes (\cite[Section 2.2]{BB}). We also have standard pullback map $\varphi^*: \operatorname{Pic}(X) \to \operatorname{Pic}(Y)$ of line bundles given by $L \mapsto \varphi^*L$. The pullback maps intertwine the natural maps $\operatorname{Pic}(\cdot) \to \pic(\cdot) \simeq H^2_F(\cdot)$ by construction.

For any $\TT$-module $\M$ over $X$, $\varphi^*\M = \OO_Y \otimes_{\varphi^{-1} \OO_X} \varphi^{-1} \M$ admits naturally the structure of a $\varphi^\circ \TT$-module.
When $\M$ is $\OO_X$-coherent, $\varphi^*\M$ is $\OO_Y$-coherent. Therefore we have a pullback functor $\varphi^*: \coh(\TT) \to \coh(\varphi^\circ \TT)$. In a similar manner, we can pull back TDOs and twisted $D$-modules. Namely, for the TDO $\mathscr{D}_\TT$ assoicated to a Picard algebroid $\TT$ over $X$, we define $\varphi^\circ \mathscr{D}_\TT$ as the sheaf of algebras $\operatorname{Diff}_{\varphi^{-1} \mathscr{D}_\TT} (\varphi^*\mathscr{D}_\TT, \varphi^*\mathscr{D}_\TT)$ of all $\OO_Y$-differential operators on $\varphi^*\mathscr{D}_\TT$ that commutes with the right $\varphi^{-1} \mathscr{D}_\TT$-action. Then $\varphi^\circ \mathscr{D}_\TT$ is a TDO over $Y$ and there is a canonical isomorphism $\varphi^\circ \mathscr{D}_\TT \simeq \mathscr{D}_{\varphi^\circ \TT}$ of TDOs.

Finally, we remark that we can define Picard algebroids and their modules in the analytic category and all the constructions above and in the remaining part of Section \ref{subsec:pic_defn} work in the same way. If $\TT$ is an algebraic Picard algebroid and $\M$ is a sheaf of modules over $\TT$, then the analytification $\TT^{an}$ of $\TT$ is an analytic Picard algebroid and the analytification $\M^{an}$ of $\M$ is a sheaf of modules over $\TT^{an}$.

\subsubsection{Equivariant setting}
Let $K$ be an algebraic group with Lie algebra $\kf$ and $X$ be a smooth variety equppied with a $K$-action $\alpha: K \times X \to X$. Then $\alpha$ induces an infinitesimal $\kf$-action $d \alpha: \kf \to \Gamma(X, \T_X)$, which is a Lie algebra morphism. Recall the following definitions from \cite[Section 1.8]{BB}.

\begin{defi}\label{defn:pic_action}
	Let $(\LL, \eta)$ be a Lie algebroid over $X$. A {\it (weak) $K$-action} or a {\it $K$-equivariant structure} $\alpha_\LL$ on $\LL$ is a $K$-action on $\LL$ by automorphisms of Lie algebroid which lifts $\alpha$ and makes $\LL$ into a $K$-equivariant coherent sheaf.

	A {\it strong $K$-action} on $\LL$ is a pair $(\alpha_\LL, \nu_\kf)$, where $\alpha_\LL$ is a weak $K$-action on $\LL$ and $\nu_\kf: \kf \to \LL$ is a morphism of Lie algebras, such that
	\begin{enumerate}[(i)]
		\item
		$\nu_\kf$ is $K$-equivariant, where $\kf$ is equipped with the adjoint $K$-action;
		\item
		the infinitesimal $\kf$-action $d\alpha_\LL: \kf \to \der(\LL)$ on $\LL$ induced by $\alpha_\LL$ coincides with the adjoint $\kf$-action $\operatorname{ad}_{\nu_\kf}$ induced by $\nu_\kf$. Here
		$\der(\LL)$ is the Lie algebra of derivations $\LL\rightarrow \LL$.
	\end{enumerate}

	When $\LL$ is a Picard algebroid $(\TT, \eta, 1_\TT)$, we require that $K$ acts on $\TT$ by automorphisms of Picard algebroid, i.e., additionally $k^* 1_\TT = 1_\TT$ for any $k \in K$.

	A morphism between Lie algebroids or Picard algebroids with strong $K$-actions are required to be $K$-equivariant and intertwine the $\nu_\kf$'s. As in the non-equivariant case, the set of all isomorphism classes of strongly $K$-equivariant Picard algebroids over $X$ forms a vector space under Baer sum, which is denoted by $\pic_\kf(X)$.
\end{defi}

\begin{Rem}\label{rem:strong_diff}
	A strong $K$-action on $\TT$ is the same as a $K$-action $\alpha_{\mathscr{D}}$ on the corresponding TDO $\mathscr{D}_\TT$ by algebra automorphisms making the embedding $\mathcal{O}_X\hookrightarrow \mathscr{D}_\TT$ equivariant together with a Lie algebra map $\nu_\kf: \kf \to \mathscr{D}_{\TT, \leqslant 1}$ that satisfies conditions similar to (i) and (ii) above. In this case, we also say that $\mathscr{D}_\TT$ carries a strong $K$-action.
\end{Rem}

\begin{defi}\label{defn:pic_module_action}
	Let $\LL$ be a Lie algebroid over $X$ equipped with a weak $K$-action $\alpha_\LL$. A {\it weakly $K$-equivariant $\LL$-module}, or {\it weak $(\LL, K)$}-module, is an $\LL$-module $\M$ equipped with a $K$-equivariant structure (as an $\OO_X$-module) that is compatible with the $K$-action $\alpha_\LL$ on $\LL$.

	Suppose now $\LL$ is equipped with a strong $K$-action $(\alpha_\LL, \nu_\kf)$. Given $\kappa \in (\kf^*)^{K}$, a {\it $(K, \kappa)$-equivariant $\LL$-module} is a weak $(\LL, K)$-module $\M$ such that the differential of the $K$-action equals $\nu_\kf - \kappa \cdot \Id_\M$ (cf. Definition \ref{eq:HC_equiv}). A {\it (strong) $(\LL, K)$-module} is a $(K,\kappa)$-equivariant $\LL$-module $\M$ with $\kappa=0$.
\end{defi}

Consider the case when $\LL$ is a Picard algebroid $\TT$. Let $\coh_K^\kappa(\TT)$ denote the abelian category of all $\OO_X$-coherent $(K,\kappa)$-equivariant $\TT$-modules over $X$, with morphisms being $K$-equivariant morphisms of $\TT$-modules.  In a similar manner, we can talk about weak and strong $(\mathscr{D}_\TT, K)$-modules and $(K, \kappa)$-equivariant $\mathscr{D}_\TT$-modules. Let $\coh_K^\kappa(\mathscr{D}_\TT)$ denote the abelian category of all $(K,\kappa)$-equivariant $\mathscr{D}_\TT$-modules that are $\OO_X$-coherent, with morphisms being $K$-equivariant morphisms of $\mathscr{D}_\TT$-modules. We also have an equivalence $\coh_K^\kappa(\TT) \simeq \coh_K^\kappa(\mathscr{D}_\TT)$ of abelian categories. When $\kappa=0$, we simply write $\coh_K^0(\TT)$ as $\coh_K(\TT)$ and write $\coh_K^{0}(\mathscr{D}_\TT)$ as $\coh_K(\mathscr{D}_\TT) $.

When the variety $X$ carries a $\cm$-action, we also consider the groupoid of {\it graded Picard algebroids}, i.e., (weakly) $\cm$-equivariant Picard algebroids in the sense of Definition \ref{defn:pic_action}, together with $\cm$-equivariant isomorphisms among them. Let $\pic(X)^{gr}$ denote the vector space of equivalence classes of graded Picard algebroids over $X$, the linear structure of which is defined in the same way as the definition of that of $\pic(X)$. Now suppose $X$ is equipped with a $K \times \cm$-action. A \emph{strongly $K$-equivariant graded Picard algebroid} is a weakly $\cm$-equivariant Picard algebroid $\TT$ equipped with a strong $K$-action commuting with the $\cm$-action, such that the image of the associated Lie algebra morphism $\nu_\kf: \kf \to \Gamma(X, \TT)$ lies in the $\cm$-invariant part of $\Gamma(X, \TT)$. Isomorphisms of two such objects are required to be $\cm$-equivariant and compatible with the strong $K$-actions (in the sense of Defintion \ref{defn:pic_action}).  Let $\pic_\kf(X)^{gr}$ to denote the space of all isomorphism classes of strongly $K$-equivariant graded Picard algebroids. Again $\pic_\kf(X)^{gr}$ admits the structure of a vector space. When $\TT$ is a (strongly $K$-equivariant) graded Picard algebroid, we always assume a $\TT$-module (strong $(\TT,K)$-module) $\M$ admits a $\cm$-action which is compatible with the $\cm$-action on $\TT$ and the $\TT$-module structure.
We say that $\M$ is a graded $\TT$-module (or a graded twisted local system).

In what follows we will need the following easy technical lemma.

\begin{Lem}\label{Lem:graded_loc_sys}
Let $\M_1,\ldots,\M_k$ be graded $\TT$-modules that are simple as ordinary $\TT$-modules.
Let $\M$ be a graded $\TT$-module that is isomorphic to the direct sum of $\M_i$
as an ordinary module. Then $\M$ is isomorphic to the direct sum of $\M_i$'s with some shifts
as a graded $\TT$-module.
\end{Lem}
\begin{proof}
For two coherent $\TT$-modules $\M,\M'$, the space $\operatorname{Hom}_{\TT}(\M,\M')$ is finite
dimensional: it is the space of flat sections of the finite rank flat vector bundle $\mathcal{H}om_{\mathcal{O}_X}(\M,\M')$. In particular, the Schur lemma holds for
coherent $\TT$-modules. If $\M,\M'$ are graded $\TT$-modules, then $\operatorname{Hom}_{\TT}(\M,\M')$
has a natural grading. Then $\M\cong \bigoplus_{i=1}^k \operatorname{Hom}_{\TT}(\M_i,\M)\otimes_{\C}\M_i$,
as a graded module.
\end{proof}

\subsubsection{Monodromic sheaves}\label{SSS:monodromic}

For our later applications, it suffices to consider Picard algebroids $\TT$ which admit at least one $\OO_X$-coherent $\TT$-module $\EE$. By Lemma \ref{Lem:algebroid_w_coh_module}, $\TT$ must be of the form $L^r = r \TT(L)$, where $L$ is an algebraic line bundle and $r \in \mathbb{Q}$. Let $L^\times$ denote the total space of $L$ with the zero section excluded and let $\pi: L^\times \to X$ denote the projection map.  Let $H = \cm$, then $L^\times$ is the (right) $H$-torsor over $X$ associated to $L$, so that $L \simeq L^\times \times_{H} \C$. We describe the following alternative equivalent way to define $\TT$-modules via $H$-monodromic $\mathscr{D}$-modules over $L^\times$ as follows. This is a special case of \cite[Section 2.5.2]{BB}, where $H$ can be a torus of any dimension.

For a moment, we treat the general case when $H$ is any torus. Let $\hf$ denote the Lie algebra of $H$ and $\exp_H: \hf \to H$ denote the exponential map for $H$. Then $\hf_{\Z} := \ker(\exp)$ is a lattice in $\hf$ such that there are identifications of groups $\hf/\hf_\Z \simeq H$ and $\hf_\Z \simeq \pi_1(H)$ via the exponential map, where $\pi_1(H)$ is the fundamental group of $H$.  Define the dual torus $H^\vee$ to be $\Hom(\pi_1(H), \cm)$. Then the Lie algebra of $H^\vee$ is naturally identified with the linear dual $\hf^*$ of $\hf$ so that the exponential map $\exp_{H^\vee}: \hf^* \to H^\vee$ sends $\mu \in \hf^*$ to the character $e^{2\pi i \mu}: X \mapsto e^{2\pi i {\langle \mu, X \rangle}}$, where $X \in \hf_{\Z} \simeq \pi_1(H)$. Let $\hf^*_\Z := \{ \mu \,|\, \operatorname{Im}\mu \subset \Z \}$, then $\hf^*_\Z$ is a lattice of $\hf^*$. Then $\exp_{H^\vee}$ induces a Lie group isomorphism $\hf^*/\hf^*_\Z \xrightarrow{\sim} H^\vee$, $\bar{\mu} \mapsto e^{2\pi i \mu}$,  where $\bar{\mu}$ stands for the image of $\mu\in \hf^*$.

Back to the case when $H=\cm$, we identify $\hf$ with $\C$ in the way that the exponential map $\exp_H$ is given by $\exp_H(z) = e^{2 \pi i z}$, $\forall \, z \in \hf = \C$, so that $\hf_\Z = \pi_1(H)$ is identified with $\Z \subset \C$. We also identify $\hf^*$ with $\C$ so that $w \in \C$ corresponds to the linear functional $f_w$ of $\hf$ given by $f_w(z) = w z$, $\forall \, z \in \hf = \C$. In this way $\hf^*_\Z$ is identified with $\Z \subset \C = \hf^*$ and $H^\vee= \Hom(\pi_1(H), \cm) = \Hom(\Z, \cm)$ is identified with $\cm$ via $\phi \mapsto \phi(1)$, $\forall \, \phi \in H^\vee$. The exponential map of $H^\vee$ is hence given by $\exp_{H^\vee} (w) = e^{2 \pi i w}$. In what follows, we regard the complex number $r$ that appears in $\TT = r \TT(L)$ as an element $r \in \hf^*$ via the chosen identification $\hf^* = \C$.

Let $\wt{\mathscr{D}} = \wt{\mathscr{D}}_{L^\times}$ denote the sheaf of (usual) differential operators over $L^\times$. Then $\wt{\mathscr{D}}$ carries a natural $H$-action. Set $\mathscr{D}_\hf := (\pi_*\mathscr{D})^H$ and  $\wt{\mathscr{D}}_\hf := \pi^{-1} \mathscr{D}_\hf$. Then $\wt{\mathscr{D}}_\hf$ is the subseaf of $\wt{\mathscr{D}}$ consisting of all $H$-invariant differential operators. The differential of the right $H$-action on $L^\times$ yields an injective morphism $i_\hf: \hf \to \wt{\mathscr{D}}$ of Lie algebras, whose image lies in $\wt{\mathscr{D}}_\hf$ and hence $\hf$ also embeds into $\mathscr{D}_\hf$. Therefore we have an embedding of $\U\hf = S\hf$ into the center of $\mathscr{D}_\hf$. Regard $r$ as an element of $\hf^* \simeq \C$ as above, then it determines a maximal ideal $I_r$ of $\U\hf = S\hf$ generated by $h - r(h)$ for all $h \in \hf$.  Set $\mathscr{D}_{r,X} (= \mathscr{D}_r) := \mathscr{D}_\hf / I_r \mathscr{D}_\hf$ to be the quotient sheaf of $\mathscr{D}_\hf$ by the two-sided ideal $I_r \mathscr{D}_\hf$. We have a natural isomorphism between $\mathscr{D}_r$ and the sheaf of twisted differential operators $\mathscr{D}_{L^r}$ with twist $L^r = r \TT(L)$.

The $H$-action on $\wt{\mathscr{D}}$ and $i_\hf: \hf \to \wt{\mathscr{D}}$ give a strong $H$-action on $\wt{\mathscr{D}}$. Therefore we can talk about $(H, r)$-equivariant $\wt{\mathscr{D}}$-modules over $L^\times$ (which are also called {\it weakly $H$-equivariant $\wt{\mathscr{D}}$-module with monodromy $r$} in \cite[Section 2.5.2]{BB}). We use  $\coh(\wt{\mathscr{D}})^{r} = \coh_H^r(\wt{\mathscr{D}})$ to denote the abelian category of all $(H, r)$-modules that are $\OO_{L^\times}$-coherent. We then have an equivalence of categories $\coh(\wt{\mathscr{D}})^{r} \xrightarrow{\sim} \coh(\mathscr{D}_r)$ given by $\wt{\M} \mapsto (\pi_* \wt{\M})^H$. The quasi-inverse functor is given by $\M \mapsto \pi^* \M = \wt{\mathscr{D}} \otimes_{\wt{\mathscr{D}}_\hf } \pi^{-1}\M$, where the $\mathscr{D}_r$-module $\M$ is regarded as a $\mathscr{D}_\hf$-module via the quotient homomorphism $\mathscr{D}_\hf \to \mathscr{D}_r$ and hence $\pi^{-1} \M$ can be regarded as a module over $\wt{\mathscr{D}}_\hf = \pi^{-1} \mathscr{D}_\hf$.

Similarly we can regard $X$ and $L^\times$ as analytic manifolds. Denote the analytic versions of relevant sheaves by $\OO_X^{an}$, $\OO_{L^\times}^{an}$, $\mathscr{D}^{an}_{r,X}$, $\wt{\mathscr{D}}^{an}$, etc. Define the categories $\coh(\wt{\mathscr{D}}^{an})^{r} \simeq \coh(\mathscr{D}^{an}_r)$ in a similar manner. We have the analytification functors $(\cdot)^{an}:  \coh(\wt{\mathscr{D}})^{r}  \to \coh(\wt{\mathscr{D}}^{an})^{r}$, $\M \mapsto \M^{an} = \OO^{an}_{L^\times} \otimes_{\OO_{L^\times}} \M$, and $(\cdot)^{an}: \coh(\mathscr{D}_r) \to \coh(\mathscr{D}^{an}_r)$, $\M \mapsto \M^{an} = \OO^{an}_{X} \otimes_{\OO_{X}} \M$. These two functors intertwine the equivalences $\coh(\wt{\mathscr{D}})^{r} \simeq \coh(\mathscr{D}_r)$ and $\coh(\wt{\mathscr{D}}^{an})^{r} \simeq \coh(\mathscr{D}^{an}_r)$.

Note that if we have a weakly $H$-equivariant $\wt{\mathscr{D}}$-module $\wt{\M}$ with monodromy $r$ and $\wt{\M}'$ with monodromy $\nu$, then $\wt{\M} \otimes_{\OO_{L^\times}} \wt{\M}'$ is a naturally $\wt{\mathscr{D}}$-module with monodromy $r + \nu$. Let $\OO(1) := \pi^* L$, then it has a natural $H$-equivaraint structure. Each point $\tilde{x}$ in the fiber $L^\times_x $ at any $x \in X$ corresponds to a linear  isomorphism $\C \simeq L_x$. Therefore there is a canonical trivialization $\OO(1) \simeq \OO_{L^\times}$, via which $\OO_{L^\times}$ becomes a $\wt{\mathscr{D}}$-module. However, this trivialization is not $H$-equivariant and we have $\OO(1) \in \coh(\wt{\mathscr{D}})^{1}$. Suppose $\nu \in \hf^*_\Z=\Z$. Set $\OO(\nu):=\OO(1)^{\otimes \nu}$. Then $\OO(\nu) \in \coh(\wt{\mathscr{D}})^{\nu}$. We then have an equivalence of categories $\coh(\wt{\mathscr{D}})^{r}  \xrightarrow{\sim} \coh(\wt{\mathscr{D}})^{r + \nu}$, $\wt{\M} \mapsto \OO(\nu) \otimes_{\OO_{L^\times}} \wt{\M}$, whose inverse is given by tensoring with $\OO(-\nu)$. The image of $\OO(\nu)$ under the equivalence $\coh(\wt{\mathscr{D}})^{\nu} \xrightarrow{\sim} \coh(\mathscr{D}_\nu)$ is $(\pi_*\OO(\nu))^H = L^{\nu} = L^{\otimes \nu}$. We have an isomorphism of algebras $\mathscr{D}_{r+\nu} \simeq  L^{\nu} \otimes_{\OO_X} \mathscr{D}_{r} \otimes_{\OO_X} L^{-\nu}$ for any $r \in \hf^*$ and $\nu \in \hf^*_\Z$, which induces an equivalence of categories $\coh(\mathscr{D}_r) \xrightarrow{\sim} \coh(\mathscr{D}_{r + \nu})$, $\M \mapsto L^{\nu} \otimes_{\OO_X} \M$, whose inverse is given tensoring with $L^{-\nu}$. Of course these two equivalences intertwine the equivalences $\coh(\wt{\mathscr{D}})^{r} \simeq \coh(\mathscr{D}_r)$ and $\coh(\wt{\mathscr{D}})^{r+\nu} \simeq \coh(\mathscr{D}_{r+\nu})$.

\subsubsection{Riemann-Hilbert correpondence} \label{subsec:RH}
We now turn to monodromic local systems over $L^\times$, considered as an analytic manifold $(L^\times)^{an}$.  Given any $r \in \hf^*$, let $\loc(L^\times)^{\bar{r}}$ denote the category of all local systems with monodromy $e^{2 \pi i r}$ over $L^\times$ (as an analytic manifold). Note that the definition only depends on the image $\bar{r}$ of $r$ in $\hf^*/\hf_\Z^*$.

Assume $X$ is a connected smooth variety, so that $X^{an}$ and $(L^\times)^{an}$ are also connected under the analytic topology. Fix a point $x$ in $X$ and choose any point $\tilde{x}$ in the fiber of $L^\times$ at $x$. Since $L^\times$ is an $H$-torsor, all groups $\pi_1(L^\times, \tilde{x})$ for different choices of $\tilde{x}$ in the fiber of $L^\times$ at a given $x$ are canonically isomorphic. Hence we will write $\pi_1(L^\times, x)$  instead of $\pi_1(L^\times, \tilde{x})$. Moreover, the action on $L^\times$ by $H = \cm$ gives rise to a canonical embedding of $\pi_1(H)$ into $\pi_1(L^\times, x)$. We then have a short exact sequence of groups
\begin{equation}\label{eq:pi_exsq}
	1 \to \pi_1(H) \to \pi_1(L^\times, x) \to \pi_1(X, x) \to 1.
\end{equation}
It is well-known that the category of local systems over $(L^\times)^{an}$ is equivalent to the category of finite dimensional representations of the fundamental group $\pi_1(L^\times, \tilde{x})$ of $L^\times$ based at $\tilde{x}$ by taking the monodromy. Given any $r \in \hf^*=\C$, we say that a local system over $(L^\times)^{an}$ \emph{has monodromy $\bar{r}$} if  $\pi_1(H)$ acts on the corresponding monodromy representation by the character $\bar{r} = e^{2\pi i r} \in H^\vee$. Let $\loc(L^\times)^{\bar{r}}$ denote the category of all local systems over $(L^\times)^{an}$ with monodromy $\bar{r}$.  The Riemann-Hilbert correspondence induces an equivalence of abelian categories $( \cdot )^\nabla : \coh(\wt{\mathscr{D}}^{an}_{L^\times})^r \to \loc(L^\times)^{\bar{r}}$, which sends a $\wt{\mathscr{D}}^{an}$-module $\wt{\M}$ to the sheaf $\wt{\M}^\nabla$ of its flat sections. The inverse functor is given by $\wt{S} \mapsto \OO_{L^\times} \otimes_{\C} \wt{S}$, where the $H$-action on $\OO_{L^\times} \otimes_{\C} \wt{S}$ is the diagonal one. Note that objects in $\loc(L^\times)^{\bar{r}}$ are automatically $H$-equivariant.

The functor $(\cdot)^{an}:  \coh(\wt{\mathscr{D}})^{r}  \to \coh(\wt{\mathscr{D}}^{an})^{r}$ is not an equivalence in general. On the other hand, let $\coh^{reg}(\wt{\mathscr{D}})^{r}$ denote the full subcategory of $\coh(\wt{\mathscr{D}})^{r}$ consisting of those $\wt{\mathscr{D}}$-modules with regular singularities. Then $(\cdot)^{an}$ restricts to an equivalence of categories $\coh^{reg}(\wt{\mathscr{D}})^{r} \xrightarrow{\sim} \coh(\wt{\mathscr{D}}^{an})^{r}$ and therefore we have the equivalence of categories $\coh^{reg}(\wt{\mathscr{D}})^{r} \xrightarrow{\sim}  \loc(L^\times)^{\bar{r}}$, $\wt{\M} \mapsto (\wt{\M}^{an})^{\nabla}$(\cite{Deligne}, \cite[Remark 2.5.5, (iii)]{BB}). Denote by $\coh^{reg}(\mathscr{D}_r)$ the essential image of $\coh^{reg}(\wt{\mathscr{D}})^{r}$ under the equivalence $\coh(\wt{\mathscr{D}})^{r} \xrightarrow{\sim} \coh(\mathscr{D}_r)$. Then we also have an equivalence  $\coh^{reg}(\mathscr{D}_r) \xrightarrow{\sim} \loc(L^\times)^{\bar{r}}$. Note the latter category depends only on the image of $r$ in $\hf^*/\hf^*_\Z$. The equivalences $\coh(\mathscr{D}_r) \xrightarrow{\sim} \coh(\mathscr{D}_{r + \nu})$ for $\nu \in \hf^*_\Z$ preserve the subcategory of objects with regular singularities and intertwine the equivalences with the category of twisted local systems.

We take the pushout of the short exact sequence \eqref{eq:pi_exsq} along the character $\bar{r}: \pi_1(H) \to \cm$, to obtain a central extension of the group $\pi_1(X,x)$ by $\cm$,
\[  \quad 1 \to \cm \to \pi^{\tilde{r}}_1(X, x) \to \pi_1(X, x) \to 1.  \]
This extension determines a Schur multiplier $\tilde{r} \in H^2(\pi_1(X, x), \cm)$ of $\pi_1(X, x)$. Let $\operatorname{Rep}(\pi^{\tilde{r}}_1(X, x))$ denote the abelian category of all finite dimensional representations $V$ of $\pi^{\tilde{r}}_1(X, x)$ such that $\cm \subset \pi^{\tilde{r}}_1(X, x)$ acts on $V$ by the usual scalar multiplication, i.e., $V$ is a projective representation of $\pi_1(X)$ with Schur multiplier $\tilde{r}$. For any $\wt{S} \in \loc(L^\times)^{\bar{r}}$, we take its fiber at any point $\tilde{x} \in L^\times$ above $x$, then its monodromy representation factors through a representation in $\operatorname{Rep}(\pi^{\tilde{r}}_1(X, x))$. This defines an equivalence of abelian categories $\loc(L^\times)^{\bar{r}} \xrightarrow{\sim} \operatorname{Rep}(\pi^{\tilde{r}}_1(X, x))$.

All the constructions and equivalences of categories have natural functoriality with respect to pullback. Namely, let $\varphi: Y \to X$ be a morphism of smooth varieties. Suppose $L$ is an algebraic line bundle over $X$, then $L':=\varphi^* L$ is a line bundle over $Y$ and $L'^\times$ is the pullback $H$-torsor of $L^\times$ via $\varphi$, together with an $H$-equivariant map $\tilde{\varphi}: L'^\times \to L^\times$ lifting $\varphi$. We then have the pullback functor $\tilde{\varphi}^{-1}: \operatorname{Loc}(L^\times)^{\bar{r}} \to \operatorname{Loc}(L'^\times)^{\bar{r}}$, $\wt{S} \mapsto \tilde{\varphi}^{-1} \wt{S}$. Suppose $X$ and $Y$ are connected as complex
analytic manifolds. Fix a point $y \in Y$ and set $x = \varphi(y)$. We then have the Cartesian diagrams of groups
\begin{equation}\label{diag:pi_groups}
	\begin{tikzcd}
		\pi_1(L'^\times, y) \arrow{r} \arrow[d, "\tilde{\varphi}_*"]  & \pi_1(Y,y)  \arrow[d, "\varphi_*"]   \\
		\pi_1(L^\times, x) \arrow{r}  & \pi_1(X, x)
	\end{tikzcd}
	\quad \text{and} \quad
	\begin{tikzcd}
		\pi^{\tilde{r}'}_1(Y, y) \arrow{r} \arrow[d, "\tilde{\varphi}_*"]  & \pi_1(Y,y)  \arrow[d, "\varphi_*"]   \\
		\pi^{\tilde{r}}_1(X, x) \arrow{r}  & \pi_1(X, x)
	\end{tikzcd}
\end{equation}
where $\tilde{r}' \in H^2(\pi_1(Y,y), \cm)$ is the Schur multiplier for the $H$-torsor $L'^\times$ and $r$, which coincides with the pullback of $\tilde{r} \in H^2(\pi_1(X,x), \cm)$ via the group homomorphism $\varphi_*: \pi_1(Y,y) \to \pi_1(X,x)$. Therefore we have a functor $\tilde{\varphi}^*: \operatorname{Rep}(\pi^{\tilde{r}}_1(X, x)) \to \operatorname{Rep}(\pi^{\tilde{r}'}_1(Y, y))$, $(V, \rho) \mapsto (V, \rho \circ  \tilde{\varphi}_*)$. The discussion above yields the following proposition.

\begin{Prop}\label{Prop:regular_local_systems}
	There are natural equivalence of categories \[ \coh^{reg}(\mathscr{D}_{r,X}) \xrightarrow{\sim} \coh^{reg}(\wt{\mathscr{D}}_{L^\times})^{r}  \xrightarrow{\sim} \loc(L^\times)^{\bar{r}}  \xrightarrow{\sim} \operatorname{Rep}(\pi^{\tilde{r}}_1(X, x)) \]
	which are functorial with respect to pullback functors, i.e., for any morphism $\varphi: Y \to X$ of smooth varieties which are connected as  complex analytic manifolds, we have natural equivalences of categories illustrated by the diagram
	\begin{equation}\label{diag:pi_functors}
		\begin{tikzcd}
			\coh^{reg}(\mathscr{D}_{r,X}) \arrow[r, "\sim"] \arrow[d, "\tilde{\varphi}^*"] & \coh^{reg}(\wt{\mathscr{D}}_{L^\times})^{r} \arrow[r, "\sim"] \arrow[d, "\tilde{\varphi}^*"]   & \loc(L^\times)^{\bar{r}}  \arrow[r, "\sim"]  \arrow[d, "\varphi^{-1}"]  & \operatorname{Rep}(\pi^{\tilde{r}}_1(X, x)) \arrow[d, "\tilde{\varphi}^*"] \\
			\coh^{reg}(\mathscr{D}_{r,Y}) \arrow[r, "\sim"]   & \coh^{reg}(\wt{\mathscr{D}}_{L'^\times})^{r} \arrow[r, "\sim"]    & \loc(L'^\times)^{\bar{r}}  \arrow[r, "\sim"]   & \operatorname{Rep}(\pi^{\tilde{r}'}_1(Y, y)).
		\end{tikzcd}
	\end{equation}
\end{Prop}

Now suppose $X$ admits an algebraic action by an algebraic group $K$ and $L$ is a $K$-equivariant line bundle over $X$. We can consider the category $\coh_K(\mathscr{D})^r$ of strongly $K$-equivariant objects in $\coh(\mathscr{D})^r$ and the category $\loc_{K}(L^\times)^{\bar{r}}$ of $K$-equivariant local systems in $\loc(L^\times)^{\bar{r}}$, i.e., local systems that carry $K$-actions which lift the $K$-action on $L^\times$
subject to the usual compatibility conditions. The category $\coh_K(\mathscr{D})^r$ is equivalent to the category $\coh_K(\mathscr{D}_r)$ of strongly $K$-equivariant objects in $\coh(\mathscr{D}_r)$.  Assume additionally that the $K$-action on $X$ has only finitely many orbits. Then $L^\times$ has finitely many $K \times H$-orbits. This implies that all objects in $\coh_K(\mathscr{D})^r$ automatically have regular singularities (\cite[Theorem 11.6.1]{Hotta}). Therefore we have the equivalences of categories $\coh_K(\mathscr{D}_r) \simeq \coh_K(\mathscr{D})^r \simeq  \loc_{K}(L^\times)^{\bar{r}}$.

\subsubsection{Homogeneous case}\label{SSS_homog_case}
Let $G$ be a connected reductive algebraic group, and $H$ be its closed subgroup. Pick $\lambda\in
\operatorname{Hom}(H,\C^\times)\otimes_{\Z}\C = (\hf^*)^H$. Similarly to Section \ref{SSS:monodromic}, to $\lambda$ we can assign a strongly $G$-equivariant Picard algebroid $\TT_\lambda$ over $G/H$ (see \cite[Section 2.4.2]{Yu} for details), or equivalently, a sheaf of twisted
differential operators $\mathscr{D}^\lambda_{G/H}$ with a strong $G$-action (see Remark \ref{rem:strong_diff}). In fact, this gives an isomoprhism $(\hf^*)^H \simeq \pic_\g(G/H)$ of vector spaces (see \cite[Proposition 2.4]{Yu}). Pick $\kappa\in \g^{*G}$. We want to classify strongly
$(G,\kappa)$-equivariant coherent $\mathscr{D}^\lambda_{G/H}$-modules. Since the action of $G$
on $G/H$ is transitive, each such module $\M$ is a $G$-equivariant coherent sheaf on
$G/H$. Set $x:=1H\in G/H$. So, the fiber $\M_x$ is a rational representation of $H$.
The following lemma is quite standard and is a variant of \cite[Lemma 2.7]{Yu}.

\begin{Lem}\label{Lem:homog_classif}
The assignment $\M\mapsto \M_x$ defines a category equivalence between
\begin{itemize}
\item[(a)] The category of coherent strongly $(G,\kappa)$-equivariant $\mathscr{D}^\lambda_{G/H}$-modules,
\item[(b)] and the category of finite dimensional rational $H$-representations, where
$\mathfrak{h}$ acts by $\lambda-\kappa|_{\mathfrak{h}}$.
\end{itemize}
\end{Lem}
\begin{proof}
Let $\pi$ denote the projection $G\rightarrow G/H, g\mapsto gH$. As in Section \ref{SSS:monodromic},
the functor $\pi^*$ gives an equivalence between (a) and
\begin{itemize}
\item[(c)] The category of coherent strongly $(G\times H, \kappa\times (-\lambda))$-equivariant
$\mathscr{D}_G$-modules (the minus sign appears because the action of $H$ on $G$ is by $h.g=gh^{-1}$).
\end{itemize}
For $\xi\in \g$, let $\xi_G$ (resp., $\xi_\M$) denote the vector field on $G$ (resp., $\C$-linear endomorphism
of $\M$ induced by $\xi$). For $\eta\in \mathfrak{h}$, let $\eta^r_H,\eta^r_\M$ have the similar meaning
for the $H$-action on $G$ from the right, like in (c).
Note that $\eta_{G,1}=-\eta^r_{H,1}$
for $\eta\in \mathfrak{h}$. The stabilizer of $1$ in $G\times H$ is $H$ embedded diagonally. So the action
of $\mathfrak{h}$ on $\M_1$ is by $\eta\mapsto \eta_{\M,1}+\eta^r_{\M,1}$. Since $\xi_G-\langle\kappa,\xi\rangle=\xi_\M$
and $\eta_G+\langle\lambda,\eta\rangle=\eta_\M$, we see that $\mathfrak{h}$ acts on $\M_1$ by
$\eta\mapsto \langle\lambda-\kappa|_{\mathfrak{h}},\eta\rangle$, so $\M\mapsto \M_{x}=(\pi^*\M)_1$
is indeed a functor from (a) to (b).

It remains to show that every object $N$ in (b) arises as $\M_1$. We have a unique $(G,\kappa)$-equivariant
$\mathscr{D}_G$-module with fiber $N$ at $1$. It is easy to see that it is also $(H,-\lambda)$-equivariant.
This finishes the proof.
\end{proof}

\begin{Rem}\label{Rem:homog_analyt_classif}
The same argument shows that the same statement as in Lemma \ref{Lem:homog_classif}
holds if in (a) we replace $\mathscr{D}^\lambda_{G/H}$ with its analytification.
\end{Rem}

Now we consider a more general situation. Let $K$ be a subgroup of $G$ that acts
on $G/H$ with finitely many orbits, say $Kx_1,\ldots, Kx_\ell$ for $x_i:=g_i H\in G/H$.
Note that $H$ acts trivially on $\operatorname{Hom}(H,\C^\times)_{\Z} \otimes \C$. So conjugating with
$g_i$ gives rise to a well-defined element $\lambda_i\in \operatorname{Hom}(G_{x_i},\C^\times)_{\Z} \otimes \C$.
Let $\Upsilon_i$ denote the set of isomorphism classes of irreducible $K_{x_i}$-modules
where $\kf_{x_i}$ acts by $-(\lambda_i|_{\kf_{x_i}}-\kappa|_{\kf_{x_i}})$.

The following result is proved in the same way as \cite[Theorem 11.6.1,(ii)]{Hotta}.

\begin{Lem}\label{Lem:fin_many_orbit_classif}
The simple (automatically regular) $(K,\kappa)$-equivariant $\mathscr{D}^{\lambda}_{G/H}$-modules
are classified by the elements of the set $\bigsqcup_{i=1}^\ell \Upsilon_i$:
to get from an element of $\Upsilon_i$ to a simple object, take the minimal extension of the twisted local
system on $Kx_i$ corresponding to the element of $\Upsilon_i$ via Lemma \ref{Lem:homog_classif}.
\end{Lem}

\subsection{Deformation quantization of symplectic varieties}\label{subsec:quan}
\subsubsection{Definitions}
We review deformation quantization in the algebraic context from \cite{BK}. Let $X$ be a smooth algebraic variety (over $\C$) equipped with an algebraic symplectic form $\Omega \in H^0(X, \Omega_X^2)$. Then $\Omega$ induces a  Poisson bracket $\{ - , -\}$ on the structure sheaf $\OO_X$ of $X$.

\begin{defi}\label{defi:formal_quan_t}
	A \emph{formal deformation quantization over $\C[[\hbar]]$}, or simply {\it quantization} of a smooth symplectic variety $(X,\Omega)$ is the following datum
	\begin{itemize}
		\item
		a sheaf $\OO_\hbar$ of flat associative $\C[[\hbar]]$-algebras on $X$, complete and separated in the $\hbar$-adic topology, such that the commutator $[a, b] = a \cdot b - b \cdot a \in \hbar \OO_\hbar$ for any $a, b \in \OO_\hbar$, and
		\item
		an isomorphism $\phi: \OO_\hbar / \hbar \OO_\hbar \xrightarrow{\sim} \OO_X$ of sheaves of commutative Poisson algebras, where the Poisson bracket on $ \OO_\hbar / \hbar \OO_\hbar$ is induced by $\hbar^{-1}[a,b]$, $a, b \in \OO_\hbar$.
	\end{itemize}
\end{defi}

There is an obvious definition of isomorphic quantizations. Denote by $\quan(X,\Omega)$ or $\quan(X)$ the set of equivalence classes of quantizations of $X$. There is also a natural notion of quantization for a scheme $X$ which is smooth over a scheme $S$ of finite type and is equipped with a relative sympletic form $\Omega \in H^0(X, \Omega^2_{X/S})$, where $\Omega^\bullet_{X/S}$ is the relative de Rham complex. All the constructions and results in this section also work in the general relative setting. See \cite{BK} for details. We will suppress $S$ for most of the time in the rest of the paper.

We recall the following main result from \cite[Theorem 1.8]{BK} which classifies all quantizations under an additional assumption: $X$ is said to be {\it admissible} if the natural homomorphism $H^i_{DR}(X) \to H^i(X, \OO_X)$ is surjective for $i = 1, 2$. This implies that the natural map $H^2_F(X) \to H^2_{DR}(X)$ is injective and its image coincides with the kernel of the map $H^2_{DR}(X) \twoheadrightarrow H^2(X, \OO_X)$.

\begin{Thm}\label{thm:quan_BK}
	Let $X$ be an admissible smooth variety equipped with a symplectic form $\Omega$. Then there is a natural injective map
	\[  \per: \quan(X, \Omega) \hookrightarrow \hbar^{-1} H^2_{DR}(X) [[\hbar]],   \]
	called the {non-commutative period map}. Futhermore, for any $\OO_\hbar \in \quan(X,\Omega)$, we have $\per(\OO_\hbar) \in \hbar^{-1} [\Omega] + H^2_{DR}[[\hbar]]$, where $[\Omega] \in H^2_{DR}(X)$ denotes the cohomology class of the symplectic form $\Omega$. Finally, any choice of linear splitting $P: H^2_{DR}(X) \twoheadrightarrow H^2_F(X)$ of the canonical embedding $H^2_F(X) \to H^2_{DR}(X)$ induces a bijection $\quan(X, \Omega) \xrightarrow{\sim} H^2_F(X)[[\hbar]]$ given by $\OO_\hbar \mapsto P( \per(\OO_\hbar) - \hbar^{-1} [\Omega])$.
\end{Thm}

For our applications, we need to consider \emph{graded} quantizations(see \cite[Section, 2.2]{quant_iso} \cite[Section 3.1]{orbit}  and \cite[Section 3.2]{BPW}). Fix a positive integer $d$.  Suppose that the symplectic variety $X$ admits a rational $\cm$-action, with respect to which the symplectic form $\Omega$ has weight $d$, i.e., the push-forward $z.\Omega = z_* \Omega$ of the form $\Omega$ by the automorphism of $X$ induced by $z \in \cm$ equals the rescaled form $z^d\Omega$. In this case we say that $(X,\Omega)$ is \emph{graded} (of weight $d$).

\begin{defi}\label{defn:graded_quan}
	Suppose $(X, \Omega)$ is a graded symplectic variety of weight $d$. A quantization $\OO_\hbar$ of $(X,\Omega)$  is said to be \emph{graded} if it is equipped with a pro-rational  $\cm$-action by algebra automorphisms such that the isomorphism $\OO_\hbar / \hbar \OO_\hbar \simeq \OO_X$ is $\cm$-equivariant and $z. \hbar = z^d \hbar$ for any $z \in \cm$. We require isomorphisms between graded quantizations to be $\cm$-equivariant. The set of isomorphism classes of graded quantizations of $(X,\Omega)$ is denoted as $\quan(X, \Omega)^{gr}$ or $\quan(X)^{gr}$.
\end{defi}

We write
\begin{equation} \label{eq:per}
	\per(\OO_\hbar) =  \hbar^{-1}[\Omega] + \omega_1(\OO_\hbar) +  \hbar \omega_2(\OO_\hbar)+  \hbar^2 \omega_3 (\OO_\hbar)+ \cdots \in  \hbar^{-1}H^2_{DR}(X)[[\hbar]],
\end{equation}
where $\omega_k(\OO_\hbar) \in H^2_{DR}(X)$. Note that when $(X, \Omega)$ is graded, the de Rham class $[\Omega] \in H^2_{DR}(X)$ of the symplectic form $\Omega$ vanishes, thanks to the well-known Cartan homotopy formula. The following result is due to first named author (\cite[Section 2.2 and Prop. 2.3.2]{quant_iso}).

\begin{Prop}\label{prop:graded}
	Let $(X,\Omega)$ be a graded smooth symplectic variety.
		If $\OO_\hbar$ is a graded quatization of $(X, \Omega)$, then $\per(\OO_\hbar) \in H^2_{DR}(X, \C) \subset  \hbar^{-1}H^2_{DR}(X)[[\hbar]]$, i.e., $\omega_k(\OO_\hbar) = 0$ for all $k \geqslant 2$ in \eqref{eq:per}.
The converse is true if $X$ is admissible.
\end{Prop}

Then Theorem \ref{thm:quan_BK} and Proposition \ref{prop:graded} together imply the following result.

\begin{Cor}\label{cor:per_graded}
	Let $(X, \Omega)$ be a graded smooth symplectic variety. Suppose $H^1(X,\OO_X) = H^2(X,\OO_X)=0$. Then the period map $\per$ identifies the set $\quan(X, \Omega)^{gr}$ of isomorphism classes of graded quantizations of $(X,\Omega)$ with $H^2_{DR}(X) \subset  \hbar^{-1}H^2_{DR}(X)[[\hbar]]$.
\end{Cor}

\begin{Rem} \label{rmk:enhanced_period}
	In  \cite{Yu}, the second named author has defined an {\it enhanced (truncated) non-commutative period map} $\overline{\per}: \quan(X, \Omega) \to H^2_F(X)$ for any (not necessarily admissible) symplectic variety $X$ and proved that, for any quantization $\OO_\hbar$ of $(X, \Omega)$, the map $	c_1^{\operatorname{top}}: H^2_F(X) \to H^2_{DR}(X)$ sends $\overline{\per}(\OO_\hbar)$ to $ \omega_1(\OO_\hbar)$, where $\omega_1(\OO_\hbar)$ is as in \eqref{eq:per}. This is compatible with Corollary \ref{cor:per_graded}, since in the context there we have $H^2_F(X) \simeq H^2_{DR}(X)$. The enhanced period map will play a role in quantizations of Lagrangian subvarieties. See Remark \ref{rmk:enhanced_period_lag}.
\end{Rem}

We finish this section with an example. Suppose $X$ be  a conical symplectic singularity
such that $\operatorname{codim}_X X^{sing}\geqslant 4$. Then
\begin{equation}\label{eq:orbit_coh_cohom_vanish}
H^i(X^{reg},\mathcal{O}_X)=0, \text{ for }i=1,2.
\end{equation}
Recall that in Section \ref{SS_quant_nilp} to $\lambda\in H^2(X^{reg},\C)$
we have assigned the filtered quantization $\A_\lambda$ of $\C[X]$. Let $\A_{\lambda,\hbar}$ denote
the completion of its modified Rees algebra. The microlocalization
$\A_{\lambda,\hbar}|_{\Orb}$ is the graded quantization with period $\lambda$,
see, e.g., \cite[Section 3.2]{orbit}.

\subsection{Quantization of Lagrangian subvarieties} \label{subsec:quan_lag}

\subsubsection{Definitions}

Suppose that $Y$ is a smooth closed Lagrangian subvariety of a symplectic variety $X$ with the embedding denoted by $\iota : Y \hookrightarrow X$. When $X$ is graded, we also assume $Y$ is $\cm$-stable. There is a natural notion of (graded) quantization of a ($\cm$-equivariant) vector bundle over $Y$ (\cite[Section 1]{BC}, \cite[Section 4.3]{LY}).

\begin{defi}
	Suppose $\OO_\hbar$ is a quantization of $(X,\Omega)$ and $\EE$ is a vector bundle of rank $r$  over $Y$. A {\it quantization} of $\EE$, or rather the sheaf $\iota_* \EE$ of $\OO_X$-modules, is an $\OO_\hbar$-module $\EE_\hbar$ flat over $\C[[\hbar]]$, complete and separated in the $\hbar$-adic topology, with an isomorphism $\EE_\hbar /\hbar \EE_\hbar \cong \iota_*\EE$ of $\OO_X$-modules.

	Now assume $(X,\Omega)$ is graded and $Y$ is a $\cm$-stable Lagrangian subvariety of $X$. Suppose $\OO_\hbar$ is a graded quantization of $X$ and $\EE$ is a $\cm$-equivariant ($\cm$-linearized) vector bundle of rank $r$ over $Y$. A {\it graded quantization} of $\EE$ is a quantization $\EE_\hbar$ of $\EE$ with a pro-rational $\cm$-action, such that the structure map $\OO_\hbar \otimes_{\C[[\hbar]]} \EE_\hbar \to \EE_\hbar$ is $\cm$-equivariant and the isomorphism $\EE_\hbar /\hbar \EE_\hbar \cong \iota_*\EE$ is also $\cm$-equivariant.
\end{defi}

\subsubsection{Picard algebroids associated to Lagrangian subvarieties} \label{subsec:lag_pic}

By \cite[\S\,4.1]{LY}, one can associate two related Picard algebroids over a Lagrangian subvariety $Y$ to a given quantization $\OO_\hbar$. Namely, let $\J_Y$ denote the ideal subsheaf of $\OO_X$ consisting of all
sections vanishing on $Y$. Set $\J_{Y,\hbar}$ and $\J_{Y,\hbar}'$ to be the preimage of the ideal $\J_Y$ and $\J_Y^2$ respectively under the projection $\OO_\hbar \twoheadrightarrow \OO_X$. We write $\J_{Y,\hbar}^2 = \J_{Y,\hbar} \cdot \J_{Y,\hbar} \subset \OO_\hbar$, which is a subsheaf of two-sided ideals of the algebra $\OO_\hbar$. Then $\hbar^{-1} \J_{Y,\hbar}^2 \subset \hbar^{-1}\J_{Y,\hbar}' \subset \hbar^{-1}\J_{Y,\hbar}$ are all subsheaves of Lie ideals in the sheaf $ \hbar^{-1}\OO_\hbar$ of Lie algebras. Consider the sheaf $\hbar^{-1} (\J_{Y,\hbar}/\J_{Y,\hbar}^2)$, which is a sheaf of Lie algebras supported on $Y$. The left multiplication of $\OO_\hbar$ on $\hbar^{-1}\J_{Y,\hbar}$ descends to a left $\OO_Y$-module structure on $\hbar^{-1} (\J_{Y,\hbar}/ \J_{Y,\hbar}^2)$. We have a short exact sequence of sheaves of Lie algebras and $\OO_Y$-modules
\begin{equation}\label{exsq:J}
	0 \to \hbar^{-1} (\J_{Y,\hbar}' / \J_{Y,\hbar}^2) \to \hbar^{-1} (\J_{Y,\hbar}/ \J_{Y,\hbar}^2) \to \hbar^{-1} (\J_{Y,\hbar} / \J_{Y,\hbar}') \to 0.
\end{equation}
By \cite[Lemma 5.4]{BGKP}, there are canonical isomorphisms
\begin{equation} \label{eq:J_isom}
	\hbar^{-1} (\J_{Y,\hbar} / \J_{Y,\hbar}') \simeq \T_Y, \quad \hbar^{-1} (\J_{Y,\hbar}' / \J_{Y,\hbar}^2) \simeq \OO_Y.
\end{equation}
Then  \eqref{exsq:J} and \eqref{eq:J_isom} together make $\hbar^{-1} (\J_{Y,\hbar}/ \J_{Y,\hbar}^2)$ into a Picard algebroid over $Y$, which is denoted as $\tatp$.

The second Picard algebroid arises from a different $\OO_Y$-module structure on $\hbar^{-1} (\J_{Y,\hbar}/ \J_{Y,\hbar}^2)$ defined as follows. There is a commutative non-assoicative \emph{symmetrized product} on $\OO_\hbar$ defined by $a \bullet b := \frac{1}{2}(a b + b a)$. The symmetrized product induces a second $\OO_\hbar$-module structure, which is also compatible with the short exact sequence \eqref{exsq:J} and isomorphisms \eqref{eq:J_isom},  and hence we get a second Picard algebroid structure on $\hbar^{-1} (\J_{Y,\hbar}/ \J_{Y,\hbar}^2)$ with the same Lie bracet. We denote the second Picard algebroid as $\tat$. By \cite[Section 1]{BC} and \cite[Section 4.1]{LY},  the identity map of $\hbar^{-1} (\J_{Y,\hbar}/ \J_{Y,\hbar}^2)$ induces a canonical isomorphism
\begin{equation}\label{eq:tatp2tat}
	\tatp \xrightarrow{\sim} \tat + \frac{1}{2}\TT(K_Y).
\end{equation}
of Picard algebroids, where $K_Y$ denotes the canonical line bundle of $Y$.

By \cite[\S\,5.3]{BGKP}, the Picard algebroid $\tat$ satisfies
\begin{equation}  \label{eq:c1_lag}
	c^{\operatorname{top}}_1(\tat)=\iota^*[\omega_1(\OO_\hbar)],
\end{equation}
where $\omega_1(\OO_\hbar)$ is as in \eqref{eq:per} and $\iota^*: H^2_{DR}(X) \to H^2_{DR}(Y)$ is the restriction map induced by the embedding $\iota: Y \hookrightarrow X$.

\begin{Rem} \label{rmk:enhanced_period_lag}
	It is shown in \cite[Proposition 6.4]{Yu} that $\iota^\circ\left(\overline{\per}(\OO_\hbar)\right) = \tat$, where $\overline{\per}: \quan(X, \Omega) \to \pic(X)$ is the enhanced period map mentioned in Remark \ref{rmk:enhanced_period} and $\iota^\circ$ stands for the pullback of Picard algebroids as in the discussion after Lemma \ref{Lem:algebroid_w_coh_module}. This provides a refinement of the equality \eqref{eq:c1_lag}. The case when $Y$ is an orbit of an algebraic group, which is the main focus of the current paper, is treated in Lemma \ref{Lem:quantum_comom_period} and Corollary \ref{Cor:equiv_condition}.
\end{Rem}

\subsubsection{Existence and Uniqueness} \label{subsec:lag_vec}

Let $\EE$ be an algebraic vector bundle  over $Y$ and $\EE_\hbar$ be a quantization of $\EE$. In particular, $\EE_\hbar$ is a module over $\OO_\hbar$ and we have $\J_{Y,\hbar} \cdot \EE_\hbar \subset \hbar \EE_\hbar$. Since $\EE_\hbar$ is flat over $\C[[\hbar]]$, we have a Lie algebra action of $\hbar^{-1} \J_{Y,\hbar}$ on $\EE_\hbar$ given by $(\hbar^{-1} a). m = \hbar^{-1}(a \cdot m)$. This action descends to a Lie algebroid action of $\tatp$ on $\EE_\hbar / \hbar \EE_\hbar \simeq \EE$, which makes $\EE$ into a $\tatp$-module. Therefore it makes sense to talk about quantizations of a $\tatp$-module $\EE$ that lift the $\tatp$-module structure.

\begin{Rem} \label{rmk:pic_mod_t2}
	Note that the $\tatp$-module structure on $\EE$ is completely determined by the $\OO_\hbar / \hbar^2 \OO_\hbar$-module structure on $\EE_\hbar / \hbar^2 \EE_\hbar$. That is, in the definition of $\tatp$ from Section \ref{subsec:lag_pic}, we can replace $\J_{Y,\hbar}$ by $\wt{\J}_{Y,\hbar} := \J_{Y,\hbar} / \hbar^2 \OO_\hbar \subset \OO_\hbar / \hbar^2 \OO_\hbar$ and get a natural identification $\tatp \simeq \hbar^{-1} \wt{\J}_{Y,\hbar} / \wt{\J}^2_{Y,\hbar}$ of Picard algebroids. Apply the same construction as above to the $\OO_\hbar / \hbar^2 \OO_\hbar$-module $\EE_\hbar / \hbar^2 \EE_\hbar$ gives the same $\tatp$-module structure on $\EE$.
\end{Rem}

Recall that we can decompose $\per(\OO_\hbar)$ as in \eqref{eq:per}, where $ \omega_k(\OO_\hbar) \in H^2_{DR}(X)$. The following theorem (\cite[Thm\,1.1]{BC}) gives a necessary and sufficient condition for $\EE$ to admit a quantization. The case when $\EE$ is a line bundle was proven in \cite[Theorem 1.4]{BGKP}.

\begin{Thm}\label{thm:lag_vec}
  Fix a quantization $\OO_\hbar$ of $(X,\Omega)$. Let $Y$ be a smooth Lagrangian subvariety $Y$ of $X$ and $\EE$ be a coherent $\tatp$-module over $Y$. Then $\EE$ admits a quantization $\EE_\hbar$ that lifts the $\tatp$-module structure on $\EE$, if and only if $\iota^*[\omega_k(\OO_\hbar)] = 0$ holds in $H^2_{DR}(Y)$, $\forall ~ k \geqslant 2$.
\end{Thm}

Note that, the quantization $\EE_\hbar$ in Theorem \ref{thm:lag_vec}, if exists, is usually far from being unique. The set of isomorphism classes of all quantizations of vector bundles with a fixed rank admit a free action of an often infinite-dimensional group. We refer the reader to \cite[Thm\,1.1]{BC} for more details.

From now on, we will restrict ourselves to the graded setting. One advantage of working under the graded assumption, besides many others, is that we have the uniqueness of a quantization under a certain mild condition, which we recall now. First of all, note that when $\EE_\hbar$ is a graded quantization, the induced $\tatp$-module structure on $\EE$ is compatible with the $\cm$-actions, i.e., $\EE$ is a graded $\tatp$-module in the sense of Definition \ref{defn:pic_module_action}. The action of $\tatp$ on $\EE$ desends to an action of an algebraic flat connection on the projectivization $\mathbb{P}(\EE)$ of $\EE$, via which the differentiation of the $\cm$-action on $Y$ lifts to a Lie algebra $\C$-action on $\mathbb{P}(\EE)$ (see the discussion before \cite[Definition 4.3.2]{LY}). On the other hand, the $\cm$-action on $\EE$ descends to a $\cm$-action on $\mathbb{P}(\EE)$, which differentiates to a second $\C$-action. The following notion was introduced in \cite[Definition 4.3.2]{LY}.

\begin{defi} \label{defn:strong_cm}
	Let $\TT$ be a graded Picard algebroid and $\EE$ be a graded $\TT$-module over $Y$ that is $\OO_Y$-coherent. The $\cm$-action on $\EE$ is said to be \emph{strong} if the Lie algebra $\C$-action on the projectivization $\mathbb{P}(\EE)$, defined via the flat connection that is induced from the $\TT$-module structure, coincides with the $\C$-action induced by the $\cm$-action on $\EE$. In this case, we also say that the graded $\TT$-module $\EE$ is \emph{strongly graded}.
\end{defi}

The strongness condition is vacuous when $\EE$ is a line bundle. More generally, we have the following sufficient condition to ensure the strongness condition, which is a generalization of \cite[Corollary 6.11]{Yu} and follows from exactly the same argument. 


\begin{Prop}[{\cite[Cor. 6.11]{Yu}}] \label{prop:irred_strongness} 
	Let $\TT$ be a strongly $K$-equivariant graded Picard algebroid over $Y$ and $\EE$ be an $\OO_Y$-coherent graded $(K,\kappa)$-equivariant $\TT$-module over $Y$, whose analytification $\EE^{an}$ is irreducible as a $(K,\kappa)$-equivariant $\TT^{an}$-module. Then the $\cm$-action on $\EE$ is automatically strong.
\end{Prop}

Now we get to the main result of this section.

\begin{Thm}[{\cite[Thm 4.3.3]{LY}, \cite[Thm 6.7]{Yu}}] \label{thm:quan_lag}
	Let $\EE$ be an $\OO_Y$-coherent strongly graded $\tatp$-module over the Lagrangian subvariety $Y$.  Then there exists a graded quantization $\EE_\hbar$ of $\EE$ which gives rise to the given graded $\tatp$-module structure on $\EE$. Such graded quantization is unique up to $\cm$-equivariant isomorphism.
\end{Thm}

\subsection{Hamiltonian quantization}\label{SS_Ham_quant}
Let $G$ be an affine algebraic group. Suppose $G$ acts on the graded symplectic variety $X$ commuting with the $\cm$-action and preserving the symplectic form $\Omega$. Let $\alpha_G: G \times X \to X, (g, x) \mapsto g \cdot x$, denote the action map. Roughly speaking, a $G$-equivariant graded quantization of $(X, \Omega)$ is a graded quantization $\OO_\hbar$ of $(X, \Omega)$ equipped with a $G$-action by automorphisms of sheaves of $\C[[\hbar]]$-algebras that lifts the $G$-action on $X$, such that
\begin{enumerate}
	\item
	it commutes with the $\cm$-action on $\OO_\hbar$;
	\item
	the quotient homomorphism $\OO_\hbar \twoheadrightarrow \OO_\hbar / \hbar \OO_\hbar \simeq \OO_X$ is $G$-equivariant.
\end{enumerate}
In more detail, define the maps
\begin{align*}
	\pr_G: G \times X \to G, & \quad (g,x) \mapsto g \\
	\pr_X: G \times X \to X, & \quad (g,x) \mapsto x \\
	i: X \to G \times X, & \quad x \mapsto (e, x)
\end{align*}
Set
\begin{equation} \label{eq:quan_pullback}
	\alpha_G^* \OO_\hbar := \pr_G^{-1}  \OO_G[[\hbar]]  \widehat{\otimes}_{\C[[\hbar]]} \alpha_G^{-1} \OO_\hbar = \varprojlim_{k \to \infty} \pr_G^{-1}  \OO_G  \otimes_\C  \alpha_G^{-1} (\OO_\hbar / \hbar^k \OO_\hbar)
\end{equation}
and define $\pr_X^* \OO_\hbar$ similarly, so that both are sheaves of $\pr_G^{-1}  \OO_G[[\hbar]]$-algebras over $G \times X$. Here $\widehat{\otimes}$ denotes the completed tensor product. We can regard $G \times X$ as a scheme over $G$ equipped with the relative symplectic form $\alpha_G^* \Omega = \pr_X^* \Omega$ and an action of $\cm$ on the second factor. The $G$-action on $X$ induces a $\cm$-equivariant isomorphism $\beta_0^G: \alpha_G^* \OO_X  \xrightarrow{\sim} \pr_X^* \OO_X$ of sheaves of Poisson algebras, satisfying the usual compatibility conditions. Then both $\alpha_G^*\OO_\hbar$ and $\pr_X^*\OO_\hbar$ can be regarded as graded quantizations of $G \times X$ over $G$. A {\it $G$-equivariant structure} on a graded quantization $\OO_\hbar$ consists of an isomorphism
\[ \beta_\hbar^G: \alpha_G^* \OO_\hbar \xrightarrow{\sim} \pr_X^* \OO_\hbar \]
of sheaves of $\pr_G^{-1}  \OO_G[[\hbar]]$-algebras, such that
\begin{itemize}
\item
it satisfies the usual associativity condition and  $i^* \beta_\hbar^G = \Id_{\OO_\hbar}$, \item $\beta_\hbar^G$ is $\cm$-equivariant (this is a more precise version of condition (1) above),
\item and $\beta_\hbar^G \mod \hbar$ equals $\beta_0^G: \alpha_G^* \OO_X  \xrightarrow{\sim} \pr_X^* \OO_X$ (this is a more precise version of condition (2)).
\end{itemize}

Given a $G$-equivariant graded quantization $(\OO_\hbar, \beta_\hbar^G)$, we can assign an action of the Lie algebra $\g$ on $\OO_\hbar$, i.e., a Lie algebra homomorphism $L_\hbar^\g : \g \to \der(\OO_\hbar)$, where $\der(\OO_\hbar) = \der_{\C[[\hbar]]}(\OO_\hbar)$ denotes the sheaf of $\C[[\hbar]]$-linear derivations of $\OO_\hbar$ (note that such derivations  are automatically continuous with respect to the $\hbar$-adic topology), as follows. Let $L_G : \g \to \Gamma(G, \T_G)$ denote the Lie algebra homomorphism generated by the left action of $G$ on itself. Then $L_G$ identifies $\g$ with the space of right invariant vector fields on $G$. Define
\begin{equation} \label{eq:inf_action_OO_h}
	L_\hbar^\g (\xi) s = i^* \langle (L_G (\xi) \boxtimes \Id) [\beta_\hbar^G (\alpha_G^*(s))] \rangle \in \send_{\C[[\hbar]]} (\OO_\hbar)
\end{equation}
for any $\xi \in \g$ and local section $s \in \OO_\hbar$, where $\alpha_G^*(s)$ stands for the local section $1 \otimes \alpha_G^{-1}(s)$ of the sheaf $ \pr_G^{-1}  \OO_G[[\hbar]]  \widehat{\otimes}_{\C[[\hbar]]} \alpha_G^{-1} \OO_\hbar = \alpha_G^* \OO_\hbar$. In other words, for any open subset $U \subset X$, $L_\hbar^\g(\xi)$ is the composite map
\begin{multline*}
	L_\hbar^\g (\xi) : \Gamma(U, \OO_\hbar) \xrightarrow{\alpha_G^*} \Gamma(\alpha_G^{-1}(U), \alpha_G^* \OO_\hbar) \xrightarrow{\beta_\hbar^G} \ \Gamma(\alpha_G^{-1}(U), \OO_G \widehat{\boxtimes} \OO_\hbar) \\ \xrightarrow{L_G(\xi) \boxtimes \Id} \Gamma(\alpha_G^{-1}(U), \OO_G \widehat{\boxtimes} \OO_\hbar) \xrightarrow{i^*} \Gamma(U, \OO_\hbar).
\end{multline*}
It is standard to check that $L_\hbar^\g (\xi) s$ lies in $\der(\OO_\hbar)$.

\begin{defi}\label{defi:Hamilt_action_quantization}
	Let $\OO_\hbar$ be a graded quantization of a graded symplectic variety of weight $d$. A \emph{(graded) quantum Hamiltonian $G$-action} on $\OO_\hbar$ consists of
	\begin{itemize}
		\item
		a $G$-equivariant structure $\beta_\hbar^G$ on $\OO_\hbar$ in the above sense;
		\item
		a $G$-equivariant $\C$-linear map $\Phi_\hbar: \g \to \Gamma(X, \OO_\hbar)$	whose image lies in the degree $d$ part of $ \Gamma(X, \OO_\hbar)$,
	\end{itemize}
	such that the adjoint action of $\hbar^{-1} \Phi_\hbar (\xi)$ on $\OO_\hbar$, $u \mapsto \hbar^{-1} [\Phi_\hbar(\xi), u ]$, for any $\xi \in \g$, agrees with the infinitesimal action $L_\hbar^\g(\xi)$ on $\OO_\hbar$ induced by the $G$-equivariant structure $\beta_\hbar^G$ as above. The map $\Phi_\hbar$ is called a {\it quantum comoment map}.
\end{defi}

Note that the definition above implies that $ [\Phi_\hbar(\xi), \Phi_\hbar(\zeta)] = \hbar \Phi_\hbar ([\xi, \zeta]_\g)$, $\forall \, \xi, \zeta \in \g$. In other words, the composite map $\Phi_\hbar: \g \to \Gamma(X, \OO_\hbar) \xrightarrow{\hbar^{-1} \cdot} \hbar^{-1} \Gamma(X, \OO_\hbar)$ is a Lie algebra homomorphism, where the Lie bracket of $\hbar^{-1} \Gamma(X, \OO_\hbar)$ is given by taking commutators. Therefore $\Phi_\hbar$ induces a $\cm$-equivariant algebra homomorphism $\Phi_\hbar: \U_\hbar \g \to \Gamma(X, \OO_\hbar)$, where $\U_\hbar \g$ is the Rees algebra of  $U(\g)$
with respect to the PBW filtration.

For any Hamiltonian $G$-action $(\beta_\hbar^G, \Phi_\hbar)$ on a graded quantization $\OO_\hbar$, set $\phi$ to be the composition $\g \xrightarrow{\Phi_\hbar} \Gamma(X, \OO_\hbar) \to \Gamma(X, \OO_\hbar / \hbar \OO_\hbar) \simeq \C[X] $, then $\phi$ is dual to a classical moment map $X \to \g^*$, which is $G \times \cm$-equivariant when $\g^*$ is defined to be of weight $d$ and is equipped with the coadjoint action of $G$. We say that $\Phi_\hbar$ lifts $\phi$ and $(\OO_\hbar, \beta_\hbar^G, \Phi_\hbar)$ is a Hamiltonian quantization of the Hamiltonian $G$-space $(X, \Omega, \phi)$.

Now suppose $K$ is an algebraic subgroup of $G$ with Lie algebra $\kf$. Suppose further that  $Y$ is a closed Lagrangian subvariety in $X$ stable under the $K \times \cm$-action. In addition, we assume that the image of $\mathfrak{k}$ under $\phi$ lies in the ideal subsheaf $\J_Y \subset \OO_X$ of $Y$. Then for any $G$-Hamiltonian quantization $(\OO_\hbar,\Phi_\hbar)$ of $(X, \Omega)$, the restriction of the $G$-action to $K$ preserves the ideal subsheaf $\J_{Y,\hbar} \subset \OO_\hbar$ defined in Section \ref{subsec:lag_pic} and hence induces a weak $K$-action on the sheaf $\hbar^{-1}(\J_{Y,\hbar} / \J_{Y,\hbar}^2)$ over $Y$, which is compatible with both the two $\OO_Y$-module structures. Moreover, the restriction $\Phi_\hbar|_{\kf} : \kf \to \Gamma(X, \OO_\hbar)$ of the quantum comoment map has its images lying in $\Gamma(X, \J_{Y,\hbar})$. The composition of $\Phi_\hbar|_{\kf}$ with the quotient map $\J_{Y,\hbar} \twoheadrightarrow \J_{Y,\hbar} / \J_{Y,\hbar}^2$ induces strong $K$-actions on both $\tat$ and $\tatp$ which lift the weak $K$-actions and hence makes them into strongly $K$-equivariant graded Picard algebroids on $Y$. 

The following definition is the sheaf theoretic counterpart to Definition \ref{eq:HC_equiv}.

\begin{defi}\label{defi:equiv_quant_modules}
	With the assumptions and notations above, suppose $\EE$ is a $\cm$-equivariant vector bundle over $Y$. A {\it $(K,\kappa)$-equivariant graded Hamiltonian quantization} of $(Y,\EE)$ is a graded quantization $\EE_\hbar$ of $\EE$ equipped with a $\C[[\hbar]]$-linear $K$-action that lifts the $K$-action on $Y$, such that
	\begin{enumerate}
		\item the $K$-action commutes with the $\cm$-action on $\EE_\hbar$;
		\item the structure morphism $\OO_\hbar \otimes \EE_\hbar \to \EE_\hbar$ is $K$-equivariant;
		\item the differential of the $K$-action on $\EE_\hbar$ is given by
		\[ L_\hbar^\kf(\xi) m = \hbar^{-1} \Phi_\hbar(\xi) \cdot m - \langle\kappa,x\rangle m,\]
		for any $\xi \in \kf$ and $m \in \EE_\hbar$.
	\end{enumerate}
\end{defi}

Again we can unpackage the definition of a $K$-action on $\EE_\hbar$ as in the case of $G$-action on $\OO_\hbar$ as follows. Let $\alpha_K, \, \wt{\pr}_X : K \times X \to X$ denote the restriction of the morphisms $\alpha_G, \, \pr_X : G \times X \to X$ to $K \times X$. Then $\alpha_K^{-1}(Y) = \wt{\pr}_X^{-1}(Y) = K \times Y$ can be regarded as a closed relative Lagrangian subvariety of the relative symplectic variety $K \times X$ over $K$. The sheaves $\alpha_K^* \EE_\hbar$ and $\wt{\pr}_X^* \EE_\hbar$, defined similarly as in \eqref{eq:quan_pullback}, are naturally modules over the sheaves of algebras $\alpha_K^* \OO_\hbar$ and $\wt{\pr}_X^* \OO_\hbar$ respectively, therefore are quantizations of $\alpha_K^* \EE$ and $\wt{\pr}_X^* \EE$ respectively, thanks to the isomoprhism $\beta_\hbar^G |_{K \times X}: \alpha_K^* \OO_\hbar \xrightarrow{\sim} \wt{\pr}_X^* \OO_\hbar$. A $K$-equivariant structure on $\EE_\hbar$ then consists of a $\cm$-equvariant isomorphism $\beta^K_\hbar: \alpha_K^* \EE_\hbar \xrightarrow{\sim} \wt{\pr}_X^* \EE_\hbar$ of sheaves, such that it preserves the module structures and is compatible with $\beta_\hbar^G |_{K \times X}: \alpha_K^* \OO_\hbar \xrightarrow{\sim} \wt{\pr}_X^* \OO_\hbar$. We can define the differentiation of the $K$-action on $\EE_\hbar$ in a similar way as in \eqref{eq:inf_action_OO_h} to get a map $L_\hbar^\kf: \kf \to \send_{\C[[\hbar]]}(\EE_\hbar)$, such that
\[ L_\hbar^\kf (\xi) ( s \cdot m) = L_\hbar^\g(\xi) (s) \cdot m + s \cdot L_\hbar^\kf(\xi) ( m ), \quad \forall \, s \in \OO_\hbar, \, m \in \EE_\hbar. \]

It is clear that, if $\EE_\hbar$ is a $(K,\kappa)$-equivariant graded quantization of $\EE$, then such a structure induces an algebraic $K$-action on $\EE$ as an $\OO_X$-module, such that it commutes with the $\cm$-action on $\EE$. Moreover, the $K$-action on $\EE$ and the strong $K$-action on $\tatp$ mentioned above make $\EE$ into a $(K, \kappa)$-equivariant $\tatp$-module.

Conversely, the following theorem serves as an equivariant version of Theorem \ref{thm:quan_lag}. It is analogous to \cite[Theorem 6.14]{Yu} (see also \cite[Theorem 6.2.5]{LY}), which addresses the case $\kappa = 0$, with the proof requiring minimal modifications. However, it is worth noting that we will apply Theorem \ref{thm:quan_lag_equiv} exclusively in the case $\kappa = 0$ when proving Theorem \ref{thm:non-normal}.

\begin{Thm}\label{thm:quan_lag_equiv}
	In the setting above,let $\EE$ be an $\OO_Y$-coherent strongly graded $(K, \kappa)$-equivariant $\tatp$-module over $Y$. Then for any graded quantization of $\EE$ (unique up to $\cm$-equivariant isomorphisms) as in Theorem \ref{thm:quan_lag}, the $K$-action on $\EE$ lifts uniquely to a $K$-action on $\EE_\hbar$ that makes $\EE_\hbar$ into a $(K, \kappa)$-equivariant Hamiltonian graded quantization of $\EE$. 
\end{Thm}

\section{W-algebras and restriction functors}
W-algebras are quantizations of Slodowy slices in $\g^*$. They were introduced
in \cite{Premet_slice}, alternative definitions and characterizations were given
in \cite{Wquant},\cite{HC} and other papers by the first named author.
In \cite[Section 6.1]{Wdim} the first named author constructed a restriction functor
for HC $(\U,\theta)$-modules that is going to play a central role in this paper.

The goal of this section is to recall these constructions. We will also relate the
restriction functor for Harish-Chandra modules with some constructions from
Section \ref{S_qclass_quant}. In particular, we will see that one can use the restriction
functor to recover some results from \cite{LY}.

Throughout the section $G$ is a semisimple algebraic group, $\g$ is its Lie algebra, and
$\sigma$ is an involution of $\g$. Let $\theta:=-\sigma$ be the corresponding anti-involution
of $\g$.

\subsection{Slodowy slices}\label{SS_Sl_sl}
Fix a nilpotent element $\chi\in \g^*$ with $\theta(\chi)=\chi$.
Let $\Orb$ denote its $G$-orbit.

For a $G$- and $\sigma$-invariant orthogonal form
$(\cdot,\cdot)$ on $\g$, we get an identification  $\g\cong \g^*$.
Let $e\in \g$ denote the image of $\chi$ under this identification.
We include $e$ into an $\slf_2$-triple $(e,h,f)$ that is \emph{normal}, i.e.,
$\sigma(h)=h, \sigma(f)=-f$. Then we can consider the
Slodowy slice $e+\mathfrak{z}_{\g}(f)$. Let $S$ be its image in
$\g^*$, also to be referred to as the Slodowy slice. Note that
$S$ is a transverse slice to $\Orb$. Note also that it is $\theta$-stable.

Let us describe two reductive group actions on $S$.
Let $Q:=Z_G(e,h,f)$.
This is a maximal reductive subgroup of the stabilizer $G_\chi$. Note that
$Q$ acts on $S$. Also we have a $\C^\times$-action on $S$ often called the {\it Kazhdan action}. Namely, let
$z\mapsto z^h$ be the one-parameter subgroup of $G$ corresponding to
the semisimple element $h$. Then $z \mapsto z^{-2} z^h$
and preserves $S$. Moreover, the induced grading on $\C[S]$ is positive.
We note that $\C^\times$ commutes with $\sigma$, while $\sigma$ naturally
acts on $Q$ and for $s\in S, q\in Q$, we have $\sigma(qs)=\sigma(q)\sigma(s)$.

The algebra $\C[S]$ is Poisson with bracket of degree $-2$, see, e.g., \cite[Section 2]{GG}. The action of
$Q$ on $S$ is Hamiltonian.

\subsection{W-algebra}\label{SS_W_alg} The algebra $\C[S]$ admits a quantization, called a
finite W-algebra. We recall a construction following \cite[Sections 2.1,2.2]{W_prim}.

Recall that we write $\U$ for the universal enveloping algebra of $\g$.
Form the Rees algebra $\U_\hbar$ of $\U$. The element $\chi$ defines a
homomorphism $\C[\g^*]\rightarrow \C$. Consider its composition
with $\U_\hbar\twoheadrightarrow \C[\g^*]$. We consider the completion
$\U_\hbar^{\wedge_\chi}$ with respect to the kernel of this composition. The completion comes with an action of
\begin{equation}\label{eq:group_acting} \C^\times \times \left(Q\rtimes (\Z/2\Z)\right)
\end{equation}
Here the action of $\Z/2\Z$ is induced by
$\theta$ on $\g$ and sends $\hbar$ to $\hbar$, this is a $\C[\hbar]$-linear
anti-involution. The group $Q$ acts by $\C[\hbar]$-linear automorphisms and
this action is Hamiltonian: the quantum comoment map, to be denoted by $\Phi_\U$, comes from $\mathfrak{q}\hookrightarrow
\U_\hbar$. The action of $\C^\times$ is given by $z.x=z^{-2}z^hx$, and it rescales $\hbar$ via $z.\hbar := z^2\hbar$.

Set $V:=T_\chi\Orb$. This is  a symplectic vector space.
The space again comes with a natural action of
(\ref{eq:group_acting}). Let $\omega$ denote the symplectic form on
$V^*$. We can form the homogeneous Weyl algebra
$$\Weyl_\hbar:=T(V^*)[\hbar]/(u\otimes v-v\otimes u-\hbar\omega(u,v)).$$
The group (\ref{eq:group_acting}) acts on $\Weyl_\hbar$:
\begin{itemize}
\item $Q$ acts by
$\C[\hbar]$-linear automorphisms coming from the $Q$-action on $V$,
\item the generator of $\Z/2\Z$ acts by a
$\C[\hbar]$-linear anti-involution induced by its action on $V$,
\item and the $\C^\times$-action is induced from the action of $\C^\times$ on $\g$
that fixes $\chi$ (introduced in Section \ref{SS_Sl_sl}).
\end{itemize}
Note that the action of $Q$ is Hamiltonian, let us
write $\Phi_{\Weyl}$ for the corresponding quantum comoment map
$\q\rightarrow \Weyl_\hbar$.

Consider the completion $\Weyl_\hbar^{\wedge_0}$. There is an embedding
$\Weyl_\hbar^{\wedge_0}\hookrightarrow \U_\hbar^{\wedge_\chi}$ that is
equivariant with respect to (\ref{eq:group_acting}) and lifts the decomposition
$T_\chi\g^*\cong T_\chi S\oplus V$. Any two such embeddings are conjugate with
an appropriate automorphism of $\U^{\wedge_\chi}_\hbar$ that commutes with
(\ref{eq:group_acting}), see \cite[Section 2.1]{W_prim}.

We write $\Walg_\hbar'$ for the centralizer of $\Weyl_\hbar^{\wedge_0}$
in $\U_\hbar^{\wedge_\chi}$. This is a $\C[[\hbar]]$-algebra with an action of
(\ref{eq:group_acting}). We have the natural isomorphism, where $\widehat{\otimes}$
stands for the completed tensor product:
\begin{equation}\label{eq:decomp_U}
\Weyl_\hbar^{\wedge_0}\widehat{\otimes}_{\C[[\hbar]]}\Walg_\hbar'
\xrightarrow{\sim} \U_\hbar^{\wedge_\chi}.
\end{equation}
We note that $\Phi_\U(\xi)-\Phi_\Weyl(\xi)\in \Walg_\hbar'$ for all
$\xi\in \q$. Set $\Phi_\Walg:=\Phi_\U-\Phi_{\Weyl}$, this is
a quantum comoment map for the action of $Q$ on $\Walg'_\hbar$.

Let $\Walg_\hbar$ stand for the $\C^\times$-finite part of $\Walg_\hbar'$.
Since the grading on $\C[S]$ coming from the $\C^\times$-action is positive,
we see that $\Walg_\hbar'$ coincides with the completion of $\Walg_\hbar$
at $\chi$. Set $\Walg:=\Walg_\hbar/(\hbar-1)\Walg_\hbar$. This is the finite W-algebra.
It comes with a filtration and also with a Hamiltonian action of $Q\rtimes (\Z/2\Z)$
preserving the filtration. It also comes with an automorphism $\zeta$ of order $2$,
it is induced by the action of $-1\in \C^\times$ on $\Walg_\hbar$.

\subsection{Restriction of Dixmier algebras}\label{SS_Dixmier_restriction}
The construction above can be adapted to construct the restriction functor between
categories of Harish-Chandra bimodules, \cite[Section 3]{HC}. We will need it in the case of Dixmier algebras. Recall that by a Dixmier algebra one means an associative algebra $\A$ that comes equipped
with a rational $G$-action by algebra automorphisms satisfying the following two conditions
\begin{itemize}
\item
this action is Hamiltonian
\item
and that the quantum comoment map $\U\rightarrow \A$ turns $\A$ into a HC bimodule.
\end{itemize}
For example,
any filtered quantization of $\C[\Orb]$ is a Dixmier algebra, see the discussion in the end of
Section \ref{SS_quant_nilp}.

Let $\A$ be a Dixmier algebra.

\begin{Lem}\label{Lem:Dixmier_good_filtration}
There is a good algebra filtration on $\A$.
\end{Lem}

Note that this is claimed in the proof of \cite[Lemma 3.4.5]{HC}, but the proof is wrong.
We use this opportunity to correct the proof.

\begin{proof}
Take a $G$-stable subspace $V\subset \A$ that generates $\A$ as a left $\U$-module and contains
$1$. We can find a positive integer $k$ such that $V^2\subset \U_{\leqslant k}V$.
Set $\A_{\leqslant i}:=\U_{\leqslant i}1$ for $i<k$, and $\A_{\leqslant i}:=\U_{\leqslant i}1+\U_{\leqslant i-k}V$ otherwise. This filtration is $G$-stable, it is a left $\U$-module filtration, and it is good.
It remains to show that it is an algebra filtration. The inclusion $\A_{\leqslant i}\A_{\leqslant j}
\subset \A_{\leqslant i+j}$ is checked as follows:
\begin{itemize}
\item for $i<k$, this follows from the claim that we have a left $\U$-module filtration,
\item for $j<k$, this follows from the claim that we have a right $\U$-module filtration,
\item for $i,j\geqslant k$, we also need to use that $\U_{\leqslant i-k}V\U_{\leqslant j-k}V
\subset \U_{\leqslant i+j-2k}V^2\subset \U_{\leqslant i+j-k}V\subset \A_{\leqslant i+j}$.
\end{itemize}
\end{proof}


Let $\A_\hbar$ denote the Rees algebra of $\A$. We get
a graded algebra homomorphism $\U_\hbar\rightarrow \A_\hbar$.  Set
$$\A_{\hbar}^{\wedge_\chi}:=\U_\hbar^{\wedge_\chi}\otimes_{\U_\hbar}\A_\hbar.$$
This $\C[[\hbar]]$-module coincides with the completion of $\A_\hbar$ with respect to a suitable
topology, where the product is continuous. So it has a natural algebra structure making the map
$\U_\hbar^{\wedge_\chi}\rightarrow \A_\hbar^{\wedge_\chi}$ an algebra homomorphism.

Let $\underline{\A}_\hbar'$ denote the centralizer of $\Weyl_\hbar^{\wedge_0}$
in $\A_\hbar^{\wedge_\chi}$, this is a subalgebra.  By \cite[Proposition 3.3.1]{HC},
we have
\begin{equation}\label{eq:decomp_A}
\Weyl_\hbar^{\wedge_0}\widehat{\otimes}_{\C[[\hbar]]}\underline{\A}_\hbar'
\xrightarrow{\sim} \A_\hbar^{\wedge_\chi}.
\end{equation}
Then we can pass from $\underline{\A}_\hbar'$ to $\underline{\A}_\hbar$
(the $\C^\times$-locally finite part). We set $\underline{\A}:=\underline{\A}_\hbar/(\hbar-1)\underline{\A}_\hbar$.
We call $\underline{\A}$ the {\it restriction} of $\A$ to $S$ -- this is motivated
by $\gr\underline{\A}$ being the pullback of the $\C[\g^*]$-algebra $\gr\A$
to $S\subset \g^*$.

Note that the action of $Q$ on $\underline{\A}_\hbar$ is Hamiltonian: if $\Phi_{\A}$
denotes the quantum comoment map for $\A_\hbar$, then
\begin{equation}\label{eq:q_comom_A}
\Phi_{\underline{\A}}:=\Phi_{\A}-\Phi_{\Weyl}
\end{equation}
is a quantum comoment map for $\underline{\A}_\hbar$.

Now suppose that $\A$ comes with a filtration preserving anti-involution
$\theta$ that is intertwined with the anti-involution $\theta$ on $\U$
by the quantum comoment map $\U\rightarrow\A$.
Then similarly to Section \ref{SS_W_alg}, $\underline{\A}$ inherits an
anti-involution, again denoted by $\theta$.

Again, as in the case of W-algebras, $\underline{\A}$ comes with an order two
automorphism $\zeta$.

\subsection{Restriction of HC modules}\label{SS:HC_mod_restr}
Now we recall a construction of the restriction functor for HC modules.
The construction was given in \cite[Section 6.1]{Wdim} for the case of
$\A=\U$. It extends to the general setting, where $\A$ is a Dixmier algebra, in a straightforward way.
For simplicity, assume that $K$ is connected. Let $\A$ be as in the previous section.
Assume that $\gr\A$ is commutative, so
it makes sense to speak about HC $(\A,\theta)$-modules (with $d=1$).

We consider the category $\wHC^{gr}(\A,\theta)^{K,\kappa}$. A functor to be constructed depends on a choice of an irreducible component in $\Orb\cap (\g/\kf)^*$.
The intersection  $\Orb\cap (\g/\kf)^*$ is a smooth Lagrangian subvariety in
the symplectic variety $\Orb$. Every irreducible component of this subvariety is
a single $K$-orbit. Pick an irreducible component $\Orb_K$. Let $\chi$ be a point of
this orbit.

We will need to rescale the gradings.
We pick an even number $d$ such that $\frac{1}{2}dh\in \kf$ gives rise to
a one-parameter subgroup $\gamma:\C^\times\rightarrow K$. For example, for $K\subset G$,
we can take $d=2$. Then we have an action of $\C^\times$ on any $M_\hbar\in \wHC^{gr}(\A_\hbar,\theta)^{K,\kappa}$
given by $t.m=t^{di}\gamma(t)^{-1}m$ for homogeneous $m\in M_\hbar$ of degree $i$. We also consider the similar $\C^\times$-actions
on $\U_\hbar,\A_\hbar,\Walg_\hbar$, and so on. Note that now $\hbar$ has degree $d$ after rescaling. In particular, a primitive root of unity $\epsilon$ of order $d$ gives an order $d$ automorphism of $\U,\A,\Walg$, to be denoted by $\zeta$. We use this $\zeta$ to form HC modules over these algebras, when we equip them with the filtrations coming from the rescaled gradings on $\U_\hbar,\A_\hbar,\Walg_\hbar$, etc.  

Consider the category
$\wHC^{gr}(\A_\hbar,\theta)^{K,\kappa}$, where now the grading on $\A_\hbar$ is rescaled.
 We are going to define a functor
$$\bullet_{\dagger,\chi}: \wHC^{gr}(\A_\hbar,\theta)^{K,\kappa}\rightarrow \wHC^{gr}(\underline{\A}_\hbar,
\theta).$$
Here
$\underline{\A}$ is the restriction of $\A$ to $S$ defined in the previous section.

Pick $M_\hbar\in \wHC^{gr}(\A_\hbar,\theta)^{K,\kappa}$.
Consider its completion
$$M_\hbar^{\wedge_\chi}:=\A_\hbar^{\wedge_\chi}\otimes_{\A_\hbar}M_\hbar.$$
This is an object $\wHC^{gr}(\A_\hbar^{\wedge_\chi},\theta)$, where now the $\C^\times$-action
comes from $t\mapsto t^{-d}\gamma(t)$ on $S$.
We are going to produce an equivalence $$\wHC^{gr}(\A_\hbar^{\wedge_\chi},\theta)
\xrightarrow{\sim} \wHC^{gr}(\underline{\A}_\hbar^{\wedge_\chi},\theta).$$

For this consider the Lagrangian subspace $L\subset V^*$ given by $L:=(\kf.\chi)^*$
(where $L$ is embedded into $V^*$ via the decomposition $V=\kf.\chi\oplus \kf^\perp.\chi$).
Note that $L\subset \left(\U_{\hbar}^{\wedge_\chi}\right)^{-\theta}$, where $V^*$
is embedded into $\U_{\hbar}^{\wedge_\chi}$ via (\ref{eq:decomp_U}).
We write $M_\hbar^L$ for the annihilator of $L$ under the Lie algebra action.

\begin{Lem}\label{Lem:wHC_equivalence}
The assignment $M_\hbar\mapsto M_\hbar^L$ defines a category equivalence
$$\wHC^{gr}(\A_\hbar^{\wedge_\chi},\theta)
\rightarrow \wHC(\underline{\A}_\hbar^{\wedge_\chi},\theta).$$
A quasi-inverse functor is given by $N_\hbar\mapsto \C[[L,\hbar]]\widehat{\otimes}_{\C[[\hbar]]} N_\hbar$.
\end{Lem}
\begin{proof}
Recall the decomposition (\ref{eq:decomp_A}). Under this decomposition $\theta$ is the tensor product
of two involutions: the restriction of $\theta$ to $\Walg_\hbar'$ and the anti-involution on
$\Weyl_\hbar^{\wedge_0}$ that is $-1$ on $L$ and $1$ on its Lagrangian complement.
The claim is now standard, compare to \cite[Proposition 3.3.1]{HC} and Lemma \ref{Lem:wHC_formal_Weyl}.
\end{proof}

Taking the locally finite elements for $\C^\times$ gives rise to an equivalence
$$\wHC^{gr}(\underline{\A}_\hbar^{\wedge_\chi},\theta)\xrightarrow{\sim} \wHC^{gr}(\underline{\A}_\hbar,\theta).$$
The functor $\bullet_{\dagger,\chi}:\wHC^{gr}(\A_\hbar,\theta)^{K,\kappa}\rightarrow \wHC(\underline{\A}_\hbar,\theta)$ we need
is the composition
$$\wHC^{gr}(\A_\hbar,\theta)^{K,\kappa}\rightarrow \wHC^{gr}(\A_\hbar^{\wedge_\chi},\theta)\xrightarrow{\sim}
\wHC^{gr}(\underline{\A}^{\wedge_\chi}_\hbar,\theta)\xrightarrow{\sim} \wHC^{gr}(\underline{\A}_\hbar,\theta),$$
where the first functor is the completion functor. Note that $\bullet_{\dagger,\chi}$
is exact and  $\C[\hbar]$-linear by the construction. Similarly to
\cite[Section 3.4]{HC}, it descends to an exact functor
$$\HC(\A,\theta,\zeta)^{K,\kappa}\rightarrow \HC(\underline{\A},\theta,\zeta).$$


Let us examine what happens with the equivariance. Let $K_Q := Z_K(e,h,f)$.
By the  construction of $\bullet_{\dagger,\chi}$, this functor lifts to a functor to the category
of weakly $K_Q$-equivariant modules. Now we proceed to the strong equivariance.

\begin{defi}\label{defi:half_character}
Let $\rho_\omega$ denote one half of the character of the
$K_Q$-action on $\Lambda^{top}L$, where $L=(\kf.\chi)^*$.
\end{defi}
Set
\begin{equation}\label{eq:shifted_character}
\kappa_Q:=\kappa|_{\kf_Q}-\rho_\omega.
\end{equation}

\begin{Lem}\label{Lem:restriction_equivariance}
The functor $\bullet_{\dagger,\chi}$ sends
$\wHC^{gr}(\A_\hbar,\theta)^{K,\kappa}$ to
$\wHC^{gr}(\underline{\A}_\hbar,\theta)^{K_Q,\kappa_Q}$.
\end{Lem}
\begin{proof}
Pick $M_\hbar\in  \wHC(\A_\hbar,\theta)^{K,\kappa}$ and let $ N_\hbar\in \wHC(\underline{\A}_\hbar, \theta)$ denote its image. For $\xi\in \kf_Q$, we write
$\xi_M,\xi_N,\xi_L$ for the operators defined by $\xi$ on $M_\hbar^{\wedge_\chi},
N_\hbar^{\wedge_\chi}$ and $\C[[L,\hbar]]$ via the differentials of the $K_Q$-actions. From the $K_Q$-equivariant isomorphism
$M_{\hbar}^{\wedge_\chi}\cong \C[[L,\hbar]]\widehat{\otimes}_{\C[[\hbar]]}N_\hbar^{\wedge_\chi}$
we get
\begin{equation}\label{eq:vect_fields}
\xi_{M}=\xi_L+\xi_N.
\end{equation}
Note that $\Phi_{\Weyl}(\xi)$ acts on $\C[[L,\hbar]]$ as $\xi_L+\langle\rho_\omega,\xi\rangle$
so from (\ref{eq:vect_fields}) and (\ref{eq:q_comom_A}) we get the following equality of
operators on $M_\hbar^{\wedge_\chi}$ (via the Lie algebra actions)
$$\Phi_{\A}(\xi)-\langle\kappa,\xi\rangle=\Phi_\Weyl(\xi)-\langle\rho_\omega,\xi\rangle+
\Phi_{\underline{\A}}(\xi)-\langle\kappa-\rho_\omega,\xi\rangle.$$
The claim of the lemma follows.
\end{proof}


\subsection{Properties of restriction functor}\label{SS_Prop_restr}
As mentioned in \cite[Section 6.1]{Wdim}, most of the properties of $\bullet_{\dagger,\chi}$
mentioned below are proved completely similarly to their counterparts for
Harish-Chandra bimodules from \cite[Sections 3.3,3.4,4]{HC}.

As we have mentioned, the functor $$\bullet_{\dagger,\chi}:
\HC(\A_\hbar,\theta,\zeta)\rightarrow\HC(\underline{\A}_\hbar,\theta,\zeta)$$ is exact. In fact,
up to an isomorphism, the functor only depends on $\Orb_K$, not on the choice of
$\chi$. An isomorphism is given by an element of $K$ that sends one point to the other.

For a closed $K$-stable subvariety $Y\subset (\g/\kf)^*$ we write
$\HC_Y(\A,\theta,\zeta)$ for the full subcategory of all modules whose associated
variety is contained in $Y$. By the construction, for $M\in \HC_Y(\A,\theta,\zeta)^{K,\kappa}$,
the module $M_{\dagger,\chi}$ is supported on $S\cap Y$, where $S$ is the Slodowy slice from
Section \ref{SS_Sl_sl}. In particular, if
$\overline{\Orb}_K$ is an irreducible component in $Y$, then $\dim M_{\dagger,\chi}<\infty$.
And if $\Orb_K\not \subset Y$, we have $M_{\dagger,\chi}=\{0\}$.

Now suppose that (\ref{eq:codim_condition_K}) holds.

Let $\operatorname{V}(\A)$ denote the associated variety of $\A$. We further assume that
\begin{equation}\label{eq:assoc_var_condition}
\operatorname{V}(\A)=\overline{\Orb}.
\end{equation}

These conditions will be assumed until the end of the section.
In this case, it is known, \cite[Theorem 1.3]{Vogan}, that if $\overline{\Orb}_K$
is an irreducible component in the associated variety of an irreducible
HC module, then it coincides with the associated variety (recall that we have assumed that $K$ is connected). Consider
the Serre quotient category
\begin{equation}\label{eq:quotient_cat}
\HC_{\Orb_K}(\A,\theta,\zeta)^{K,\kappa}:=\HC_{\overline{\Orb}_K}(\A,\theta,\zeta)^{K,\kappa}/\HC_{\partial\Orb_K}(\A,\theta,\zeta)^{K,\kappa}.
\end{equation}
The functor $\bullet_{\dagger,\chi}$ factors through $\HC_{\Orb_K}(\A,\theta,\zeta)^{K,\kappa}$ and
maps to the category $(\underline{\A},\zeta)\operatorname{-mod}^{K_Q,\kappa_Q}_{fin}$ of
finite dimensional $(K_Q,\kappa_Q)$-equivariant $\Z/d\Z$-graded $\underline{\A}$-modules.
The functor has the
right adjoint
$$\bullet^{\dagger,\chi}:(\underline{\A},\zeta)\operatorname{-mod}^{K_Q,\kappa_Q}_{fin}
\rightarrow \HC_{\overline{\Orb}_K}(\A,\theta,\zeta)^{K,\kappa}.$$
The functor $\bullet^{\dagger,\chi}$  between the weakly HC categories is constructed as in the case of Harish-Chandra bimodules in
\cite[Section 3.3]{HC}, see \cite[Proposition 3.3.4]{HC} there and the preceding discussion.
Similarly to \cite[Section 3.4]{HC}, it carries to the HC categories.

Further  properties of $\bullet_{\dagger,\chi},\bullet^{\dagger,\chi}$
are summarized in the following proposition.

\begin{Prop}\label{Prop:adj_properties}
Suppose that (\ref{eq:codim_condition_K}) and (\ref{eq:assoc_var_condition}) hold. Then we have the following:
\begin{enumerate}
\item The kernel and the cokernel of the adjunction unit morphism
$M\rightarrow (M_{\dagger,\chi})^{\dagger,\chi}$ are supported on $\partial\Orb_K$.
In particular, $\bullet_{\dagger,\chi}: \HC_{\Orb_K}(\A,\theta,\zeta)^{K,\kappa}\rightarrow (\underline{\A},\zeta)\operatorname{-mod}_{fin}^{K_Q,\kappa_Q}$ is a full embedding.
\item The image of $\HC_{\Orb_K}(\A,\theta,\zeta)^{K,\kappa}$ in $(\underline{\A},\zeta)\operatorname{-mod}_{fin}^{K_Q,\kappa_Q}$
is closed under taking subquotients.
\item If $\operatorname{codim}_{\overline{\Orb}_K}(\partial\Orb_K)>2$,
then $$\bullet_{\dagger,\chi}: \HC_{\Orb_K}(\A,\theta,\zeta)^{K,\kappa}\xrightarrow{\sim} (\underline{\A},\zeta)\operatorname{-mod}_{fin}^{K_Q,\kappa_Q}.$$
\end{enumerate}
\end{Prop}
\begin{proof}
(1) is proved in the same way as \cite[Proposition 3.4.1(5)]{HC}. (2) is proved as
\cite[Theorem 4.1.1]{HC}. (3) is proved as \cite[Theorem 4.4]{reg}.
\end{proof}


To finish this section, we want to understand the interaction between $\bullet_{\dagger,\chi}$
and the tensor products with HC bimodules. Here it will be convenient for us to work
with modules and bimodules over $\U$. Let $\B$ be a HC $\U$-bimodule and $M$ be a HC
$(\g,K)$-module. We have the restriction functor $\B\mapsto \B_\dagger: \HC^G(\U)
\rightarrow \HC^Q(\Walg)$, \cite[Section 3.4]{HC}.

\begin{Lem}\label{Lem:restriction_compatibilty}
We have a functorial $K_Q$-linear isomorphism of $\Walg$-modules $\B_\dagger\otimes_\Walg M_{\dagger,\chi}\xrightarrow{\sim}(\B\otimes_\U M)_{\dagger,\chi}$.
\end{Lem}
\begin{proof}
Pick good filtrations on $\B$ and $M$. We have the induced tensor product filtration on
$\B\otimes_{\U} M$, which is also good. The completion functor is tensor so we have a
natural isomorphism
\begin{equation}\label{eq:tensor_product_iso}\B_\hbar^{\wedge_\chi}\otimes_{\U_\hbar^{\wedge_\chi}}M_\hbar^{\wedge_\chi}
\xrightarrow{\sim} (\B_\hbar\otimes_{\U_\hbar}M_\hbar)^{\wedge_\chi}.
\end{equation}
Let
$\underline{\B}_\hbar'$ denote the centralizer of $\Weyl_\hbar^{\wedge_0}$
in $\B_\hbar^{\wedge_\chi}$. The left hand
side is
\begin{align*}
\left(\Weyl_\hbar^{\wedge_0}\widehat{\otimes}\underline{\B}_\hbar'\right)
\widehat{\otimes}_{\Weyl_\hbar^{\wedge_0}\widehat{\otimes}_{\C[[\hbar]]}
\Walg_\hbar^{\wedge_\chi}}(\C[[L,\hbar]]\widehat{\otimes}_{\C[[\hbar]]}N_\hbar^L)
\xrightarrow{\sim} \C[[L,\hbar]]\widehat{\otimes}_{\C[[\hbar]]}
(\underline{\B}_\hbar'\widehat{\otimes}_{\C[[\hbar]]}N_\hbar^L).\end{align*}
This isomorphism combined with (\ref{eq:tensor_product_iso}) and the constructions
of $\bullet_{\dagger,\chi}$ (Section \ref{SS:HC_mod_restr}) and
$\bullet_\dagger$ (\cite[Section 3.4]{HC}) yields
the claim of the lemma.
\end{proof}

\begin{Rem}\label{Rem:restriction_compatibility}
The lemma generalizes to the weakly HC (bi)modules. It also generalizes to Hom's instead of tensor
products (this will be needed in a subsequent paper by the first named author). Let $B_\hbar$
be a graded weakly HC $\U_\hbar$-bimodule, while $M_\hbar\in \wHC^{gr}(\U_\hbar,\theta)^{K,\kappa}$.
Then $\Hom_{\U_\hbar}(B_\hbar,M_\hbar)$ has a natural structure of an object in
$\wHC^{gr}(\U_\hbar,\theta)^{K,\kappa}$, and we have a natural isomorphism
$\Hom_{\U_\hbar}(B_\hbar,M_\hbar)_{\dagger,\chi}\xrightarrow{\sim}\Hom_{\Walg_\hbar}(B_{\hbar,\dagger},
M_{\hbar,\dagger,\chi})$. The proof repeats that of the lemma.
\end{Rem}

\subsection{The case of quantizations of nilpotent orbits}\label{subsec:W-algebra_nilpotent}
Now we would like to treat the case when the Dixmier algebra $\A$ in
Section \ref{SS_Dixmier_restriction} is a quantization of
$\C[\Orb]$ so that (\ref{eq:assoc_var_condition}) holds. We assume that (\ref{eq:codim_condition}) holds
as well. Clearly, in this case $\underline{\A}=\C$.

In what follows we need to understand how the quantum comoment map
$\Phi_{\underline{\A}}:\q\rightarrow \underline{\A}$ behaves. Since $\underline{\A}=\C$,
the map $\Phi_{\underline{\A}}$ is just an element of $\q^{*Q}$.

Recall that the quantizations of $\C[\Orb]$ are classified by $H^2(\Orb,\C)$,
Section \ref{SS_quant_nilp}.
Note that $(\q^*)^Q$  is identified with $H^2(\Orb,\C)$ as a vector space, Lemma \ref{Lem:G_orbit_cohomology}.
On the lattice $\Hom(Q,\C^\times)\subset \q^{*Q}$ this identification sends
a character of $Q$ to the first Chern class of the corresponding $G$-equivariant line  bundle
on $\Orb$.

\begin{Lem}\label{Lem:quantum_comom_period}
Let $\lambda \in H^2(\Orb, \C)$ be the period of $\A$, see
Proposition \ref{prop:graded}. Then $\Phi_{\underline{\A}}=\lambda \in  \q^{*Q}$.
\end{Lem}
\begin{proof}
Let us write $\A_\lambda$ for the quantization $\A$ with period $\lambda$. Also recall the universal
quantization $\A_{\param}$, where $\param=H^2(\Orb,\C)$, Section \ref{SS_quant_nilp}.  We still have the quantum comoment map $\g\rightarrow \A_{\param}$,
which gives rise to a linear $Q$-equivariant map $\q\rightarrow \underline{\A}_{\param}=\C[\param]$. By degree reasons, the image of this map lies in $\q^{Q}\oplus \C$. The claim of the lemma
boils down to the claim that this map is the $Q$-equivariant  projection
$\q\twoheadrightarrow \q^{Q}$.

The first observation is that the $\q^Q$-part of $\q\rightarrow \q^Q\oplus \C$ is the projection.
To see this, one can consider the quasi-classical limit of the situation. Take a line
bundle $\mathcal{L}$ on $\Orb$, it corresponds to a character of $Q$, say $\mu$. Let $\A^\Orb_\lambda$ denote
the microlocalization of $\A_\lambda$ to $\Orb$.
We have $H^i(\Orb,\mathcal{O}_{\Orb})=0$ for $i=1,2$, (\ref{eq:orbit_coh_cohom_vanish}).
It follows that $\mathcal{L}$ deforms to a unique $G$-equivariant $\A^\Orb_{\lambda+\mu}$-$\A^\Orb_{\lambda}$-module, denote it by $\mathcal{L}'$, compare to
\cite[Proposition 5.2]{BPW} and its proof. Since $\operatorname{codim}_{\overline{\Orb}}\partial\Orb\geqslant 4$,  the global sections
$\Gamma(\mathcal{L}')$ is a HC $\A_{\lambda+\mu}$-$\A_\lambda$-bimodule whose microlocalization
back to $\Orb$ is $\mathcal{L}'$, compare to the proof of \cite[Theorem 4.4]{reg}.
Apply the restriction functor $\bullet_\dagger$ for Harish-Chandra bimodules, \cite[Section 3.4]{HC},
to $\Gamma(\mathcal{L}')$. We land in the category of $Q$-equivariant $\Walg$-bimodules. By \cite[Lemma 3.3.2]{HC}, we have that $\gr \Gamma(\mathcal{L}')_\dagger$
is the pullback of $\gr\mathcal{L}'$ to $S\cap \Orb$. So $\Gamma(\mathcal{L}')_\dagger$ is a one-dimensional
space and $Q$ acts via $\mu$. On the other hand, the adjoint action of $\mathfrak{q}\subset \Walg$
is by $\Phi_{\underline{\A}_{\lambda+\mu}}-\Phi_{\underline{\A}_\lambda}$. This difference is $\mu$.
Equivalently, the $\q^Q$-part of the map $\q\rightarrow \q^Q\oplus \C$ is the projection.
%

Finally, we need to show that the $\C$-part of the map $\q\rightarrow \q^Q\oplus \C$
is zero.  This is equivalent to the claim that for the canonical quantization $\A_0$, we have
$\Phi_{\underline{\A}_0}=0$. For this, consider the Rees algebra $R_\hbar(\A_0)$.
It comes with a parity anti-automorphism, denote it by $\varsigma$, see
\cite[Corollary 2.3.3]{quant_iso}. The anti-automorphism $\varsigma$ is
characterized by the property that it is the identity modulo $\hbar$ and sends
$\hbar$ to $-\hbar$. In this situation it makes sense to speak about
the symmetrized quantum comoment map to $R_\hbar(\A_0)$, the homogeneous
map with image in the $\varsigma$-invariants,  see \cite[Section 5.4]{quant_iso}.
Note that $\Phi_{\A_0}:\g
\rightarrow R_\hbar(\A_0)$ is symmetrized. We can assume that the image of the embedding
$V^*\hookrightarrow R_\hbar(\A_0)^\wedge$ is $\varsigma$-stable. So $\underline{\A}_{0,\hbar}$
inherits the anti-automorphism $\varsigma$. The quantum comoment map $\Phi_{\mathbb{A}}$
is symmetrized. From here and (\ref{eq:q_comom_A}) we see that
$\Phi_{\underline{\A}_0}:\q\rightarrow \underline{\A}_{0, \hbar}=\C$
is also symmetrized. So,   $\Phi_{\underline{\A}_0}=0$.
\end{proof}

Here is an immediate corollary of Lemma \ref{Lem:quantum_comom_period}.

\begin{Cor}\label{Cor:equiv_condition}
The category $\wHC^{gr}(\underline{\A}_\hbar,\theta)^{K_Q,\kappa_Q}$ is nonzero if and only if $\lambda|_{\kf_Q}-\kappa_Q$
integrates to a character of $K_Q^\circ$. This is the character by which $K_Q^\circ$
acts on every object in $\wHC^{gr}(\underline{\A}_\hbar,\theta)^{K_Q,\kappa_Q}$.
\end{Cor}

Now note that $(\underline{\A},\zeta)\operatorname{-mod}^{K_Q,\kappa_Q}$ is a semisimple category:
by Corollary \ref{Cor:equiv_condition} this is the category of $\Z/d\Z$-graded $K_Q$-modules, where
$\mathfrak{k}_Q$ acts by $\lambda|_{\kf_Q}-\kappa_Q$.
Recall that (\ref{eq:codim_condition}) implies (\ref{eq:codim_condition_K}) for each $\Orb_K$.
Proposition \ref{Prop:adj_properties} implies that $\HC_{\Orb_K}(\A,\theta,\zeta)^{K,\kappa}$
is the full subcategory that consists of direct sums of certain simple objects in
$(\underline{\A},\zeta)\operatorname{-mod}^{K_Q,\kappa_Q}$. So to determine $\HC_{\Orb_K}(\A,\theta,\zeta)^{K,\kappa}$ we just need to understand which simple
objects in $(\underline{\A},\zeta)\operatorname{-mod}^{K_Q,\kappa_Q}$ lie in the image of $\HC_{\Orb_K}(\A,\theta,\zeta)^{K,\kappa}$. The following lemma is
a straightforward corollary of Proposition \ref{Prop:adj_properties}.

\begin{Lem}\label{Lem:image_description}
Let $N$ be a simple object in $(\underline{\A},\zeta)\operatorname{-mod}^{K_Q,\kappa_Q}$. Then the following claims
are equivalent:
\begin{enumerate}
\item $N$ lies in the image of $\bullet_{\dagger,\chi}$.
\item $N^{\dagger,\chi}\neq 0$.
\item $(N^{\dagger,\chi})_{\dagger,\chi}\xrightarrow{\sim} N$.
\end{enumerate}
\end{Lem}

We finish this section with two remarks.

\begin{Rem}\label{Rem:removing_zeta}
Note that $\zeta$ acts trivially on both $\A$ and $\underline{\A}=\C$. It follows that $\HC(\A,\theta,\zeta)$
decomposes as the direct sum of $d$ copies of $\HC(\A,\theta):=\HC(\A,\theta,1)$ (where $\zeta$
acts on an object in $i$th copy by $\exp(2\pi\sqrt{-1}i/d)$). The same is true of $(\underline{\A},\zeta)\operatorname{-mod}^{K_Q,\kappa_Q}$.
Therefore, it makes sense to define $M_{\dagger,\chi}$ for $M\in \HC(\A,\theta,1)$
and $N^{\dagger,\chi}$ for $N\in \underline{\A}\operatorname{-mod}^{K_Q,\kappa_Q}$.
Below we also write $\operatorname{Rep}(K_Q,\lambda|_{\kf_Q}-\kappa_Q)$ for $\underline{\A}\operatorname{-mod}^{K_Q,\kappa_Q}$.
\end{Rem}

\begin{Rem}\label{Rem:distinguished_filtration}
In fact, as in \cite[Section 3.4]{HC}, $\bullet^{\dagger,\chi}$ comes from
the functor $$\bullet^{\dagger,\chi}:
\underline{\A}_\hbar\operatorname{-mod}^{K_Q,\kappa_Q,gr}_{fin}
\rightarrow \wHC_{\overline{\Orb}_K}^{gr}(\A_\hbar,\theta)^{K,\kappa},$$
where in the source we have graded modules finitely generated over $\C[\hbar]$.
This shows that for $N\in \underline{\A}\operatorname{-mod}^{K_Q,\kappa_Q}_{fin}$,
the HC module $N^{\dagger,\chi}$ comes with a distinguished filtration: we equip
$N$ with the trivial filtration in degree $0$. Then $N^{\dagger,\chi}=(R_\hbar(N))^{\dagger,\chi}/(\hbar-1)$ inherits the filtration.
We always consider this filtration on $N^{\dagger,\chi}$.
\end{Rem}

\subsection{Restriction functor vs quantizations}
The purpose of this section is relate the restriction functor in Sections \ref{SS:HC_mod_restr}-\ref{subsec:W-algebra_nilpotent} to
quantizations of vector bundles on Lagrangian subvarieties studied in Section
\ref{S_qclass_quant}. This will ultimately give us tools to examine the
equivalent conditions of Lemma \ref{Lem:image_description}.

In this section, $\Orb$ is a nilpotent orbit in $\g^*$ and $\A$ is a filtered quantization
of $\C[\Orb]$. We assume that $K$ is connected.
We do not put any conditions on the codimension of boundary at this point. Let
$\Orb_K$ denote a $K$-orbit in $\Orb\cap (\g/\kf)^*$. The goal of this section
is to produce an equivalence between $\wHC^{gr}(\underline{\A}_\hbar,\theta)^{K_Q,\kappa_Q}$
and a category of microlocal sheaves on $\Orb_K$. The functor $\bullet_{\dagger,\chi}$
will become the microlocalization functor, while  $\bullet^{\dagger,\chi}$ that exists
under condition (\ref{eq:codim_condition_K}) will
become the global sections functor.

To relate  $\wHC^{gr}(\underline{\A}_\hbar,\theta)^{K_Q,\kappa_Q}$ to the category of microlocal sheaves
we will use jet bundles.

\subsubsection{Microlocal version}
In this section we will write $\hat{\A}_\hbar$ for the $\hbar$-adic completion of the
Rees algebra $\A_\hbar$ of $\A$. Let $\hat{\A}_\hbar^{\Orb}$ denote the microlocalization of
$\hat{\A}_\hbar$ to $\Orb$ viewed as an open subvariety in $\operatorname{Spec}(\C[\Orb])$.
This is a sheaf in Zariski topology.

We can talk about  weakly HC $(\hat{\A}_\hbar^\Orb,\theta)$-modules, these are sheaves in the Zariski
topology on $\Orb$ defined similarly to Definition \ref{defi:wHC}.
Denote the corresponding category by $\wHC(\hat{\A}_\hbar^\Orb,\theta)$. Note that we have
the microlocalization functor from the category $\wHC_{\overline{\Orb}_K}(\hat{\A}_\hbar,\theta)$
to the subcategory $\wHC_{\Orb_K}(\hat{\A}^\Orb_\hbar,\theta)$ of all weakly HC $\hat{\A}^\Orb_\hbar$-
modules supported (set theoretically) on $\Orb_K$.

\begin{Lem}\label{Lem:global_sections}
If (\ref{eq:codim_condition_K}) holds, then we have the right
adjoint functor $\wHC_{\Orb_K}(\hat{\A}^\Orb_\hbar,\theta)\rightarrow
\wHC_{\overline{\Orb}_K}(\hat{\A}_\hbar,\theta)$ given by taking the global sections.
\end{Lem}
\begin{proof}
Let $M_\hbar^0\in \wHC_{\Orb_K}(\hat{\A}^\Orb_\hbar,\theta)$. We need to show that
\begin{itemize}
\item[(*)]
$\Gamma(M_\hbar^0)$ is a finitely generated $\hat{\A}_\hbar$-module.
\end{itemize}
Note that $\Gamma(M_\hbar^0)$ is automatically
complete and separated in the $\hbar$-adic topology because $M_\hbar^0$ is. So it is
enough to prove (*) when $M_\hbar^0$ is killed by $\hbar$.
Note that $M_\hbar^0/\hbar M_{\hbar}^0$
is supported on $\Orb_K$ scheme theoretically, this follows from (i)
of Definition \ref{defi:wHC}.
So $M_\hbar^0$
is a $K$-equivariant coherent sheaf on $\Orb_K$, hence a vector bundle. Since $\operatorname{codim}_{\overline{\Orb}_K}\partial\Orb_K\geqslant 2$, the global sections
$\Gamma(M^0_\hbar)$ is a finitely generated module over $\C[\overline{\Orb}_K]$. This finishes
the proof.
\end{proof}

Similarly, we can talk about the category $\wHC(\hat{\A}_\hbar^\Orb,\theta)^{K,\kappa}$.
Our goal is to  establish an equivalence $\wHC_{\Orb_K}(\hat{\A}_\hbar^{\Orb},\theta)^{K,\kappa}
\xrightarrow{\sim} \wHC(\underline{\hat{\A}}_\hbar,\theta)^{K_Q,\kappa_Q}$ that intertwines
$\bullet_{\dagger,\chi}$ with the microlocalization functor (and
$\bullet^{\dagger,\chi}$ with the global section functor).

To accomplish this, we will need (quantum) jet bundles.

\subsubsection{Jet bundles}
We can form the jet bundle of $\hat{\A}_\hbar^{\Orb}$ to be denoted
by $\jet \hat{\A}_\hbar^{\Orb}$. This is pro-vector bundle on $\Orb$ whose fiber at $x\in \Orb$
is the completion $\A_{\hbar}^{\wedge_x}$. It comes equipped with a flat connection to be
denoted by $\nabla$. We recover $\hat{\A}_{\hbar}^{\Orb}$ as the subalgebra of flat sections
$(\jet \A_\hbar^{\Orb})^\nabla$. Compare to \cite[Section 3.2]{BK},\cite[Section 3.2]{1dim_quant},\cite[Section 2.6]{HC_symp}.

Now let $M_\hbar^0\in \wHC_{\Orb_K}(\hat{\A}_{\hbar}^{\Orb},\theta)$.
The algebra $\A_\hbar^{\wedge_\chi}$ is identified with the formal Weyl algebra of the symplectic vector space $T_\chi\Orb\cong T_\chi\Orb_K\oplus T^*_\chi\Orb_K$
so that the involution $\theta$ acts by $1$ on $T^*_\chi\Orb_K$
and by $-1$ on $T_\chi \Orb_K$. By Lemma
\ref{Lem:wHC_formal_Weyl} we have
\begin{equation}\label{eq:wHC_decomp}
M_\hbar^{0\wedge_\chi}\cong \C[[T_\chi\Orb_K,\hbar]]\otimes_{\C[[\hbar]]}N_\hbar.
\end{equation}

%

We want to define the jet bundle of $M_\hbar^0$.  First,
consider the  sheaf of algebras $\hat{\A}_\hbar^{\Orb}\widehat{\boxtimes} \Str_{\Orb_K}$ on $\Orb\times \Orb_K$ (the completed outer tensor product).
Let $\mathcal{J}_\Delta$ denote the ideal sheaf of the diagonal in $\Str_{\Orb}\boxtimes
\Str_{\Orb_K}$ and let $\mathcal{J}_{\Delta,\hbar}$ denote its preimage in
$\hat{\A}_\hbar^{\Orb}\widehat{\boxtimes} \Str_{\Orb_K}$. Let $\jet_{\Orb_K} \hat{\A}_\hbar$
denote the completion of  $\hat{\A}_\hbar^{\Orb}\widehat{\boxtimes} \Str_{\Orb_K}$ with respect to
this ideal. It is nothing else but the pullback of the quantum jet bundle $\jet \hat{\A}_\hbar^{\Orb}$
to $\Orb_K$. It is a $K$-equivariant pro-vector bundle of algebras on $\Orb_K$ that comes with a flat
connection. The fiber at $\chi$ is $\A_\hbar^{\wedge_\chi}$.

We can use a similar construction to define the jet bundle $\jet_{\Orb_K} M_\hbar^0$.
This is a $\jet_{\Orb_K} \hat{\A}_\hbar$-module that is equipped with a compatible flat connection.
Note that the sheaf $\jet_{\Orb_K} \hat{\A}_\hbar$ comes with an $\Str_{\Orb_K}$-linear
involution induced by the anti-involution $\theta\boxtimes \operatorname{id}$
on $\hat{\A}_\hbar^{\Orb}\boxtimes \Str_{\Orb_K}$. It will also be denoted by $\theta$. Using this
involution, we can define the notion of a weakly HC $(\jet_{\Orb_K} \hat{\A}_\hbar,\theta)$-module
similarly to what was done in \cite[Section 2.6]{HC_symp} for HC bimodules.

We now explain the details of the definition.
Define $\mathfrak{J}_\hbar\subset \jet_{\Orb_K} \hat{\A}_\hbar$ to be the preimage of the
ideal in $\jet_{\Orb_K} \mathcal{O}_\Orb$ generated by $(\jet_{\Orb_K} \mathcal{O}_\Orb)^{-\theta}$.

Consider locally finitely generated
$\jet_{\Orb_K} \hat{\A}_\hbar$-modules $\mathfrak{M}^0_\hbar$ that come equipped with a compatible
flat connection $\nabla$ and the Lie algebra action of $\mathfrak{J}_\hbar$ subject
to the following conditions:

\begin{itemize}
\item[(i)] The action map $\mathfrak{J}_\hbar\otimes \mathfrak{M}^0_\hbar\rightarrow
\mathfrak{M}^0_\hbar$ is flat (i.e. is compatible with the flat connection).
\item[(ii)] $\mathfrak{M}^0_\hbar$ is complete and separated in the $\hbar$-adic topology.
\item[(iii)] The successive quotients $\hbar^{k-1}\mathfrak{M}^0_\hbar/\hbar^k \mathfrak{M}^0_\hbar$
are the jet bundles of some vector bundles on $\Orb_K$.
\end{itemize}

Let $\wHC(\jet_{\Orb_K} \hat{\A}_\hbar,\theta)$ denote the category
of such modules $\mathfrak{M}^0_\hbar$. The morphisms are $\jet_{\Orb_K} \hat{\A}_\hbar$-
and $\mathfrak{J}_\hbar$-linear maps that respect the flat connections.

Then we have an exact functor
$$\jet: \wHC_{\Orb_K}(\hat{\A}_\hbar^{\Orb},\theta) \to \wHC(\jet_{\Orb_K} \hat{\A}_\hbar,\theta)$$
of taking the jet bundle.

\begin{Lem}\label{Lem:HC_jet}
The following claims are true:
\begin{enumerate}
\item
For $\mathfrak{M}^0_\hbar\in \wHC(\jet_{\Orb_K} \hat{\A}_\hbar,\theta)$, its submodule of flat sections, $(\mathfrak{M}^0_\hbar)^\nabla$, lies in
$\wHC_{\Orb_K}(\hat{\A}_\hbar^{\Orb},\theta)$.
\item  Moreover,
$$\bullet^\nabla: \wHC(\jet_{\Orb_K} \hat{\A}_\hbar,\theta)\rightarrow
\wHC_{\Orb_K}(\hat{\A}_\hbar^{\Orb},\theta) $$ is quasi-inverse to $\jet$.
\end{enumerate}
\end{Lem}
\begin{proof}
For (1) we just need to show that $(\mathfrak{M}^0_\hbar)^\nabla$ is locally finitely
generated. Thanks to conditions (ii), it is enough to do this when
$\mathfrak{M}^0_\hbar$ is annihilated by $\hbar$. In this case, by (iii), it is the jet bundle of a vector
bundle and the flat sections recover the vector bundle.

(2) is proved as \cite[Lemma 2.16]{HC_symp}.
\end{proof}

%

We can also define the graded analog $\wHC^{gr}(\jet_{\Orb_K} \hat{\A}_\hbar,\theta)$:
the objects $\mathfrak{M}^0_\hbar$ come with a pro-rational action of $\C^\times$ compatible
with the action on $\jet_{\Orb_K} \hat{\A}_\hbar$ and such that the connection is
$\C^\times$-invariant. A direct analog of Lemma \ref{Lem:HC_jet} holds.

\subsubsection{Equivariant version}\label{SSS_equiv_jets}
Note that the $G$-action on $\hat{\A}_\hbar^{\Orb}$ is Hamiltonian in the sense of
Definition \ref{defi:Hamilt_action_quantization}.
We now discuss the equivariance for weakly HC modules: we want to define a jet counterpart of $\wHC(\hat{\A}_\hbar^\Orb,\theta)^{K,\kappa}$, the category of $(K,\kappa)$-equivariant
objects in $\wHC(\hat{\A}_\hbar^\Orb,\theta)$, defined as in Definition \ref{defi:equiv_quant_modules}. Note that $K$ naturally acts on $\jet_{\Orb_K}\hat{\A}_\hbar^\Orb$.
For $\xi\in \kf$, we write $\xi_{\jet\A}$ for the induced derivation of
$\jet_{\Orb_K}\hat{\A}_\hbar$, $\xi_{\Orb_K}$ for the vector field on $\Orb_K$
induced by $\xi$, and $\xi_{\A}$ for the image of $\xi$ in $\A_\hbar$. Then we have
\begin{equation}\label{eq:deriv_equation}
\xi_{\jet\A}=\nabla_{\xi_{\Orb_K}}+\hbar^{-2}[\xi_{\A},\cdot].
\end{equation}
Compare to \cite[(3.1)]{1dim_quant}.

\begin{defi}\label{defi:wHC_jet_equiv}
By a {\it $(K,\kappa)$-equivariant weakly HC} $\jet_{\Orb_K}\hat{\A}_\hbar$-module
we mean a weakly $K$-equivariant object $\mathfrak{M}^0_\hbar\in \wHC(\jet_{\Orb_K}\hat{\A}_\hbar,\theta)$
satisfying the following condition:
\begin{equation}\label{eq:jet_equiv}
\xi_{\mathfrak{M}^0_\hbar}m=\nabla_{\xi_{\Orb_K}}m+\xi_{\A}.m-\langle\kappa,\xi\rangle m
\end{equation}
for all local sections $m$. The category of $(K,\kappa)$-equivariant weakly HC
$\jet_{\Orb_K}\hat{\A}_\hbar$-modules will be denoted by
$\wHC(\jet_{\Orb_K}\hat{\A}_\hbar,\Orb)^{K,\kappa}$.
\end{defi}

With this definition, the following is a straightforward consequence of
Lemma \ref{Lem:HC_jet}.

\begin{Lem}\label{Lem:HC_jet_equiv}
The functor $\jet$ is an equivalence
$$\wHC(\hat{\A}_\hbar^{\Orb},\theta)^{K,\kappa}\xrightarrow{\sim} \wHC(\jet_{\Orb_K}\hat{\A}_\hbar,\theta)^{K,\kappa}.$$
(with quasi-inverse $\bullet^{\nabla}$).
\end{Lem}

\subsubsection{Restriction to a point}\label{SSS_point_restriction}
Our goal here is to produce an equivalence
$$\wHC(\jet_{\Orb_K}\A_\hbar,\theta)^{K,\kappa}\xrightarrow{\sim} \wHC(\underline{\A}_\hbar,\theta)^{K_Q,\kappa_Q}.$$
First, we have the functor
\begin{equation}\label{eq:restr_to_pt}\wHC(\jet_{\Orb_K}\hat{\A}_\hbar,\Orb)^{K,\kappa}\rightarrow
\wHC(\A_\hbar^{\wedge_\chi},\theta)^{K_Q,\kappa|_{\kf_Q}}
\end{equation}
of restriction to $\chi$.

\begin{Lem}\label{Lem:restr_pt_equiv}
The functor (\ref{eq:restr_to_pt})  is a category equivalence.
\end{Lem}
\begin{proof}
We construct a quasi-inverse functor. Let $N_\hbar\in \wHC(\A_\hbar^{\wedge_\chi},\theta)^{K_Q,\kappa|_{\kf_Q}}$. The Lie algebra
$\mathfrak{k}_\chi$ acts on $N_\hbar$ via the Lie algebra action of $(\A_\hbar^{\wedge_\chi})^{-\theta}$
shifted by $\kappa$.
This action is compatible with that of $K_Q$ by the equivariance condition. So we get a pro-rational
action of $K_\chi$ on $N_\hbar$. So
we can form the corresponding homogeneous bundle $K\times^{K_\chi}N_\hbar$ on $\Orb_K$ with fiber $N_\hbar$ at $\chi$.
We recover the flat connection on this bundle
by (\ref{eq:jet_equiv}). We get an object of
$\wHC(\jet_{\Orb_K}\hat{\A}_\hbar,\theta)^{K,\kappa}$. The functors sending $\mathfrak{M}_\hbar^0$ to its fiber at $\chi$ and sending $N_\hbar$ to $K\times^{K_\chi}N_\hbar$ are quasi-inverse.
\end{proof}

Now recall, Section \ref{SS:HC_mod_restr}, that we have the equivalence
$$\wHC(\hat{\A}_\hbar^{\wedge_\chi},\theta)^{K_Q,\kappa|_{\kf_Q}}\xrightarrow{\sim}
\wHC(\underline{\A}_\hbar,\theta)^{K_Q,\kappa_Q}$$
given by $\C[[L,\hbar]]\otimes_{\C[[\hbar]]}\bullet$.
Composing this equivalence with
(\ref{eq:restr_to_pt}) we get the required equivalence
\begin{equation}\label{eq:interm_jet_equiv}
\wHC(\jet_{\Orb_K}\A_\hbar,\theta)^{K,\kappa}\xrightarrow{\sim}
\wHC(\underline{\A}_\hbar,\theta)^{K_Q,\kappa_Q}.
\end{equation}

\subsubsection{Conclusion}\label{SSS_conclusion_microloc}
Combining Lemma \ref{Lem:HC_jet_equiv} with (\ref{eq:interm_jet_equiv})  we conclude that we have an equivalence
\begin{equation}\label{eq:final_jet_equiv}
\wHC(\hat{\A}_\hbar^{\Orb},\theta)^{K,\kappa}\xrightarrow{\sim}
\wHC(\underline{\hat{\A}}_\hbar,\theta)^{K_Q,\kappa_Q}.
\end{equation}
By the construction of $\bullet_\dagger$ in Section \ref{SS:HC_mod_restr},  this equivalence intertwines the microlocalization functor
$\wHC(\hat{\A}_\hbar,\theta)^{K,\kappa}\rightarrow \wHC(\hat{\A}_\hbar^\Orb,\theta)^{K,\kappa}$
and the restriction functor $\bullet_{\dagger,\chi}:
\wHC(\A_\hbar,\theta)^{ K,\kappa}\rightarrow \wHC(\underline{\hat{\A}}_\hbar,\theta)^{K_Q,\kappa_Q}$.
It follows that the global section functor
$\wHC(\hat{\A}^\Orb_\hbar,\theta)^{K,\kappa}\rightarrow \wHC(\hat{\A}_\hbar, \theta)^{K,\kappa}$
(the right adjoint to the microlocalization functor)
gets intertwined with $\bullet^{\dagger,\chi}$ (the right adjoint of $\bullet_{\dagger,\chi}$).

The same works for the graded categories.

\subsubsection{Twisted local systems and quantizations}\label{SSS_tw_loc_quant}
Let $M$ be a $(K,\kappa)$-equivariant HC $(\A,\theta)$-module. Equip it with a good filtration
and form the completed Rees module $M_\hbar$. Consider the microlocalization $M_\hbar|_{\Orb}$.
The quotient $M_\hbar|_{\Orb}/ \hbar M_\hbar|_{\Orb}$ is a $K$-equivariant coherent sheaf on
$\Orb\cap (\g/\kf)^*$. Since the latter is a Lagrangian subvariety and the union of finitely
many $K$-orbits, $M_\hbar|_{\Orb}/ \hbar M_\hbar|_{\Orb}$ is a $K$-equivariant vector
bundle on $\Orb\cap (\g/\kf)^*$, and $M_\hbar|_{\Orb}$ is its graded quantization.
So $M_\hbar|_{\Orb}/(\hbar)$ is a $(K,\kappa)$-equivariant twisted local system on $\Orb^\theta$,
see the last paragraph in Section \ref{SS_Ham_quant}.

Let $V:=(M_\hbar/\hbar M_\hbar)_{\dagger,\chi}$. If the period of $\A$ is $\lambda$, then
$\kf_Q$ acts on $V$ by $\lambda|_{\kf_Q}-\kappa_Q$, see Corollary \ref{Cor:equiv_condition}.
This determines the twist of the restriction of $M_\hbar/\hbar M_\hbar$ to $\Orb_K$,
see Lemma \ref{Lem:homog_classif}. In particular, consider the situation when $\lambda=\kappa=0$.
Then this restriction has half-canonical twist. This gives an analog of \cite[Theorem 8.7]{Vogan}
in our setting.


\section{Extension}\label{sec:extn}
Here we will give sufficient conditions for an object of
$\underline{\A}_\lambda\operatorname{-mod}^{K_Q,\kappa_Q}$
to lie in the image of $\bullet_{\dagger,\chi}$.  Our results
in this section are similar to what was obtained in \cite[Section 3]{HC_symp}.

\subsection{Extension to codimension $2$} \label{subsec:extn_codim2}
Recall that we assume that $\operatorname{codim}_{\overline{\Orb}_K}\partial\Orb_K\geqslant 2$.
Let $\Orb_K^2$ denote the union of $\Orb_K$ and all codimension $2$ $K$-orbits
in $\overline{\Orb}_K$. The closure $\overline{\Orb}_K$ is taken in $\operatorname{Spec}(\C[\Orb])$
and may be different from the closure in $\overline{\Orb}$.

Set $\partial \Orb_K^2:=\overline{\Orb}_K\setminus \Orb_K^2$, this is
closed subvariety in $\overline{\Orb}_K$ of codimension at least $3$.
Let $\iota$ denote the inclusion $\Orb_K\hookrightarrow \Orb_K^2$ and
$\iota'$ denote the inclusion $\Orb_K^2\hookrightarrow \overline{\Orb}_K$.

Recall that $\hat{\A}_\hbar$ stands for the $\hbar$-adic completion of
$\A_\hbar$, and  $\hat{\A}_\hbar^\Orb$ is the microlocalization of $\hat{\A}_\hbar$ to $\Orb$.
We write $\HC^{gr}(\hat{\A}_\hbar^{\Orb},\theta)$ for the full subcategory of
$\C[\hbar]$-flat objects in $\wHC^{gr}(\hat{\A}_\hbar^{\Orb},\theta)$.

Let $\mathcal{E}_\hbar$ be an object in $\HC^{gr}(\hat{\A}_\hbar^{\Orb},\theta)$. Assume that $\mathcal{E}_\hbar$ is supported on $\Orb_K$
so that $\mathcal{E}_\hbar/\hbar \mathcal{E}_\hbar$ is a twisted local system on $\Orb_K$, see
Section \ref{SSS_tw_loc_quant}.
We want to get a sufficient
condition for $(\iota'\circ\iota)_* \mathcal{E}_\hbar$ to be coherent.
See \cite[Section 2.5]{HC_symp} for a discussion of coherent modules and the pushforwards.

To start with note that if $(\iota'\circ\iota)_* \mathcal{E}_\hbar$  is coherent,
then so is $\iota_*\mathcal{E}_\hbar$. So, as the first step, we will find a necessary and sufficient condition for $\iota_*\mathcal{E}_\hbar$ to be coherent.

Take a $K$-orbit $\Orb_K'\subset \Orb^2_K\setminus \Orb_K$ and a point $\chi'\in \Orb_K'$. Let $\underline{\chi}$
denote its image in $\kf^\perp$. Consider a normal $\slf_2$-triple $(e',h',f')$
corresponding to $\underline{\chi}$ and the $\theta$-stable  Slodowy slice $\underline{S}\subset \g^*$
associated to this triple. Let $\check{X}$ denote a connected component of the preimage of $\underline{S}$
in $X=\operatorname{Spec}(\C[\Orb])$, this is a transverse slice to $G\Orb_K'$ in $X$  and a 4-dimensional conical
symplectic singularity. Also note that $\check{X}\subset X$ is $\theta$-stable.

Let $\Orb'=G \cdot \Orb'_K$ be the $G$-saturation of $\Orb'_K$ in $\g^*$.
Set $\check{V}:=T_{\chi'}\Orb'$, this is a symplectic vector space. Set
$\breve{X}:=\check{X}\times \check{V}$. Taking the product of contracting torus
actions on $\check{X}$ and $\check{V}$ we get a contracting action on $\breve{X}$,
such that the weight of the Poisson bracket is $-2$. Also $\check{X},\check{V}$ come with
Poisson anti-involutions induced from $X$. The resulting Poisson anti-involution on
$\breve{X}$ will also be denoted by $\theta$.
We set $\check{X}^\times:=\check{X} \setminus \{0\}, \breve{X}^\times=\check{X}^\times\times \check{V}$.

Using the quantum slice construction, Section \ref{SS_Dixmier_restriction}, from $\A$ we get a quantization
of $\check{X}$ to be denoted by $\check{\A}$. Note that in Section \ref{SS_Dixmier_restriction},
the quantum slice algebra was denoted by $\underline{\A}$ but from now on we will use the notation
$\underline{\A}$ for the quantum slice associated to the open orbit $\Orb$. We let $\check{\A}_\hbar$
denote the Rees algebra of $\check{\A}$.

We form the quantization
$\breve{\A}_\hbar:=\check{\A}_\hbar\otimes_{\C[\hbar]}\Weyl_\hbar(\check{V})$ of $\breve{X}$.
Let $\breve{\A}_\hbar^\times$ (resp. $\check{\A}_\hbar^\times$) be the microlocalization of the $\hbar$-adic completion of
$\breve{\A}_\hbar$ (resp. $\check{\A}_\hbar$) to $\breve{X}^\times$ (resp. $\check{X}^\times$).

We write $X^{\wedge_{\chi'}}$ for the spectrum of the completion of
$\C[X]$ at $\chi'$. It coincides with $\breve{X}^{\wedge_0}$ and, by the construction
in Section \ref{SS_Dixmier_restriction}, we have
an isomorphism of quantizations $\A_\hbar^{\wedge_{\chi'}}\cong
\breve{\A}_\hbar^{\wedge_0}$.
Our first task is to define an object $\breve{\Ecal}_\hbar\in \HC^{gr}(\breve{\A}_\hbar,\theta)$ 
whose restriction to $\breve{X}^\times\cap \breve{X}^{\wedge_0}$ coincides with
the restriction of $\Ecal_\hbar$ to $\Orb\cap X^{\wedge_{\chi'}}$.

 Take $n\in \Z_{>0}$.
Consider the quotient $\mathcal{E}_{\hbar,n}:=\mathcal{E}_\hbar/\hbar^{n+1}\mathcal{E}_\hbar$. This is
a weakly HC $(\hat{\A}_\hbar^\Orb,\theta)$-module. Since $\mathcal{E}_{\hbar,0}$ is a vector
bundle on $\Orb_K$ and $\operatorname{codim}_{\Orb_K^2}\partial \Orb_K\geqslant 2$,
the sheaf $\iota_* \mathcal{E}_{\hbar,0}$ is coherent. It follows that
$\iota_*\mathcal{E}_{\hbar,n}$ is coherent for all $n \ge 0$ as well. Consider the completion
$(\iota_*\mathcal{E}_{\hbar,n})^{\wedge_{\chi'}} \in \wHC^{gr}(\A_\hbar^{\wedge_{\chi'}},\theta)$.
By Lemma  \ref{Lem:wHC_equivalence}, $(\iota_*\mathcal{E}_{\hbar,n})^{\wedge_{\chi'}}$
decomposes as
\begin{equation}\label{eq:decomposition11}
	[(\iota_*\mathcal{E}_{\hbar,n})^{\wedge_{\chi'}}]^{\check{L}}\widehat{\otimes}_{\C[[\hbar]]}
\C[[\check{L},\hbar]],
\end{equation}
where $\check{L}=\check{V}^{-\theta}$. This decomposition is $\C^\times$-equivariant
with respect to the induced $\C^\times$-action on $\check{L}$. We define
the weakly HC $(\breve{\A}_\hbar^{\wedge_0}, \theta)$-module $\breve{\Ecal}_{\hbar,n}'$ as (\ref{eq:decomposition11}), where we redefine
the $\C^\times$-action on $\C[[\check{L},\hbar]]$ to come from the dilation action.
We also define the weakly HC $(\check{\A}_\hbar^{\wedge_0}, \theta)$-module $\check{\Ecal}_{\hbar,n}'$ as the first factor in (\ref{eq:decomposition11}).
Let $\breve{\Ecal}_{\hbar,n},\check{\Ecal}_{\hbar,n}$ denote the $\C^\times$-finite parts
in $\breve{\Ecal}_{\hbar,n}', \check{\Ecal}_{\hbar,n}'$, so that they are weakly HC $(\breve{\A}_\hbar, \theta)$- and $(\check{\A}_\hbar, \theta)$-modules respectively. So we have
\begin{equation}\label{eq:decomposition115}
\breve{\Ecal}_{\hbar,n}=\check{\Ecal}_{\hbar,n}\otimes_{\C[\hbar]}\C[L,\hbar].
\end{equation}

Now we proceed to constructing a sheaf $\breve{\Ecal}^\times_\hbar$ of HC $(\breve{\A}_\hbar^\times, \theta)$-modules. Consider the restriction
$\breve{\Ecal}_{\hbar,n}^\times$ of $\breve{\Ecal}_{\hbar,n}$ to $\breve{X}^\times$.
It follows from the construction of $\breve{\Ecal}_{\hbar,n}$
that the kernel and the cokernel of
$$\breve{\Ecal}_{\hbar,n}/\hbar^{n}\breve{\Ecal}_{\hbar,n}
\rightarrow \breve{\Ecal}_{\hbar,n-1}$$
are supported on $\{0\}\times \check{L}$. Indeed, a similar claim holds for
the sheaves $\Ecal_{\hbar,n}$, hence for $(\iota_*\mathcal{E}_{\hbar,n})^{\wedge_{\chi'}}$,
hence for $\breve{\Ecal}_{\hbar,n}'$, hence for $\breve{\Ecal}_{\hbar,n}$.
Therefore
\begin{equation}\label{eq:sheaf_map}\breve{\Ecal}^\times_{\hbar,n}/\hbar^{n}\breve{\Ecal}^\times_{\hbar,n}
\xrightarrow{\sim} \breve{\Ecal}^\times_{\hbar,n-1}\end{equation}
We set
$$\breve{\Ecal}_\hbar^\times:=\varprojlim \breve{\Ecal}^\times_{\hbar,n}.$$
Similarly, we can define the sheaf $\check{\Ecal}^\times_\hbar$ of $(\check{\A}_\hbar^\times, \theta)$-modules. Observe that, by the construction, the sheaves $\breve{\Ecal}^\times_\hbar$
and $\check{\Ecal}^\times_\hbar$ are flat over $\C[[\hbar]]$ and hence HC.

Note that, thanks
to (\ref{eq:sheaf_map}) and its analog for $\check{\Ecal}_\hbar$,
$$\breve{\Ecal}^\times_{\hbar}/\hbar^n \breve{\Ecal}^\times_\hbar\xrightarrow{\sim}
\breve{\Ecal}^\times_{\hbar,n-1}, \quad \check{\Ecal}^\times_{\hbar}/\hbar^n \check{\Ecal}^\times_\hbar\xrightarrow{\sim}
\check{\Ecal}^\times_{\hbar,n-1}.$$
And since
$(\iota_*\mathcal{E}_{\hbar,n})^{\wedge_{\chi'}}\cong \breve{\Ecal}_{\hbar,n}^{\wedge_0}$,
we conclude that
\begin{itemize}
\item
the restrictions of $\Ecal_\hbar,\breve{\Ecal}^\times_\hbar$
to $X^{\wedge_{\chi'}}\cap \Orb\cong \breve{X}^{\wedge_0}\cap \breve{X}^\times$
are isomorphic.
\end{itemize}

Here is a criterion for $\iota_*\mathcal{E}_\hbar$ to be coherent.

\begin{Lem}\label{Lem:coh_extension1}
The following two conditions are equivalent:
\begin{enumerate}
\item $\iota_*\mathcal{E}_\hbar$ is coherent.
\item For each codimension two orbit $\Orb'_K$ in $\overline{\Orb}_K$,
the restriction of $\Gamma(\check{\mathcal{E}}^\times_\hbar)$ (where $\check{\mathcal{E}}^\times_\hbar$
is constructed for that orbit) to $\check{X}^\times$ coincides with $\check{\Ecal}^\times_\hbar$.
\end{enumerate}
\end{Lem}
\begin{proof}
Consider the third condition:
\begin{itemize}
\item[(3)] For each codimension two orbit $\Orb'_K$ in $\overline{\Orb}_K$,
the restriction of $\Gamma(\mathcal{E}_\hbar|_{X^{\wedge_{\chi'}}\cap \Orb})$ to $X^{\wedge_{\chi'}}\cap\Orb$ coincides with $\mathcal{E}_\hbar|_{X^{\wedge_{\chi'}}\cap \Orb}$.
\end{itemize}
The equivalence of (1) and (3) is established in exactly the same way as in
the proof of \cite[Lemma 3.4]{HC_symp}. Now we explain why (2)$\Leftrightarrow$(3).

By the discussion before the lemma,
$$\mathcal{E}_\hbar|_{X^{\wedge_{\chi'}}\cap \Orb}\cong \breve{\Ecal}^\times_\hbar|_{\breve{X}^{\wedge_0}\cap\breve{X}^\times}.$$
Note that (3) is equivalent to the claim that the kernel and the cokernel
of
$$\Gamma(\breve{\Ecal}^\times_\hbar|_{\breve{X}^{\wedge_0}\cap\breve{X}^\times})/
\hbar^{n+1} \Gamma(\breve{\Ecal}^\times_\hbar|_{\breve{X}^{\wedge_0}\cap\breve{X}^\times})
\rightarrow \Gamma(\breve{\Ecal}^\times_{\hbar,n}|_{\breve{X}^{\wedge_0}\cap\breve{X}^\times})$$
are supported on $\{0\}\times L^{\wedge_0}$.

Now we are going to pass to the subspaces of $\C^\times$-finite elements.  Note that
$$\Gamma(\breve{\Ecal}^\times_\hbar|_{\breve{X}^{\wedge_0}\cap\breve{X}^\times})=
\varprojlim \Gamma(\breve{\Ecal}^\times_{\hbar,n}|_{\breve{X}^{\wedge_0}\cap\breve{X}^\times}).$$
Arguing as in Step 2 of the proof of \cite[Lemma 3.4]{HC_symp}, we see that
\begin{equation} \label{eq:completion_coincide10}
	\Gamma(\breve{\Ecal}^\times_{\hbar,n}|_{\breve{X}^{\wedge_0}\cap\breve{X}^\times}) =
	\Gamma(\breve{\Ecal}^\times_{\hbar,n})^{\wedge_0}
	\simeq (\iota_*\mathcal{E}_{\hbar,n})^{\wedge_{\chi'}}.
\end{equation}
Therefore the subspace of $\C^\times$-finite elements in $\Gamma(\breve{\Ecal}^\times_{\hbar,n}|_{\breve{X}^{\wedge_0}\cap\breve{X}^\times})$
coincides with $\Gamma(\breve{\Ecal}^\times_{\hbar,n})$. Since the $\C^\times$-action on
$\breve{X}$ is contracting, we see that
\begin{equation}\label{eq:completion_coincide11}
\Gamma(\mathcal{E}_\hbar|_{X^{\wedge_{\chi'}}\cap \Orb})\cong \Gamma(\breve{\mathcal{E}}_\hbar^\times)^{\wedge_0}.
\end{equation}

So (3) is equivalent to the claim that
the kernel and the cokernel of
$$\Gamma(\breve{\Ecal}^\times_\hbar)/
\hbar^{n+1} \Gamma(\breve{\Ecal}^\times_\hbar)
\rightarrow \Gamma(\breve{\Ecal}^\times_{\hbar,n}).$$
to be supported on $\{0\}\times L$.

Let us show that 
\begin{equation} \label{eq:decomposition125}
	\Gamma(\breve{\Ecal}^\times_{\hbar,n})=
	\Gamma(\check{\Ecal}^\times_{\hbar,n})\otimes_{\C[\hbar]}\C[L,\hbar].
\end{equation}
We cover $\check{X}^\times$ with principal open affine subsets
$\check{X}_{f_i}\subset \check{X}$. Then $\Gamma(\check{\Ecal}^\times_{\hbar,n})$
is the zeroth cohomology of the \v{C}ech complex for $\check{\Ecal}_{\hbar,n}$
associated with the cover $\check{X}^\times=\bigcup \check{X}_{f_i}$, its terms
are direct sums of localizations of $\check{\Ecal}_{\hbar,n}$. Similarly,
$\Gamma(\breve{\Ecal}^\times_{\hbar,n})$ is the zeroth cohomology of the \v{C}ech
complex associated to the cover $\breve{X}^\times=\bigcup \breve{X}_{f_i}$. This \v{C}ech
complex is obtained from the previous one by tensoring with $\C[L,\hbar]$ over $\C[\hbar]$.
Since $\C[L,\hbar]$ is flat over $\C[\hbar]$ our claim follows.

Since
$\Gamma(\breve{\Ecal}^\times_\hbar)=\varprojlim \Gamma(\breve{\Ecal}^\times_{\hbar,n})$ and $\Gamma(\check{\Ecal}^\times_\hbar)=\varprojlim \Gamma(\check{\Ecal}^\times_{\hbar,n})$, we see that
\begin{equation}\label{eq:decomposition12}
\Gamma(\breve{\Ecal}^\times_{\hbar})=
\Gamma(\check{\Ecal}^\times_\hbar)\widehat{\otimes}_{\C[[\hbar]]}\C[L][[\hbar]].
\end{equation}
It follows that (3) is equivalent to the claim that the kernels and the cokernels of
$$\Gamma(\check{\Ecal}^\times_\hbar)/
\hbar^{n+1} \Gamma(\check{\Ecal}^\times_\hbar)
\rightarrow \Gamma(\check{\Ecal}^\times_{\hbar,n})$$
are finite dimensional for all $n$. But this condition is equivalent to (2).

\end{proof}

\subsection{Extension outside of codimension $2$}
Now suppose that the equivalent conditions of Lemma
\ref{Lem:coh_extension1} hold. The following claim
is an analog of \cite[Lemma 3.5]{HC_symp}. Note that the conditions
in the statement are more restrictive: unlike in the setting of
\cite{HC_symp} not all orbits in $\overline{\Orb}_K\setminus \Orb_K^2$ have even
codimension so we can get an orbit of codimension $3$, while in the setting
of \cite{HC_symp} the minimal possible codimension is $4$.

\begin{Lem}\label{Lem:coh_extension2}
	The following conditions are equivalent:
	\begin{enumerate}
		\item 
			the microlocalization of $\Gamma(\mathcal{E}_\hbar)$ to $\Orb_K$
			coincides with $\mathcal{E}_\hbar$ and the coherent sheaf $\Gamma(\mathcal{E}_\hbar) / \hbar \Gamma(\mathcal{E}_\hbar)$ restricted to $\Orb_K^2$ is maximal Cohen-Macaulay (or, equivalently, satisfying Serre's $S_2$ condition);
		\item 
			the following assumption holds for all codimension $2$ orbits
			in $\overline{\Orb}_K$ (and the corresponding sheaves $\check{\mathcal{E}}^\times_\hbar$):
			\begin{itemize}
				\item[(*)] $\Gamma(\check{\mathcal{E}}_\hbar)/
				\hbar \Gamma(\check{\mathcal{E}}^\times_\hbar)\xrightarrow{\sim}\Gamma(\check{\mathcal{E}}^\times_{\hbar}/\hbar \check{\mathcal{E}}^\times_{\hbar})$.
			\end{itemize}
	\end{enumerate}
\end{Lem}

\begin{proof}
	Combining (\ref{eq:completion_coincide11}) and (\ref{eq:decomposition12}), we see that condition (*) is equivalent to the claim that
	\begin{equation}\label{eq:quot_mod_h}
		\Gamma(\mathcal{E}_\hbar|_{X^{\wedge_{\chi'}}\cap \Orb})/
		\hbar \Gamma(\mathcal{E}_\hbar|_{X^{\wedge_{\chi'}}\cap \Orb})\xrightarrow{\sim}
		\Gamma(\mathcal{E}_{\hbar,0}|_{X^{\wedge_{\chi'}}\cap \Orb}).
	\end{equation}
	In its turn, (\ref{eq:quot_mod_h}) for all codimension $2$ orbits is equivalent to
	$$(\iota_*\mathcal{E}_\hbar)/\hbar (\iota_*\mathcal{E}_\hbar)
	\xrightarrow{\sim} \iota_*(\mathcal{E}_\hbar/\hbar \mathcal{E}_\hbar).$$
	
	Note that $\check{\mathcal{E}}^\times_{\hbar}/\hbar \check{\mathcal{E}}^\times_{\hbar}$ is a vector bundle
	on $\check{X}^\theta\setminus\{0\}$. From \cite[Corollary 3.12]{Burban_Drozd} we see that
	$\Gamma(\check{\mathcal{E}}^\times_{\hbar}/\hbar \check{\mathcal{E}}^\times_{\hbar})$ is a maximal Cohen-Macaulay
	module over $\C[\check{X}^\theta\setminus\{0\}]$, the normalization of $\C[\check{X}^\theta]$.
	By \cite[Lemma 2.24]{Burban_Drozd},  $\Gamma(\check{\mathcal{E}}^\times_{\hbar}/\hbar \check{\mathcal{E}}^\times_{\hbar})$ is Cohen-Macaulay over $\C[\check{X}^\theta]$. (Or we can directly apply Lemma 3.1 and Proposition 3.10 of \cite{Burban_Drozd} without passing to normalizations.)
	Now combining Corollary 2.1.8 and Theorem 2.1.9 of \cite{BH} (cf. \cite[Remark 12.2]{KP2}), \eqref{eq:completion_coincide11} and \eqref{eq:decomposition125} (for $n=0$), we conclude that $\iota_*(\mathcal{E}_\hbar/\hbar \mathcal{E}_\hbar) = \iota_*(\EE_{\hbar,0})$
	is maximal Cohen-Macaulay over $\Orb_K^2$.
	
	Now we prove that (1) implies (2). Let $\F := (\iota_*\mathcal{E}_\hbar)/\hbar (\iota_*\mathcal{E}_\hbar)$. Then by the assumption of (1), $\F$ is a maximal Cohen-Macaulay coherent sheaf over $\Orb^2_K$ whose restriction to $\Orb_K$ is $\mathcal{E}_\hbar/\hbar \mathcal{E}_\hbar$. Therefore the canonical map $\F \to \iota_* \iota^* \F \simeq \iota_*(\mathcal{E}_\hbar/\hbar \mathcal{E}_\hbar)$ is an isomorphism by \cite[Proposition 1.11]{Hartshorne}. 
	
	For (2)$\Rightarrow$(3), we only need to show that the microlocalization of $\Gamma(\mathcal{E}_\hbar)$ to $\Orb_K$
	coincides with $\mathcal{E}_\hbar$. Since the complement to
	$\Orb_K^2$ in $\overline{\Orb}_K$ has codimension $\geqslant 3$, we see that
	$H^1(\Orb_K^2, \iota_*(\mathcal{E}_\hbar/\hbar \mathcal{E}_\hbar))$
	is a finitely generated module over $\C[\overline{\Orb}_K]$. Now we can
	finish the proof as in \cite[Lemma 3.5]{HC_symp}.
\end{proof}

It turns out that (*) is not convenient to check, as it is hard to determine $\check{\mathcal{E}}^\times_\hbar$.
It is easy to determine $\check{\mathcal{E}}^\times:=\check{\mathcal{E}}^\times_\hbar/\hbar \check{\mathcal{E}}^\times_\hbar$: we will see in the next section that it is obtained by restricting the twisted local system $\mathcal{E}:=\mathcal{E}_\hbar/\hbar \mathcal{E}$ to the slice.
In the next section we will also see that $\check{\mathcal{E}}^\times$ admits a unique quantization to
an $\check{\A}_\hbar^\times$-module if we assume that the $\C^\times$-action on $\mathcal{E}$ is strong.

\subsection{Restrictions to slices}
We adopt the settings and notations of Section \ref{subsec:extn_codim2}.
Take an object $\mathcal{E}_\hbar \in \wHC^{gr}(\hat{\A}_\hbar^{\Orb},\theta)^{K,\kappa}$.  By Section \ref{subsec:lag_vec}, the sheaf $\EE = \EE_\hbar / \hbar \EE_\hbar$ is a coherent $(K,\kappa)$-equivariant $\TT^+(\hat{\A}_\hbar^\Orb, \Orb_K)$-module and hence a $K$-equivariant vector bundle over $\Orb_K$.
The following claim follows by combining Remark \ref{Rem:homog_analyt_classif} with Proposition \ref{prop:irred_strongness}.

\begin{Lem}\label{Lem:irred_strongness_HC}
	Suppose the geometric fiber of $\mathcal{E}_\chi$ at $\chi \in \Orb_K$ is an irreducible $K_Q$-module.
	Then the induced $\cm$-action on $\EE$ is strong in the sense of Definition \ref{defn:strong_cm}.
\end{Lem}

Let $\check{Y} = \check{X}^\theta$ and $\check{Y}^\times = \check{Y} \setminus \{ 0 \} = (\check{X}^\times)^\theta$. In Section \ref{subsec:extn_codim2} we have produced a module $\check{\Ecal}^\times_\hbar$ over a quantization of $\check{X}^\times$ from $\EE_\hbar$, which is supported on $\check{Y}^\times$, so that $\check{\mathcal{E}}^\times = \check{\mathcal{E}}_\hbar^\times/\hbar \check{\mathcal{E}}_\hbar^\times$ is a graded $\TT^+(\check{\A}_\hbar^\times, \check{Y}^\times)$-module,
where the $\cm$-action is restricted from the action $t.m:=t^{-2}\gamma'(t)m$ on $\mathcal{E}$; here $\gamma$ is the one-parameter subgroup in $G$ corresponding to the semisimple element $h' \in \kf$ of the normal $\slf_2$-triple $(e',h',f')$ corresponding to $\chi'$. Let $K'_Q$ be the $K$-centralizer of $(e',h',f')$. We have the strong $K$-action on $\TT^+(\hat{\A}_\hbar^{\Orb}, \Orb_K)$ with the Lie algebra morphism $\nu_\kf: \kf \to  \TT^+(\hat{\A}_\hbar, \Orb_K)$ induced by restricting the quantum comoment map $\Phi_\hbar^\g: \g \to  \hat{\A}_\hbar$ to $\kf$ as in Section \ref{SS_Ham_quant}. Its restriction to $\kf'_Q$ descends to a Lie algebra morphism
\[ \nu_{\kf'_Q} := \nu_\kf|_{\kf'_Q}: \kf'_Q \to \iota_{\check{Y}^\times}^\circ \TT^+(\hat{\A}_\hbar^{\Orb}, \Orb_K),\]
where $\iota_{\check{Y}^\times} : \check{Y}^\times \hookrightarrow \Orb_K$ denotes the closed embedding, and $\iota^\circ$ stands for the pullback of Picard algebroids as in the discussion after Lemma \ref{Lem:algebroid_w_coh_module}. We also have the quantum comoment map $\Phi_\hbar^{\q}: \q \to \check{\A}_\hbar$ defined similarly as in \eqref{eq:shifted_character}. Then we also have a strong $K'_Q$-action on $\TT^+(\check{\A}_\hbar^\times, \check{Y}^\times)$ with the Lie algebra morphism
\[ \check{\nu}_{\kf'_Q}: \kf'_Q \to \TT^+(\check{\A}_\hbar^\times, \check{Y}^\times)\]
induced by $\Phi_\hbar^{\q}$.


Here is the main result of this section.

\begin{Prop}\label{Prop:slice_strongness}
We have the following results:
	\begin{enumerate}
		\item
		 There is a natural $K'_Q$-equivariant isomorphism
		  \[ \upsilon^+: \TT^+(\check{\A}_\hbar^\times, \check{Y}^\times) \xrightarrow{\sim} \iota_{\check{Y}^\times}^\circ \TT^+(\hat{\A}_\hbar^{\Orb}, \Orb_K),\]
		  such that $\nu_{\kf'_Q} = \upsilon^+ \circ \check{\nu}_{\kf'_Q} + \rho'_\omega$, where $\rho'_\omega$ is the character of $\kf'$ as defined in Section \ref{SS_intro_approach} for $\Orb'_K$.
		\item
		 Let $\mathcal{E}_\hbar\in\wHC^{gr}(\hat{\A}_\hbar^{\Orb},\theta)^{K,\kappa}$ and keep the notations in Section \ref{subsec:extn_codim2}. Then there is a natural $K'_Q$-equivariant isomorphism between the $\TT^+(\check{\A}_\hbar^\times, \check{Y}^\times)$-module $\check{\mathcal{E}}^\times$ and the pullback $\iota_{\check{Y}^\times}^* \EE$ of the $\TT^+(\hat{\A}_\hbar^{\Orb}, \Orb_K)$-module $\EE$ with respect to the isomorphism in (1).
		\item
		If the $\cm$-action on $\EE$ is strong, so is the induced $\cm$-action on $\check{\mathcal{E}}^\times$.
	\end{enumerate}
\end{Prop}

To prove this proposition, we need to first understand the relation between the Picard algebroids $\TT^+(\hat{\A}_\hbar^\Orb, \Orb_K)$ and $\TT^+(\check{\A}_\hbar^\times, \check{Y}^\times)$ (as well as $\TT(\hat{\A}_\hbar^\Orb, \Orb_K)$ and $\TT(\check{\A}_\hbar^\times, \check{Y}^\times)$ ). We start with defining analogues of them for (not necessarily smooth) affine schemes. The following definition is similar to Definition \ref{defn:pic} of Picard algebroids. Let $\der(?)$ denote the Lie algebra of derivations of an algebra ``?''.

\begin{defi}
	Suppose $\B$ is a commutative (Noetherian) $\C$-algebra. A  {\it (weak) Picard Lie algebra} over $\B$ is a quadruple $(\T, [-,-], \eta_\T, 1_{\T})$, where
	\begin{itemize}
		\item
		$\T$ is a finitely generated $\B$-module;
		\item
		$[-,-]$ is a $\C$-linear Lie bracket;
		\item
		$\eta_\T: \T \to \der(\B)$ is a Lie algebra morphism and a morphism of $\B$-modules;
		\item
		$1_\T$ is a central element in the Lie algebra $\T$ which lies in the kernel of $\eta_\T$,
	\end{itemize}
	such that the Lie bracket $[-,-]$, $\eta_\T$ and the $\B$-module structure of $\T$ are compatible with each other, in the sense that they satisfy a Leibniz rule similar to that in the definition of a Lie algebroid and $\eta_\T$ behaves like an anchor map. Note that we have a $\B$-linear morphism $\iota_\T: \B \to \T$, $b \mapsto b \cdot 1_\T$, and the composite map $\B \xrightarrow{\iota_\T} \T \xrightarrow{\eta_\T} \der(\B)$ is zero. This sequence, however, does not have to be exact.

	As in the case of Picard algebroids, one can similarly define the notion of a morphism $\tau: \T \to \F$ between two Picard Lie algebras $\T=(\T, [-,-]_\T, \eta_\T, 1_{\T})$ and $\F=(\F, [-,-]_\F, \eta_\F, 1_{\F})$ over a base algebra $\B$, by requiring $\tau$ to be a $\B$-linear Lie algebra homomorphism that intertwines the anchor maps $\eta_\T$ and $\eta_{\F}$, and $\tau(1_\T)=1_\F$. Unlike the case of Picard algebroids, however, a morphism of Picard Lie algebras needs not to be an isomorphism. We write $\mathrm{wPA}(\B)$ for the category of all (weak) Picard algebras over $\B$.
\end{defi}

Given a homomorphism $f: \B \to \mathcal{C}$ of commutative algebras over $\C$, we can also define the pullback Picard Lie algebra $f^\circ \T$ over $\mathcal{B}$ of $\T$ as the fiber product of the natural maps
  \[  \der(\mathcal{C}) \to \mathcal{C} \otimes_{\B} \der(\B) \quad \text{and} \quad 1 \otimes \eta_\T :  \mathcal{C} \otimes_{\B} \T \to \mathcal{C} \otimes_{\B} \der(\B).\]
The central element $1_{f^\circ \T}$ is taken to be $(0, 1_\T)$ in this fiber product. The anchor map of $f^\circ \T$ is just the projection map $f^\circ \T \to \der(\mathcal{C})$. The Lie bracket of $f^\circ \T$ can be defined in a similar way as how pullbacks of Picard algebroids are defined. We have the obvious functoriality: if $\tau: \T \to \F$ is a morphism of Picard algebras over $\B$, then it induces a morphism $f^\circ(\tau): f^\circ \T \to f^\circ \F$ of Picard algebras over $\mathcal{C}$ in the obvious way, such that  $\T \mapsto f^\circ \T$ defines a functor $f^\circ: \mathrm{wPA}(\B) \to \mathrm{wPA}(\mathcal{C})$; moreover,
 if $f: \B \to \mathcal{C}$ and $g: \mathcal{C} \to \Dcal$ are two homomoprhisms of commutative $\C$-algebras and $\T$ is a Picard algebra over $\B$, then we have a natural isomorphism $(g \circ f)^\circ  \simeq g^\circ \circ f^\circ : \mathrm{wPA}(\B) \to \mathrm{wPA}(\mathcal{\Dcal})$ of functors. We can hence take various localizations $\B \to S^{-1} \B$ of $\B$ and localize a Picard algebra $\T$ over $\B$ to a Lie algebroid over $\spec \B$ or any of its open subset.

One can also talk about modules over the Picard Lie algebra $(\T, [-,-], \eta_\T, 1_\T)$ similarly to modules over Picard algebroids. That is, a module over $\T$ is a $\B$-module $N$  equipped with a Lie algebra action of $\T$, which satisfies the Leibniz rule with respect to $\eta_\T$ and the $\B$-module structure such that $1_\T$ acts by the identity. Given a homomorphism $f: \B \to \mathcal{C}$ of commutative $\C$-algebras, for any $\T$-module $M$, the $\mathcal{C}$-module $f^*M := \mathcal{C} \otimes_{\B}  M$ is naturally a $f^\circ \T$-module. If $p: M \to N$ is a morphism of modules over $\T$, then it induces a morphism $f^*(p) : f^*M \to f^*N$ of modules over the Picard algebra $f^\circ \T$ over $\mathcal{C}$, so that $f^*(-)$ defines a functor from the category of modules over $\T$ to the category of modules over $f^\circ \T$.

We adopt the setting and notations of Section \ref{SS_weakly_HC}. In particular, let $\Dcal_\hbar$ be an associate flat $\C[\hbar]$-algebra. Assume that $\Dcal_0 = \Dcal_\hbar / \hbar \Dcal_\hbar$ is a Noetherian commutative $\C$-algebra. Then the map $\Dcal_\hbar \times \Dcal_\hbar \to \Dcal_\hbar : (a,b) \mapsto \hbar^{-1}[a,b]$ descends to a Possion bracket on $\Dcal_0$. As in Section \ref{SS_weakly_HC}, suppose $\theta$ is a $\C[\hbar]$-linear anti-involution on $\Dcal_\hbar$. We are mostly interested in the case when $\Dcal_\hbar$ is the quantization $\A_\hbar$ of $\C[\Orb]$ or the quantum slice $\check{\A}_\hbar$. In what follows, we will define  in a similar manner the  analogues for $(\Dcal_\hbar, \theta)$ of the Picard algebroids $\TT^+(\hat{\A}_\hbar^\Orb, \Orb_K)$ and $\TT^+(\check{\A}_\hbar^\times, \check{Y}^\times)$ introduced in Section \ref{subsec:lag_pic}. According to Remark \ref{rmk:pic_mod_t2}, those Picard algebroids can be defined using only the quantization mod $\hbar^2$. Therefore it is natural to set $\Dcal_\hbar^{(2)} := \Dcal_\hbar / \hbar^2 \Dcal_\hbar$ to replace $\Dcal_\hbar$.

 Let $\J$ be the ideal in $\Dcal_0$ generated by $\Dcal_0^{-\theta} $. Note that $\J$ can also be regarded as a Lie algebra via commutator. Set $\Dcal'_0: = \Dcal_0 / \J$. Let $\bar{\J}_\hbar$ be the preimage of $\J$ in $\Dcal_\hbar^{(2)}$ under the quotient map $\Dcal_\hbar^{(2)} \to \Dcal_0$. Then the ideal $\bar{\J}_\hbar^2\subset \Dcal_\hbar^{(2)}$ is also a Lie ideal of $ \bar{\J}_\hbar$ and we can form the Lie algebra $\hbar^{-1} (\bar{\J}_\hbar / \bar{\J}_\hbar^2)$. Since $\hbar \bar{\J}_\hbar \subset \bar{\J}_\hbar^2$ by definition, the natural inclusion $\hbar \Dcal_\hbar^{(2)}  \subset  \bar{\J}_\hbar$ descends to a map
   \[  \iota_\T: \Dcal'_0 \simeq \Dcal_\hbar / \bar{\J}_\hbar \to \hbar^{-1} (\bar{\J}_\hbar / \bar{\J}_\hbar^2). \]
The action of $\hbar^{-1} \bar{\J}_\hbar$ on $\Dcal_\hbar^{(2)}$ by taking commutators preserves $\bar{\J}_\hbar$ and hence induces a Lie algebra morphism
   \[ \eta_\T: \hbar^{-1} (\bar{\J}_\hbar / \bar{\J}_\hbar^2) \to \der(\Dcal_\hbar^{(2)} / \bar{\J}_\hbar) \simeq \der(\Dcal'_0). \]
 As in Section \ref{subsec:lag_pic}, let $\T^+(\Dcal_\hbar, \theta)$ and $\T(\Dcal_\hbar, \theta)$ denote the same Lie algebra $\hbar^{-1} (\bar{\J}_\hbar / \bar{\J}_\hbar^2)$, but with the two $\Dcal'_0$-module structures induced by the left multiplication of $\Dcal_\hbar^{(2)}$ on $\bar{\J}_\hbar$ and the symmetrized product respectively. Clearly, the  maps $\eta_\T$ and $\iota_\T$ are morphisms of $\Dcal'_0$-modules with respect to both $\Dcal'_0$-module structures on $\hbar^{-1} (\bar{\J}_\hbar / \bar{\J}_\hbar^2)$. Moreover, $\eta_\T \circ \iota_\T = 0$. It is then routine to check that both $\T^+(\Dcal_\hbar, \theta)$ and $\T(\Dcal_\hbar, \theta)$ with the structure maps  $\eta_\T$ and $\iota_\T$ are weak Picard Lie algebras over $\Dcal'_0$. All the constructions are compatible with $\cm$-actions. Again our constructions are compatible with localizations. This together with Remark \ref{rmk:pic_mod_t2} immediately implies the following lemma.

 \begin{Lem}\label{lem:Picard_loc}
 	The localizations of $\T^+( \A_\hbar, \theta)$ and $\T( \A_\hbar, \theta)$ to $\Orb_K$ are naturally isomorphic to the Picard algebroids $\TT^+(\hat{\A}_\hbar^\Orb, \Orb_K)$ and $\TT(\hat{\A}_\hbar^\Orb, \Orb_K)$ respectively. Similarly the localizations of $\T^+( \check{\A}_\hbar, \theta)$ and $\T( \check{\A}_\hbar, \theta)$ to $\check{Y}^\times$ are naturally isomorphic to the Picard algebroids $\TT^+(\check{\A}_\hbar^\times, \check{Y}^\times)$ and $\TT(\check{\A}_\hbar^\times, \check{Y}^\times)$ respectively.
 \end{Lem}

As in the case of quantization of vector bundles over Lagrangian subvarieties (see Section \ref{subsec:lag_vec}), if $M_\hbar$ is a HC $(\Dcal_\hbar, \theta)$-module, then the $\Dcal'_0$-module $M_0 := M_\hbar / \hbar M_\hbar$ is naturally a module over $\T^+(\Dcal_\hbar, \theta)$. Again all the constructions are compatible with $\cm$-actions.

Now we come back to the setting of Proposition \ref{Prop:slice_strongness}.

\begin{proof}[Proof of Proposition \ref{Prop:slice_strongness}]
We first prove part (1). We have a decomposition $\A_\hbar^{\wedge_{\chi'}} \simeq  \check{\A}_\hbar^{\wedge 0} \widehat{\otimes}_{\C[[\hbar]]}  \Weyl_\hbar(\check{V} )^{\wedge_0}$. Set $\Dcal_\hbar = \A_\hbar^{\wedge_{\chi'}}$ and $\mathcal{C}_\hbar = \check{\A}_\hbar^{\wedge 0}$, then $\Dcal_0 = \Dcal_\hbar/\hbar\Dcal \simeq \C[X^{\wedge_{\chi'}}]$ and $\mathcal{C}_0 = \mathcal{C}_\hbar / \hbar \mathcal{C}_\hbar \simeq \C[\check{X}^{\wedge_0}]$. Let $\J_\mathcal{C}$ and $\J_\Dcal$ be the ideals in $\mathcal{C}_0$ and $\Dcal_0$ generated by $\mathcal{C}_0^{-\theta}$ and $\Dcal_0^{-\theta}$ respectively. Let $\bar{\J}_{\mathcal{C},\hbar}$ and $\bar{\J}_{\Dcal,\hbar}$ be the preimage of $\J_\mathcal{C}$ and $\J_\Dcal$ in $\mathcal{C}_\hbar^{(2)}$ and $\Dcal_\hbar^{(2)}$, respectively. Set $\mathcal{C}'_0: = \mathcal{C}_0 / \J_\mathcal{C}$ and $\Dcal'_0: = \Dcal_0 / \J_\Dcal$. We have a natural morphism $\Dcal_0 \to \mathcal{C}_0$ intertwining $\theta$'s, which is the completion of the morphism $\C[X] \to \C[\check{X}]$ corresponding the inclusion $\check{X} \hookrightarrow X$. This induces a morphism $\hat{f}: \Dcal_0' \to \mathcal{C}_0'$. It is the completion of the morphism $f: \C[X] / \langle \C[X]^{-\theta} \rangle \to \C[\check{X}] / \langle \C[\check{X}]^{-\theta} \rangle$.

The subspace $\bar{\J}_{\mathcal{C},\hbar} \widehat{\otimes} 1 \subset \mathcal{C}_\hbar^{(2)} \widehat{\otimes}_{\C[[\hbar]]}  \Weyl_\hbar(\check{V} )^{\wedge_0}  \simeq \Dcal_\hbar^{(2)}$ lies inside $\bar{\J}_{\Dcal,\hbar}$. The inclusion map $\hbar^{-1}\bar{\J}_{\mathcal{C},\hbar} \widehat{\otimes} 1 \hookrightarrow \hbar^{-1} \bar{\J}_{\Dcal,\hbar}$ induces a $\mathcal{C}'_0$-linear map $\T^+(\mathcal{C}_\hbar, \theta) \to \mathcal{C}'_0 \otimes_{\Dcal'_0} \T^+(\Dcal_\hbar, \theta) $. Take the composition of the maps
\[ \hbar^{-1} \bar{\J}_{\mathcal{C},\hbar} \widehat{\otimes} 1 \hookrightarrow \hbar^{-1} \bar{\J}_{\Dcal,\hbar} \twoheadrightarrow  \hbar^{-1} (\bar{\J}_{\Dcal,\hbar} / \bar{\J}_{\Dcal,\hbar}^2) = \T^+(\Dcal_\hbar, \theta) \to \mathcal{C}'_0 \otimes_{\Dcal'_0} \T^+(\Dcal_\hbar, \theta) \xrightarrow{\mathbbm{1} \otimes \eta_\T} \mathcal{C}'_0 \otimes_{\Dcal'_0} \der(\Dcal'_0),\]
where the map $\T^+(\Dcal_\hbar, \theta) \to \mathcal{C}'_0 \otimes_{\Dcal'_0} \T^+(\Dcal_\hbar, \theta)$ is one induced by tensoring the identity map of $\T^+(\Dcal_\hbar, \theta)$ with the morphism $\hat{f}: \Dcal_0' \to \mathcal{C}_0'$,  and the last one $\mathbbm{1} \otimes \eta_\T$ is induced by the anchor map $\eta_\T$ of $\T^+(\Dcal_\hbar, \theta)$. The resulting composite map $\hbar^{-1} \bar{\J}_{\mathcal{C},\hbar}\rightarrow \mathcal{C}'_0 \otimes_{\Dcal'_0} \der(\Dcal'_0)$  coincides with the composite map
$$\hbar^{-1} \bar{\J}_{\mathcal{C},\hbar}\twoheadrightarrow
\hbar^{-1} (\bar{\J}_{\mathcal{C},\hbar}/\bar{\J}_{\mathcal{C},\hbar}^2) = \T^+(\mathcal{C}_\hbar, \theta) \xrightarrow{\eta_\T} \der(\mathcal{C}'_0) \to \mathcal{C}'_0 \otimes_{\Dcal'_0} \der(\Dcal'_0).$$
This implies that we have a natural $\cm$-equivariant morphism $\T^+(\mathcal{C}_\hbar, \theta) \to \hat{f}^\circ [\T^+(\Dcal_\hbar, \theta)]$ of Picard Lie algebras over $\mathcal{C}'_0$. Note that $\hat{f}^\circ [\T^+(\Dcal_\hbar, \theta)]$ is nothing else but the completion of $f^\circ [\T^+(\A_\hbar, \theta)]$. 
Taking the $\cm$-finite vectors gives a $\cm$-equivariant morphism
\begin{equation} \label{eq:res_Picard_algebra}
	\upsilon^+: \T^+(\check{\A}_\hbar, \theta) \to f^\circ [\T^+(\A_\hbar, \theta)]
\end{equation}
of Picard Lie algebras over $\C[\check{Y}]$. Similarly we have a $\cm$-equivariant morphism $\upsilon: \T(\check{\A}_\hbar, \theta) \to f^\circ [\T(\A_\hbar, \theta)]$.
The localization of the map $\upsilon^+$ to $\check{Y}^{\times}$ is a $\cm$-equivariant morphism $\nu^+: \TT^+(\check{\A}_\hbar^\times, \check{Y}^\times) \to \iota_{\check{Y}^\times}^\circ [ \TT^+(\hat{\A}_\hbar^{\Orb}, \Orb_K)]$ of Picard algebroids by Lemma \ref{lem:Picard_loc} and hence must be an isomorphism. The $K'_Q$-equivariance of $\nu^+$ is clear. The claim about the strong $K_Q'$-actions follows from the same argument as in the proof of Lemma \ref{Lem:restriction_equivariance}.

Now we prove part (2). Keeping the notations from Section \ref{subsec:extn_codim2}, we have a natural isomorphism $\check{\Ecal}_{\hbar,1}' \xrightarrow{\sim}  \mathcal{C}'_0 \otimes_{\Dcal'_0} (\iota_*\mathcal{E}_{\hbar,1})^{\wedge_{\chi'}}$ induced by the decomposition \eqref{eq:decomposition11} for $n=1$ and the inclusion $\check{\Ecal}_{\hbar,1}' \simeq \check{\Ecal}_{\hbar,1}' \widehat{\otimes} 1 \subset (\iota_*\mathcal{E}_{\hbar,1})^{\wedge_{\chi'}}$. By \eqref{eq:decomposition11} for $n=2$, we see that the restriction of the Lie algebra action of $\bar{\J}_{\Dcal, \hbar}$ on $(\iota_*\mathcal{E}_{\hbar,2})^{\wedge_{\chi'}}$ to an action of $\bar{\J}_{\mathcal{C},\hbar} \widehat{\otimes} 1$ preserves the subspace $\check{\Ecal}_{\hbar,2}' \widehat{\otimes} 1$. This implies that the isomorphism $\check{\Ecal}_{\hbar,1}' \xrightarrow{\sim}  \mathcal{C}'_0 \otimes_{\Dcal'_0} (\iota_*\mathcal{E}_{\hbar,1})^{\wedge_{\chi'}} = (\iota_{\check{Y}}^*\iota_*\mathcal{E})^{\wedge_{0}}$ intertwines the $\T^+(\mathcal{C}_\hbar, \theta)$-module structure on the domain and the $\hat{f}^\circ [\T^+(\Dcal_\hbar, \theta)]$-module structure on the codomain, via the morphism $\T^+(\mathcal{C}_\hbar, \theta) \to \hat{f}^\circ [\T^+(\Dcal_\hbar, \theta)]$ of Picard Lie algebras. Taking $\cm$-finite parts and localizing to $\check{Y}^\times$, we get (2). Again the statement about $K'_Q$-equivariance is clear.

Finally we prove part (3). Recall that to construct  $\check{\mathcal{E}}^\times_\hbar$ and $\check{\mathcal{E}}^\times$ together with their $\cm$-actions, we first take the action $\alpha_{h'}$ of the one-parameter subgroup of $K$ generated by the semisimple element $h'$ of the normal $\slf_2$-triple $(e',h',f')$ corresponding to $\chi'$, which commutes with the original $\cm$-action on $\EE_\hbar$, then twist the $\cm$-action on $\EE_\hbar$ by the action $\alpha_{h'}$ to make it compatible with the Kazhdan action. Since $\EE_\hbar$ is a $(K, \kappa)$-equivariant graded Hamiltonian quantization of $(\Orb_K, \EE)$, the sheaf $\EE$ is a $(K, \kappa)$-equivariant $\tatp$-module. This implies that the differential of the $K$-action on the projectivization $\mathbb{P}(\EE)$ of $\EE$ coincides with the $\kf$-action induced by the projective flat connection and the $\kf$-action on $\Orb_K$. The Kazhdan $\cm$-action from Section \ref{SS_Sl_sl} is obtained by twisting the original strong $\cm$-action by $\alpha_{h'}$. Now the action $\alpha_{h'}$ is the restriction of the $K$-action on $\EE$ to a torus subgroup, hence also satisfies the strongness condition. The action $\alpha_{h'}$  commutes with the orginal $\cm$-action, so the twisted $\cm$-action on $\EE$ is still strong.
Then in the construction of quantum slice and $\check{\mathcal{E}}^\times_\hbar$, we restrict the twisted action on $\EE$ to $\check{\mathcal{E}}^\times$ and the natural isomorphisms in part (1) and (2) are $\cm$-equivariant with respect to the twisted $\cm$-actions. Therefore the resulting $\cm$-action on $\check{\mathcal{E}}^\times$ is also strong.
\end{proof}

\subsection{Description of the image}\label{SS_image_description}
The following proposition gives a description of the image of $\bullet_{\dagger,\chi}$
that will be used to characterize the image of
$\overline{\HC}(\A_\lambda,\theta)^{K,\kappa}$ in $\operatorname{Rep}(K_Q,\lambda|_{\kf_Q}-\kappa_Q)$.

Let $\Orb_K^{i},i=1,\ldots,\ell,$ be all
codim $2$ orbits in $\overline{\Orb}_K$, $\check{X}^i$ be the corresponding
4-dimensional slices in $X$ (at points $\chi^i\in \Orb_K^i$), and $\theta^i$ be induced Poisson anti-involutions of $X^i$. We write $\check{X}^{i\times}$ for $\check{X}^i\setminus \{0\}$.
Set $\check{\A}^i_\lambda:=(\A_\lambda)_{\dagger,\chi^i}$, these are quantizations of $\C[\check{X}^i]$.
For $V\in \operatorname{Rep}(K_Q,\lambda|_{\kf_Q}-\kappa_Q)$, we write $\mathcal{E}$ for the corresponding
twisted local system on $\Orb_K$. We write $\check{\mathcal{E}}^i$ for the restriction of $\mathcal{E}$ to $\check{X}^{i\times}$. We note that $\mathcal{E}$ is an equivariant twisted system on an orbit,
so it is regular, see the last paragraph of Section \ref{subsec:RH}. By Proposition \ref{Prop:regular_local_systems}, $\check{\mathcal{E}}^i$ is regular as well.

\begin{Prop}\label{Lem:image_description11}
Let $V$ be an irreducible object in $\operatorname{Rep}(K_Q,\lambda|_{\kf_Q}-\kappa_Q)$. The following claims are true:
\begin{enumerate}
\item If $V\in \operatorname{Im}\bullet_{\dagger,\chi}$, then there are
$M^i\in \HC(\check{\A}^i_\lambda,\theta^i)$ with good filtrations such that
$\gr M^i|_{\check{X}_i^\times}\cong \check{\mathcal{E}}^i$, a $\C^\times$-equivariant isomorphism of twisted
local systems.
\item Conversely, assume that $V$ satisfies the following condition:
there are $M^i\in \HC(\check{\A}^i_\lambda,\theta^i)$ with good filtrations such that
$\gr M^i\cong \Gamma(\check{\mathcal{E}}^i)$, an isomorphism of graded weakly HC modules.
Then $V\in \operatorname{Im}\bullet_{\dagger,\chi}$.
\end{enumerate}
\end{Prop}
\begin{proof}
Set $M=V^{\dagger,\chi}$, see Section \ref{SS_Prop_restr}.
This module comes with the natural filtration, Remark \ref{Rem:distinguished_filtration}.

To prove (1), observe that $M_{\dagger,\chi}=V$, Lemma \ref{Lem:image_description}. Then
$(\gr M)_{\dagger,\chi}=\gr (M_{\dagger,\chi})=V$.
We have $\gr M|_{\Orb_K}=\mathcal{E}$ by the description of $\bullet_{\dagger,\chi}$
as the microlocalization functor, Section \ref{SSS_conclusion_microloc}.
It follows that $\gr (M_{\dagger,\chi^i})=(\gr M)_{\dagger,\chi^i}$, so we can
take $M^i:=M_{\dagger,\chi^i}$ in (1). This finishes the proof of (1).

We proceed to proving (2). Let $\mathcal{E}_\hbar$ be the object
in $\HC(\hat{\A}_\hbar^{\Orb},\theta)^{K,\kappa}$ corresponding to the
$\hbar$-adic completion of $R_\hbar(V)$ (with trivial filtration) under
equivalence (\ref{eq:final_jet_equiv}).
We need to show $M_{\dagger,\chi}\cong V$. Thanks to Section
\ref{SSS_conclusion_microloc}, this is equivalent to the claim
that the microlocalization of $\Gamma(\mathcal{E}_\hbar)$ to
$\Orb_K$ is $\mathcal{E}_\hbar$.

We are now going to show that
the assumption in (2) implies that (*) of Lemma \ref{Lem:coh_extension2}
holds for all $i$. Let $M^i_\hbar$ denote the $\hbar$-adic completion of $R_\hbar(M^i)$. We have $M^i_\hbar/\hbar M^i_\hbar\xrightarrow{\sim} \Gamma(\check{X}^{i\times}, \check{\mathcal{E}}^i)$. Let $\check{\mathcal{E}}^i_\hbar$
denote the microlocalization of $M^i_\hbar$ to $\check{X}^{i\times}$.
By Lemma  \ref{Lem:irred_strongness_HC}, the $\C^\times$-action on $\mathcal{E}$ is strong.
By Proposition \ref{Prop:slice_strongness}, the $\C^\times$-action on $\check{\mathcal{E}}^i$ is strong as well.
So $\check{\mathcal{E}}^i_\hbar$ is the unique quantization
of $\check{\mathcal{E}}^i$ by Theorem \ref{thm:quan_lag}.
We have natural homomorphisms $M^i_\hbar\rightarrow \Gamma(\check{\mathcal{E}}^i_\hbar)$.
The composition of $M^i_\hbar/\hbar M^i_\hbar\rightarrow \Gamma(\check{\mathcal{E}}^i_\hbar)/\hbar \Gamma(\check{\mathcal{E}}^i_\hbar)$ with the inclusion $\Gamma(\check{\mathcal{E}}^i_\hbar)/\hbar \Gamma(\check{\mathcal{E}}^i_\hbar)
\hookrightarrow \Gamma(\check{\mathcal{E}}^i)$ is an isomorphism. It follows that
$M^i_\hbar\xrightarrow{\sim} \Gamma(\check{\mathcal{E}}^i_\hbar)$ and (*) of Lemma
\ref{Lem:coh_extension2} holds for all $i$.

Now Lemma \ref{Lem:coh_extension2} implies that
the microlocalization of $\Gamma(\mathcal{E}_\hbar)$ coincides with $\mathcal{E}_\hbar$.
This finishes the proof.
\end{proof}

\section{Harish-Chandra modules in dimension $4$}
\subsection{Goal}\label{SS_goals}
We start by introducing some notation. Let $X$ be a conical symplectic singularity,
 and $\theta$ its Poisson anti-involution.
Assume that $\operatorname{codim}_X X^{sing}\geqslant 4$. In this case, as was explained in
Section \ref{SS_quant_nilp}, we can consider
the universal filtered quantization $\A_{\param}$ of $\C[X]$.
The anti-involution $\theta$ naturally lifts to $\A_{\param}$, denote the
lift by the same letter. Recall the category $\HC(\A_\lambda,\theta,\zeta)$ and its Serre quotient $\overline{\HC}(\A_\lambda,\theta,\zeta)$. See Section \ref{SS_quant_nilp}.

In the present section we describe the categories $\overline{\HC}(\A_\lambda,\theta,\zeta)$, where
\begin{itemize}
\item
$\A_\lambda$ is a quantization of a suitable four dimensional conical symplectic singularity $X$
with $X^{sing}=\{0\}$,
\item and $\theta$ is as above.
\end{itemize}

More precisely, we care about the following situation.
Let $\Orb$ be a nilpotent orbit satisfying (\ref{eq:codim_condition}). For $X$ we take
a slice to a codimension $4$ orbit in $\operatorname{Spec}(\C[\Orb])$. This is a conical
symplectic singularity. The list of possible options for $X$ is known in  all
cases, see \cite{KP1} for $\g=\mathfrak{sl}_n$, \cite{KP2} for the other classical Lie algebras,
and \cite{FJLS1} and \cite{FJLS2} for the exceptional Lie algebras.
Note that these papers deal with slices to open orbits in $\partial\Orb$,
where $\Orb$ is viewed inside of $\g^*$. Our slices will be obtained for those by passing to the normalization and taking a connected component.

Here is the list of possible varieties $X$ that can appear, we will elaborate on the detailed properties of the isomorphisms between $X$ and the model symplectic singularities in Section \ref{subsec:anti_involution}, and 
the orbits $\Orb$ giving rise to these slices in Section \ref{S_classif_thm}.

1) Type $a_2$: the variety $X$ is isomorphic to the closure of the minimal nilpotent orbit $\overline{\Orb_{min}(\mathfrak{sl}_3)}$ in
$\mathfrak{sl}_3$. (Note that when $\g$ is exceptional, the Slodowy slice in $\overline{\Orb}$ can be $2a_2$, the union of two irreducible components, both isomorphic to $\overline{\Orb_{min}(\mathfrak{sl}_3)}$, with the two isolated points glued together.)

2) Type $c_2$: the variety $X$ is the closure of the minimal nilpotent orbit
in $\mathfrak{sp}_4$. Equivalently, $X \simeq \C^4/\{\pm 1\}$.

3) Type $\tau$: $X=\C^4/\mathbb{Z}_3$. Here we write $\mathbb{Z}_3$ for the cyclic group with
$3$ elements thought of as the group of roots of unity of order $3$.
The action on $\C^4$ is as follows: an element $\zeta\in \mathbb{Z}_3$
acts by $\operatorname{diag}(\zeta,\zeta,\zeta^{-1},\zeta^{-1})$.

4) Type $a_2/S_2$ (see \cite[Prop. 3.7]{FJLS2}): Consider the involution
of $\overline{\Orb_{min}(\mathfrak{sl}_3)}$ obtained by restricting the outer involution of $\mathfrak{sl}_3$,
it is Poisson. This involution gives rise to an $S_2$-action. We take $X \simeq \overline{\Orb_{min}(\mathfrak{sl}_3)}/S_2$.

5) Type $\chi$. Let $\Orb$ be the principal orbit in $\mathfrak{sl}_5$. Let $\tilde{\Orb}$
be its universal cover (with Galois group $\mathbb{Z}_5$). Consider the variety
$\Spec(\C[\tilde{\Orb}])$. It is smooth in codimension $2$. For $X$ we take
the slice to a codimension $4$ orbit in $\Spec(\C[\tilde{\Orb}])$. It is studied in detail in \cite{BBFJLS1}.

We note that cases 3)-5) only appear for exceptional $\g$.

It will be convenient for us to use the following terminology.

\begin{defi}\label{defi:loc_sys_quantized}
Let $X$ be a conical symplectic singularity with $\operatorname{codim}_X X^{sing}\geqslant 4$, $\theta$ its graded Poisson anti-involution,
$\A$ a filtered quantization of $\C[X]$, and $\mathcal{E}$ a graded regular twisted local system on $X^{reg}\cap X^\theta$.
\begin{itemize}
\item[(i)]
We say that $\mathcal{E}$ is {\it $\A$-quantizable} if there is
a HC $(\A,\theta)$-module $M$ with a $\C^\times$-equivariant isomorphism $\gr M|_{X^{reg}}\cong \mathcal{E}$
of twisted local systems on $X^{reg}\cap X^\theta$.
\item[(ii)] Moreover, if in this case we additionally have
$\gr M\xrightarrow{\sim} \Gamma(\mathcal{E})$ (for some HC $(\A,\theta)$-module $M$ as in (i)), then we say that $\mathcal{E}$ is {\it strongly $\A$-quantizable}.
\end{itemize}
\end{defi}

In what follows, we will always be in the following situation:
\begin{itemize}
\item The category of regular twisted local systems on $X^\theta\cap X^{reg}$ is semisimple with finitely many simples,
\item and each simple admits a grading.
\end{itemize}

In this case, thanks to Lemma \ref{Lem:graded_loc_sys}, it is enough to check (i) and (ii) only
for irreducible regular twisted local systems.

\begin{defi}\label{defi:obstructive_slice}
Let $X$ be one of the five 4-dimensional conical symplectic singularities mentioned above
and $\theta$ be a graded Poisson anti-involution of $\C[X]$. We say that $(X,\theta)$
is {\it unobstructive} if for every quantization $\A$, every graded regular twisted local system 
(with twist recovered from the microlocalization of $\A$ to $\Orb$, compare to Theorem
\ref{thm:lag_vec}) whose $\C^\times$-action  is strong on $X^\theta\setminus \{0\}$ is strongly $\A$-quantizable.
\end{defi}

The latter definition is motivated by Proposition \ref{Lem:image_description11}: for an unobstructive slice the conditions of this proposition are always satisfied, so it does not obstruct extending
a HC module from $X^{reg}$ to $X$.

\subsection{Slices}
As was mentioned in the previous section, we have a description of the (formal) slices to all codimension 4 $G$-orbits, $\Orb'$, in
$X:=\operatorname{Spec}(\C[\Orb])$. However, we need more: we need to know the Poisson structure, the subalgebra of $\C^\times$-finite elements,
and the Poisson anti-involution $\theta$. In this and the next section, we will see that these structures
are recovered essentially uniquely.

Here we are going to prove that
\begin{itemize}
\item[(*)]
the slice to $\Orb'$ in $X$ at a $\theta$-stable
point $\chi'$ is well-defined as a formal Poisson scheme with an anti-Poisson involution
and an action of $\C^\times$ that rescales the bracket.
\end{itemize}

First, we explain what we mean by a formal slice. The completion $\C[X]^{\wedge_{\chi'}}$ decomposes as
\begin{equation}\label{eq:decomp_101}\C[X]^{\wedge_{\chi'}}=\C[[x_1,y_1,\ldots,x_k,y_k]]\widehat{\otimes}
A\end{equation}
where $2k=\dim \Orb'$, the factor $\C[[x_1,\ldots,y_k]]$ is the algebra of functions on the formal
symplectic polydisc with Darboux basis $x_1,\ldots,y_k$, and $A$ is a complete local ring that coincides
with the centralizer of $\C[[x_1,\ldots,y_k]]$ in $\C[X]^{\wedge_{\chi'}}$. See, for example, \cite[Theorem 2.3]{Kaledin_symplectic}, for a more general result. A quotient $A$ of $\C[X]^{\wedge_{\chi'}}$ that factors through (\ref{eq:decomp_101}) will be called a {\it transverse slice} in $X^{\wedge_{\chi'}}$. Note that if a reductive group $\tilde{Q}$ acts on
$\C[X]^{\wedge_{\chi'}}$ by algebra automorphisms rescaling the bracket,
then one can find $x_1,\ldots,y_k$
such that $\operatorname{Span}_{\C}(x_1,\ldots,y_k)$ is $\tilde{Q}$-stable.
This gives a $\tilde{Q}$-action on $A$ viewed as a quotient.

In particular, taking the formal neighborhood at $\chi'$ of an actual slice in $X$ we get a formal slice in the sense of the
previous paragraph. So, the following lemma implies (*).

\begin{Lem}\label{Lem:slice_well_defined}
Suppose $\tilde{Q}$ is a reductive group acting on $\C[X]^{\wedge_{\chi'}}$ by algebra automorphism
and rescaling the Poisson bracket such that the decomposition is $\tilde{Q}$-stable. Then every transverse slice to $\operatorname{Spec}(\C[[x_1,\ldots,y_k]])$ in $X^{\wedge_{\chi'}}$ is conjugate
to the standard slice $x_1=\ldots=y_k=0$ via a  Hamiltonian
automorphism. If the transverse slice is $\tilde{Q}$-stable, we can pick this
Hamiltonian automorphism to be  $\tilde{Q}$-equivariant.
\end{Lem}
\begin{proof}
By the implicit function theorem (for formal completions), every transverse slice
in $X^{\wedge_{\chi'}}$ is defined by the ideal
$$(x_i-F_i, y_i-G_i| i=1,\ldots,k)$$
for suitable elements $F_i,G_i$ that lie in the ideal $I\subset \C[X]^{\wedge_{\chi'}}$ generated by the maximal ideal of a fixed transverse slice $A$
and the homogeneous quadratic polynomials in $x_i,y_i$. We can modify the elements $F_i,G_i$
without changing the ideal and assume that they actually lie in the maximal ideal of $A$.

Now consider the element $f:=\sum_{i=1}^k (-y_i F_i+x_i G_i)$. The Hamiltonian automorphism
$\exp(\{f,\cdot\})$ sends $x_i$ to $x_i-F_i$ and $y_i$ to $y_i-G_i$ for all $i$. The condition
that the slice is $\tilde{Q}$-stable is equivalent to the map $\operatorname{Span}_\C(x_1,\ldots,y_k)
\rightarrow \operatorname{Span}_\C(F_1,\ldots,G_k)$ being $\tilde{Q}$-equivariant. From here it is
easy to see that $f$ transforms under $\tilde{Q}$ with the character opposite of that of the bracket.
Hence $\exp(\{f,\cdot\})$ is $\tilde{Q}$-equivariant.
\end{proof}

\subsection{Anti-involutions} \label{subsec:anti_involution}

Now we are going to deduce two corollaries that we are going to use in our analysis of Harish-Chandra modules.

We follow the notations in Sction \ref{subsec:extn_codim2}. Let $\Orb'$ be a $G$-orbit in $\overline{\Orb} \subset \g^*$ of codimension $4$. Pick a point $\underline{\chi} \in \Orb'$ and  consider an $\slf_2$-triple $(e',h',f')$
corresponding to $\underline{\chi}$ and the Slodowy slice $\underline{S} = e'+\mathfrak{z}_{\g}(f') \subset \g$
associated to this triple. Let $Q'$ denote the centralizer of $(e',h',f')$ in $G$. Let $p: X = \operatorname{Spec}(\C[\Orb]) \to \overline{\Orb}$ denote  the natural projection and $\check{X}$ denote any connected component of the preimage of $\underline{S}$
in $X$. When an anti-involution $\theta$ of $\g$ is present, we assume additionally that the $\slf_2$-triple $(e',h',f')$ is normal (see the beginning of Section \ref{SS_Sl_sl}), so that both $\underline{S}$ and $\check{X}$ are $\theta$-stable.

Define the $Q'$-invariant subspaces
\[
\mathfrak{z}_{\g}(f')_k=\{ x \in \mathfrak{z}_{\g}(f') : [h',x]=kx\}
\] 
of $\mathfrak{z}_{\g}(f')$ for $k \leqslant 0$, so that we have a $Q'$-equivariant direct sum decomposition $\mathfrak{z}_{\g}(f') = \bigoplus_{k \le 0} \mathfrak{z}_{\g}(f')_k$, which is preserved by $\theta$.
Note that $\mathfrak{z}_{\g}(f')_0$ is simply $\mathfrak{q}'$, the centralizer of the $\slf_2$-triple $(e',h',f')$ in $\g$. 

Let $\underline{S}_\Orb := \underline{S} \cap \overline{\Orb}$ and $\pi_{0}: \underline{S}_\Orb \to \mathfrak{q}'$ be the restriction of the 
$Q' \times \cm$-equivariant linear projection of $\underline{S}$ onto $\mathfrak{q}' = \mathfrak{z}_{\g}(f')_0$.  Then $\phi := \pi_0 \circ (p |_{\check{X}}): \check{X}  \to \q' \simeq (\q')^*$ is a moment map for the Hamiltonian $Q'$-action on $\check{X}$.


\begin{Prop} \label{prop:model_a2c2}
	Assume $\check{X}$ is of type $a_2$ (respectively $c_2$), then the morphism $\phi$ induces a $G' \times \cm$-equivariant Poisson isomorphism between $\check{X}$ and $\overline{\Orb_{min}(\g')}$, where $\g' = \mathfrak{sl}_3$ (resp. $\mathfrak{sp}_4$.) and $G' = \SL_3$ (resp. $\Sp_4$). Here $\g'$ is a simple factor of $\mathfrak{q}'$ preserved by $\theta$, and hence we have a (nontrivial) group homomorphism $G' \to Q'$ which gives rise to the action of $G'$ on $\check{X}$. The $\cm$-action on $\overline{\Orb_{min}(\g')} \subset \g'$ is induced by the squared dilation action on $\g'$.
	
	Moreover, when the $\slf_2$-triple $(e',h',f')$ is normal, the isomorphism above intertwines the anti-involution on $X$ (induced from $\theta$ on $\g$) and the anti-involution on $\overline{\Orb_{min}(\g')}$ induced from the restriction of $\theta$ to $\g'$.
\end{Prop}

\begin{proof}
We split the proof into three cases.

{\it Case 1: $\g$ is of one of the exceptional types}.  The claims follow from case (1) of \cite[Proposition 3.3]{FJLS1} (see also Remark 3.4 there).
	In our case, the results there imply that there exists a set $J \subset \mathbb N$, which is either empty or $\{2\}$, satisfying the following properties: for each $i \in J$, there exists
	a highest weight vector $x_i \in \mathfrak{z}_{\g}(f')_{-i}$ for the action of $Q'^\circ$, 
	and there exists a minimal nilpotent element $x_0 \in \mathfrak{q}'$, such that for $v = x_0 + \sum_{i \in J}  x_i$, we have $e' + v \in \underline{S}_\Orb$ and $\underline{S}_\Orb = \overline{Q'\cdot(e'+v)}$. Moreover, $\pi_0$ induces an isomorphism between each irreducible component of $\underline{S}_\Orb$ and the closure of the $Q'^\circ$-orbit through the minimal nilpotent element $x_0 \in \mathfrak{q}'$. We can assume that $x_0$ lies in a simple factor $\g' \in \mathfrak{q}'$, which is either isomorphic to $\mathfrak{sl}_3$ or $\mathfrak{sp}_4$ (which are the only simple Lie algebras whose minimal orbits have dimension $4$.)

{\it Case 2: $\g$ is classical}.
	Note that, even though \cite[Proposition 3.3]{FJLS1} is only stated for exceptional Lie algebras, similar results also hold for classical Lie algebras $\g$, in which case $J = \varnothing$, i.e., $e' + x_0 \in \underline{S}_\Orb$ and $\underline{S}_\Orb = \overline{Q'\cdot(e'+v)} = \overline{Q'^\circ \cdot(e'+v)}$ (for classical $\g$, a dimension $4$ slice $\underline{S}_\Orb$ is always irreducible). 

{\it Case 2.1: $\g$ is of type $B,C$ or $D$}. The only slices of dimension $4$ are of type $c_2$ by Table I of \cite[Section 3.4]{KP2} (note that the singularity of type $b_2$ in the same table is isomorphic to $c_2$). Namely, let $\g$ be a classical Lie algebra and $G$ be the corresponding classical Lie group. Let $\Orb$ be a coadjoint nilpotent $G$-orbit with Jordan type $\tau=(\tau_1,\ldots,\tau_k)$ (we append $\tau$
with zeros as necessary). It is easy to
show, for example, using results of \cite{KP2}, that the condition $\operatorname{codim}_{\overline{\Orb}}\overline{\Orb}\setminus \Orb\geqslant 4$ is equivalent to $\tau_i-\tau_{i+1} \leqslant 1$ for all $i$. There are two types of codimension $4$ orbits $\Orb'$ in $\overline{\Orb}$ (which is normal in this case): 

\begin{itemize}
	\item Type $c_2$:
		those $\Orb'$ with the corresponding slice $\underline{S}_{\Orb}$ of type $c_2$ are in bijection with the indices $i$ such that $\tau_i = \tau_{i+1} + 1 = \tau_{i+2} + 1 = \tau_{i+3} + 2$ and $\tau_i$ is even (resp. odd) if $\g$ is symplectic (resp. orthogonal). The corresponding codimension $4$-orbit $\Orb^i$ has Jordan type \[\tau^i=(\tau_1,\ldots,\tau_{i-1},\tau_{i+1},\tau_{i+1},\tau_{i+1}, \tau_{i+1}, \tau_{i+4},\ldots).\]
		Pick $e^i\in \Orb^i$. The $4$ copies of $\tau_{i+1}$ in $\tau^i$ contributes to a simple factor $\g' \simeq \mathfrak{sp}_4$ of the reductive centralizer $\q'$ of $e'$ in $\g$.
	\item Type $b_2$:	
		Those $\Orb'$ with the corresponding slice $\underline{S}_{\Orb}$ of type $b_2$ are in bijection with the indices $i$ such that $\tau_i = \tau_{i+1}  = \tau_{i+2} + 1 = \tau_{i+3} + 2 = \tau_{i+4} + 2$ and $\tau_i$ is even (resp. odd) if $\g$ is orthognal (resp. symplectic). The corresponding codimension $4$-orbit $\Orb^i$ has Jordan type 
		\[\tau^i=(\tau_1,\ldots,\tau_{i-1}, \tau_{i+2}, \tau_{i+2}, \tau_{i+2}, \tau_{i+2}, \tau_{i+2}, \tau_{i+5},\ldots).\]
		Pick $e^i\in \Orb^i$. The $5$ copies of $\tau_{i+1}$ in $\tau^i$ contributes to a simple factor $\g' \simeq \mathfrak{so}_5 \simeq \mathfrak{sp}_4$ of the reductive centralizer $\q'$ of $e'$ in $\g$.
\end{itemize}

For either type, let $e_0 \in \g' \subset \mathfrak{q}^i$ be an element from the minimal nilpotent orbit $\Orb_{min}$ in $\g'$.  By \cite[Lemma 4.1]{JLS}, we have $e^i+e_0\in \Orb$. By the same argument as in the proof of \cite[Proposition 4.2]{JLS}, the map
$x\mapsto e^i+x$ embeds $\overline{\Orb}_{min}$ into $\underline{S}_{\Orb}^i := (e^i+\mathfrak{z}_{\g}(f^i))\cap \overline{\Orb}$
as a closed subvariety. As the latter is 4-dimensional and irreducible (in fact, normal), we get an isomorphism and this is clearly the inverse of $\phi$.

{\it Case 2.2: $\g$ is of type $A$}. The only slices of dimension $4$ are of type $a_2$ by \cite{KP1}. 

Let $\g=\mathfrak{sl}_n$. We can replace $G$ with any connected reductive group such that
$[\g,\g]=\mathfrak{sl}_n$. So we take $G=\operatorname{GL}_n$ instead.

Let $\Orb$ be a coadjoint nilpotent $G$-orbit with Jordan type $\tau=(\tau_1,\ldots,\tau_k)$ (we append $\tau$
with zeros as necessary). It is easy to
show, for example, using results of \cite{KP1}, that the condition $\operatorname{codim}_{\overline{\Orb}}\overline{\Orb}\setminus \Orb\geqslant 4$ is equivalent to $\tau_i-\tau_{i+1}\leqslant 1$ for all $i$. The codimension $4$ orbits in $\overline{\Orb}$
are in bijection with the indices $i$ such that $\tau_i=\tau_{i+1}+1=\tau_{i+2}+2$.
The corresponding codimension $4$-orbit $\Orb^i$ has Jordan type $$\tau^i=(\tau_1,\ldots,\tau_{i-1},\tau_{i+1},\tau_{i+1},\tau_{i+1},\tau_{i+3},\ldots).$$
Pick $e^i\in \Orb^i$. The reductive part $Q^i$ of the centralizer of $e^i$ is
$\prod_j \operatorname{GL}_{d_j}$, where $j$ is a part that occurs in $\tau^i$
and $d_j$ is its multiplicity. So we can write $Q^i=\operatorname{GL}_3\times \underline{Q}^i$,
where the first factor corresponds to the part $\tau_{i+1}$ and the second factor is the product of
the $\operatorname{GL}$-factors corresponding to the remaining parts. Let $\g' = \slf_3 = [\gl_3, \gl_3]$ be the semisimple part of the Lie algebra of first factor $\GL_3$ of $Q^i$.

Now let $e_0 \in \g' \subset \mathfrak{q}^i$ be an element from the minimal nilpotent orbit
$\Orb_{min}$ in $\g'$.  We claim that $e^i+e_0\in \Orb$. For this it is enough to consider the similar question
when $e^i$ corresponds to the partition $(\tau_{i+1},\tau_{i+1},\tau_{i+1})$. If $\tau_{i+1}=1$, there
is nothing to prove. If $\tau_{i+1}>1$, then it is easy to see that $e^i+e_0$ has the same rank as
$e^i$. Also $(e^i+e_0)^{k}=0$ if and only if $k>\tau_{i+1}$ and $(e^i+e_0)^{\tau_{i+1}}\neq 0$. Since the orbit of $e^i+e_0$ has
$e^i$ in its closure, we see that $e^i+e_0\in \Orb$. Alternatively, this also also follows from \cite[Lemma 4.1]{JLS}, as in the case of type $B$, $C$ and $D$ as above.Then we can again argue as in Case 2.1 that the map $x\mapsto e^i+x$ gives an isomorphism $\overline{\Orb}_{min} \simeq \underline{S}_{\Orb}^i := (e^i+\mathfrak{z}_{\g}(f^i))\cap \overline{\Orb}$.


Now observe that the isomorphism $\check{X} \xrightarrow{\sim} \Orb_{min}$ in each case above is $G' \times \cm$-equivariant and intertwines the moment maps to $\g' \simeq (\g')^*$. Since $\Orb_{min}$ is a $G'$-orbit, the symplectic form on $\Orb_{min}$ is completely determined by the moment map. Therefore the isomorphism $\check{X} \xrightarrow{\sim} \Orb_{min}$ preserves the symplectic forms and hence is Poisson. 

Let $R := \C[\check{X}] \simeq \C[\overline{\Orb}_{min}]$ be the coordinate algebra, then the moment map induces a Lie algebra morphism $\g' \to R$ that identifies $\g'$ with the degree $2$ component $R_2$ of $R$ and $R_2$ generates $R$ by Lemma 4.10 and Lemma 4.14 of \cite{BK}. Therefore, when $e' \in \Orb'$ is chosen to be in $\kf^\perp$ and the $\slf_2$-triple $(e',h',f')$ is normal, the action of the anti-Poisson involution $\theta$ on $\check{X}$, or equivalently $R$, is completely determined by its restriction to the Lie subalgebra $R_2 \simeq \g'$.
\end{proof}

\begin{Lem}\label{lem:tau_a_2S_2}
Suppose that $\check{X}$ is of type $\tau$ or $a_2/S_2$. For the contracting $\cm$-action on the model symplectic singularities, we take
\begin{itemize}
\item for $\C^4/\Z_3$, the action induced from the standard dilation action on $\C^4$,
\item for $\overline{\Orb_{min}(\mathfrak{sl}_3)} /S_2$, the action  induced from the $\cm$ action on $\overline{\Orb_{min}(\mathfrak{sl}_3)}$ that is the restriction of the squared dilation action on $\mathfrak{sl}_3$.
\end{itemize}
Then there exists a $\cm$-equivariant Poisson isomorphism between $\check{X}$ and the corresponding model symplectic singularity.
\end{Lem}
\begin{proof}
First consider the case when $\check{X}$ is of type $a_2/S_2$. This type of singularity only appears when $\g$ is a simple Lie algebra of type $E_7$ and $\Orb'$, $\Orb$ are the nilpotent orbits in $\g$ with respective Bala-Carter labels $A_3+A_2+A_1$, $A_4+A_1$. In this case $Q‘$ is isomorphic to $SL_2$. By the discussions in \cite[Section 3.1]{FJLS2}, especially Lemma 3.2, Lemma 3.6 and Proposition 3.7 of \cite{FJLS2} (and their proofs), we have an $\SL_2 \times \cm$-equivariant isomorphism $\check{X} \simeq \overline{\Orb_{min}(\mathfrak{sl}_3)} / S_2$. Here $\slf_2 \simeq \mathfrak{so}_3$ is regarded as the fixed subalgebra of $\slf_3$ by an outer involution, giving rise to an $\SL_2$-action on $\overline{\Orb_{min}(\mathfrak{sl}_3)} / S_2$.

When $\check{X}$ is of type $\tau$, it follows from Theorem 12.6 and Corollary 12.7 of \cite{FJLS1} (and their proofs) that there is a $\cm$-equivariant isomorphism $X \simeq \C^4 / \Z_3$, with respect to the Kazhdan $\cm$-action on $X$ and the $\cm$-action on $\C^4 / \Z_3$ induced by the dilation action on $\C^4$.

Finally, by \cite[Theorem 3.1]{Namikawa_equiv}, we can always modify the above isomorphisms to make them Poisson and still $\cm$-equivariant, since all the relevant symplectic structures have degree $2$ with respect to the relevant $\cm$-actions.
\end{proof}

We will describe possible anti-involutions $\theta$ in these two cases later.

Starting the next section we proceed to classifying Harish-Chandra modules with full support
for the four out of five singularities of interest: $a_2,c_2,\C^4/\mathbb{Z}_3, a_2/S_2$.
We will see that the varieties of types $c_2,\C^4/\mathbb{Z}_3$ are unobstructive for all
choices of involutions. The variety of type $a_2$ is unobstructive for one of the two possible
choices of an involution.

\subsection{Type $a_2$}\label{SS_type_a2}
The first case to consider is when $X$ is the closure of the minimal orbit (Jordan type $(2,1)$)
in $\g:=\mathfrak{sl}_3$. In this case, the anti-involution $\theta$ on $X$ is induced from a Lie algebra anti-involution of $\g$ by Proposition \ref{prop:model_a2c2}. Consider
the universal filtered quantization $\A_{\param}$ of $\C[X]$. We have $\param=\C$. The quantizations of $\C[X]$ have the form
$\mathscr{D}^{\lambda}(\mathbb{P}^2)$, the algebra of $\lambda$-twisted
differential operators on $\mathbb{P}^2$. The period of such a quantization is $\lambda-\frac{1}{2}c_1(K_{\mathbb{P}^2})$. We consider anti-involutions $\theta$ of the universal quantization $\A_\param$. 

We have $\theta=-\sigma$ for an involution $\sigma$ of $\g$. Let $\Orb$ be the open $\SL_3$-orbit in $X$, then $\param=H^2(\Orb,\C)$ and the action of $\theta$ on $\param$ must be the natural one induced by the action on $\Orb$ by universality 
(see \cite[Proposition 3.5]{orbit} for the description of what it means to be universal).
Up to conjugation, the Lie algebra $\mathfrak{sl}_3$ has two involutions
$\sigma$. One has $\kf=\mathfrak{gl}_2$ and the other has $\kf=\mathfrak{so}_3$. When $\kf=\mathfrak{gl}_2$,
the involution $\sigma$ is inner, and $\sigma$ acts trivially
on $\param$. When $\kf=\mathfrak{so}_3$, the involution $\sigma$ is outer, and
$\sigma$ act on $\param$ by $-1$. Note that the  action of $-1$ on $T^*\mathbb{P}^2$
is contained in the dilation action of $\C^\times$, a connected group. So the actions of
$\sigma$ and $\theta$ on the $H^2(T^*\mathbb{P}^2,\C)$ coincide.

Let us write $\mathscr{D}^\lambda_{\Pb^2}$ for the sheaf of $\lambda$-twisted
differential operators on $\Pb^2$. We have the global section
functor $\mathscr{D}^\lambda_{\Pb^2}\operatorname{-mod} \rightarrow \mathscr{D}^\lambda(\Pb^2)\operatorname{-mod}$,
where the source category is that of coherent $\mathscr{D}^\lambda_{\Pb^2}$-modules and the target is the category
of finitely generated $\mathscr{D}^\lambda(\Pb^2)$-modules.

We claim that
the functor $\Gamma$ induces an equivalence between the
quotients of $\mathscr{D}^\lambda_{\Pb^2}\operatorname{-mod}$ and $\mathscr{D}^\lambda(\Pb^2)\operatorname{-mod}$
by the Serre subcategories of $\mathcal{O}$-coherent twisted $\mathscr{D}$-modules and
finite dimensional $\mathscr{D}^\lambda(\Pb^2)$-modules, respectively. Indeed, let $\operatorname{Loc}: \mathscr{D}^\lambda_{\Pb^2}\operatorname{-mod} \to \mathscr{D}^\lambda(\Pb^2)\operatorname{-mod}$
denote localization functor, which is the left adjoint of $\Gamma$ (see \cite[Section 3.3]{BB}). The kernels and cokernels of the adjunction unit and
counit morphisms for the adjoint pair $(\operatorname{Loc},\Gamma)$ are killed by the microlocalization
to $\Orb=\overline{\Orb}\setminus \{0\}$ because the resolution morphism $T^*\Pb^2\rightarrow \overline{\Orb}$
is an isomorphism over $\Orb$. It follows that they are supported at $\{0\}$, which implies our claim.

The group $K$ acts on $\Pb^2$ with finitely many orbits. Thanks to this we can classify
irreducible $(K,\kappa)$-equivariant $\mathscr{D}^\lambda_{\Pb^2}$-modules, see Lemma \ref{Lem:fin_many_orbit_classif}. To apply the lemma, we need the following notation.
Let $v\in \C^3$ be a nonzero vector and $[v]$ be the corresponding line
in $\Pb^2$. Let $P_{[v]}$ denote the stabilizer of $[v]$ in $\operatorname{SL}_3$. We identify the character
lattice $\mathfrak{X}(P_{[v]})$ with $\Z$, with $1$ corresponding to the character of
$P_{[v]}$ in $\C v$ (and so the $G$-equivariant line bundle $\mathcal{O}(1)$ corresponds to $-1$). So $\lambda$ can be viewed as
a complex number. This also gives an identification of
$H^2(T^*\mathbb{P}^2,\C)$ with $\C$ (with $c_1(\mathcal{O}(1))=-1$ and $c_1(K_{\mathbb{P}^2})=3$).

\subsubsection{The case of $\kf=\mathfrak{gl}_2$}\label{SS_HC_sl3_gl2}
Here $\theta$ comes from the inner involution of $\mathfrak{sl}_3$. We identify $\kf^{*K}$ with $\C$ so that the element $1\in \C$
corresponds to $-\operatorname{tr}\in \kf^{*K}$. So we can view $\kappa$ as a complex number.

Our goal is to establish the following claim.

\begin{Lem}\label{Lem:sl3_gl2_unobstructive}
Let $X$ be the closure of the minimal orbit in $\mathfrak{sl}_3$ and $\theta:=-\sigma$,
where $\sigma$ is an inner involution of $\mathfrak{sl}_3$. Then $(X,\theta)$ is unobstructive.
\end{Lem}
\begin{proof}
{\it Step 1}.
In this case we have three $K$-orbits on $\mathbb{P}^2$. Let us write points of $\Pb^2$
as $[a:b:c]$. We embed $\operatorname{GL}_2$ into
$\operatorname{SL}_3$ as $x\mapsto \operatorname{diag}(\det(x)^{-1},x)$.
The representatives of the three orbits are $[1:1:0],[0:1:0]$ and
$[1:0:0]$.

Let us see what the classification of irreducible twisted equivariant twisted local systems (see Lemma
\ref{Lem:homog_classif}) on the three orbits looks like:

1) {\it The case of $[1:1:0]$}. The group $K_{[v]}$ consists of all matrices of the form
$$\begin{pmatrix}x&0&0\\0&x&y\\0&0&x^{-2}\end{pmatrix},$$
and the orbit $K \cdot [v]$ is open in $\Pb^2$. We identify $\mathfrak{X}(K_{[v]})$ with $\Z$, where $1$
corresponds to the character that sends the above matrix to $x$. 
The restriction map $\mathfrak{X}(P_{[v]})\rightarrow \mathfrak{X}(K_{[v]})$ sends
$1$ to $1$.
So, by Lemma \ref{Lem:homog_classif}, the existence of a $(K,\kappa)$-equivariant $\mathcal{O}$-coherent
$\mathscr{D}^\lambda_{K/K_{[v]}}$-module is equivalent to $\lambda-\kappa\in \Z$.
If this is the case we have a unique twisted
equivariant irreducible $\mathcal{O}$-coherent module because $K_{[v]}$ is connected.

2) {\it The case of $[0:1:0]$}. The group $K_{[v]}$ consists of all matrices of the form
$$\begin{pmatrix}x_1&0&0\\0&x_2&y\\0&0&x_1^{-1}x_2^{-1}\end{pmatrix},$$
and therefore the orbit $K \cdot [v]$ is closed and isomorphic to $\Pb^1$. The restriction of $1\in \mathfrak{X}(P_{[v]})$ sends the matrix above to $x_2$,
while the restriction of $\det^{-1}$ from $K$ to $K_{[v]}$ sends that matrix
to $x_1$. We conclude that the existence of a $(K,\kappa)$-equivariant
$\mathcal{O}$-coherent $\mathscr{D}^\lambda_{K/K_{[v]}}$-module is equivalent to
$\lambda,\kappa\in \Z$. Again, an irreducible module is unique.

3) {\it The case of $[1:0:0]$}. The group $K_{[v]}$ is $K$ itself and the orbit $K \cdot [v]$ only consists of the point $[1:0:0]$. The restriction
map $\mathfrak{X}(P_{[v]})\rightarrow \mathfrak{X}(K)$ sends $1$ to $\det^{-1}$.
The conclusions regarding the existence of a twisted equivariant coherent
module and the classification of irreducibles are as in 1).

The above case by case considerations show that in all cases we have
\begin{itemize}
\item three irreducible
$(K,\kappa)$-equivariant $\mathscr{D}^\lambda_{\Pb^2}$-modules when both $\kappa,\lambda$
are integral,
\item two irreducible $(K,\kappa)$-equivariant $\mathscr{D}^\lambda_{\Pb^2}$-modules when $\lambda-\kappa$
is integral but $\kappa$ is not integral,
\item no irreducible $(K,\kappa)$-equivariant $\mathscr{D}^\lambda_{\Pb^2}$-modules when $\lambda-\kappa$
is not integral.
\end{itemize}
We only have one $\mathcal{O}$-coherent $(K,\kappa)$-equivariant $\mathscr{D}^\lambda_{\mathbb{P}^2}$-module
when $\kappa$ and $\lambda$ are both integral. We conclude that we have two non $\mathcal{O}$-coherent
$(K,\kappa)$-equivariant $\mathscr{D}^{\lambda}_{\mathbb{P}^2}$-modules if $\lambda-\kappa$ is an integer
and none otherwise.

{\it Step 2}.
Now let us examine $\Orb\cap \kf^\perp$. We identify $\g$ with $\g^*$ using the trace pairing. The space $\kf^\perp$ consists of the nonzero
matrices of the following form:
$$\begin{pmatrix}0&?&?\\
?&0&0\\?&0&0\end{pmatrix}.$$
Recall that $\Orb$ consists of the nilpotent rank one matrices. So $\Orb\cap \kf^\perp$
consists of two components
\[
  \Orb^1_K = \left\{
     \begin{pmatrix}
       0&?&?\\
       0&0&0\\
       0&0&0
    \end{pmatrix}
  \right\} \cong \C^{2*}\setminus \{0\}
  \quad
   \Orb^2_K = \left\{
  \begin{pmatrix}
  0&0&0\\
  ?&0&0\\
  ?&0&0
  \end{pmatrix}
  \right\} \cong \C^{2}\setminus \{0\},
\]
with the usual linear actions of $\operatorname{GL}_2$ on $\C^2$ and its dual $\C^{2*}$.
We claim that the condition $\lambda-\kappa\in \Z$ is equivalent to the condition of
Corollary \ref{Cor:equiv_condition}.
Take $\chi\in \Orb^2_K$ to be the unit matrix $E_{21}$. The corresponding group $K_Q$
is $\{\operatorname{diag}(x,x,x^{-2})| x\in \C^\times\}$.
The element $\rho_\omega$ from Definition \ref{defi:half_character}
is easily seen to lie in $\Z+1/2$. And since the period is $\lambda+3/2$, the condition
$\lambda-\kappa\in \Z$ is indeed equivalent to the condition in Corollary \ref{Cor:equiv_condition}.
The case of $\Orb^1_K$ is similar.

Thanks to Corollary \ref{Cor:equiv_condition}, if $\HC_{\Orb^i_K}(\mathscr{D}^\lambda(\Pb^2))^{K,\kappa}\neq \{0\}$
then $\lambda-\kappa\in \Z$.
Our case by case analysis tells us for each of the orbits $\Orb_K=\Orb^i_K,i=1,2,$ in $\Orb\cap \kf^\perp$ we have
\begin{align*}
\HC_{\Orb_K}(\mathscr{D}^\lambda(\Pb^2))\xrightarrow{\sim}\operatorname{Rep}(K_Q, \lambda|_{\kf_Q}- \kappa_Q).
\end{align*}

Note that $\lambda - \kappa \in \Z$ is the condition for $\kappa$ to have
a strong $(K,\kappa)$-action on $\mathcal{E}$, this also follows from
Lemma \ref{Lem:homog_classif}.
So every weakly $\C^\times$-equivariant twisted local system on $X^\theta\setminus \{0\}$
is $\A$-quantizable for all quantizations $\A$ of $\C[X]$.

{\it Step 3}. The goal of this step is to show that every weakly $\C^\times$-equivariant twisted local
system on $X^\theta\setminus \{0\}$ is strongly $\A$-quantizable for any quantization $\A$. This will finish the proof. Note that the category of weakly $\C^\times$-equivariant twisted local systems on
$X^\theta\setminus \{0\}$ is semisimple so we can restrict to irreducible objects $\mathcal{E}$ there.

Consider the closed $K$-orbit of the single point $[1:0:0] \in \Pb^2$ as in 3). The unique $(K,\kappa)$-equivariant irreducible $\mathscr{D}^\lambda_{\Pb^2}$-module supported on this point has a natural filtration such that the associated graded is a line bundle over the characteristic variety of this $\mathscr{D}^\lambda_{\Pb^2}$-module, which is the cotangent vector space of $[1:0:0]$ in $\Pb^2$. Note that this cotangent vector space is mapped isomorphically to the closure $\overline{\Orb}^1_K \subset \overline{\Orb}\cap \kf^\perp$ via the resolution map $T^*\Pb^2 \to \overline{\Orb}$. By identifying this $\mathscr{D}^\lambda_{\mathbb{P}^2}$-module with its global sections, we see that the corresponding HC $(\A,\theta)$-module $M$ has a filtration whose associated graded module $\gr M$ is a line bundle over $\overline{\Orb}^1_K \cong \C^{2*}$. Since $\overline{\Orb}^1_K$ is normal and the natural map $\gr M \to \Gamma((\gr M)|_{\Orb^1_K})$ is an isomorphism, we conclude that $\Gamma((\gr M)|_{\Orb^1_K})$ is Cohen-Macaulay and coincides
with the associated graded of $\Gamma(M|_{\Orb^1_K})$ for any $\lambda \in \C$ satisfying $\lambda - \kappa \in \Z$.

To analyze the HC-module supported on $\overline{\Orb}^2_K$, we take $\varepsilon: g \to (g^t)^{-1}$ to be the outer automorphism of $\SL_3$ which preserves $K=\operatorname{GL}_2$. We also denote the induced automorphism of the Lie algebra $\slf_3$ by $\varepsilon$, so that $\varepsilon$ commutes with $\theta$ and interchanges the two $\operatorname{GL}_2$-orbits in $\Orb\cap \kf^\perp$. Since  $\A_\mu = \mathscr{D}^{\mu + \frac{3}{2}}(\mathbb{P}^2)$, the involution $\varepsilon$ induces an isomorphism $\A_\mu \cong \A_{-\mu}$ of algebras ($\A_0 = \mathscr{D}^{ \frac{3}{2}}(\mathbb{P}^2)$ is the canonical quantization of $\overline{\Orb}$). Twisting the $\A_\mu$-module structures and the $K$-actions by $\varepsilon$ gives equivalences of categories $\HC(\A_\mu,\theta)^{K,\kappa} \cong \HC(\A_{-\mu},\theta)^{K,-\kappa}$ and $\HC_{\overline{\Orb}^1_K}(\A_\mu,\theta)^{K,\kappa} \cong \HC_{\overline{\Orb}^2_K}(\A_{-\mu},\theta)^{K,-\kappa}$ since $\varepsilon$ interchanges $\Orb^1_K$ and $\Orb^2_K$. Therefore we end up with the same conclusion for HC-modules supported on $\overline{\Orb}^2_K$ as for those  HC-modules supported on $\overline{\Orb}^1_K$  above.
\end{proof}

\subsubsection{The case of $\kf=\mathfrak{so}_3$}\label{SS:sl3_so3}
Here $\theta=-\sigma$ for an outer involution $\sigma$ of $\mathfrak{sl}_3$.

\begin{Lem}\label{Lem:sl3_so3_obstructive}
The following claims are true:
\begin{enumerate}
\item The variety $X^\theta\setminus \{0\}$ is $(\C^2\setminus \{0\})/\Z_4$ (where a generator of $\Z_4$ acts on $\C^2$ by $\sqrt{-1}$). In particular, we consider the ordinary local systems on
    $X^\theta\setminus \{0\}$. Their category is $\operatorname{Rep}(\Z_4)$.
\item The two $\operatorname{SO}_3$-equivariant (corresponding to representation of $\Z_4$, where
$-1$ acts by $1$) local systems on $X^\theta\setminus \{0\}$
are strongly $\A$-quantizable for any quantization $\A$.
\item If $\A=\mathscr{D}^\lambda(\mathbb{P}^2)$  and $\lambda\not\in \Z+\frac{1}{2}$, then any $\A$-quantizable
local system is $\operatorname{SO}_3$-equivariant.
\item If $\A=\mathscr{D}^\lambda(\mathbb{P}^2)$  and $\lambda\in \Z+\frac{1}{2}$, then only one of the two non-$\operatorname{SO}_3$-equivariant
irreducible local systems is $\A$-quantizable. Which one, depends on the parity of $\lambda$. The $\A$-quantizable irreducible local system is strongly $\A$-quantizable if and only if
$\lambda\in \{1/2,3/2,5/2\}$.
\end{enumerate}
\end{Lem}
\begin{Rem}
	The case when $\lambda=3/2$ in part (4) is well-known, see \cite{RawnsleySternberg} and \cite{Torasso}. It was also pointed out by Vogan in \cite[Example 12.4]{Vogan}.
\end{Rem}

\begin{proof}
{\it Step 1}.
Let $K=\operatorname{SL}_2$ mapping to $\operatorname{SL}_3$ with image $\operatorname{SO}_3$.
Proposition \ref{prop:model_a2c2} shows that the HC $(\A_\lambda,\theta)$-modules are
just usual HC $(\A_\lambda,K)$-modules.

In this case we have two $K$-orbits in $\mathbb{P}^2$ that correspond to the non-degenerate and
isotropic loci with respect to the form. Let  the $K$-invariant quadratic form
on $\C^3$ be given by $q(x_1,x_2,x_3)=x_1^2+x_2^2+x_3^2$.  Then the representatives
are $[0:1:0],[1:\sqrt{-1}:0]$. Let $\tau_n$ denote the unique irreducible representation of $\SL_2$ of dimension $n$.

1) The case of $[v]=[0:1:0]$. Then $K_{[v]}$ is  the normalizer $N_K(T_K)$ of a maximal torus $T_K$ in $K$ and the $K$-orbit $Q_o := K \cdot [v]$ is open in $\Pb^2$. The Lie algebra $\kf_{[v]} = \operatorname{Lie}(T_K)$ has no nonzero
$K_{[v]}$-invariant characters. Since $N_K(T_K) / T_K \simeq \Z_2$, we use Lemma \ref{Lem:homog_classif} to conclude that for all values of $\lambda$, there are two irreducible $K$-equivariant $\mathcal{O}$-coherent
$\mathscr{D}^\lambda_{K/K_{[v]}}$-modules $\phi_\chi$, where $\chi \in  \Z_2$ is regarded as a character of $K_{[v]}$. Then $\phi_\chi$ are actually $\SO(3)$-equivariant. Let $i: Q_o \hookrightarrow \Pb^2$ denote the inclusion map, then $\I(\lambda, Q_o, \phi_\chi): = i_+ \phi_\chi = i_* \phi_\chi$ are $\mathscr{D}^\lambda_{\Pb^2}$-modules. Let $a\in \{0,1\}$ be such that $\chi$ sends $1+2\Z \in \Z/2\Z$ to $(-1)^a$,
we write $\chi_a$ for $\chi$ to indicate the dependence. Set $M_{2a} = \Gamma(\Pb^2, \I(\lambda, Q_o, \phi_\chi))$. An explicit calculation using Frobenius reciprocity shows that there is an isomorphism
  \begin{equation}\label{eq:K_Q_open}
  	 M_{2a} \big|_K \simeq \operatorname{Ind}_{K_{[v]}}^K \chi_a = \bigoplus_{n \equiv 2a+1 \, (\operatorname{mod} 4)} \tau_n,
\end{equation}
of $K$-representations.

2) The case of $[v]=[1:\sqrt{-1}:0]$. Here the stabilizer is the Borel subgroup of
$\operatorname{SL}_2$,  therefore the $K$-orbit $Q_c: = K \cdot [v]$ is closed and isomorphic to $\Pb^1$. The restriction map $\mathfrak{X}(P_{[v]})
\rightarrow \mathfrak{X}(K_{[v]})$ sends $1$ to $2$. So if $\lambda$
is an integer, then we have a unique irreducible HC module and it is $\operatorname{SO}_3$-equivariant.
When $\lambda\in \frac{1}{2}+\Z$, then we still get a unique irreducible but it is only
$\operatorname{SL}_2$-equivariant -- the action of the center is nontrivial.

{\it Step 2}. Let us construct the HC module corresponding to the closed orbit explicitly.
Let $j: Q_c \hookrightarrow \Pb^2$ denote the inclusion map. Let $\mathscr{D}^\lambda_{Q_c} $ denote the sheaf of twisted differential operators over $Q_c$ corresponding to the restriction of $\lambda$ to
the stabilizer of the $K$-orbit $Q_c$, see the discussion before Lemma \ref{Lem:fin_many_orbit_classif}.
Define the $\mathscr{D}^\lambda_{\Pb^2}$ - $\mathscr{D}^\lambda_{Q_c}$-bimodule
\[ \mathscr{D}^\lambda_{\Pb^2 \leftarrow Q_c} := j^{-1} \mathscr{D}^\lambda_{\Pb^2} \otimes_{j^{-1} \OO_{\Pb^2}} \omega_{Q_c/\Pb^2},  \]
where $\omega_{Q_c/\Pb^2} = j^* \omega_{\Pb^2}^{-1} \otimes_{\OO_{Q_c}} \omega_{Q_c} = \OO_{\Pb^1}(4)$ is the relative canonical bundle of $Q_c \simeq \Pb^1$ in $\Pb^2$, which is also the normal bundle $N_{Q_c}$ of $Q_c$. Each $\lambda \in \frac{1}{2} \Z$ determines a  $K$-equivariant line bundle $L_\lambda = \OO_{\Pb^2}(-2 \lambda)$ on $Q_c$, which is a $(\mathscr{D}^\lambda_{Q_c}, K)$-module. We push $L_\lambda$ forward to $\Pb^2$ via $j$ to get a $\mathscr{D}^\lambda_{\Pb^2}$-module
\[  \mathcal{I}(\lambda, Q_c, L_\lambda) = j_+(L_\lambda) := j_* (\mathscr{D}^\lambda_{\Pb^2 \leftarrow Q_c} \otimes_{\mathscr{D}^\lambda_{Q_c}} L_\lambda ).  \]
We set $M(\lambda, Q_c, L_\lambda) = \Gamma(\Pb^2, \I(\lambda, Q_c, L_\lambda))$. There is a natural $K$-stable ascending filtration $\{ \mathcal{F}_k \I \}_{k \geqslant 0}$ of coherent $\OO_{Q_c}$-submodules in $\mathcal{I}(\lambda, Q_c, L_\lambda)$, where $\mathcal{F}_k \I$ is the subsheaf of local sections of degree of transversal derivatives
less or equal than $k$, so that $\mathcal{F}_0 \I = \omega_{Q_c/\Pb^2} \otimes_{\OO_{Q_c}} L_\lambda$ and  the associated graded sheaf is
  \[ \gr \I(\lambda, Q_c, L_\lambda) = \bigoplus_{k \geqslant 0} S^k N_{Q_c} \otimes  \omega_{Q_c/\Pb^2} \otimes  L_\lambda \simeq \bigoplus_{k \geqslant 0} \OO_{\Pb^1}(-2\lambda + 4k + 4). \]
When $\lambda \in \frac{1}{2} \Z$ and $\lambda \leqslant \frac{5}{2}$, all the line bundle summands in the last sum have vanishing higher cohomology. Therefore the higher cohomology of $\mathcal{F}_k \I$ also vanish and $\{ \mathcal{F}_k \I \}_{k \geqslant 0}$ induces an ascending $K$-stable good filtration of $M(\lambda, Q_c, L_\lambda)$, such that there is an isomorphism
  \begin{equation}\label{eq:K_Q_closed}
	\gr M(\lambda, Q_c, L_\lambda) \simeq \bigoplus_{k \geqslant 0} \Gamma(\Pb^1, \OO_{\Pb^1}(-2\lambda + 4k + 4)) = \bigoplus_{k \geqslant 0} \tau_{-2\lambda + 4k + 5}.
\end{equation}
of $K$-representations.

{\it Step 3}. Let us conclude the classification of irreducible HC modules.
When $\lambda\not\in \frac{1}{2}\Z$,
we get two irreducible $\SO_3$-equivariant $\mathscr{D}^\lambda_{\mathbb{P}^2}$-modules, coming from case 1) above. When $\lambda$ is integral, one of the two $\operatorname{SO}_3$-equivariant $\mathscr{D}^\lambda_{\mathbb{P}^2}$-modules $\I(\lambda, Q_o, \phi_\chi)$ from case 1) is irreducible and the other one has a $\OO_{\mathbb{P}^2}$-coherent submodule whose quotient is the $\operatorname{SO}_3$-equivariant $\mathscr{D}^\lambda_{\mathbb{P}^2}$-module $\I(\lambda, Q_c, L_\lambda)$ from case 2). Therefore we end up with two irreducible $\SO(3)$-equivariant $\mathscr{D}^\lambda_{\mathbb{P}^2}$-modules that are not $\mathcal{O}$-coherent.
When $\lambda\in \frac{1}{2}+\Z$ we get two $\SO(3)$-equivariant modules and one which is only $\SL_2$-equivariant.

{\it Step 4}.
Now we examine $\Orb\cap \kf^\perp$.  It consists of all symmetric matrices of rank $1$. These matrices form one $K$-orbit $\Orb_K$. The stabilizer in $\SO(3)$ is the semidirect product of the subgroup of strictly upper triangular matrices (where we now realize $\operatorname{O}(3)$ as the subgroup of matrices stabilizing the quadratic form $x\mapsto x_2^2+x_1x_3$) and the subgroup $\{\operatorname{diag}(\epsilon,1,\epsilon)| \epsilon^2=1\}$. So the orbit
is $(\C^2\setminus \{0\})/\mathbb{Z}_4$. This proves (1) of the lemma.

One conclusion is that if we take $K=\operatorname{SO}_3$, the functor  $\bullet_{\dagger}$ associated
to the open $K$-orbit  is an equivalence. For $K=\operatorname{SL}_2$ it is not: in the interesting case
of $\lambda\in \frac{1}{2}+\Z$, we have exactly one object in the target (identified with $\Z/4\Z\operatorname{-mod}$) which is not in the image.

This proves part (3) and the part of (4) about quantizable twisted local systems.

{\it Step 5.} Let us prove the claims about strongly quantizable local systems ((2) and the remaining part of (4)).
We can choose $e \in \Orb_K$ so that
  \[  K_e = \left\{
     \begin{pmatrix}
  	   t & x\\
  	   0 &t^{-1}\\
      \end{pmatrix}
      \,\Bigg| \, t^4 = 1
  \right\}
  \]
We only consider those representations of $K_e$ that are trivial on the identity component. There are four irreducible ones given by
  \[  \chi_k
    \begin{pmatrix}
  	  t & x\\
  	  0 &t^{-1}\\
    \end{pmatrix}
    = t^k, \quad k = 0,1,2,3. \]
Let $L_k$ denote the corresponding $K$-equivariant line bundles over $\Orb_K$. Then we have isomorphisms of $K \times \cm$-representations
  \begin{equation} \label{eq:K_Lk}
  	 \Gamma(\Orb_K, L_k) = \operatorname{Ind}_{K_e}^K \chi_k = \bigoplus_{n \equiv k+1 \, (\operatorname{mod} 4)} \tau_n.
\end{equation}

We write $\mu$ for $\lambda-3/2$ and $\A_\mu$ for the quantization with period $\mu$, so that
$\A_\mu=\mathscr{D}^{\lambda}(\mathbb{P}^2)$.
The outer automorphism $\varepsilon: g \to (g^t)^{-1}$ of $\SL_3$ fixes $\SO(3)$ and its action on $\slf_3$ coincides with $-\theta$. Moreover, it preserves $\Orb_K$ and maps each line bundle  $L_k$ to itself. Again twisting the $\A_\mu$-module structures by $\varepsilon$ gives equivalences of categories $\HC(\A_\mu,\theta)^K \cong \HC(\A_{-\mu},\theta)^K$ and $\HC_{\overline{\Orb}_K}(\A_\mu,\theta)^K \cong \HC_{\overline{\Orb}_K}(\A_{-\mu},\theta)^K$. Therefore we can focus on the case when $ \mu = \lambda- \frac{3}{2} \leqslant 0$.

When $\lambda \notin \frac{1}{2} \Z$, the global section modules $M_{2a}$ from case 1) are irreducible as $(\A_\mu, \theta)$-modules.  By the last part of the proof of Lemma \ref{Lem:HC_bijection2}, there exists a $K$-stable good filtration on each $M_{2a}$ such that the natural map $\gr M_{2a} \to \Gamma(\gr M_{2a}|_{\Orb_K})$ is injective. By the K-type decompositions \eqref{eq:K_Q_open} and \eqref{eq:K_Lk}, we see that this map is actually an isomoprhism and $\gr M_{2a}|_{\Orb_K} \simeq L_{2a}$, $a = 1,2$.

When $\lambda \in \frac{1}{2} + \Z$ and $\lambda \leqslant \frac{3}{2}$, consider the irreducible $(\A_\mu, \theta)$-module $M(\lambda, Q_c, L_\lambda)$ from case 2) and denote it by $M_k$, where $k \in \{1, 3\}$ and $k \equiv 2 \lambda \, (\operatorname{mod} 4)$. Recall that $M_k$ has a natural $K$-stable filtration. By \eqref{eq:K_Q_closed} and \eqref{eq:K_Lk}, we have $\Gamma(\gr M_k |_{\Orb_K}) \simeq L_{k}$, $k \in \{1, 3\}$. We have $\gr M_k \simeq \Gamma(\Orb_K, L_{k})$ only when $\lambda = \frac{3}{2}$ and $\lambda = \frac{1}{2}$.

This finishes the proof.
\end{proof}

\subsection{Type $c_2$}\label{SS_c2}
Here we consider the algebra $\A=D(\C^2)^{\{\pm 1\}}$,
the only quantization of $\overline{\Orb}$, where
$\Orb$ is the minimal nilpotent orbit in $\mathfrak{sp}_4$.
We claim  that $\A$ is a simple algebra. Here is a more general result.

\begin{Lem}\label{Lem:simplicity}
Suppose that $\tilde{\A}$ is a simple algebra and a  filtered quantization of $\C[\tilde{X}]$, where
$\tilde{X}$ is a conical symplectic singularity. Further, let $\Gamma$ be a finite
group acting on $\tilde{\A}$ by filtered algebra automorphisms such that the induced
action on $\tilde{X}$ is faithful. Then the algebras $\tilde{\A}\#\Gamma$ and $\tilde{\A}^\Gamma$
are Morita equivalent and both are simple.
\end{Lem}
\begin{proof}
The proof is standard. For $\gamma\in \Gamma$,  let $\tilde{\A}_\gamma$ denote the bimodule obtained
from the regular bimodule by twisting with $\gamma$ on the right. This $\tilde{\A}$-bimodule is simple,
and $\tilde{\A}_\gamma\not\cong \tilde{\A}_{\gamma'}$ for $\gamma\neq \gamma'$, the latter can be seen by comparing the associated varieties of these bimodules (that are the graphs of $\gamma,\gamma'$ in $X\times X$). We conclude that the only $\Gamma$- and $\tilde{\A}$- sub-bimodules in $\tilde{\A}\#\Gamma$ are
$\{0\}$ and $\tilde{\A}\#\Gamma$. Therefore, $\tilde{\A}\#\Gamma$ is a simple algebra. The remaining claims follow from here.
\end{proof}
In this section we will apply this lemma to $\tilde{X}=\C^4$, its only quantization $\A$, the Weyl algebra,
and $\Gamma=\{\pm 1\}$.

The algebra $\mathfrak{sp}_4$ has two involutions up to conjugation:
for one we have $\kf=\mathfrak{sl}_2\times \mathfrak{sl}_2$
and for the other $\kf=\mathfrak{gl}_2$. For the former,
$\kf^\perp\cap \Orb=\varnothing$, so we only consider the case
$\kf=\mathfrak{gl}_2$.

\begin{Lem}\label{Lem:c2_unobstructive}
Let $X:=\C^4/\{\pm 1\}$ and $\theta$ be the graded Poisson anti-involution corresponding to the
symmetric subalgebra $\mathfrak{gl}_2$. Then the pair $(X,\theta)$ is unobstructive.
\end{Lem}

We will prove this lemma after some preparation.

\subsubsection{Lifts of HC modules} \label{SSS_lift_HC}
This section contains a general result used to classify irreducible $(\A,\theta)$-modules
in this and the subsequent two cases.

Consider the following situation. Let $\tilde{X}$ be a conical symplectic singularity
such that $\operatorname{codim}_{\tilde{X}} \tilde{X}^{sing}\geqslant 4$.

Let $\tilde{\A}$ be a filtered
quantization of $\tilde{X}$ equipped with an anti-involution $\theta$ and a compatible $\Z/d\Z$-grading, let
$\zeta$ denote the corresponding order $d$ automorphism of $\tilde{\A}$. Let $\Gamma$ be
a group acting on $\tilde{\A}$ by filtration preserving automorphisms so that
\begin{enumerate}
\item  the action of $\Gamma$ commutes with $\theta$ and $\zeta$,
\item
and  the induced action of $\Gamma$ on $\tilde{X}^{reg}$
is free.
\end{enumerate}

Let $\tilde{Y}$ be an irreducible component of $\tilde{X}^{\theta}$.

Set $X:=\tilde{X}/\Gamma, \A:=\tilde{\A}^\Gamma$. We have the induced anti-involution
of $\A$ to be also denoted by $\theta$. Let $Y$ be the image of $\tilde{Y}$ in $X$,
this is an irreducible component of $X^\theta$.

It makes sense to speak about $\Gamma$-equivariant $\Z/d\Z$-graded HC $(\tilde{\A},\theta)$-modules.
Denote their category by $\HC^\Gamma(\tilde{\A},\theta,\zeta)$. We can consider the subcategory
$\HC^\Gamma_{\Gamma\tilde{Y}}(\tilde{\A},\theta,\zeta)$ of all modules supported on $\Gamma\tilde{Y}$
and its quotient $\overline{\HC}^\Gamma_{\Gamma\tilde{Y}}(\tilde{\A},\theta,\zeta)$ of all modules with full support.

Taking the $\Gamma$-invariants, we get a functor $\HC^\Gamma(\tilde{\A},\theta,\zeta)\rightarrow
\HC(\A,\theta,\zeta)$ that restricts to $\HC^\Gamma_{\Gamma\tilde{Y}}(\tilde{\A},\theta,\zeta)\rightarrow
\HC_{Y}(\A,\theta,\zeta)$ and then descends to
\begin{equation}\label{eq:equiv_HC_invariants}
\overline{\HC}^\Gamma_{\Gamma\tilde{Y}}(\tilde{\A},\theta,\zeta)\rightarrow
\overline{\HC}_{Y}(\A,\theta,\zeta).
\end{equation}

The main result of this section is as follows.

\begin{Lem}\label{Lem:lift}
The functor (\ref{eq:equiv_HC_invariants}) is an equivalence.
\end{Lem}
\begin{proof}
Let us write $\pi$ for the functor (\ref{eq:equiv_HC_invariants}). Note that $\pi$ is exact and, since
the action of $\Gamma$ on $X^{reg}$ is free, faithful. We claim that $\tilde{\A}\otimes_{\A}\bullet$
is right inverse of $\pi$, this will finish the proof. What needs to be proved is that for
$M\in \HC_Y(\A,\theta,\zeta)$, there is an object $\tilde{M}\in \HC^\Gamma_{\Gamma\widetilde{Y}}(\tilde{\A},\theta,\zeta)$
whose microlocalization to $X^{reg}$ coincides with that of $\tilde{\A}\otimes_{\A}M$.

Clearly, $\tilde{\A}\otimes_\A M$ is supported on $\Gamma\tilde{Y}$. We will work with the
modified Rees algebras and modules, $M_\hbar,\A_\hbar, \tilde{\A}_\hbar$. Note that for
$a\in \A_\hbar^{-\theta},b\in \tilde{\A}_\hbar, m\in M$ we have
$$a(b\otimes m)=[a,b]\otimes m+ b\otimes am=\hbar [(\hbar^{-1}[a,b])\otimes m+ b\otimes (\hbar^{-1}(am))].$$
It follows that for $b'\in  \A_\hbar^{-\theta}\tilde{\A}_\hbar$ (equivalently -- due to $[\tilde{\A}_\hbar,\tilde{\A}_\hbar]\subset \hbar \tilde{\A}_\hbar$-- for $b'\in  \tilde{\A}_\hbar\A_\hbar^{-\theta}$) and $m'\in \tilde{\A}_\hbar\otimes_{\A_\hbar}M_\hbar$ we have
$b'm\in \hbar (\tilde{\A}_\hbar\otimes_{\A_\hbar}M_\hbar)$.
We write $(\tilde{\A}_\hbar\otimes_{\A_\hbar}M_\hbar)^{reg}$ for  the microlocalization of the $\tilde{\A}_\hbar\otimes_{\A_\hbar}M_\hbar$ to $\tilde{X}^{reg}$.
Set $\tilde{M}_\hbar:=\Gamma((\tilde{\A}_\hbar\otimes_{\A_\hbar}M_\hbar)^{reg})$.

We observe that $\tilde{M}_\hbar$ is a finitely generated $\tilde{\A}_\hbar$-module and the kernel and
cokernel of the natural homomorphism $\tilde{\A}_\hbar\otimes_{\A_\hbar}M_\hbar
\rightarrow \tilde{M}_\hbar$ are supported on $\tilde{X}^{sing}$. This is because $\operatorname{codim}_{\tilde{Y}}(\tilde{Y}\setminus \tilde{X}^{reg})\geqslant 2$, compare
to the proof of Lemma \ref{Lem:HC_bijection2}. So, it remains to check that
$\tilde{M}_\hbar$ is a HC $(\tilde{\A}_\hbar,\theta)$-module. This will follow if we show that
$\tilde{\A}_\hbar^{-\theta}\tilde{M}_\hbar\subset \hbar \tilde{M}_\hbar$.

Note that the
$\tilde{\A}_\hbar$-module $\tilde{M}_\hbar$ is supported on $\Gamma\tilde{Y}$. Its
microlocalization to $\tilde{X}^{reg}$ coincides with $(\tilde{\A}_\hbar\otimes_{\A_\hbar}M_\hbar)^{reg}$.
By the construction, we have $b' \tilde{m}\in \hbar \tilde{M}_\hbar$ for all $b'\in \tilde{\A}_\hbar\A_\hbar^{-\theta}$
 and $\tilde{m}\in \tilde{M}_\hbar$. Fix a $\C^\times$- and  $\theta$-stable affine subset $U\subset X^{reg}$
 and its preimage $\tilde{U}\subset \tilde{X}$. We have $b'\tilde{m}\in \hbar^d\tilde{M}_\hbar(\tilde{U})$ for all $b'\in \tilde{\A}_\hbar(\tilde{U})\A_\hbar(U)^{-\theta}$
and $\tilde{m}\in \tilde{M}_\hbar(\tilde{U})$.

We claim that $b'\tilde{m}\in \hbar\tilde{M}_\hbar(\tilde{U})$
for all $b'\in \tilde{\A}_\hbar(\tilde{U})^{-\theta}$. This will follow once we show that
$\tilde{\A}_\hbar(\tilde{U})\A_\hbar(U)^{-\theta}\supset\tilde{\A}_\hbar(\tilde{U})^{-\theta}$.
The latter equality will follow from its classical analog:
\begin{equation}\label{eq:classic_analog}\C[\tilde{U}]^\theta
\C[U]^{-\theta}=\C[\tilde{U}]^{-\theta}.
\end{equation}
 And that equality will follow if we prove its analog
for the completions at an arbitrary point $u\in Y\cap U$. We write $\hat{A}$ for
$\C[U]^{\wedge_u}$. Then $$\C[U]^{-\theta,\wedge_u}=\hat{A}^{-\theta},
\C[\tilde{U}]^{\theta,\wedge_u}=(\hat{A}^\theta)^{\oplus \Gamma},
\C[\tilde{U}]^{-\theta,\wedge_u}=(\hat{A}^{-\theta})^{\oplus \Gamma}.$$
The completed analog of (\ref{eq:classic_analog}) follows.

So we see that $b'\tilde{m}\in \hbar\tilde{M}_\hbar(U)$ for all $b'\in \tilde{\A}_\hbar(U)^{-\theta}$
and $\tilde{m}\in \tilde{M}_\hbar(U)$. It follows that $\tilde{\A}_\hbar^{-\theta} \tilde{M}_\hbar\subset \hbar \tilde{M}_\hbar$. This finishes the proof.
\end{proof}

\subsubsection{Proof of Lemma \ref{Lem:c2_unobstructive}}
\begin{proof}
{\it Step 1}.
Consider the algebra $\tilde{\A}=D(\C^2)$ so that $\A$ is the invariants in $\tilde{\A}$
under the action of $\{\pm 1\}$. We will view $\tilde{\A}$ as the Weyl algebra
of the symplectic vector space $V:=\C^4$. The anti-involution $\theta$ of $\A$
comes from an anti-symplectic involution of $V$. To give such an involution is the
same thing as to decompose $V$ into the direct sum of two Lagrangian subspaces,
$L_+$ and $L_-$, the involution is $1$ on $L_+$ and $-1$ on $L_-$. In particular,
there are two irreducible components of $X^\theta$, these are $L_+/\{\pm 1\}$
and $L_-/\{\pm 1\}$.

{\it Step 2}. We claim that for each of $L_{\pm}$, there are two
non-isomorphic simple HC $(\A,\theta)$-modules supported at $L_\pm/\{\pm 1\}$ and the category
$\HC_{L_{\pm}/\{\pm 1\}}(\A,\theta)$ is semisimple.

To prove this claim observe first that
the involution $\theta$ admits two lifts to $\tilde{\A}$ that differ by
$-1$. The fixed point set in $V$ for one of them is $L_+$ and for the
other it is $L_-$. Both are $\{\pm 1\}$-stable.

The algebra $\A$ is simple by Lemma \ref{Lem:simplicity}. So $\bar{\HC}(\A)=\HC(\A)$
thanks to Lemma \ref{Lem:lag_antiinvolution}.
We take $Y=L_+$.
Thanks to Lemma \ref{Lem:lift}, $\HC_{L_{+}/\{\pm 1\}}(\A,\theta)\xrightarrow{\sim}
\HC^{\pm 1}_{L_+}(\tilde{\A},\theta)$. Analogously to Lemma \ref{Lem:wHC_formal_Weyl},
every HC $(\tilde{\A},\theta)$-module supported on $L_+$ is the direct sum of several copies of
$\C[L_+]$. So taking the annihilator of $L_+\subset \tilde{\A}$ defines an equivalence
$\HC^{\pm 1}_{L_+}(\tilde{\A},\theta)\xrightarrow{\sim} \operatorname{Rep}(\{\pm 1\})$.
This proves the claim in the beginning of the step.

{\it Step 3}. In Step 2 we have seen that the simple objects in
$\HC(\A,\theta,\zeta)$ are precisely the $\{\pm 1\}$-isotypic components in $\tilde{\A}/\tilde{\A}L_-$.
So they have good filtrations whose associated graded modules are the isotypic components
of $\C[L_{\pm}]$. In particular, both associated graded modules are Cohen-Macaulay.
\end{proof}

\subsection{The case of $X=\C^4/\mathbb{Z}_3$}\label{SS_tau_check}
Set $\Gamma=\mathbb{Z}_3,\tilde{\A}=D(\C^2)$.
The variety $X$ is $\mathbb{Q}$-factorial terminal, and the only filtered quantization of $X$ is $\A=\tilde{\A}^\Gamma$, see, e.g., \cite[Section 3.6]{orbit}.

The goal of this section is to prove the following result.

\begin{Lem}\label{Lem:tau_unobstructive}
Let $X:=\C^4/\mathbb{Z}_3$. For any graded Poisson anti-involution $\theta$ of $\C[X]$,
the pair $(X,\theta)$ is unobstructive.
\end{Lem}
\begin{proof}
{\it Step 1}.
Here we are going to classify the anti-symplectic involutions of $X$ that commute
with the standard contracting $\C^\times$-action, see Lemma \ref{lem:tau_a_2S_2}. The classification will be up to an action of $\operatorname{GL}_2$.
Note that we can lift the involution from $X$ to an anti-symplectic transformation of
$\C^4$ that commutes with $\C^\times$, equivalently, linear. This lift, together with
the $\Gamma$-action gives a linear action on $\C^4$ of a group $\hat{\Gamma}$ of order $6$ that contains
$\Gamma$ as a normal subgroup. So we need to consider two cases:

{\it i}) $\hat{\Gamma}=\Z_3\times \Z_2$. In this case, the lift of $\theta$ is unique, denote it also
by $\theta$. Let $L_\pm\subset \C^4$ denote the eigenspaces for $\theta$ with eigenvalues
$1,-1$. They are $\Gamma$-stable. There are three choices for the pair $(L_+,L_-)$
(up to $\operatorname{GL}_2$-conjugacy). The first two are when $L_+,L_-$ are the
eigenspaces for $\Gamma$. The remaining choice is when $L_+$ intersects both of these
eigenspaces by 1-dimensional spaces. In all of these cases, $X^\theta=\C^2/\Gamma$
(where the actions of $\Gamma$ are by $\zeta\mapsto \operatorname{diag}(\zeta,\zeta)$
in the first case and $\zeta\mapsto \operatorname{diag}(\zeta,\zeta^{-1})$ in the second case).

{\it ii}) $\hat{\Gamma}=S_3$. In this case, a lift of $\theta$ permutes the eigenspaces for $\Gamma$
and this determines a lift uniquely up to $\operatorname{GL}_2$-conjugation. The preimage of
$X^\theta$ in $\C^4$ is the union of three Lagrangian subspaces permuted simply transitively
by $\Gamma$. The regular locus of $X^\theta$ is $\C^2\setminus \{0\}$.

{\it Step 2}. We now show that all (twisted) local systems on $X^\theta\setminus \{0\}$
are strongly quantizable.
We start with case {\it i}). This is very similar to the case of $c_2$ considered above
and is based on Lemma \ref{Lem:lift}.
For each of the three involutions, the category of HC modules is semisimple and has
three simple objects. All of these simples admit a good filtration whose associated graded
is Cohen-Macaulay.

Consider case {\it ii}). Since $X^\theta\cap X^{reg}=\C^2\setminus \{0\}$, there is no twist.
The category of regular local systems is semisimple with one simple object, the structure
sheaf of $X^\theta\cap X^{reg}$. Thanks to the discussion after Definition \ref{defi:loc_sys_quantized}, it is enough to show that the structure sheaf is strongly quantizable.
Fix a lift of $\theta$ and consider the corresponding subspaces $L_\pm$.
Then $\tilde{\A}/\tilde{\A}L_-$ is an irreducible HC $(\A,\theta)$-module
(the irreducibility follows because the associated variety of this module is irreducible,
the multiplicity is $1$, and the algebra $\A$ is simple).
Its usual filtration is good, and the associated graded is Cohen-Macaulay.
\end{proof}

\subsection{The case of $a_2/S_2$}\label{SS_a2S_2}
Here we write $\tilde{X}$ for the closure of the  minimal nilpotent orbit in
$\mathfrak{sl}_3$ and $\Gamma$ for the group (of two elements) generated by the
outer involution of $\mathfrak{sl}_3$ so that $X=\tilde{X}/\Gamma$. Note
that this variety is $\mathbb{Q}$-factorial terminal. Hence it has a unique quantization, $\A=\tilde{\A}^\Gamma$, where
$\tilde{\A}$ is the canonical quantization of $\tilde{X}$, i.e.,
$\tilde{\A}=\Gamma(\mathscr{D}^{3/2}_{\mathbb{P}^2})$.

To state the main result we need to understand the symmetry of the situation. Note that $\operatorname{SO}_3$ acts on $X$ and $\A$. The actions are Hamiltonian. We consider the grading
on $\C[X]$ with $d=1$, it is restricted from the grading on $\C[X]$.  So $\theta$ preserves $\mathfrak{so}_3$ and therefore normalizes $\operatorname{SO}_3$
acting on $\C[X]$. It follows, in particular, that $X^\theta$ is stable under the action of a maximal
torus of $\operatorname{SO}_3$ whose Lie algebra is fixed by $\theta$, denote it by $T_0$.

\begin{Lem}\label{Lem:a2S2_strongly_quantizable}
Let $\chi$ be a character of $T_0$.
Any graded $(T_0,\chi)$-equivariant twisted local system on $X^\theta\setminus \{0\}$
is strongly quantizable.
\end{Lem}

In the process of proof we will fully classify HC $(\A,\theta)$-modules.

\subsubsection{Classification of anti-involutions}\label{SSS_a2S2_involutions}
Here we are going to classify Poisson anti-involutions of $X$ that commute
with the obvious contracting $\C^\times$-action. Again, the anti-Poisson
involution lifts to $\tilde{X}$ giving an action of an order $4$ group
$\tilde{\Gamma}$. Note that the generator of $\Gamma$ is not a square of
any Poisson anti-automorphism of $X$. Indeed, such an anti-automorphism
must come from an anti-isomorphism of $\mathfrak{sl}_3$ hence its square is an inner
automorphism.

It follows that
$\tilde{\Gamma}\cong \Z_2\times \Z_2$, i.e. $\theta$ admits two different
lifts to Poisson anti-involutions of $\tilde{X}$. We need to consider
two cases:
\\

\noindent {\it Case 1}. None of the lifts comes from the anti-involution $x\mapsto -x$
of $\mathfrak{sl}_3$.
Since the generator
of $\Gamma$ is an outer involution we recover both  Poisson anti-involutions
of $\tilde{X}$ considered in Section \ref{SS_type_a2}
as possible lifts of $\theta$. Denote them by $\theta_{in}$
(the inner lift) and $\theta_{out}$ (the outer lift). So $\theta_{in}$
is $-\operatorname{Ad}(g)$ for a non-trivial involution $g\in \operatorname{SO}_3$.
And $\theta_{out}$ is given by $x\mapsto \operatorname{Ad}(g)(x^t)$.

Now we describe the fixed point locus $X^\theta$. It is the union of
$\tilde{X}^{\theta_{in}}/\Gamma$ and $\tilde{X}^{\theta_{out}}/\Gamma$.
Recall that $\tilde{X}^{\theta_{in}}$ consists of two irreducible components.
These components are permuted by $\Gamma$. So $\tilde{X}^{\theta_{in}}/\Gamma$
is an irreducible variety whose smooth locus is $\C^2\setminus \{0\}$.

Now we consider the $\tilde{X}^{\theta_{out}}/\Gamma$. In analysing its
structure it will be convenient for us to make a modification. Namely,
we can assume that $\theta_{out}$ sends $x$ to $x^t$ and $\Gamma$
acts by $x\mapsto -\operatorname{Ad}(g)(x^t)$, where $g$ is an involution in $\operatorname{SO}_3$.
This choice is conjugate to the initial one. So $\tilde{X}^{\theta_{out}}$
consists of symmetric nilpotent rank one matrices. Such matrices
have the form
$$\begin{pmatrix} x_1^2& x_1x_2& x_1x_3\\
x_2x_1& x_2^2& x_2x_3\\
x_3x_1& x_3x_2& x_3^2\end{pmatrix},$$
where the numbers $x_1,x_2,x_3$ satisfy $x_1^2+x_2^2+x_3^2=0$ and are determined uniquely up to simultaneous multiplication by $-1$.  We can assume that $g=\operatorname{diag}(1,-1,-1)$.
Then the action of the generator of $\Gamma$ on $\tilde{X}^{\theta_{out}}$ is induced by $(x_1,x_2,x_3)
\mapsto (\sqrt{-1}x_1,-\sqrt{-1}x_2,-\sqrt{-1}x_3)$. It follows that the smooth part of
$\tilde{X}^{\theta_{out}}/\Gamma$ is $(\C^2\setminus \{0\})/\Z_8$.
\\

\noindent {\it Case 2}. One of the lifts of $\theta$ is $x\mapsto -x$. The only lift
that has fixed points in the smooth locus of $\tilde{X}$ is $x\mapsto x^t$. So $X^\theta$
is the quotient of $\{x\in \tilde{X}| x=x^t\}$ by $\{\pm 1\}$. The smooth
locus is still $(\C^2\setminus \{0\})/\Z_8$.

\subsubsection{Classification of HC modules}
Since the algebra $\tilde{\A}$ is simple, so is $\A=\tilde{\A}^\Gamma$, Lemma \ref{Lem:simplicity}.

We concentrate on case 1.
We need to treat two (very different) cases. First, consider the HC $(\A,\theta)$-modules supported
on $\tilde{X}^{\theta_{in}}/\Gamma$. This case is treated as case {\it ii}) in
Section \ref{SS_tau_check}. We get a single irreducible HC module whose associated graded
is $\C[\C^2]$.

Now we proceed to HC modules supported on $\tilde{X}^{\theta_{out}}/\Gamma$.
By Lemma \ref{Lem:lift}, taking the $\Gamma$-invariants defines an equivalence
$$\HC^\Gamma(\tilde{\A},\theta_{out}) \xrightarrow{\sim}  \HC_{\tilde{X}^{\theta_{out}}/\Gamma}(\A,\theta).$$

Now we need to understand the structure of
$\HC^\Gamma(\tilde{\A},\theta_{out})$. First, we claim that $\Gamma$ acts trivially
on the set of irreducibles in $\HC(\tilde{\A},\theta_{out})$. Indeed, an irreducible
object is completely determined by the irreducible local system on $(\tilde{X}^{\theta_{out}})^{reg}
=(\C^2\setminus \{0\})/\Z_4$. The action of $\Gamma$ on the irreducible HC modules coincides
with the action on the set of isomorphism classes of irreducible local systems, which is trivial because the nontrivial element in
$\Gamma$ acts by $-1$ on $(\C^2\setminus \{0\})/\Z_4$.

We conclude that there are six irreducibles in $\HC_{\tilde{X}^{\theta_{out}}/\Gamma}(\A,\theta)$.
Each is a $\Gamma$-isotypic component in one of the three irreducible objects in $\HC(\tilde{\A},\theta_{out})$.
In particular, each admits a good filtration whose associated graded is Cohen-Macaulay.

In case 2, the only component of $X^\theta$ is handled completely analogously to $\tilde{X}^{\theta_{out}}/\Gamma$ in case 1.

\begin{proof}[Proof of Lemma \ref{Lem:a2S2_strongly_quantizable}]
By the analysis in this section, all local systems on $[\tilde{X}^{\theta_{in}}/\Gamma]\setminus \{0\}=\C^2 \setminus \{0\}$ are strongly quantizable. So we need to concentrate on the local systems
on $[\tilde{X}^{\theta_{out}}/\Gamma]\setminus \{0\}=(\C^2\setminus \{0\})/\Z_8$. The action of $T_0$
on $\C^2/\Z_8$ is induced from the action of the maximal torus of $\operatorname{SO}_3$ on $\C^2/\{\pm 1\}$. So, under the equivalence of the category of local systems with $\operatorname{Rep}(\Z_8)$,
the $T_0$-equivariant ones correspond to the representations of $\Z_8$, where $\Z_2\subset \Z_8$
acts trivially. The analysis above in this section shows that all such local systems are
strongly quantizable.
\end{proof}


\section{Classification theorems}\label{S_classif_thm}

\subsection{Nonlinear covers of real groups} \label{subsec:cover}

Throughout this subsection, we assume that $G$ is a simply connected semisimple algebraic group defined over $\R$ and $\GR = G(\R)$ is the group of real points of $G$. Then $\GR$ is automatically connected. Let $\sigma$ be the Cartan involution of $G$ corresponding to $\GR$, then $\KR = \GR^\sigma$ is a maximal compact subgroup of $\GR$ and $K = G^\sigma$ is a complexification of $\KR$. 
We have natural identification of fundamental groups $\pi_1(\GR) \simeq \pi_1(\KR) \simeq \pi_1(K)$.  

Let $H$ be a $\sigma$-stable Cartan subgroup of $G$ with Lie algebra $\hf$ a $\sigma$-stable Cartan subalgebra of $\g$. Then $\sigma$ acts on the set $\Delta(\g,\hf)$ of roots of $\hf$ in $\g$. A root $\alpha$ is said to be {\it imaginary} (resp. {\it real}) if $\sigma(\alpha) = \alpha$ (resp. $\sigma(\alpha) = -\alpha$). If $\alpha$ is neither imaginary or real, we say that $\alpha$ is \emph{complex}. If $\alpha$ is imaginary, then the involution $\sigma$ preserves the (one-dimensional) root subspace $\g_\alpha$ of $\alpha$ and hence acts by $\pm 1$. In this case, $\alpha$ is said to be compact (resp. non-compact) imaginary if $\sigma$ acts on $\g_\alpha$ by $1$ (resp. $-1$). By convention, all roots are long if the root system of $G$ is simply laced.

Now assume $H$ (or $\hf$) to be fundamental (i.e., maximally compact, which means that $\dim \hf^\sigma$ is as large as possible), so that $T := H^\sigma$ is a Cartan subgroup of $K$. It is known that $H$ is maximally compact if and only if there are no real roots in $\Delta(\g,\hf)$ (\cite[Proposition 6.70]{Knapp}). Then one can choose the positive root system $\Delta^+$ of $\Delta(\g,\hf)$ to be $\sigma$-stable and hence $\sigma$ permutes the simple roots (\cite[Chapter VI, Section 8]{Knapp}). Moreover, by a result of Borel and de Siebenthal (\cite[Theorem 6.96]{Knapp}), when $\g$ is simple (or equivalently, $\g_\R$ is a non-complex simple real Lie algebra), one can always choose the simple root system $\Pi$ for which at most one simple root is non-compact imaginary.

For any imaginary $\alpha \in \Delta(\g,\hf)$, we identify the corresponding coroot $\alpha^\vee$ with a one-parameter subgroup $\alpha^\vee: \cm \to T \subset K$. Let $\gamma: [0,1] \to \cm, t \mapsto \exp(2\pi \sqrt{-1} t)$, be the continuous loop in $\cm$, which gives rise to a particular generator $[\gamma]$ of $\pi_1(\cm) \simeq \Z$. Then we may also identify $\alpha^\vee$ with the image of $[\gamma]$ in $\pi_1(K) \simeq \pi_1(G_\R)$ under the map $\pi_1(\cm) \to \pi_1(K) \simeq \pi_1(G_\R)$ induced by  $\alpha^\vee: \cm \to K$.

We summarize below the main results of \cite{Adams_nonlinear}.

\begin{Thm}\label{thm:pi_GR}
	Let $G$ be a simply connected semisimple algebraic group defined over $\R$ and $\GR = G(\R)$ be the group of real points of $G$. Then the following are true:
	\begin{enumerate}
		\item 
		   	For any compact imaginary root $\alpha$, the element in $\pi_1(K) \simeq \pi_1(G_\R)$ corresponding to $\alpha^\vee$ is always trivial.
		\item 
			$\pi_1(\GR) \neq 1$ if and only if $H$ has a long non-compact imaginary root.
		\item 
			If $\g_\R$ is simple, then $\pi_1(\GR)$ is either trivial or is isomorphic to $\Z$ or $\Z_2$. 
		\item 
		  	If $\g_\R$ is simple and $\pi_1(\GR) \neq 1$, then $\pi_1(\GR)$ is generated by a coroot $\alpha^\vee$ corresponding to an arbitrary long non-compact imaginary root $\alpha$ in $\Delta(\g, \hf)$. 
		\item
			If $\g_\R$ is simple, then $\pi_1(\GR) \simeq \Z$ if and only if $\GR$ is Hermitian symmetric. In this case, $\tf = \hf$ and one can choose the simple root system for $\Delta(\g,\hf)$ so that exact one simple root is noncompact imaginary.
	\end{enumerate}
\end{Thm}

\begin{proof}
	(1) is standard, see, for example, equation (2.2) of \cite{Adams_nonlinear}. (2) is \cite[Theorem 1.1]{Adams_nonlinear}. (3) is \cite[Theorem 1.4]{Adams_nonlinear}. (4) follows from \cite[Theorem 1.5]{Adams_nonlinear}. (5) follows from the Case 2 in the proof of Theorem 1.4 on page 4043 of \cite{Adams_nonlinear} and \cite[Chapter VII, Section 9]{Knapp}. 
\end{proof}

It follows that for simple $\GR$ with $\pi_1(\GR) \neq 1$, there exists, up to isomorphism, a unique connected nonlinear $2$-fold covering group $\tilde{G}_\R$ of $\GR$, which is analogous to the metaplectic groups. We will see that the representations constructed in Section \ref{SS:double_SLR}, Section \ref{SS_exceptional} and Appendix \ref{sec:appendix_A} are all defined over $\tilde{G}_\R$ and often do not factor through the linear groups $\GR$.

\subsection{Main results}\label{SS_main_results}
In this section $G$ is a \emph{simply connected} simple algebraic group. Let $\sigma$ be an involution of
$G$ and $K$ a connected algebraic group equipped with a homomorphism $K\rightarrow G$ that induces
an isomorphism $\kf\xrightarrow{\sim} \g^\sigma$. We fix a nilpotent orbit $\Orb\subset \g^*$
with codimension of the boundary at least $4$. Further, fix a $K$-orbit $\Orb_K\subset \Orb\cap
\kf^\perp$. Fix a character $\kappa\in (\kf^*)^{K}$ and $\lambda\in H^2(\Orb,\C)$.

Set $X:=\operatorname{Spec}(\C[\Orb])$. The anti-involution $\theta=-\sigma$ of $\g$ induces a
Poisson anti-involution on $X$ also to be denoted by $\theta$.
Let $\Orb'_K$ be a $K$-orbit of codimension $2$ in the
closure of $\Orb_K$ in $X$. For $\chi'\in \Orb'_K$, we have an induced Poisson anti-involution
on a transverse slice $\check{X}$ to $\chi'$ in $X$.

%

Pick $\chi\in \Orb_K$ and let $K_Q,\kappa_Q$ have the same meaning as in Section \ref{SS:HC_mod_restr}. Recall the restriction functor $\bullet_{\dagger,\chi}: \HC_{\Orb_K}(\A_\lambda,\theta)^{K,\kappa}
\rightarrow \underline{\A}_\lambda\operatorname{-mod}^{K_Q,\kappa_Q}=\operatorname{Rep}(K_Q,\lambda|_{\kf_Q}-\kappa_Q)$. Here the target category
is the category of finite dimensional rational representations of $K_Q$, where $\kf_Q$ acts by
$\lambda|_{\kf_Q}-\kappa_Q$.
The following is the first of our main results.

\begin{Thm}\label{Thm:main1}
Suppose one of the following conditions holds:
\begin{itemize}
\item[(i)] $K\hookrightarrow G$.
\item[(ii)] There are no slices of type $\chi$.
For all slices of type $a_2$, the anti-involution of this slice comes from the inner involution of $\mathfrak{sl}_3$.
\end{itemize}
Then the restriction functor $\HC_{\Orb_K}(\A_\lambda,\theta)^{K,\kappa}\rightarrow
\operatorname{Rep}(K_Q,\lambda|_{\kf_Q}-\kappa_Q)$ is an equivalence.
\end{Thm}

\begin{Cor}\label{Cor:c_2_tau_slices}
Suppose $\g$ is of types $B,C,D$, or $\Orb$ is one of the following orbits in the exceptional types:
\begin{itemize}
\item $(E_6,2A_2+A_1), (E_8, D_4(a_1)+A_1), (E_8, A_3+A_2+A_1), (E_8,A_3+2A_1)$.
\end{itemize}
Then $\HC_{\Orb_K}(\A_\lambda,\theta)^{K,\kappa}\xrightarrow{\sim}
\operatorname{Rep}(K_Q,\lambda|_{\kf_Q}-\kappa_Q)$.
\end{Cor}
\begin{proof}
In the cases listed all slices $\check{X}$ are $\C^4/\{\pm 1\}$ (i.e., the minimal orbit of $\mathfrak{sp}_4$),  $\C^4/\mathbb{Z}_3$ or $\C^4$.
For the classical Lie algebras this follows from \cite{KP2}, while for the exceptional
groups this follows from \cite{FJLS1}. We are done by (ii) of Theorem \ref{Thm:main1}.
\end{proof}

The following corollary will be proved in subsequent sections.

\begin{Cor}\label{Cor:type_A_inner}
Suppose that $\g\cong \mathfrak{sl}_n$ and $(\g, \kf)$ does not correspond to the split real form $\mathfrak{sl}(n ,\R)$ of $\mathfrak{sl}_n$ (i.e., $\kf\not\cong \mathfrak{so}_n$). Then
$$\HC_{\Orb_K}(\A_\lambda,\theta)^{K,\kappa}\xrightarrow{\sim}
\operatorname{Rep}(K_Q,\lambda|_{\kf_Q}-\kappa_Q).$$
\end{Cor}

Note that Theorem \ref{Thm:main1} and Corollaries \ref{Cor:c_2_tau_slices} and
\ref{Cor:type_A_inner} imply Theorem \ref{Thm:omnibus} from Introduction.

The situation when neither of the conditions of Theorem \ref{Thm:main1} hold
is more complicated. In the classical types this only occurs when
$(\g,K)=(\mathfrak{sl}_n, \operatorname{Spin}_n)$ (with $n\geqslant 3$).
We will state a classification result in this case in a later section.
We will also analyze some cases in the exceptional types.

We will adopt the following terminology in this section.

\begin{defi}\label{defn:linearly-admissible}
	Recall $G$ is a \emph{simply connected} simple algebraic group. Let $K= G^\sigma$, then $K$ is automatically connected. If $\Orb_K$ is admissible for $K$ (in the sense of Section \ref{SS_intro_approach}), we say that it is \emph{linearly admissible}. If $\Orb_K$ is admissible for some finite covering group $\tilde{K}$ of $K$ but not $K$ itself, then we say that $\Orb_K$ is \emph{nonlinearly admissible}. If it is neither the case, then we say that $\Orb_K$ is \emph{not admissible}.
\end{defi}

\begin{Rem}\label{rem:sc-adm}
		We compare Definition \ref{defn:linearly-admissible} with the terminology in \cite{Noel1, Noel2}. No{\"e}l called $\Orb_K$ \emph{admissible} if $\Orb_K$ is admissible for an adjoint real group $\GR^{ad}$ (in terms of the complex groups this means that $G$ is of adjoint type and $K=(G^\sigma)^\circ$). He called $\Orb_K$ \emph{sc-admissible} if $\Orb_K$ is admissible for the simply connected real group but not for the adjoint group $\GR^{ad}$. It is clear that admissibility in the sense of No{\"e}l implies linearly admissibility but not vice versa, and nonlinearly admissibility implies sc-admissibility but not vice versa. See also Remark \ref{rem:Noel1}.
\end{Rem}

\subsection{Proof of Theorem \ref{Thm:main1}}
In this section we prove Theorem \ref{Thm:main1}.

We use the notation of Section \ref{SS_image_description}, in particular,
$\Orb_K^i, i=1,\ldots,\ell$, are all codimension $2$ orbits in $\overline{\Orb}_K$.

Pick an irreducible representation $V\in \operatorname{Rep}(K_Q,\lambda|_{\kf_Q}-\kappa_Q)$ and let
$\mathcal{E}:=\mathcal{E}_V$ be the corresponding $(K,\kappa)$-equivariant $\lambda$-twisted local system
on $\Orb_K$. Then it makes sense to speak about the restrictions $\check{\mathcal{E}}^i$
of $\mathcal{E}$ to $\check{X}^{i\times}$.

%

\subsubsection{Three lemmas}
In the proof of the theorem we will need three lemmas.

Assume for a moment that $\check{X}^i$ is of type $a_2$. Then the connected component of the  group of graded Poisson automorphisms of $\check{X}^i$ is $\operatorname{PGL}_3$ (indeed, the Lie algebra of this group is $\mathfrak{sl}_3\subset \C[\check{X}^i]$).
And if $\check{X}^i$ is of type $a_2/S_2$, then $\operatorname{SO}_3$ acts. Let
$Q^i$ denote the reductive part of $Z_G(\chi^i)$, this reductive group acts on
$\check{X}^i$ (from the realization of $\check{X}^i$ as a Slodowy slice). As before, we can assume that $Q^i$ is $\sigma$-stable and $\check{X}^i$ is $\theta$-stable.
Let $K^i_Q$ denote the preimage of $Q^i\subset G$ in $K$.

\begin{Lem}\label{Lem:Qi_action}
In the $a_2$ case, the action of $(Q^i)^\circ$ on $\check{X}^i$ factors through $(Q^i)^\circ\twoheadrightarrow
\operatorname{PGL}_3$.
\end{Lem}
\begin{proof}
The algebra $\g$ is either type $A$ or exceptional. We divide the proof into two cases.

{\it Case 1}. As in the proof of Case 2.1 in Proposition \ref{prop:model_a2c2},
let $\g=\mathfrak{sl}_n$ and take $G=\operatorname{GL}_n$. We take $\Orb,\tau, \tau^i$ as in that proof.
It is clear from the construction there that $Q^i$ acts on the slice $\overline{\Orb}_{min}$
via the projection $Q^i\twoheadrightarrow \operatorname{PGL}_3$.

{\it Case 2}. Now suppose $\g$ is exceptional and $\check{X}^i$ is of type $a_2$. One can extract
the information about the pairs $(\Orb,\Orb^i)$ from \cite{FJLS1} and the information about $\mathfrak{q}$
and $\mathfrak{q}^i$, for example, from the tables in \cite[Section 13.1]{Carter}. 

Suppose, for the sake of contradiction, that the $\slf_3$-factor in $\mathfrak{q}^i$
(or both of such factors in the case when $\Orb$ is $A_2+2A_1$ in $E_6$) act trivially
on $\check{X}^i$. Then we have $\mathfrak{q}^i\hookrightarrow \mathfrak{q}$.
This is clearly impossible in all cases but when $\Orb$ is $A_4+A_1$ in type $E_8$.
In that case, suppose $\mathfrak{q}^i\subset \mathfrak{q}$. The centralizer of a Cartan subalgebra
in $\mathfrak{q}^i$ is a minimal Levi containing $e^i$, i.e., of type $D_4+A_2$. Then we can choose $e$
in the centralizer of a semisimple element in that Levi. This is impossible, because the minimal
Levi containing $e$ is of type $A_4+A_1$ and there's no inclusion of $\mathfrak{sl}_5\times \mathfrak{sl}_2$
into $\mathfrak{so}_8\times \mathfrak{sl}_3$. This contradiction finishes the proof (this conclusion can be also
extracted from the treatment of these orbits in \cite{FJLS1}).
\end{proof}

The second lemma concerns the case of $a_2/S_2$. According to \cite{FJLS1}, the singularity of type $a_2/S_2$ occurs only for the orbit $\Orb$ of type $A_4+A_1$ in $E_7$  along the orbit $\Orb^i$ of type $A_3 + A_2 + A_1$ in its boundary. Suppose $G$ is a  simple algebraic group of type $E_7$. Let $\sigma$ be an involution of $G$ and $K$ a connected algebraic group equipped with a homomorphism $K\rightarrow G$ that induces
an isomorphism $\kf\xrightarrow{\sim} \g^\sigma$.  Let $i$ be such that $\check{X}^i$ has this type.

\begin{Lem}\label{Lem:character_integrality}
The following claims hold:
\begin{enumerate}
\item
Suppose $\dim \kf^i_Q=1$.  Then the element $\rho_{L^i}$ for the orbit $\Orb^i_K$ (see Definition \ref{defi:half_character})
restricts to character of the 1-dimensional torus $T_0\subset \operatorname{SO}_3$ with Lie algebra $\kf^i_Q$.
\item Suppose $\dim \kf^i_Q\neq 1$. Then the action of $(K^i_Q)^\circ$ on $\check{X}^i$ factors through $\operatorname{SO}_3$.
\end{enumerate}
\end{Lem}

\begin{proof}
	According to the classification of nilpotent orbits for real forms of $E_7$ and their closure relations (see \cite{Dk3, Dk4, Dk5}), if $\Orb$ and $\Orb^i$ intersect $\kf^\perp$, the only possible real form of $E_7$ is $E_{7(7)}$. In this case the closure $\overline{\Orb}_K$ intersect $A_3 + A_2 + A_1$ at four $K$-orbits, labelled by 39, 40, 41 and 42.
The Lie algebra $\mathfrak{q}^i$ is $\mathfrak{sl}_2$ (see, e.g., tables in \cite[Section 13.1]{Carter}), while the Lie algebras $\kf^i_Q$ for $39$ and $40$ are $\slf_2$ and 1-dimensional for $41$ and $42$ by \cite[Table XI, pg. 513]{Dk0}. Since the $G$-orbit $A_3 + A_2 +A_1$ is even, by \cite[Chapter 2]{Schwartz} (see also \cite[Corollary 7.28]{Vogan}), all the four $K$-orbits in $A_3 + A_2 +A_1$ are admissible in the sense of Definition \ref{defn:linearly-admissible}.

Note that $\mathfrak{q}$ is a two-dimensional abelian Lie algebra, \cite[Section 13.1]{Carter}. It follows that the action of $(Q^i)^\circ$
on $\check{X}_i$ is nontrivial. Since $\Orb^i$ is an even orbit, this action factors through $\operatorname{SO}_3$.
In particular,  the conclusion of (1) holds for the $K$-orbits $\#41$ and $42$, while the conclusion of (2) holds for
the $K$-orbits $\#39$ and $40$.
\end{proof}

The third lemma concerns type $\chi$ (the case (5) from Section \ref{SS_goals}).

\begin{Lem}\label{Lem:type_chi}
Suppose $\Orb$ is the unique orbit $A_4+A_3$ in $E_8$ whose closure has type $\chi$ singularity.
The conclusion of (i) of Theorem \ref{Thm:main1} is true in this case.
\end{Lem}

\begin{proof}
	The orbit $\Orb$ of type $A_4+A_3$ in $E_8$ is rigid by \cite[Table 3]{GraafElashvili}  and $\dim_\C \Orb=200$ by \cite[pg. 133]{CM}.
According to \cite[Appendix D]{LMM}, the ideal $\J$, the kernel of the quantum comoment map for the unique quantization of $\C[\Orb]$
is the maximal ideal with infinitesimal character $\rho/5$, where $\rho$ is the half sum of positive roots of $E_8$. Inside $\Orb$, there is only one $K$-orbit $\Orb_K$ for the split real form $E_{8(8)}$, which is labeled as $\# 57$ in \cite[Table XIV]{Dk0}, and no $K$-orbit for the other real form $E_{8(-24)}$ by \cite[Table XV]{Dk0}.
Since the adjoint group $G$ of type $E_8$ is simply connected, $K=G^\sigma$ is semisimple and connected by \cite[Table 3]{realgrp}, and corresponds to the adjoint $E_{8(8)}$. Hence $\kappa=0$. By \cite[Table VII, pg. 484]{Noel1}, orbit $\# 57$ is (linearly) admissible. By \cite[Table 11, pg. 273]{King}, the $K$-component group of orbit $\# 57$ is trivial, so there is exactly one $K$-admissible line bundle over $\Orb_K$ and hence a unique irreducible object in $\operatorname{Rep}(K_Q,\lambda|_{\kf_Q}-\kappa_Q)$.  The \verb|atlas| command  \verb|all_parameters_gamma| lists the Langlands parameters of all irreducible HC $(\g,K)$-modules of a given infinitesimal character, and the command \verb|GK_dim| computes the Gelfand-Kirillov dimension of the representation corresponding to a given parameter. We find that there is only one irreducible HC $(\g,K)$-module infinitesimal character $\lambda = \frac{1}{5} \rho$ of Gelfand-Kirillov dimension $100$. Therefore the unique irreducible object in $\operatorname{Rep}(K_Q,\lambda|_{\kf_Q}-\kappa_Q)$ is in the image of the restriction functor $\bullet_{\dagger,\chi}$, so the claim holds.
\end{proof}

\begin{Rem}
	In a forthcoming paper \cite{quant_cover}, the second named author will provide a computer-free proof of Lemma \ref{Lem:type_chi}.
\end{Rem}

\subsubsection{Proof of Theorem \ref{Thm:main1}} \label{SS:proof_main1}
Here we deduce the theorem from the lemmas.

\begin{proof}
If (ii) holds, then all pairs $(\check{X}^i,\theta^i)$ except, possibly, $a_2/S_2$, are unobstructive, by Lemma
\ref{Lem:sl3_gl2_unobstructive}, Lemma \ref{Lem:c2_unobstructive} and Lemma \ref{Lem:tau_unobstructive}.
The claim of the theorem follows from Proposition \ref{Lem:image_description11}.

So we need to prove the theorem under assumption (i) or if the singularity $a_2/S_2$ is present. We can assume that the singularity $\chi$
does not occur, thanks to Lemma \ref{Lem:type_chi}. The only two singularities we need to consider
are $a_2$ and $a_2/S_2$.

Consider the case when $\check{X}^i$ is of type $a_2$, and the involution in question is outer (and $K\subset G$).
We note that $\mathcal{E}$ is weakly $K$-equivariant, hence $\check{\mathcal{E}}^i$ is weakly $K_{Q}^i=Q^i\cap K$-equivariant. Recall
that $(Q^i)^\circ$ acts on $\check{X}^i$ via $(Q^i)^\circ\twoheadrightarrow
\operatorname{PGL}_3$, Lemma \ref{Lem:Qi_action}. The corresponding epimorphism $\mathfrak{q}^i\twoheadrightarrow \mathfrak{sl}_3$
is equivariant with respect to the action of the reductive group $\langle\sigma\rangle\ltimes (Q^i)^\circ$. So we can find a $\langle\sigma\rangle\ltimes (Q^i)^\circ$-equivariant
section $\mathfrak{sl}_3\rightarrow \mathfrak{q}^i$ (where $\sigma$ acts as an outer involution on the source) that lifts uniquely to an algebraic group homomorphism $\operatorname{SL}_3\rightarrow (Q^i)^\circ$ whose composition with $(Q^i)^\circ\twoheadrightarrow \operatorname{PGL}_3$ is the natural epimorphism
$\operatorname{SL}_3\twoheadrightarrow \operatorname{PGL}_3$.   Passing to the invariants of the involutions, we get homomorphisms $\operatorname{SO}_3\rightarrow K_{Q}^i\twoheadrightarrow \operatorname{SO}_3$ whose composition is the identity. The action of $(K_{Q}^i)^\circ$ on $\check{X}^i$
factors through the epimorphism $K_{Q}^i\twoheadrightarrow \operatorname{SO}_3$. It follows that $\check{\mathcal{E}}$ is
weakly $\operatorname{SO}_3$-invariant. Thanks to Lemma \ref{Lem:sl3_so3_obstructive},
$\check{\mathcal{E}}$ is strongly $\check{\A}^i_\lambda$-quantizable. Applying Proposition
\ref{Lem:image_description11} finishes the proof in the $a_2$-case.

Now we handle the case of $a_2/S_2$. Recall, Section \ref{SS_a2S_2}, that there may be two components of $(\check{X}^i)^\theta$, denote them by $Y_{in}$ and $Y_{out}$. If the former is present, then its nonzero locus is $(\C^2\setminus \{0\})$ and there is a HC $(\check{A}^i_\lambda,\theta^i)$-module $M_{in}$ with a good filtration such that $\gr M_{in}=\Gamma(Y_{in}^\times, \check{\mathcal{E}}^i)$. So we just need to show that
there is a HC module $M_{out}$ with $\gr M_{out}=\Gamma(Y_{out}^\times, \check{\mathcal{E}}^i)$.

In Case 2 of Section \ref{SSS_a2S2_involutions}, the anti-involution $\theta$
of $\check{X}^i$ is invariant under $\operatorname{SO}_3$. This case is handled as the $a_2$-case above. Suppose now that we are in Case 1 of
Section \ref{SSS_a2S2_involutions}. Then $\theta$ is equivariant with respect to the maximal torus  of $\operatorname{SO}_3$ obtained as the image of $(K^i_Q)^\circ$.

As was mentioned before Lemma \ref{Lem:character_integrality}, the singularity of type $a_2/S_2$ occurs only for the orbit of type $A_4+A_1$ in $E_7$. This orbit is birationally rigid in the sense of \cite[Definition 1.2]{orbit}, see, e.g., \cite[Proposition 3.1]{Fu}.
So $\lambda=0$.
As noted in the proof of Lemma  \ref{Lem:character_integrality}, we only need to consider the real form $E_{7(7)}$. In this case, the group $K$ is semisimple (see \cite[Table 3]{realgrp}).
So $\kappa=0$. It follows that $\check{\mathcal{E}}^i$
is $(K^i_Q,-\rho^i_{L})$-equivariant. By Lemma \ref{Lem:character_integrality},
$\rho^i_L$ comes from a character of $(K^i_Q)^\circ$. By Lemma \ref{Lem:a2S2_strongly_quantizable},
$\check{\mathcal{E}}^i|_{Y_{out}^\times}$ is strongly quantizable.

Now the claim that $\bullet_{\dagger,\chi}$ is essentially surjective under the assumptions of the theorem
follows from Proposition \ref{Lem:image_description11}.
\end{proof}

\subsection{Restriction of twisted local systems to slices}\label{SS_loc_sys_slice_restriction}
As we have seen in Section \ref{SS:sl3_so3}, when the conditions of Theorem \ref{Thm:main1} are not satisfied, the functor $\bullet_{\dagger,\chi}$ may fail to be essentially surjective. In this case,
to apply Proposition \ref{Lem:image_description11}, we need to be able to compute the restriction
of the $(K,\kappa)$-equivariant twisted local system $\mathcal{E}$ to $(\check{X}^{i\times})^\theta$.
This is what the present section is concerned with.

Recall, Proposition \ref{Prop:regular_local_systems},
that the category of regular twisted local systems on $\Orb_K$
is equivalent to the category of projective representations of $\pi_1(\Orb_K)$ (with Schur multiplier
determined by the twist). All $(K,\kappa)$-equivariant local systems on $\Orb_K$ are regular, and the category of projective representations is the category $\operatorname{Rep}(K_Q,\lambda|_{\kf_Q}-\kappa_Q)$.
So the Schur multiplier for the projective representation of $K_Q/K_Q^\circ$ (that is a quotient
of $\pi_1(\Orb_K)$) is read off from the character $\lambda|_{\kf_Q}-\kappa_Q$.

Now pick an irreducible component $\check{Y}$ of $(\check{X}^{i})^\theta$ and write
$\check{Y}^\times$ for $Y\setminus \{0\}$. The inclusion $\check{Y}^\times\hookrightarrow \Orb_K$
gives rise to a group homomorphism $\pi_1(Y^\times)\rightarrow \pi_1(\Orb_K)$ and hence
$\pi_1(Y^\times)\rightarrow K_Q/K_Q^\circ$.

The following lemma is a direct consequence of Proposition \ref{Prop:regular_local_systems}.

\begin{Lem}\label{Lem:local_sys_restr1}
Let $\mathcal{E}$ correspond to a projective representation $V$ of $K_Q/K_Q^\circ$.
Then the restriction of $\mathcal{E}$ to $\check{Y}^\times$ corresponds to the pullback of
$V$ to $\pi_1(\check{Y}^\times)$.
\end{Lem}

We will apply this construction in the case when $\check{X}^i$ is of type $a_2$
and its anti-involution $\theta$ corresponds to the outer involution of $\mathfrak{sl}_3$.
By (1) of Lemma \ref{Lem:sl3_so3_obstructive}, we have $\pi_1(Y^\times)\cong \Z_4$.
So what we need to determine is a homomorphism $\Z_4\rightarrow K_Q/K_Q^\circ$.

Here is how we are going to compute this homomorphism.
Choose a base point $\chi$ used
to compute $\pi_1(\Orb_K)$ on the image of $Y^\times$ in $\Orb_K$.
This choice gives rise to an embedding $(K_{Q}^i)_\chi\rightarrow K_\chi$.
Recall, Lemma \ref{Lem:Qi_action}, that $(Q^i)^\circ$ acts on $\check{X}^i$ via an epimorphism to $\operatorname{PGL}_3$.
The action of $K_Q^i$ on $\check{X}^i$ factors through that of $Q^i$.
Hence $K_{Q}^{i,\circ}$ acts on $\check{X}^i$ via an epimorphism onto $\operatorname{SO}_3$. There is the unique homomorphism
$\operatorname{SL}_2 \rightarrow K_{Q}^{i,\circ} $ such that the composition
$\operatorname{SL}_2 \to K_{Q}^{i,\circ} \twoheadrightarrow \operatorname{SO}_3$
is the projection $\operatorname{SL}_2\twoheadrightarrow \operatorname{SO}_3$.
The homomorphism $\Z_4\rightarrow K_Q/K_Q^\circ$ we need is the composition
\begin{equation} \label{eq:hom_Z4}
	\Z_4\hookrightarrow (\SL_2)_\chi \rightarrow (K_{Q}^{i,\circ})_\chi\hookrightarrow K_Q\twoheadrightarrow K_Q/K_Q^\circ.
\end{equation}
We note that the first map is an identification of $\Z_4$ with a maximal reductive subgroup of $(\SL_2)_\chi$
and so is not canonically defined, but its composition with the projection of
$(\SL_2)_\chi$ to its component group is defined canonically.


\subsection{Type A} \label{SS:type_A}
Throughout this section, $G=\operatorname{SL}_n, \g=\mathfrak{sl}_n$ (unless otherwise
noted), $\Orb$
is a nilpotent orbit in $\g$ with codimension of the boundary at least $4$, $\tau$
is the corresponding partition. The goal of this section is to complete the classification of irreducible
$(K,\kappa)$-equivariant HC $(\A_\lambda,\theta)$-modules with GK dimension equal to $\frac{1}{2}\dim \Orb$.

Let $I$ be the set of all indices $i$
such that $\tau_{i}=\tau_{i+1}+1=\tau_{i+2}+2$. As mentioned in Case 2.2
of the proof of Proposition \ref{prop:model_a2c2},
the set $I$ parameterizes codimension 4 orbits in $\overline{\Orb}$.
For $i\in I$, we write $\Orb^i$ for the corresponding orbit, it corresponds
to the partition $\tau^i$ with $\tau^i_i=\tau^i_{i+2}=\tau_i-1$ and $\tau^i_j=\tau_j$
for $j\neq i,i+2$.

We note that, by \cite{KP_normal}, $\overline{\Orb}$ is normal, so $X=\overline{\Orb}$.

\begin{Rem}\label{Rem:outer_involutions}
Now we discuss involutions $\sigma$ of $\mathfrak{sl}_n$ that are outer automorphisms. They correspond
to orthogonal or, for even $n$, symplectic forms on $\C^n$. We note that each irreducible representation
of $\slf_2$ carries a unique (up to proportionality) bilinear form, where $e,f$ are self-adjoint.
This form is symmetric. It follows that for an $\slf_2$-triple $(e,h,f)$ with $\sigma(e)=-e, \sigma(f)=-f$, the multiplicity space of each irreducible representation carries a nondegenerate form
of the same type as the form on $\C^n$. In particular, if $\kf=\mathfrak{sp}_n$, then each multiplicity
space is even dimensional. It follows that there are no codimension $4$ orbits in $\overline{\Orb}$
in this case.
\end{Rem}

\subsubsection{Quantizations and their slices}\label{SS_type_A_quant}
Let us discuss quantization parameters for the quantizations of $\C[\Orb]$.
Let $P$ be the parabolic subgroup in
$G$ corresponding to the partition $\tau^t$ so that $T^*(G/P)$ is a symplectic resolution of
$\overline{\Orb}$. Then the space of quantization
parameters is $H^2(G/P,\C)=\operatorname{Pic}(G/P)\otimes_{\mathbb{Z}} \C$,
\cite[Section 3.2]{orbit}.
This space is described as follows. To each column
of $\tau$ we assign a complex number, we get $\tau_1$ numbers $\lambda_1,\ldots,
\lambda_{\tau_1}$. Then $\operatorname{Pic}(G/P)\otimes_{\mathbb{Z}} \C$
is the space $\C^{\tau_1}$ modulo the line given by $\lambda_1=\lambda_2=\ldots=\lambda_{\tau_1}$.

We can identify
$H^2(G/P,\C)$ with $(\mathfrak{l}/[\mathfrak{l},\mathfrak{l}])^*$.
The subalgebra $\mathfrak{l}$ is the algebra of block diagonal matrices
with blocks of sizes $\tau^t_1,\ldots, \tau^t_{\tau_1}$.
For $\lambda_i$ we take the coefficient in $\lambda$ of the trace of its $i$th block.

The quantization of $X$ corresponding to $\lambda\in (\mathfrak{l}/[\mathfrak{l},\mathfrak{l}])^*$
is $\mathscr{D}^{\lambda+\rho_{G/P}}(G/P)$, where $\rho_{G/P}$ is one half of the character of $\mathfrak{l}$
in $\Lambda^{top}(\g/\mathfrak{p})$.

Now take the codimension $4$ $G$-orbit in $X$ corresponding
to index $i$.
Consider the slice algebra $\check{\A}^i_\lambda$. It is a quantization of $\overline{\Orb}_{min}$,
where $\Orb_{min}$ is the minimal nilpotent orbit in $\mathfrak{sl}_3$, so is of the form
$\mathscr{D}^{\lambda^i+3/2}(\mathbb{P}^2)$ for $\lambda^i\in \C$.

\begin{Lem}\label{Lem:param_corresp_A}
We have $\lambda^i=\lambda_{i}-\lambda_{i+1}$.
\end{Lem}
\begin{proof}
By \cite[Lemma 3.3]{Perv}, the period of the slice quantization $\check{\A}^i_\lambda$
is the pullback of the period of $\A_\lambda$ under the inclusion map $\mathbb{P}^2
\hookrightarrow T^*(G/P)$. So we just need to compute the pullback. Since the pullback map is linear, in the computation we can assume that all $\lambda_i$ are integers. 

Consider the basis in $\C^n$ labelled by the boxes of the Young diagram $\tau$.
Let $e$ be the element of $\Orb$ that acts by sending the basis element labelled by a
box to the one labelled
by the box directly to the right (and to zero if there's no such).
Let $e^i$ be the nilpotent element defined similarly
to $e$ but from $\tau^i$.

The preimage of $\check{X}^i(\cong \overline{\Orb}_{min})$ in $T^*(G/P)$ is $T^*\mathbb{P}^2$.
The zero fiber $\mathbb{P}^2$ is the parabolic Springer fiber consisting of all
partial flags $\{0\}=V_{\tau_1}\subset V_{\tau_1-1}\subset\cdots \subset\cdots V_{1}=\C^n$ with
$\dim V_i/V_{i+1}=\tau^t_i$ and $e^i V_i\subset V_{i+1}$ for all $i$. These conditions
fix all subspaces in the flag but $V_{i+1}$, while the condition on $V_{i+1}$ is that
it lies between $\operatorname{Im}(e^i)^{\tau_{i+1}-1}$ and $\operatorname{Im}(e^i)^{\tau_{i+1}}$.
The subspace       $V_j/\operatorname{Im}(e^i)^{\tau_{i+1}-1}$ is one-dimensional,
while the space $\operatorname{Im}(e^i)^{\tau_{i+1}-1}/\operatorname{Im}(e^i)^{\tau_{i+1}}$
is three-dimensional. This is how the parabolic Springer fiber is identified with
$\mathbb{P}^2$. So the pullback map $\operatorname{Pic}(G/P)\rightarrow
\operatorname{Pic}(\mathbb{P}^2)$ sends $(\lambda_1,\ldots,\lambda_{\tau_1})$
to the class of 
\begin{equation}\label{eq:line_bundle}
\mathcal{L}^{\otimes \lambda_{i+1}}\otimes (\Lambda^2[\mathcal{V}/\mathcal{L}])^{\otimes \lambda_{i}}. 
\end{equation}
Here $\mathcal{L}(\cong \mathcal{O}(-1))$ is the tautological line bundle
and $\mathcal{V}$ is the trivial rank 3 vector bundle. Under the identification of the target lattice
with $\Z$ that we use, the class of (\ref{eq:line_bundle}) is
$\lambda_{i}-\lambda_{i+1}$.
\end{proof}

\subsubsection{The case of inner involution}\label{SS:inner}
\begin{proof}[Proof of Corollary \ref{Cor:type_A_inner}]
By Remark \ref{Rem:outer_involutions}, we only need to consider the case when
$\mathfrak{k}=(\mathfrak{gl}_k\times \mathfrak{gl}_{n-k})\cap
\mathfrak{sl}_n$. Here the nilpotent orbits of $K$ in $\mathfrak{k}^\perp$ are parameterized
by ``$ab$-diagrams'' due to \cite[Section 4]{KP_normal}, also see \cite[Section 2.2]{Ohta1}. These are Young diagrams with an $a$ or a $b$ in each box so that there are $k$ $a$'s and $n-k$ $b$'s, and the labels $a$ and $b$ alternate in each row.
By \cite[Theorem 3]{Ohta2}, the closure relation is generated by pairs $(\tau>\tau')$, where $\tau'$ is obtained from
$\tau$ by moving one box up (we draw Young diagrams so that the number of boxes decreases
bottom to top) preserving the $ab$-label.

Let us describe $\mathfrak{k}^i_Q$ for a nilpotent $K$-orbit labelled by a given $ab$-diagram
of shape $\tau^i$.
We have $\mathfrak{k}_Q^i=(\prod_{j=1}^{\tau_1}\mathfrak{k}^i_{Q,j})\cap \slf_n$
with $\mathfrak{k}^i_{Q,j}\subset \mathfrak{gl}_{n_j}$, where $n_j$ is the number of rows of length $j$ in $\tau^i$, determined as follows.  Fix $j$ and consider the rows in $\tau^i$ of length $j$.
Suppose that the first column of these rows contains $n_j^a$ $a$'s and
$n_j^b$ $b$'s.   We have $\mathfrak{k}^i_{Q,j}=\mathfrak{gl}_{n_j^a}\times
\mathfrak{gl}_{n_j^b}$.

Now suppose that $j=\tau_{i+1}$ so that $n_j=3$. By Case 1 in the proof of Lemma
\ref{Lem:Qi_action}, we have the inclusion $\operatorname{SL}_3\hookrightarrow Q^i$
and the projection $(Q^i)^\circ\twoheadrightarrow \operatorname{PGL}_3$ such that the composition
$\operatorname{SL}_3\rightarrow \operatorname{PGL}_3$ is surjective and $Q^i$
acts on $\check{X}^i$ via its projection to $\operatorname{PGL}_3$. So $\mathfrak{k}^i_Q$
acts on $\check{X}^i$ via the image of $\operatorname{GL}_2$ in $\operatorname{PGL}_3$.
This excludes the case when $\theta^i$ comes from the outer involution.
Now we can use Theorem \ref{Thm:main1} to finish the proof of the corollary.
%
\end{proof}

Until the end of the section we concentrate on the case when $\tilde{K}=\operatorname{Spin}_n$ (and $K=\operatorname{SO}_n$). The pair $(\slf_n, \operatorname{Spin}_n)$ corresponds to the (nonlinear) universal ($2$-fold) cover of $\SL(n,\R)$.

\subsubsection{The group $\tilde{K}_Q/\tilde{K}_Q^\circ$} 
Let $\Orb_K$ be a $K$-orbit in $\kf^\perp\cap \Orb$.
We want to determine the reductive part (of the centralizer in $\tilde{K}$ of a point in $\Orb_K$) $\tilde{K}_Q$ and, especially, its component group.

For $K:=\operatorname{SO}_n$ (instead of $\tilde{K}$), this is easy. Let $K_Q$ denote the group defined similarly as $\tilde{K}_Q$, which is the image of $\tilde{K}_Q$ in $K$. For $i\in \{1,\ldots,\tau_1\}$, let $m_i$
denote the multiplicity of $i$ in $\tau$. Consider the group $\prod_{i=1}^{\tau_1}
\operatorname{O}_{m_i}$, it is embedded into $\operatorname{O}_n$ via
$(g_i)\mapsto \operatorname{diag}(g_i^{\oplus i})$. This is the reductive part of
the centralizer in $\operatorname{O}_n$. The element  $\operatorname{diag}(g_i^{\oplus i})$
belongs to $\operatorname{SO}_n$ if and only if $\prod_{i\text{ odd}}\det(g_i)=1$.
In particular, $K_Q/K_Q^\circ\cong (\Z/2\Z)^{\oplus (\tau_1-1)}$ (we cannot have only even
parts in $\tau$ because of the inequality $\tau_i-\tau_{i+1}\leqslant 1$). We also note that outside of the trivial case of $\tau=(1^2)$, every character
of $K_Q$ factors through $K_Q/K_Q^\circ$. If there is a codimension $4$ orbit in
$\overline{\Orb}$, then $(\mathfrak{k}_Q^*)^{\tilde{K}_Q}=\{0\}$.
Hence $\operatorname{Rep}(\tilde{K}_Q,\lambda|_{\kf_Q}-\kappa_Q)=\operatorname{Rep}(\tilde{K}_Q/\tilde{K}_Q^\circ)$.

Recall that $\tilde{K}=\operatorname{Spin}_n$. The group $\tilde{K}_Q/\tilde{K}_Q^\circ$ either coincides with
$K_Q/K_Q^\circ\cong(\Z/2\Z)^{\oplus \tau_1-1}$ or is a central extension of the latter group by
$\Z/2\Z$. Our first task is to determine which option holds.

\begin{Lem}\label{Lem:comp_group_Spin1}
If $m_i>1$ for some odd $i$, then $\tilde{K}_Q/\tilde{K}_Q^\circ\cong (\Z/2\Z)^{\oplus \tau_1-1}$.
Otherwise, $\tilde{K}_Q/\tilde{K}_Q^\circ$ is a central extension of $(\Z/2\Z)^{\oplus \tau_1-1}$
by $\Z/2\Z$.
\end{Lem}
\begin{proof}
The group $\tilde{K}$ is a central extension of $K$ by $\Z/2\Z$. Such an extension is parameterized by a class in $H^2(K,\Z/2\Z)$,
denote it by $s$. Let $\iota$ denote the inclusion $K_Q\hookrightarrow K$.
The group $\tilde{K}_Q$ is the central extension of $K_Q$ by $\Z/2\Z$, which corresponds
to the class $\iota^*(s)\in H^2(K_Q,\Z/2\Z)$. Let $\iota^\circ$ be the composition
of this inclusion with $K_Q^\circ\hookrightarrow K_Q$. We have
$\tilde{K}_Q/\tilde{K}_Q^\circ\xrightarrow{\sim} K_Q/K_Q^\circ$ if and only if
$\iota^{\circ*}(s)\neq 0$, otherwise $\tilde{K}_Q/\tilde{K}_Q^\circ$ is a central extension of $(\Z/2\Z)^{\oplus \tau_1-1}$
by $\Z/2\Z$.

Let $s_{m_i}$ denote the class of  the Spin extension of $\operatorname{SO}_{m_i}$. We claim
that $\iota^{\circ*}(s)=\sum_{j \text{ odd}} s_j$. Indeed, the pullback of $s$ under the
obvious inclusion of $\operatorname{SO}_{m_i}$ is $s_{m_i}$. So $\iota^{\circ*}(s)=\sum is_{m_i}$,
which implies our claim because $H^2(?,\Z/2\Z)$ is 2-torsion.

We have $s_{m_i}=0$ if and only if $m_i=1$.
 Now if we have
two connected groups $H_1,H_2$ and a central extension
$$1\rightarrow \Z/2\Z\rightarrow H\rightarrow H_1\times H_2\rightarrow 1,$$
then this extension splits if and only if the induced extensions of both $H_1$ and $H_2$
split. So $\iota^{\circ*}(s)=0$ if and only if
$is_{m_i}=0$ for all $i$. We arrive at the conclusion of the lemma.
\end{proof}

Now we describe $\tilde{K}_Q/\tilde{K}_Q^\circ$ in the case when this group is different from
$K_Q/K_Q^\circ$. Note that $K_Q/K_Q^\circ$ embeds
into the subgroup $\operatorname{SO}_{1^n}$ of $\operatorname{SO}_n$ consisting
of all elements that are diagonalizable in a given orthogonal basis. Namely,
we view a generator of $\operatorname{O}_i/\operatorname{SO}_i$ as an
element $\operatorname{diag}(-1,1,\ldots,1)\in \operatorname{O}(i)$. Taking the direct
sum over all $i$ we get the required embedding.
We also have a similarly defined subgroup $\operatorname{O}_{1^n}\subset \operatorname{O}_n$.

Now we define the ``unit Pin group'' $\operatorname{Pin}_{1^n}$. By definition,
it has generators $e_1,\ldots,e_n,-1$, where $-1$ is central and the relations
are $e_i^2=1, e_ie_j=-e_je_i$ for $i\neq j$. So, in particular, every permutation
of $\{1,\ldots,n\}$ gives an automorphism of $\operatorname{Pin}_{1^n}$.

Inside, we have the unit Spin group
$\operatorname{Spin}_{1^n}$ generated by $-1$ and the elements $e_ie_j, i\neq j$.
Then the preimage of $\operatorname{SO}_{1^n}$ in $\operatorname{Spin}_n$
is $\operatorname{Spin}_{1^n}$.

Now we describe the image of $\tilde{K}_Q/\tilde{K}_Q^\circ$ in $\operatorname{Spin}_{1^n}$.
For $i=1,\ldots,\tau_1$, we define $E_i\in \operatorname{Pin}_{1^n}$
as the product  $e_{j+1}\ldots e_{j+i}$, where $j=i(i-1)/2$. So
$E_1=e_1, E_2=e_2e_3$, etc. Let $\Gamma$ denote the subgroup of $\operatorname{Pin}_{1^n}$
generated by $E_1,\ldots,E_{\tau_1}$. Then, with a suitable choice of an orthogonal
basis in $\C^n$,  $\tilde{K}_Q/\tilde{K}_Q^\circ$ embeds as
the intersection of  $\Gamma$ with $\operatorname{Spin}_{1^n}$.
Note that the following relations among the $E_i$'s hold:
\begin{equation}\label{eq:E_relations}
\begin{split}
& E_i E_j=(-1)^{ij}E_j E_i,\\
& E_i^2=1, \text{ if } i\equiv 0,1\text{ mod }4,\\
& E_i^2=-1, \text{ else}.
\end{split}
\end{equation}
The elements of the form $E_iE_j$, where $i,j$ are both odd and $i+j$ is divisible by
$4$ are central in $\Gamma$. The square of such an element is equal to $1$. The elements $E_k$
with even $k$ are central as well.
Note that
\begin{enumerate}
\item[(i)]  The subgroup, $\Gamma_0$, generated
by the elements $E_{4i}$ and $E_{4i+1}E_{4i+3}$ is a  central subgroup
isomorphic to $(\Z/2\Z)^{q}$, where $q=\lfloor \tau_1/4\rfloor+\lfloor (\tau_1-3)/4\rfloor$.
\item[(ii)] The subgroup, $\Gamma_1$, generated by the elements $E_{4i+1}$ is naturally
identified with $\operatorname{Pin}_{(1^r)}$, where $r=\lfloor (\tau_1-1)/4\rfloor$.
\item[(iii)] The subgroup, $\Gamma_2$, generated by the elements $E_{4i+2}$ is
isomorphic to the quotient of $(\Z/4\Z)^{t}$ by $(\Z/2\Z)^{t-1}$ generated
by the differences of the elements of order $2$ in each factor $\Z/4\Z$. Here $t=\lfloor (\tau_1-2)/4\rfloor$.
\end{enumerate}
Note that both (ii) and (iii) have distinguished central elements of order $2$.
Then $\Gamma$ is the quotient of  the product of (i), (ii) and (iii)
by the difference of these elements.
Define the element $-1$ in the quotient to be the image of either
of the two elements.

Then $\tilde{K}_Q/\tilde{K}_Q^\circ$ is identified with the intersection of $\Gamma$ with $\operatorname{Spin}_{1^n}$.

\subsubsection{Homomorphism $\Z_4\rightarrow \tilde{K}_Q/\tilde{K}_Q^\circ$}\label{SSS_Z4_homom_typeA}
As we have seen in Section \ref{SS_loc_sys_slice_restriction}, the image of $\bullet_{\dagger,\chi}$
is controlled by homomorphisms $\Z_4\rightarrow \tilde{K}_Q/\tilde{K}_Q^\circ$, one for each $\Orb_K^i$.
The goal of this section is to compute these
homomorphisms. We have one homomorphism $\Z_4\rightarrow \tilde{K}_Q/\tilde{K}_Q^\circ$ for each codimension $2$ $K$-orbit in
$\overline{\Orb}_K$, i.e.,
for each part of $\tau$ that occurs with multiplicity $1$ and is different from
$\tau_1$.

Note that the map $\Orb_K\mapsto \Orb$
defines a bijection between the $\operatorname{O}_n$-orbits in $\kf^\perp$
and the $\operatorname{SL}_n$-orbits in $\g^*$. The $\operatorname{O}_n$-orbit
with Jordan type $\tau$ splits into the union of two $\operatorname{SO}_n$-orbit if
and only if all parts in $\tau$ are even. The partition $\tau$ of $\Orb$ never satisfies this condition.
And the partition  $\tau^i$ of a codimension $4$ orbit $\Orb^i\subset \overline{\Orb}$
only satisfies this condition when $\tau=(3,2,1)$ (and $\tau^i=(2,2,2)$).
In this case $\Orb_K^i$ will denote any of the two orbits.

Our first step is to understand the embedding $\underline{\tilde{K}}^i\hookrightarrow \tilde{K}_Q$, where we write
$\underline{\tilde{K}}^i$ for the  reductive part of the stabilizer $(\tilde{K}^{i,\circ}_Q)_\chi$ of the point $\chi \in \check{Y}^{i\times}$.  
We start with constructing a point in $\check{Y}^{i\times}$, the locus of $\theta$-stable points in
$\check{X}^{i\times}$. We present $\chi$
as a Jordan matrix with blocks of sizes $(\tau_1,\ldots,\tau_k)$ coming in this order.
It is symmetric with respect to the following form. For each
$\ell\in \{1,\ldots,\tau_1\}$, consider the linear span $V_\ell$
of all Jordan vectors that belong to blocks of size $\ell$
so that $\C^n=\bigoplus_\ell V_\ell$ (as a based
space). On each
$V_\ell$ we consider the form given by the unit anti-diagonal matrix
so that every Jordan matrix with blocks of the same size is symmetric. Then
for the form on $\C^n$ we take the direct sum of the forms on $V_\ell$'s.

Set $\ell=\tau_{i+1}$. Set $\chi=\chi^i+e_0$, where $e_0$ is the matrix
\[
\begin{pmatrix}0&0&1\\0&0&0\\0&0&0\end{pmatrix} \in \mathfrak{gl}_3
\] 
embedded
into $\operatorname{End}(V_{\ell}^{\oplus 3})$ via $x\mapsto x\otimes \operatorname{id}_{V_{\ell}}$,
compare to Case 2.2 of the proof of Proposition \ref{prop:model_a2c2}.
%
As was mentioned in that proof,
the Jordan type of $\chi$ is $\tau$, and $\chi\in \check{X}^i$.
And $\chi$ is symmetric with respect to our form, hence $\theta$-stable.
So we are done.

Now we can describe the homomorphism $\underline{\tilde{K}}^i\rightarrow \tilde{K}_Q$. For this, we will
describe the corresponding embedding $\underline{Q}^i\hookrightarrow Q$, where
$\underline{Q}^i$ is the reductive part of the stabilizer of $\chi$ in $Q^i$.
It is convenient to replace $\operatorname{SL}_n$ with $\operatorname{GL}_n$
(to get to the original setting we just restrict the embedding to the subgroups of matrices with determinant $1$).

We have $Q=\prod_{i}\operatorname{GL}_{m_i}\times \GL_{m_{\ell-1}}\times \GL_1
\times \GL_{m_{\ell+1}}$, where in the first product $i$ runs over
$\{1,\ldots,\tau_1\}\setminus \{\ell-1,\ell,\ell+1\}$. The involution preserves
each factor $\GL_{m_i}$,  and the restriction comes from the orthogonal form on $\C^{m_i}$ from Remark \ref{Rem:outer_involutions}. Next, we have
$Q^i=\prod_{i}\operatorname{GL}_{m_i}\times \GL_{m_{\ell-1}-1}\times \GL_3
\times \GL_{m_{\ell+1}-1}$, again with an involution that preserves each factor.
Then $\underline{Q}^i=\prod_{i}\operatorname{GL}_{m_i}\times \GL_{m_{\ell-1}-1}\times \GL_1^2
\times \GL_{m_{\ell+1}-1}$. Here $\GL_1^2$ is embedded into $\GL_3$ via
$(t_1,t_2)\mapsto \operatorname{diag}(t_1,t_2,t_1)$. On the common factor $\prod_{i}\operatorname{GL}_{m_i}$
the embedding $\underline{Q}^i\hookrightarrow Q$ is the identity. Note also that the
embedding $\underline{Q}^i\hookrightarrow Q$
\begin{itemize}
\item[(i)] is compatible with the involution,
\item[(ii)] and intertwines the restrictions of the tautological representation of
$\operatorname{GL}_n$.
\end{itemize}

Thanks to (ii), $\operatorname{GL}_{m_{\ell-1}-1}$ embeds into $\operatorname{GL}_{m_{\ell-1}}$
as the stabilizer of a vector and a complimentary subspace. Thanks to (i), this vector
is non-isotropic. This determines the embedding up to an orthogonal automorphism.
The same consideration applies to $\operatorname{GL}_{m_{\ell+1}-1}
\hookrightarrow\operatorname{GL}_{m_{\ell+1}}$. Similar considerations show that, in the case when $\ell>1$,
the embedding of $\operatorname{GL}_1^2$ into $\GL_{m_{\ell-1}}\times \GL_1
\times \GL_{m_{\ell+1}}$ is as follows:
\begin{equation}\label{eq:embed_GL}
	(t_1,t_2)\mapsto (\operatorname{diag}(t_1,1,\ldots,1), t_2, \operatorname{diag}(t_1,1,\ldots,1)),
\end{equation}
where the first entries in the first and the third components correspond to the non-isotropic
vectors in $\C^{m_{\ell-1}},\C^{m_{\ell+1}}$, when we write diagonal matrices, we use an orthogonal basis in each $\C^{m_i}$. When $\ell=1$, we omit the first component in
the image.

Once the embedding $\underline{Q}^i\hookrightarrow Q$ is known, we get the embedding
$\underline{K}^i\hookrightarrow K_Q$, where $\bar{\bullet}$ indicates that we are considering the subgroups in $\operatorname{SO}_n$ (and not in
$\operatorname{Spin}_n$) by passing to the fixed points under the involution.
And $\underline{\tilde{K}}^i$ is obtained from $\underline{K}^i$ by taking the central extension pulled back from $\tilde{K}_Q\twoheadrightarrow K_Q$.

Now we proceed to describing the homomorphism $\Z_4\rightarrow \tilde{K}_Q/\tilde{K}_Q^\circ$
associated to $\ell$. Recall, Section \ref{SS_loc_sys_slice_restriction}, that it is computed in the following steps:

\begin{itemize}
\item[(I)] The choice of $\ell<\tau_1$ such that $\ell$ occurs in $\tau$
with multiplicity $1$ (equivalently, of a codimension $2$ $K$-orbit)
corresponds to a direct summand $\mathfrak{so}_3$
of $\kf^i_Q$. We consider the corresponding homomorphism $\operatorname{SL}_2\rightarrow \tilde{K}^i_Q$.
\item[(II)] We pass to the homomorphism $\Z_4\rightarrow \underline{\tilde{K}}^i$ between the reductive
parts of the stabilizers.
\item[(III)] Then we compose $\Z_4\rightarrow \underline{\tilde{K}}^i$ with the homomorphism
$\underline{\tilde{K}}^i\rightarrow \tilde{K}_Q/\tilde{K}_Q^\circ$ to get the required homomorphism
$\Z_4\rightarrow \tilde{K}_Q/\tilde{K}_Q^\circ$.
\end{itemize}

The cases of even $\ell$ and odd $\ell$ behave differently. We start with even $\ell$.
As we have seen in the proof of Lemma \ref{Lem:comp_group_Spin1}, in this case, the pullback of the central extension
from $\operatorname{SO}_n$ to $\operatorname{SO}_3$ is trivial. Equivalently,
the homomorphism $\operatorname{SL}_2\rightarrow \tilde{K}^i_Q$ factors through $\operatorname{SO}_3$.
The resulting homomorphism $\Z_4\rightarrow \tilde{K}_Q/\tilde{K}_Q^\circ$ is not injective and
this is the only thing that we need to know about it: by Lemma \ref{Lem:sl3_so3_obstructive}
the quantization of any $K$-equivariant local system on $\Orb_K$ extends to $\Orb_K^i$.

Now we consider the case when $\ell$ is odd. Here we do get an injective homomorphism
$\operatorname{SL}_2\rightarrow \tilde{K}^i_Q$. The stabilizer in $\operatorname{SO}_3$ of the
point $\chi$ as above is generated by $\operatorname{diag}(-1,1,-1)$. By \eqref{eq:embed_GL}, this element goes to
$E_{\ell-1}E_{\ell+1}\in \tilde{K}_Q/\tilde{K}_Q^\circ$ (up to $\pm 1$) when $\ell>1$ and to
$E_2$ when $\ell=1$.  In particular, the images of the homomorphisms
$\Z_4\rightarrow \tilde{K}_Q/\tilde{K}_Q^\circ$ all land in the abelian normal subgroup
$\Gamma_0\Gamma_2$.

\subsubsection{The main result} \label{SS:double_SLR}
Note that each $E_i$ is actually defined up to a sign. We can change the signs any way we want.
In particular, we arrive at the following characterization of the irreducible representations
of $\tilde{K}_Q/\tilde{K}_Q^\circ$ that lie in the image of $\bullet_{\dagger,\chi}$. Let us start with the special case of canonical quantizations.

\begin{Prop}\label{Prop:spin_HC_classif_canon}
Assume the quantization parameter $\lambda$ is equal to $0$.
For a suitable choice of the generators $E_i$ for even $i$, the following two conditions are
equivalent:
\begin{enumerate}
\item a representation
$V$ of $\tilde{K}_Q/\tilde{K}_Q^\circ$ lies in the image of $\bullet_{\dagger,\chi}$,
\item
for all odd $\ell<\tau_1$, the element $E_{\ell-1}E_{\ell+1}\in \tilde{K}_Q/\tilde{K}_Q^\circ$
(or $E_2$ for $\ell=1$) acts on $V$ without eigenvalue $-\sqrt{-1}$.
\end{enumerate}
\end{Prop}
\begin{proof}
By Lemma \ref{Lem:sl3_so3_obstructive}, all irreducible $\operatorname{SL}_2$-equivariant
local systems $\check{\mathcal{E}}^i$ but one (where the action of $\operatorname{SL}_2$ is faithful)
there is a HC $(\check{\A}^i_\lambda,\operatorname{SL}_2)$-module $M^i$ with
$\gr M^i=\Gamma(\check{\mathcal{E}}^i)$. The remaining local system does not correspond to any HC module.

By changing the generator of $\Z_4$, we can assume that the
remaining local system corresponds to the eigenvalue $-\sqrt{-1}$. We can further assume that
the generator maps to $E_{i-1}E_{i+1}$ (or $E_2$ for $i=1$) -- here we may change their signs.  Now the claim of this proposition
follows from Proposition \ref{Lem:image_description11}.
\end{proof}

We can also describe the image of $\bullet_{\dagger,\chi}$ for a general quantization parameter.
Let $\lambda^i$ be the quantization parameter for $\check{\A}^i_\lambda$, by Lemma
\ref{Lem:param_corresp_A}, we have $\lambda^i=\lambda_i-\lambda_{i+1}$ in the notation
of Section \ref{SS_type_A_quant}.

\begin{Thm}\label{Thm:spin_HC_classif_general}
Fix a choice of the generators $E_i$ for even  $i$ as in Proposition \ref{Prop:spin_HC_classif_canon}.
Then $V\in \operatorname{Im}\bullet_{\dagger,\chi}$ if and only if for all odd $i$ we have the following:
\begin{itemize}
\item if $\lambda^i\not\in \Z$, then the corresponding generator ($E_{i-1}E_{i+1}$ for $i>1$, or
$E_2$ for $i=1$) acts on $V$ with eigenvalues $\pm 1$,
\item if  $\lambda^i\in \Z$, then the corresponding generator acts on $V$ without eigenvalue
$\sqrt{-1}^{2\lambda^i-1}$.
\end{itemize}
\end{Thm}
\begin{proof}
First, we are going to reduce to the case when all integer $\lambda^i$ (for odd $i$) are zero.
The proof of Lemma \ref{Lem:param_corresp_A} shows that there is a line bundle $\mathcal{L}$
on $T^*(G/P)$ such that the restriction of $\mathcal{L}$ to the parabolic Springer fiber $\mathbb{P}^2$
of $\chi^i$ is $\mathcal{O}(-\lambda^i)$ if $\lambda^i$ is an integer. Let $\tilde{\lambda}$ denote the parameter obtained from $\lambda$ by replacing the integral entries by $0$.

For every line bundle
$\mathcal{L}'$ on $G/P$ and every quantization parameter $\lambda'$, consider
$\mathcal{L}'\otimes D^{\lambda'}_{G/P}$, this is a $D^{\lambda'+c_1(\mathcal{L}')}_{G/P}$-$D^{\lambda'}_{G/P}$-bimodule. Let $\A_{\lambda'}(\mathcal{L}')$
denote the global sections of this bimodule. Note that since $\mathcal{L}'\otimes D^{\lambda'}_{G/P}$ is an invertible
bimodule, and $T^*(G/P)\rightarrow X$ is an isomorphism over $\Orb$, the functor of tensoring
with $\A_{\lambda'}(\mathcal{L}')$ is an equivalence $\overline{\HC}(\A_{\lambda'},\theta)\xrightarrow{\sim} \overline{\HC}(\A_{\lambda'+c_1(\mathcal{L}')},\theta)$.

Let $\pi$ denote the projection $T^*(G/P)\twoheadrightarrow G/P$
and $\check{\mathcal{L}}'^i$ denote the restriction of $\mathcal{L}'$
to the copy of $\mathbb{P}^2\subset G/P$ corresponding to $i$.
Suppose $\mathcal{L}'$ is sufficiently ample, in particular,
$H^1(T^*(G/P),\pi^*\mathcal{L}')=0$. Then by \cite[Proposition 3.8]{Perv},
we have $\A_{\lambda'}(\mathcal{L}')_{\dagger,\chi^i}=\check{\A}^i_\lambda(\check{\mathcal{L}}'^i)$,
here we use the restriction functor for HC $\U$-bimodules.
Thanks to Lemma \ref{Lem:restriction_compatibilty}, for any $M\in \HC(\A_{\lambda},\theta)$,
we have a $\U$-linear isomorphism
\begin{equation}\label{eq:restr_iso}
\left(\A_{\lambda'}(\mathcal{L}')\otimes_{\U}M\right)_{\dagger,\chi^i}\cong
\check{\A}^i_\lambda(\check{\mathcal{L}}'^i)\otimes_{\U_{\dagger,\chi^i}}M_{\dagger,\chi^i}.
\end{equation}
Note that the action of $\U$ on $M$ factors through $\A_{\lambda}$, hence the action of
$\U_{\dagger,\chi^i}$ on $M_{\dagger,\chi^i}$ factors through $\check{\A}^i_\lambda$.
The latter is a simple algebra because its period is integral (and hence the twist for
differential operators on $\mathbb{P}^2$ is in $\frac{1}{2}+\Z$). So in the right hand side
of (\ref{eq:restr_iso}) we have $\check{\A}^i_\lambda(\check{\mathcal{L}}'^i)\otimes_{\check{\A}^i_\lambda}M_{\dagger,\chi^i}$.

Let $\psi$ denote the nontrivial automorphism of $\Z_4$ (so that $\psi^2=1$)
and let $\zeta_i$ denote the homomorphism $\Z_4\rightarrow \tilde{K}_Q/\tilde{K}_Q^\circ$
corresponding to the index $i$.
If $\mathcal{L}'^i=\mathcal{O}(d_i)$, then, on the level of representations of $\Z_4$, tensoring with $\check{\A}^i_\lambda(\check{\mathcal{L}}'^i)$ results in twisting with $\psi^{d_i}$. We can choose $\lambda'\in (\mathfrak{l}/[\mathfrak{l},\mathfrak{l}])^*$ such that $\lambda'-\lambda, \lambda'-\tilde{\lambda}$
are sufficiently dominant. We get equivalences
$$\overline{\HC}(\A_{\lambda},\theta)\xrightarrow{\sim} \overline{\HC}(\A_{\lambda'},\theta)
\xleftarrow{\sim} \overline{\HC}(\A_{\tilde{\lambda}},\theta).$$
The resulting self-equivalence $\eta,\tilde{\eta}$ of $\operatorname{Rep}(\tilde{K}_Q/\tilde{K}_Q^\circ)$ satisfy the following properties
for each $i$:
\begin{equation}\label{eq:twists_reps}
\zeta_i\circ\eta=\psi_*^{\lambda_i'-\lambda_i}\circ \zeta_i,
\zeta_i\circ\tilde{\eta}=\psi_*^{\lambda_i'-\tilde{\lambda}_i}\circ \zeta_i.
\end{equation}
Here the subscript $*$ indicates the twist by the corresponding automorphism.
Equations (\ref{eq:twists_reps}) reduce the claim of the theorem to the case of $\lambda=\tilde{\lambda}$.
This is what we assume until the rest of the proof.

Recall, Lemma \ref{Lem:sl3_so3_obstructive}, that if $\lambda_i$ is not an integer, then the image of $\overline{\HC}(\check{\A}^i_\lambda,\theta)$  in $\operatorname{Rep}(\Z_4)$ is $\operatorname{Rep}(\Z_2)$. Moreover, every $\operatorname{SO}_3$-equivariant local system
is strongly $\check{\A}^i_\lambda$-quantizable, by Lemma \ref{Lem:sl3_so3_obstructive}. We now apply Proposition \ref{Lem:image_description11} to finish the proof of the theorem.
\end{proof}

\subsection{Exceptional types}\label{SS_exceptional}

Throughout this section, $\g$ is a simple complex Lie algebra of exceptional type, $G$ is the corresponding simply connected simple group, $\Orb$ is a nilpotent $G$-orbit in $\g$ whose closure contains an orbit $\Orb'$ of codimension $4$, such that $\Orb'$ is open in $\overline{\Orb} \backslash \Orb$ and the singularity of $\overline{\Orb}$ along $\Orb'$ is of type $a_2$, $a_2 / S_2$ or $2a_2$. Note that in all the cases discussed here, $\Orb$ is always birationally rigid in the sense of \cite[Definition 1.2]{orbit}, except for the two orbits of type $A_2 + A_1$ and $A_2+2A_1$ in $\mathfrak{e}_6$ (for the latter one see Section \ref{SS:orbit_a} below).

\begin{Lem}\label{lem:only_a2}
	Under the assumption above, let $X = \spec \C[\Orb]$. Then $X$ contains a unique codimension $4$ symplectic leaf $X'$ with $a_2$ singularity and has no codimension $2$ leaves. Moreover, in this case $X'$ is the only  codimension $4$ leaf, except when $\g = \mathfrak{e}_7$ and $\Orb$ is of type $A_4+A_1$. For this orbit, there is another codimension $4$ leaf with $a_2/S_2$-singularity.
\end{Lem}

\begin{proof}
	The lemma follows by inspecting the closure diagram of nilpotent orbits of Lie algebra of exceptional types and the generic singularities of their closures in \cite{FJLS1}. The only case that requires some care is the orbit $\Orb$ of type $A_2 + A_1$ in $\mathfrak{e}_6$. There is only one orbit $\Orb'$ in $\overline{\Orb}$ with codimension $4$, which is of type $A_2$, and the singularity here is $2a_2$. Let $f: X=\spec \C[\Orb] \to \overline{\Orb}$ be the normalization morphism and $e$ be any point in $\Orb'$. By the discussions in \cite{FJLS1} above Proposition 2.1 there, the action of the component group $A(e')$ (with respect to the adjoint group $G$) of $\Orb'$ at a point $e' \in \Orb'$ on the connected components of $f^{-1}(e')$, i.e., the branches of $\overline{\Orb}$ at $e'$, is transitive.
	Now $A(e) \simeq S_2$ by the table on page 402 of \cite{Carter}. Therefore the preimage of the orbit $\Orb'$ under the normalization map $f$ is a connected double cover $\tilde{\Orb}'$ of $\Orb'$. This is a symplectic leaf $X'$ in $X$, along which the signularity is $a_2$.
\end{proof}

From now on, we make the following convention unless otherwise mentioned. Assume $K$ is the fixed subgroup of the simply connected group $G$ with respect to some involution $\sigma$, which is automatically connected. Note that $K$ has finite fundamental group except for the symmetric pairs of type $E_{6(-14)}$ and $E_{7(-25)}$, for which $\pi_1(K) \simeq \Z$. Also note that there is no relevant $K$-orbit of $E_{7(-25)}$. If $\pi_1(K)$ is finite, let $\tilde{K}$ denote the simply connected cover of $K$. Otherwise $\tilde{K}$ can mean any finite connected cover of $K$. Sometimes we also consider the adjoint group $G_{ad}$ and write $\bar{K} = (G_{ad}^\sigma)^\sigma$.

Let $X = \spec \C[\Orb]$ and $X'$ be the unique codimension $4$ symplectic leaf along which the singularity of $X$ is of type $a_2$ as in Lemma \ref{lem:only_a2}. Let $\Orb_K$ be a nilpotent $K$-orbit in $\Orb \cap \kf^\perp$ and $\overline{\Orb}_K$ be its closure in $X$. Let $\Orb'_K$ be a $K$-orbit contained in $\overline{\Orb}_K \cap X'$. Fix a point $\chi$ in $\Orb_K$. Let $Q, K_Q, \tilde{K}_Q$ denote the reductive parts of the centralizers of $\chi$ in $G,K,\tilde{K}$,
respectively. Let $Z_K = K_Q / K_Q^\circ$ and $\tilde{Z}_K = \tilde{K}_Q / \tilde{K}_Q^\circ$ denote the component groups. Let $Q'$, $K'_Q$ and $\tilde{K}'_Q$ be the groups defined similarly for a fixed element $\chi' \in \Orb'_K$ as for $\chi \in \Orb_K$. As in Section \ref{SS_main_results}, we take a $\theta$-stable slice $\check{X}$ in $X$ that is transversal to $X'$ at $\chi'$, then $\check{X}$ is isomorphic to the minimal nilpotent orbit in $\mathfrak{sl}_3$. Then the group $Q'$ acts on $\check{X}$, $K'_Q$ and $\tilde{K}'_Q$ acts on $\check{X}^\theta$.



By \cite[Proposition 3.3]{FJLS1} or Lemma \ref{Lem:Qi_action}, the action of $Q'$ on $\check{X}$ has a dense orbit. Therefore the restriction of $\theta$ on $\check{X}$ can be read off from the restriction of $\sigma$ on $\q'$, which can be found in \cite{Dk0} (in fact it is the corresponding anti-holomorphic involution that is given there). We assume that $\chi$ lies in $\check{X} \cap \Orb_K$ and the induced involution on $\check{X}$ is outer. 
Let $\underline{K}$ ($\underline{\tilde{K}}$ resp.) denote the reductive centralizer (=the component group of the centralizer) of $\chi$ for the action of $K_Q'^\circ$ ($\tilde{K}'^\circ_Q$ resp.) on $\check{X}\cap \Orb_K$. Then $\underline{K} \simeq \Z_2$. Analogous to \eqref{eq:hom_Z4}, we have the composite map
  \[  \Z_4\hookrightarrow (\SL_2)_\chi \rightarrow (\tilde{K}'^\circ_{Q})_\chi \hookrightarrow \tilde{K}_Q\twoheadrightarrow \tilde{K}_Q / \tilde{K}_Q^\circ = \tilde{Z}_K.\]
and a similarly defined map $\Z_4 \to Z_K$ for $K$. As we have seen in Section \ref{SS_loc_sys_slice_restriction}, the image of $\bullet_{\dagger,\chi}$ is controlled by homomorphism $\Z_4 \rightarrow \tilde{Z}_K$ in \eqref{eq:hom_Z4}. If this map is not injective, then similarly to the discussion in the end of Section \ref{SSS_Z4_homom_typeA}, the quantization of every irreducible $(\tilde{K},\kappa)$-equivariant
suitably  twisted local system on $\Orb_K$ extends to $\Orb'_K$. In order to determine whether the map $\Z_4 \to \tilde{Z}_K$ is injective, we will need to compute $\tilde{Z}_K$ and $\underline{\tilde{K}}$. We have a natural surjective map $\underline{\tilde{K}} \to \underline{K}$, whose kernel is naturally identified with the kernel of the covering map $\tilde{K}'^\circ_Q \twoheadrightarrow K_Q'^\circ$. Therefore we will have to compute the covering morphism $\tilde{K}'^\circ_Q \twoheadrightarrow K_Q'^\circ$ to determine $\underline{\tilde{K}}$. In most cases, knowing the covering map $\tilde{K}_Q^\circ \twoheadrightarrow K_Q^\circ$, the group $\underline{\tilde{K}}$ and the map $\underline{\tilde{K}} \to \tilde{Z}_K$ could help us determine $\tilde{Z}_K$, but not in all cases. For our purposes, however, we only need to check if the quotient map $\tilde{Z}_K \to Z_K$ is an isomorphism. 

Let $\tilde{K}_Q^\bullet$ denote the preimage of $K_Q^\circ$ in $\tilde{K}$ under the covering map $p: \tilde{K} \to K$. Then we have the inclusions $\tilde{K}_Q^\circ \subset \tilde{K}_Q^\bullet \subset \tilde{K}_Q$. It is clear that $\tilde{Z}_K \to Z_K$ is an isomorphism if and only if $\tilde{K}_Q^\bullet$ is connected. To determine if this is the case, we first introduce the following notations. Let $\iota: K_Q^\circ \hookrightarrow K$ and $\tilde{\iota}: \tilde{K}_Q^\bullet \hookrightarrow \tilde{K}$ denote the inclusion maps. We will use $p: \tilde{K}_Q^\bullet \twoheadrightarrow K_Q^\circ$ to denote the restriction of $p$ to $\tilde{K}_Q^\bullet$. Let $\langle \pi_1(K_Q^\circ), \pi_1(\tilde{K}) \rangle$ denote the subgroup of $\pi_1(K)$ generated by the images of the natural group homomorphisms $\pi_1(K_Q^\circ) \to \pi_1(K)$ and $\pi_1(\tilde{K}) \to \pi_1(K)$. Finally, let $m(\tilde{K}) := [\pi_1(K) : \langle \pi_1(K_Q^\circ), \pi_1(\tilde{K}) \rangle]$ be the index.

We then have the following elementary lemma:

\begin{Lem}\label{lem:connectedness}
	For any connected $m$-fold cover $p: \tilde{K} \to K$, the component group $\pi_0(\tilde{K}_Q^\bullet)$ is canonically isomorphic to the quotient group $\pi_1(K) / \langle \pi_1(K_Q^\circ), \pi_1(\tilde{K}) \rangle$. Moreover, the degree of the covering map $\tilde{K}_Q^\circ \twoheadrightarrow K_Q^\circ$ is given by $\frac{m}{m(\tilde{K})}$.
	
	In particular, $\tilde{K}_Q^\bullet$ is a connected ($m$-fold) cover of $K_Q^\circ$ if and only if the images of the natural group homomorphisms $\pi_1(K_Q^\circ) \to \pi_1(K)$ and $\pi_1(\tilde{K}) \to \pi_1(K)$ generate $\pi_1(K)$.
\end{Lem}

\begin{proof}
	We have a short exact sequence:
	\[
	\tilde{K}_Q^\bullet \xrightarrow{(p, \tilde{\iota})} K_Q^\circ \times \tilde{K} \xrightarrow{(\iota, p^{-1})} K,
	\]
	where $p^{-1}: \tilde{K} \to K$ stands for the composition of $p:  \tilde{K} \to K$ and the inverse map $K \to K$. Note that the map $\pi_1(p^{-1}): \pi_1(\tilde{K}) \to \pi_1(K)$ induced by $p^{-1}$ equals $[\pi_1(p)]^{-1}$, the composition of $\pi_1(p): \pi_1(\tilde{K}) \to \pi_1(K)$ induced by $p$ and the inverse map of the group $\pi_1(K)$. The claims then follow from the associated long exact sequence of homotopy groups based at the identities of the Lie groups.
\end{proof}

We now explain how to compute the map $\pi_1(K_Q^\circ) \to \pi_1(K)$. For each pair $(G,K)$, we choose a maximally compact $\sigma$-stable Cartan subalgebra $\hf$ of $\g$ with the corresponding (connected) Cartan subgroup $H$. We also choose simple roots $\Pi = \{ \alpha_1, \ldots, \alpha_l \}$, following the labeling of simple roots in \cite{Bourbaki}. Then, $\tf := \hf^\sigma$ is a Cartan subalgebra of $\kf$. For each $K$-orbit $\Orb_K$ and each $\chi \in \Orb_K$, we define $\tf_\chi := \tf \cap \kf_\chi = \tf \cap \kf_Q$, which is a Cartan subalgebra of the Lie algebra $\kf_Q$ with the corresponding Cartan subgroup $T_\chi$ of $K_Q^\circ$. 

In \cite{Noel1, Noel2}, No{\"e}l gave an explicit nilpotent element $\chi$ as a sum of root vectors, so that $\tf_\chi$ can be computed explicitly (cf. Remark \ref{rem:Noel1}). Since $G = G_{sc}$ is simply connected, the coweight lattice $X_*(H) = \ker(\exp_H(2 \pi i \, \bullet): \hf \to H)$ of $G$ is the coroot lattice of $\g$, generated by simple coroots $\alpha^\vee_i$ of $\g$ (denoted $H_i$ by No{\"e}l) as a basis, where $\exp_H$ is the exponential map of $H$. Hence, the coweight lattice $X_*(T_\chi) = \ker(\exp_{T_\chi}( 2 \pi i \, \bullet): \tf_\chi \to T_\chi)$ of $K_Q^\circ$ is the intersection of $\tf_\chi$ with the coroot lattice of $\g$. A basis $\{\vartheta^{\chi}_i \}_{1 \leq i \leq k}$ for $X_*(T_\chi)$ can be computed easily and has been provided in \cite{Noel1, Noel2} for each orbit listed there. Each $\vartheta^{\chi}_i$ is chosen to be of the form $\sum_{j=1}^{l} b_i^j \alpha^\vee_{j}$, where $b_i^j$, for $1 \leq j \leq l$, are coprime integers for each $i$. Since the fundamental group of a connected complex reductive group is naturally isomorphic to the quotient of its coweight lattice by its coroot lattice, the information above is sufficient to determine the natural map $\pi_1(K_Q^\circ) \to \pi_1(K)$.

Since $\g$ is simple, by Theorem \ref{thm:pi_GR}, we have $\pi_K \simeq \pi_1(G_\R) \simeq 1$, $\Z_2$, or $\mathbb{Z}$. When $\pi_1(K) = 1$, the situation is trivial. Therefore we only consider the remaining two nontrivial cases below.

\subsubsection{The case $\pi_1(K) \simeq \mathbb{Z}$} \label{SSS_pi1=Z}
For simple exceptional $\g$, the only possible cases for $\g_\R$ are $E_{6(-14)}$ and $E_{7(-25)}$ (we will only encounter $E_{6(-14)}$ in Theorem \ref{thm:exceptional}). By Theorem \ref{thm:pi_GR} (5), $\GR$ is Hermitian symmetric, $\tf = \hf$, and we can choose the set of simple roots $\Pi$ so that exactly one simple root is non-compact imaginary, denoted as $\alpha_{nc}$, with the corresponding coroot $\alpha_{nc}^\vee$. The choice of roots can be read off from the Vogan diagrams of $\g_\R$ (see \cite[Chapter VI, Sections 8 and 10]{Knapp}).

We can choose the basis $\{\vartheta^{\chi}_i \}$ so that at most one basis vector has a nonzero (integral) coefficient for $\alpha_{nc}^\vee$ when written as a linear combination of the simple coroots $\alpha^\vee_i$. If such a basis vector exists, we denote it by $\vartheta^{\chi}_{nc}$, and its coefficient for $\alpha_{nc}^\vee$ is denoted by $b_{nc}$. Otherwise, let $b_{nc} = 0$. By Theorem \ref{thm:pi_GR}, (1) and (4), the image of $\pi_1(K_Q^\circ) \to \pi_1(K) \simeq \mathbb{Z}$ is generated by $b_{nc} \alpha^\vee$. By Lemma \ref{lem:connectedness}, for a connected $m$-fold cover $\tilde{K} \to K$, the preimage $\tilde{K}_Q^\bullet$ of $K_Q^\circ$ in $\tilde{K}$ is a connected cover of $K_Q^\circ$, or equivalently, the degree of the covering map $\tilde{K}_Q^\circ \twoheadrightarrow K_Q^\circ$ equals $m$ if and only if $b_{nc}$ is coprime to $m$.

We can also re-examine\footnote{which is necessary, see Remark \ref{rem:Noel2}} the results in \cite{Noel1} about admissibility for arbitrary $\Orb_K$ of $E_{6(-14)}$ and $E_{7(-25)}$ (without restrictions on codimension) and refine them to include arbitrary finite covers $\tilde{K}$ of $K$. We can express the restriction of $d \delta_\chi := (\omega_{\Orb_K})_\chi \in (\kf_Q^*)^{K_Q}$ (as in Section \ref{SS_intro_approach}) to $\tf_\chi$ in terms of the basis $\{\vartheta^{\chi}_i \}$: for any element 
\[
z = \sum_{j=1}^k z_j \vartheta^{\chi}_j \in \tf_\chi
\]
with $z_j \in \mathbb{C}$, write 
\[
d \delta_\chi (z) = \sum_{j=1}^k c_j z_j,
\]
where $c_j \in \mathbb{Z}$. Then, $\Orb_K$ is linearly admissible if and only if all $c_j$'s are even. 

If this is not the case, consider the values of $b_{nc}$: if $b_{nc} = 0$, i.e., $\alpha_{nc}^\vee$ does not appear in any basis vector $\{\vartheta^{\chi}_i \}$, then by Lemma \ref{lem:connectedness}, the map $\tilde{K}_Q^\circ \twoheadrightarrow K_Q^\circ$ is always an isomorphism, and $\Orb_K$ is not admissible (for any finite cover $\tilde{K}$ of $K$). If $b_{nc} = 1$, then Lemma \ref{lem:connectedness} implies that $\tilde{K}_Q^\bullet$ is always connected for any $m$-fold cover $\tilde{K}$ of $K$. In this case, we write
\[
z = \sum_{j=1}^k z_j \vartheta^{\chi}_j = z_1 \vartheta^{\chi}_1 + \cdots + z_{nc} \vartheta^{\chi}_{nc} + \cdots + z_k \vartheta^{\chi}_k \in \tf_\chi
\]
and
\[
d \delta_\chi (z) = \sum_{j=1}^k c_j z_j = c_1 z_1 + \cdots + c_{nc} z_{nc} + \cdots + c_k z_k.
\]
Then, $\Orb_K$ is admissible for $\tilde{K}$ if and only if all $c_j$'s, except for $c_{nc}$, are even (so $c_{nc}$ is automatically odd) and $m$ is even. In this case, for any even $m$, any $\tilde{K}$-equivariant admissible vector bundle is, in fact, equivariant with respect to the unique $2$-fold cover of $K$. It turns out that these are all the cases we need to consider by examining the tables in \cite{Noel1} for {\bf E III }$= E_{6(-14)}$ and {\bf E VII }$= E_{7(-25)}$.

\begin{Rem}\label{rem:Noel2}
	Noel wrote down $6$ simple roots of $\kf = \mathfrak{so}(10) \oplus \C$ for $E_{6(-14)}$ (\cite[pg. 467]{Noel1}) and $7$ simple roots of $\kf = \mathfrak{e}_6 \oplus \C$ for $E_{7(-25)}$ (\cite[pg. 477]{Noel1}), which is confusing. Probably he just extended the set of simple roots of the semisimple part of $\kf$ to a basis of the Cartan subalgebra $\tf$ of $\kf$. He also used the term ``fundamental weights" for $\kf$ of $E_{6(-14)}$, which form a basis of $\tf$. We do not understand what this means and how this helps him determine admissibility. However, our re-examination seems to detect no mistakes regarding admissibility from his tables for $E_{6(-14)}$ and $E_{7(-25)}$.
\end{Rem}

\subsubsection{The case $\pi_1(K) \simeq \mathbb{Z}_2$}\label{SSS_pi1=Z2}
In this case, $\tilde{K}_Q^\circ \twoheadrightarrow K_Q^\circ$ is either an isomorphism or a connected $2$-fold covering map. To determine which is the case, we can compute the map $\pi_1(K_Q^\circ) \to \pi_1(K)$ and apply Lemma \ref{lem:connectedness} as we did in the case when $\pi_1(K) \simeq \mathbb{Z}$, using the Kac diagrams associated with $\g_\R$ (see \cite{Adams_nonlinear}) in a more involved way. However, we choose to adapt an alternative shortcut here. 

Namely, let $\{ \varpi_1, \ldots, \varpi_n \}$ denote the fundamental weights of $\kf$, which have also been provided in \cite{Noel1, Noel2}. Evaluation of $\{ \varpi_1, \ldots, \varpi_n \}$ at each basis vector $\vartheta^{\chi}_i$ gives us information about the order of the element $\exp(2\pi i \vartheta^{\chi}_i)$ in $\tilde{K}$. Specifically, let $m_i$ be the least positive integer such that all $\varpi_k$, for $1 \leq k \leq n$, have integral values at $m_i \vartheta^{\chi}_i$. Then $m_i$ is the order of $\exp(2\pi i \vartheta^{\chi}_i)$ in $\tilde{K}$. Since $\pi_1(K) \simeq \mathbb{Z}_2$, $\exp(2\pi i \vartheta^{\chi}_i)$ is uniquely determined by its order, and hence we can determine $\ker(\tilde{K}_Q^\circ \twoheadrightarrow K_Q^\circ)$ as a subgroup of $\pi_1(K)$. In the same way, we can compute $\ker(\tilde{K}'^\circ_Q \twoheadrightarrow K_Q'^\circ)$.

\begin{Rem}\label{rem:Noel1}
	If one wants to determines the admissibility for the the adjoint real group $\GR^{ad}$ instead, one needs to replace in the general strategy above the coroot lattice of $\g$ by the coweight lattice of $\g$, and takes its intersection with $\tf_\chi$ to get a lattice in $\tf_\chi$. In general, it might contains strictly the lattice $X_*(T_\chi)$ in the proof of Theorem \ref{thm:exceptional}. 
	In \cite{Noel1, Noel2}, however, all computations seem to be done using the coroot lattice of $\g$. This means that the results of \cite{Noel1, Noel2} do not give as claimed the admissibility of nilpotent orbits for the simple adjoint real groups and sc-admissibility, as defined in Remark \ref{rem:sc-adm}. Instead, the orbits labelled as ``admissible" there are actually linearly admissible and those labelled as ``sc-admissible" are nonlinearly admissible in the sense of Definition \ref{defn:linearly-admissible}. Therefore to fullfill No{\"e}l's original goal, one has to redo the calculations using the correct basis, for groups of type $E_6$ and $E_7$. 
	
	Note that the tables of admissbile $K$-orbits at the end of \cite{LY} are taken from \cite{Noel1, Noel2}, therefore should also be reinterpreted as above. However, this does not affect any theoretical result in \cite{LY}. Also note that \cite{Nevins} only considers the case of $\R$-points of a split simply connected algebraic group of one of the exceptional types $G_2$, $F_4$, $E_6$ or $E_7$.
	
	There are $K$-orbits, for which $\operatorname{codim}_{\overline{\Orb}}\partial\Orb\geqslant 6$, that are not admissible for adjoint groups from our computations, but are marked as admissible by No{\"e}l. See Example \ref{Ex:adm} below.
\end{Rem}

\begin{Ex}\label{Ex:adm}
	Let $G_{ad}$ be the adjoint complex group of type $E_7$ and $K = (G_{ad}^\sigma$ be the fixed locus of the Cartan involution $\sigma$ of $G_{ad}$ corresponding to the split real form $E_{7(7)}$, which is disconnected and has component group $\Z_2$. Let $K^\circ$ be the identity component of $K$. There are three $K^\circ$-orbits $\#10, 11, 12$ in the $G_{ad}$-orbit $A_2+A_1$ in $\mathfrak{e}_7$. The union of orbit 10 and 11 forms one $K$-orbit according to Djokovic \cite{Dk5}, while 12 is a $K$-orbit by itself.
	
	\cite[Table IV]{Noel1} claims that all the three orbits $\#10, 11, 12$ are admissible for the adjoint group. However, our recalculation shows that orbit 10 and 11 are not admissible for the adjoint group, while they become admissible for the group $G_{sc}(\R)$ of $\R$-points of the simply connected group $G_{sc}$ of type $E_7$ (i.e., they are linearly admissible in the sense of Definition \ref{defn:linearly-admissible}). We will see below that this discrepancy leads to different predictions of the number of HC modules.
	
	The $G_{ad}$-orbit $A_2+A_1$ is birationally rigid and has boundary of codimension greater than $4$, therefore any admissible vector bundles can be quantized to a HC-module. According to \cite[Table 8]{King}, the $K^\circ$-component groups $Z_K$ of orbit 10 and 11 are both $\Z_2$, while $Z_K$ for orbit 12 is $\Z_2 \times \Z_2$. 
	This means that there should be 6 irreducible HC $(\g,K)$-modules with complex associated variety $A_2+A_1$ if all the three $K^\circ$-orbits are admissible (for $K^\circ$), two of which have $K$-associated varieties equal to the union of orbit 10 and 11, while the other 4 irreducible HC modules are all supported on orbit 12 corresponding to the 4 admissible local systems ($\rho_\omega = 0$ for orbit 12). However, calculations with \texttt{atlas} gives only 4 HC modules. Note that $A_2+A_1$ is special and its Barbasch-Vogan dual orbit is $E_6(a_1)$ in $\mathfrak{e}_7$, so the representations in question are all special unipotent. These 4 HC modules can also be found on line 38 of the table at:\\
	\url{http://www.liegroups.org/tables/unipotentExceptional/E7/E7_ad/E7_ad_summary.txt}\\
	They are nothing else but quantizations of the 4 admissible local systems over orbit 12, which matches with our calculations.
	
\end{Ex}

\subsubsection{Main result} We will list in Theorem \ref{thm:exceptional} below all the relevant $K$-orbits as above by cases and make conclusions about them. We follow the labeling of $K$-orbits in  \cite{Dk1, Dk6, Dk5, Dk2, Dk3, Dk4, Dk7, Dk8}, where the closure relations of the orbits can also be found. The papers \cite{Noel1, Noel2} also use the same labeling. The information about admissibility of the orbits of exceptional groups can be found in \cite{Noel1, Noel2} (after correctly reinterpreted as in Remark \ref{rem:Noel1}) and partially in \cite{Nevins}. 

\begin{Thm}\label{thm:exceptional}
	The list of either linearly or non-linear admissible $K$-orbits $\Orb_K \subset \Orb$ satisfying $\operatorname{codim}_{\overline{\Orb}}\partial\Orb = 4$ can be divided into the following three cases:
	\begin{enumerate}
		\item[Case 1:]
		the involution on the $a_2$-slice is inner or $\overline{\Orb}_K$ does not intersect the leaf $X'$ with the $a_2$-singularity. In this case, all orbits are nonlinearly admissible.
		\begin{itemize}
			\item
			$E_{6(2)}$: $\#12$, $\#13$ and $\#14$ in $A_2 + 2A_1$; 
			\item
			$E_{6(-14)}$: $\# 7$ in $A_2+A_1$; $\# 8$ in $A_2+A_1$; 
			\item
			$E_{7(-5)}$: $\#27$ in $A_4 + A_1$; 
			\item
			$E_{8(8)}$: $\#42$ in $A_4+2A_1$;  $\#43$ in $A_4+A_1$;
			\item
			$E_{8(-24)}$: $\#26$ in $A_4+A_1$. 
		\end{itemize}
		\item[Case 2:]
		the involution on the $a_2$-slice is outer and $\tilde{Z}_K = Z_K$. In this case, all orbits are linearly admissible.
		\begin{itemize}
			\item
			$E_{6(6)}$: $\# 8$ in $A_2+A_1$; 
			\item
			$E_{8(8)}$:  $\# 38$ in $A_4+A_1$; 
		\end{itemize}
		\item[Case 3:]
		the involution on the $a_2$-slice is outer and $\tilde{Z}_K \neq Z_K$. In this case, all orbits are linearly admissible.
		\begin{enumerate}[(a)]
			\item
			$E_{6(6)}$: $\# 10$ in $A_2+2A_1$; 
			\item
			$E_{7(7)}$: $\# 50$ in $A_4+A_1$; 
			\item
			$E_{8(8)}$: $\#44$ in $A_4 + 2A_1$; 
		\end{enumerate}
	\end{enumerate}
	We have $\pi_1(K) \simeq \Z_2$ in all cases except for $E_{6(-14)}$, for which $\pi_1(K) \simeq \Z$. 
	Moreover, we have the following conclusions:
	\begin{enumerate}[(i)]
		\item
		For Case 1 and 2, the functor \[ \bullet_{\dagger,\chi} : \HC_{\Orb_K}(\A_\lambda,\theta)^{\tilde{K},\kappa}\xrightarrow{\sim} \operatorname{Rep}(\tilde{K}_Q,\kappa_Q)\]
		is an equivalence for any quantization parameter $\lambda$.  
		\item
		For Case 3, we have $\underline{\tilde{K}} \simeq \Z_4$ and the map $\Z_4 \to \tilde{Z}_K$
		is injective. Moreover, $K$ is simple and $\rho_\omega=0$, so $\kappa =0$ and $\kappa_Q=0$ (for $\lambda=0$). These imply that the functor
		\begin{equation} \label{eq:HC_bad_exc_orbits}
			\HC_{\Orb_K}(\A_0,\theta)^{\tilde{K},\kappa} = \HC_{\Orb_K}(\A_0,\theta)^{\tilde{K},0} \to \operatorname{Rep}(\tilde{K}_Q,\kappa_Q) = \operatorname{Rep}(\tilde{K}_Q)
		\end{equation}
		is not essentially surjective. This means that there always exists at least one irreducible $\tilde{K}$-equivariant local system over $\Orb_K$, which does not quantize to a HC $(\A_0, \tilde{K})$-module.
		\item
		Now assume $\kappa=0$ (which is automatic in all cases except for $E_{6(-14)}$) and $\lambda = 0$ (which is automatic if $\Orb$ is not $A_2+A_1$ or $A_2+ 2A_1$  in $\mathfrak{e}_6$). Then the category $\operatorname{Rep}(\tilde{K}_Q,\kappa_Q)$ is nonempty for the (unique) $2$-fold cover $\tilde{K}$ of $K$. For the two orbits of $E_{6(-14)}$, $\operatorname{Rep}(\tilde{K}_Q,\kappa_Q)$ is nonempty if and only if $\tilde{K}$ is a connected cover of $K$ of even degree, in which case all representations in $\operatorname{Rep}(\tilde{K}_Q,\kappa_Q)$ and hence $\HC_{\Orb_K}(\A_0,\theta)^{\tilde{K},0}$ factor through the $2$-fold cover of $K$.
	\end{enumerate}

\end{Thm}

\begin{proof}
	The lists are obtained by case-by-case inspections and computations, using the general strategy outlined above this theorem.
	
	We first prove (i). Lemma \ref{lem:only_a2} implies that, in both Case 1 and 2, $X=\spec \C[\Orb]$ contains only one codimenison $4$ leaf, along which the singularity is of type $a_2$. Therefore the orbits in Case 1 satisfy (ii) of Theorem \ref{Thm:main1} and we are done.  For Case 2,  both orbits have $\kf_Q=\mathfrak{so}(3)$ according to \cite{Dk0.5}, which are semisimple. Therefore by Section \ref{SSS_homog_case}, the only strongly $K$-equivariant Picard algebroid over $\Orb_K$ is the trivial one (up to isomorphism). Together with the fact that $\tilde{Z}_K = Z_K$, this implies that the $\tilde{K}$-equivariant admissible vector bundles are in fact $K$-equivariant and are nothing else but (untwisted) $K$-equivariant local systems (of rank $1$), for any quantization parameter $\lambda$ (recall that $A_2+A_1$ in $\mathfrak{e}_6$ is not birationally rigid). Therefore (i) of Theorem \ref{Thm:main1} applies.
	
	
	For (ii), we will give a detailed computation of the component group $\tilde{Z}_K$ for the orbit (a) in Section \ref{SS:orbit_a} and a less detailed one for the orbit (b) in Section \ref{SS:orbit_b}, as demonstrations of the general strategy outlined above. The analysis of orbit (c) is similar and we omit it. These computations confirm $\underline{\tilde{K}} \simeq \Z_4$ and the injectivity of the map $\Z_4 \to \tilde{Z}_K$. From \cite{Noel1, Noel2}, we also know the Lie algebra character $d \delta_\chi = (\omega_{\Orb_K})_\chi \in (\kf_Q^*)^{K_Q}$ (as in Section \ref{SS_intro_approach}) for $\Orb_K$, so that $\rho_\omega = \frac{1}{2} d \delta_\chi$ is also known. Then we can again apply Proposition \ref{Lem:image_description11} and Lemma \ref{Lem:sl3_so3_obstructive}. 
	
	(iii) follows from computations following the strategies outlined in Section \ref{SSS_pi1=Z} and \ref{SSS_pi1=Z2}. We will give a detailed computation for the $K$-orbit $\# 7$ of $E_{6(-14)}$ in $\Orb = A_2+A_1$ below in Section \ref{SS:orbit_Hermitian}, for which $\pi_1(K) \simeq \Z_2$.
\end{proof}

\subsubsection{The $K$-orbit $\# 7$ of $E_{6(-14)}$ in $\Orb = A_2+A_1$}\label{SS:orbit_Hermitian} \, \\

Let $(\g, \kf)$ be the Hermitian symmetric pair of type $E_{6(-14)}$. Then $\kf \simeq \mathfrak{so}(10,\C) \oplus \C$ and $\pi_1(K) \simeq \pi_1(G_\R) \simeq \Z$. Choose the set of simple roots $\Pi = \{ \alpha_1, \alpha_2, \ldots, \alpha_6 \}$ of $\g$ as in \cite{Bourbaki}. Set
\[  \beta_1 = \alpha_1, \beta_2 = \alpha_3, \beta_3 = \alpha_4, \beta_4 = \alpha_2, \beta_5 = \alpha_5, \beta_6 = - \alpha_1  - 2 \alpha_2   - 2\alpha_3  - 3 \alpha_4 - 2 \alpha_5 - \alpha_6. \]
Then $\Pi_\kf = \{\beta_1, \ldots, \beta_6\}$ forms a set of simple roots for $\kf$. Here we choose $\alpha_6$ to be the unique non-compact imaginary simple root.

Let $\Orb_K$ be the $K$-orbit $\# 7$ in $\Orb=A_2+A_1$. 
By \cite[Table III]{Noel1}, we can take the Cartan subalgebra of $\kf_Q$ to be $\tf_\chi = \C \vartheta^{\chi}_1 \oplus \C \vartheta^{\chi}_2 \oplus \C \vartheta^{\chi}_3$ where
\[ \vartheta^\chi_1 = \alpha^\vee_4, \quad  \vartheta^\chi_2 = \alpha^\vee_5, \quad \vartheta^\chi_3 = 2 \alpha^\vee_1 + \alpha^\vee_3 + \alpha^\vee_6.\]
We have $d \delta_\chi = 3z_3$, therefore $\vartheta^{\chi}_{nc} = \vartheta^{\chi}_3$, $c_{nc} = 3$ and $b_{nc} = 1$, with the notations introduced in Section \ref{SSS_pi1=Z}. Therefore by the discussions there, $\Orb_K$ is nonlinearly admissible and is admissible for an $m$-fold cover $\tilde{K}$ or $K$ if and only in $m$ is even.

\subsubsection{The $K$-orbit $\# 10$ of $E_{6(6)}$ in $\Orb = A_2 + 2A_1$}\label{SS:orbit_a} \, \\

Let $(\g, \kf)$ be the symmetric pair of type $E_{6(6)}$, then $\bar{K} \simeq \Sp(4, \C)/\Z_2$ (recall that $\bar{K} = (G_{ad}^\sigma)^\sigma$). Since $G_{sc}$ is a $3$-fold cover of $G_{ad}$, $K \simeq \bar{K} \simeq \Sp(4, \C)/\Z_2$ and $\tilde{K} = \Sp(4, \C)$.  
Let $\alpha_i$, $1 \leqslant i \leqslant 36$, be the positive roots of $\g = \mathfrak{e}_6$ as in \cite{Dk2} and choose $\Pi = \{ \alpha_1, \alpha_2, \ldots, \alpha_6 \}$ to be the set of simple roots as in \cite{Bourbaki}. All $36$ positive roots written as integral linear combinations of the simple roots $\alpha_i$, $1 \leqslant i \leqslant 6$, can be found in \cite[Table 13]{Dk2}. A negative root $-\alpha_i$ will be also written as $\alpha_{-i}$. For each root $\alpha$, let $\g^\alpha$ denote the root space of $\alpha$. Choose root vectors $X_i \in \g^{\alpha_i}$ for $\pm i \in \{ 1, \ldots, 36\}$ so that together with the coroots $\alpha^\vee_i$, $i= 1, \dots, 6$, they form a Chevalley basis of $\g$. 

Set
\[ \beta_1 = \alpha_2, \quad \beta_2 = \alpha_4, \quad \beta_3 = \frac{\alpha_3 + \alpha_5}{2}, \quad \beta_4 = \frac{\alpha_1 + \alpha_6}{2} , \quad \beta_0 = \beta_1 + 2 \beta_2 + 3 \beta_3 + 2 \beta_4. \]
Then $\Pi_\kf = \{-\beta_0, \beta_4, \beta_3, \beta_2 \}$ forms a set of simple roots for $\kf = \mathfrak{sp}(4,\C)$. We have $\pi_1(K) \simeq \Z_2$.
The fundamental weights of $\kf$ are $(\varpi_1, \ldots, \varpi_4)$ given by
\[
\begin{pmatrix}
	\varpi_1 \\ \varpi_2 \\ \varpi_3 \\ \varpi_4
\end{pmatrix}
=
\begin{pmatrix}
	1  &  1  &  1  &  \frac{1}{2}   \\
	1  &  2  &  2 &   1  \\
	1  &  2  &  3  &   \frac{3}{2}  \\
	1  &  2  &  3  &   2 \\
\end{pmatrix}
\begin{pmatrix}
	-\beta_0 \\ \beta_4 \\ \beta_3 \\ \beta_2
\end{pmatrix}.
\]

Let $\Orb_K$ be the $K$-orbit $\# 10$ in $\Orb=A_2+2A_1$. We have $Z_K = \bar{Z}_K \simeq \Z_2$ by \cite[Table 4]{King} and $d \delta_\chi = 0$ by \cite[Table 4]{Noel2}, so the orbit is admissible for any finite cover of $K$, so in particular $K$-admissible. Using the explicit nilpotent element $\chi$ in \cite[Table 4]{Noel2}, 
we can take the Cartan subalgebra of $\kf_Q$ to be $\tf_\chi =  \C \vartheta^{\chi}$ where \footnote{There is a typo in No{\"e}l's expression for $\chi$: the last root vector $X_{\alpha_1 - \alpha_2 - \alpha_3 - 2 \alpha_4 - \alpha_5 - \alpha_6}$ should be $X_{-\alpha_1 - \alpha_2 - \alpha_3 - 2 \alpha_4 - \alpha_5 - \alpha_6}$ instead. The expression in the former subscript is not even a root since the coefficients have different signs. The latter root is  noncompact imaginary by \cite[Table 3]{Noel2}.}
\[ \vartheta^{\chi} = \alpha^\vee_1 -2 \alpha^\vee_4 + \alpha^\vee_6.\]
Evaluating $(\varpi_1, \ldots, \varpi_4)$ at the basis vector $\vartheta^{\chi}$ of $\tf_\chi$ gives $(0, 1, 0, -2)$. This implies that $\exp(2\pi i \vartheta^{\chi})$ is the identity in $\tilde{K}$. This means that the nontrivial element in the center of $\tilde{K}$ does not lie in $\tilde{K}_Q^\circ$, hence  $\tilde{K}_Q^\circ=K_Q^\circ$ and $\tilde{Z}_K$ is a central extension of $Z_K \simeq \Z_2$ by $\Z_2$. We will show below that $\tilde{Z}_K \simeq \Z_4$.


The closure $\overline{\Orb}$ is normal in codimension $4$ and hence we can identify the symplectic leaf $X'$ in $X$ with the orbit $\Orb'$ of type $A_2 + A_1$. The closure $\overline{\Orb}_K$ intersect with $\Orb'$ of type $A_2 + A_1$ at the $K$-orbit $\Orb'_K  = \# 8$ and the involution on the corresponding slice is outer. As in \cite[Table 4]{Noel2}, choose a nilpotent element $\chi'$ in $\Orb'_K$ so that the Cartan subalgebra of $\kf'_Q$ can be taken as $\tf_{\chi'} = \C \vartheta^{\chi'}$ where
\[ \vartheta^{\chi'} = \alpha^\vee_1 + \alpha^\vee_2 + 2\alpha^\vee_3 + 2\alpha^\vee_4 + 2\alpha^\vee_5+ \alpha^\vee_6.\]
Evaluating $(\varpi_1, \ldots, \varpi_4)$ at the basis vector $\vartheta^{\chi}$ of $\tf_{\chi'}$ gives $(-1/2, 0, 1/2, 0)$. By the discussions in Section \ref{SSS_pi1=Z2}, this means that $\exp(2\pi i \vartheta^{\chi'})$ is the unique non-trivial element of order $2$ in the center of $\tilde{K}$ 
and $\tilde{K}_Q'^\circ$ is a connected $2$-fold cover of $K_Q'^\circ \simeq \SO_3$, hence $\tilde{K}_Q'^\circ \simeq \SL_2$. We see that $\underline{\tilde{K}} \simeq \Z_4$ with $2 \underline{\tilde{K}} = 2 \Z_4$ being the center of $\tilde{K}$. Since the center $2 \underline{\tilde{K}}$ is not in the kernel of the map $\underline{\tilde{K}} \to \tilde{Z}_K$ by the discussion above, we conclude that $\underline{\tilde{K}} \to \tilde{Z}_K$ is injective. This also implies that the map $\underline{K} \to Z_K$ is also injective. 
We then conclude that $\tilde{Z}_K \simeq \Z_4$ and the natural map $\underline{\tilde{K}} \to \tilde{Z}_K$ is an isomorphism. The kernel of the natural map $\tilde{Z}_K \to Z_K$ equals $\ker(\tilde{K} \twoheadrightarrow K) \simeq \Z_2$. Therefore we only get three irreducbile HC $(\A_0, \tilde{K})$-modules. This finishes the proof of part (ii) of Theorem \ref{thm:exceptional} in this case.

 As mentioned at the beginning of this section, $A_2 + 2A_1$ in $\mathfrak{e}_6$ is not birationally rigid. Let us explain the classification of irreducible HC modules with full support over $\Orb_K$ for any quantization parameter $\lambda$ in this case. The component group
 for the orbit $A_2+2A_1$ is trivial, \cite[Section 8.4]{CM}. According to \cite[Table 6]{GraafElashvili}, the orbit $A_2 + 2A_1$ is (automatically birationally) induced from the zero nilpotent orbit of a Levi subgroup $L \subset G$ of type $A_4 + A_1$. Choose a parabolic subgroup $P$ in $G$ with Levi factor $L$ and form the generalized Springer map $T^*(G/P) \to \overline{\Orb}$. One can show that the parabolic Springer fiber over any point $\chi' \in \Orb'$ is identified with $\mathbb{P}^2$ and the pullback map $\operatorname{Pic}(G/P)\rightarrow
 \operatorname{Pic}(\mathbb{P}^2)$ is an isomorphism of free abelian groups of rank $1$. Using Lemma \ref{Lem:sl3_so3_obstructive}, we can also apply the same argument in the proof of Theorem \ref{Thm:spin_HC_classif_general} and get similar classification as in the case of type $A$ groups. We leave the details to interested readers. Note that in this case, $\kf_Q = \mathfrak{so}(2) \subset \gl(2) = \q$, hence the natural striction map $(\q^*)^Q \to (\mathfrak{k}_Q^*)^{K_Q}$ is zero. Therefore restriction of the twist $\lambda$ to $\Orb_K$ is always trivial by Corollary \ref{Cor:equiv_condition}. Moreover $d \delta_\chi = 0$, therefore $\kappa_Q=0$ and $\operatorname{Rep}(\tilde{K}_Q, \kappa_Q)$ is equivalent to $\operatorname{Rep}(\tilde{Z}_K) \simeq \operatorname{Rep}(\Z_4)$.

\subsubsection{$K$-orbit $\# 50$ of $E_{7(7)}$ in $\Orb=A_4+A_1$}\label{SS:orbit_b} \, \\

Let $(\g,\kf)$ be the symmetric pair of type $E_{7(7)}$ with $\bar{K} = (G_{ad}^\sigma)^{\circ} \simeq \SL(8) / \Z_4$, so $K = (G_{sc}^\sigma)^{\circ} \simeq \SL(8)  / 2\Z_4$ (since $\pi_1(K) \simeq \Z_2$) and $\tilde{K} \simeq \SL(8)$. Let $\Orb_K$ be the $K$-orbit $\# 50$ in $\Orb=A_4+A_1$. We have $Z_K \simeq \Z_4$. In this case $\kf_Q=0$, so $d\delta_\chi = 0$ and the $K$-orbit $\Orb_K$ is admissible.

Let $X=\spec\C[\Orb]$, then $X$ is normal in codimenision $4$ and it has two codimension $4$ leaves, identified with the orbit $A_4$ with $a_2$-singualrity and the orbit $A_3 + A_2 + A_1$ with $a_2 /S_2$-singularity, respectively. By Lemma \ref{Lem:character_integrality} and the same argument in the proof of Theorem \ref{Thm:main1} (Section  \ref{SS:proof_main1}), the $a_2/S_2$-singularity poses no obstruction to the quantization of (twisted) local systems over $\Orb_K$. Therefore we only need to consider the $a_2$-singularity along the orbit $A_4$.

The closure $\overline{\Orb}_K$ intersects with $A_4$ at the $K$-orbit $\Orb'_K  = \# 43$ and the involution on the corresponding slice comes from an outer involution of $\slf_3$. Using the explicit representative element in $\Orb'_K$ from \cite[Table 3]{Dk5}, we can compute that $K_Q'^\circ \simeq \bar{K}_Q'^\circ$ while $\tilde{K}'^\circ_Q$ is a double cover of $K_Q'^\circ$, in a similar way as in Section \ref{SS:orbit_a}. We see that $\underline{K} \simeq \underline{\bar{K}} \simeq \Z_2$ and $\underline{\tilde{K}} \simeq \Z_4$, so that $2 \underline{\tilde{K}}$ is identified the unique subgroup of order $2$ of the center $\Z_4$ of $\tilde{K}$.

Since $\kf_Q=0$, it follows that $K_Q$, $\bar{K}_Q$ and $\tilde{K}_Q$ are finite groups and coincide with their component groups $Z_K$, $\bar{Z}_K$ and $\tilde{Z}_K$. Therefore the maps from $\underline{K}$, $\underline{\bar{K}}$ and $\underline{\tilde{K}}$ respectively to $K_Q$, $\bar{K}_Q$ and $\tilde{K}_Q$ respectively are all injective, and so we can regard the former as subgroups of the latter. The group $\tilde{Z}_K$ is a central extension of $\bar{Z}_K \simeq \Z_4$ (by \cite[Table 8]{King}) by the center $\Z_4$ of $\tilde{K}$, hence $\tilde{Z}_K$ has to be abelian and so is $Z_K$.
Pick a generator $a$ of $\bar{Z}_K = \Z_4$ and a preimage $\bar{a}$ of $a$ under the projection map $Z_K \twoheadrightarrow \bar{Z}_K$. If $\bar{a}$ is of order $8$, then $Z_K \simeq \Z_8$. But then $\underline{K} \subset Z_K$ is the unique subgroup of $Z_K$ of order $2$, which is also the kernel of the projection $Z_K \twoheadrightarrow \bar{Z}_K$. This contradicts with the fact that $\underline{\bar{K}}$ is the isomorphic image of $\underline{K}$ under $Z_K \twoheadrightarrow \bar{Z}_K$. Therefore we conclude that $\bar{a}$ is of order $4$ and generates a cyclic subgroup of $Z_K$ of order $4$, so that $Z_K \simeq \Z_4 \times \Z_2$, where the second factor $\Z_2 \simeq \Z_4 / 2 \Z_4$ is the center of $\bar{K}$.

We now apply Theorem \ref{Thm:main1} and conclude: there are totally $4$ distinct irreducible admissible $\bar{K}$-equivariant line bundles on $\Orb_K$ and give rise to $4$ irreducible HC $(\g, \bar{K})$-modules supported on $\Orb_K$. There are $8$ irreducible admissible $K$-equivariant line bundles on  $\Orb_K$, $4$ of which are $\bar{K}$-equivariant ones mentioned before. The other $4$ genuine ones quantize to $4$ distinct genuine irreducible HC $(\g, K)$-modules. 


We are unable to determine $\tilde{Z}_K$. Continuing the analysis above for $Z_K$ shows that either $\tilde{Z}_K \simeq \Z_8 \times \Z_2$ such that the embedding $\Z_4 \simeq \underline{\tilde{K}} \hookrightarrow \tilde{Z}_K \simeq \Z_8 \times \Z_2$ is given by the embedding of $\Z_4$ to the first factor $\Z_8$, or $\tilde{Z}_K \simeq \Z_4 \times \Z_4$, where the second factor is the center of $\tilde{G}_\R$, and the embedding $\Z_4 \simeq \underline{\tilde{K}} \hookrightarrow \tilde{Z}_K \simeq \Z_4 \times \Z_4$ sends the generator $\bar{1}$ of $\Z_4$ to $(\bar{2}, \bar{1})$. The irreducible admissible $\tilde{K}$-equivariant vector bundles are equivalent to irreducible characters of $\tilde{Z}_K$, so in either case there are totally $16$ of them. We now apply Lemma \ref{Lem:sl3_so3_obstructive} and make the following conclusions: among the $16$ irreducible admissible $\tilde{K}$-equivariant vector bundles (in fact, local systems), $4$ admit no quantization, $4$ are quantized to irreducible Harish-Chandra $(\g, \bar{K})$-modules, $4$ are quantized to irreducible genuine Harish-Chandra $(\g, K)$-modules and $4$ are quantized to genuine Harish-Chandra $(\g, \tilde{K})$-modules that do not factor through $K$.\\

\appendix
\section{The case of non-normal orbit closures}\label{sec:appendix_A}

\subsection{Setup} \label{appendix_A_setup}
In this appendix, we study the case when the closure $\overline{\Orb}$ of $\Orb$ in $\g^*$ has non-empty codimension $2$ boundaries, but the affinization $X=\operatorname{Spec}(\C[\Orb])$ has singularities of codimension greater than $2$. In particular, $\overline{\Orb}$ is non-normal. Any such orbit can only appear for $\g$ of exceptional types by \cite{KP2} and has a unique codimenision $2$ orbit $\Orb'$ in its closure. Moreover, the transversal slice to $\Orb'$ in $\overline{\Orb}$ is isomorphic to the variety $m$ defined as follows: 
Consider the irreducible representation $V(i)$ of highest weight $i \in \mathbb{Z}_{\geq 0}$ of $\SL_2$ (so that $V(i)$ is of dimension $i+1$) and fix a heighest weight vector $v_i \in V(i)$ associated with a fixed Borel subgroup of $\SL_2$. Then the variety $m$ is defined as the closure in $V$ of the $\SL_2$-orbit through $v = v_2 + v_3$ in $V(2)\oplus V(3)$, which is a two-dimensional variety with an isolated singularity at the origin. This variety is not normal, and its  normalization is isomorphic to the affine plane $\C^2$. In particular, this implies that the preimage of $\Orb'$ in $X$ is isomorphic to $\Orb'$. See \cite[Section 3.2.1, 3.2.2]{FJLS1} for details. 

Here is the list of all such orbits $\Orb$:
\begin{itemize}
	\item
	$G_2$: $\Orb = \tilde{A}_1$, $\Orb' = A_1$;
	\item
	$F_4$: $\Orb=\tilde{A}_2 + A_1$, $\Orb'=A_2 + \tilde{A}_1$;
	\item
	$E_6$: $\Orb=A_3+A_1$, $\Orb'=2A_2+A_1$;
	\item
	$E_7$: $\Orb=(A_3+A_1)'$, $\Orb' = 2A_2+A_1$;
	\item
	$E_8$: $\Orb=A_5+A_1$, $\Orb'=A_4+A_3$;
	\item
	$E_8$: $\Orb=D_5(a_1)+A_2$, $\Orb'=A_4+A_3$;
	\item
	$E_8$: $\Orb=A_3+A_1$, $\Orb'=2A_2+A_1$.
\end{itemize}
 
One can check that $\operatorname{codim}_{\overline{\Orb}'}(\partial \Orb')>2$, hence the affinization $X$ has singularities of codimension strictly greater than $4$ in all cases except for $\Orb =A_3+A_1$ in $E_6$. In this case, there is another codimenion $4$ orbit $\Orb''=A_3$ in the smooth locus of $\overline{\Orb}^{sing}$ along which the singularity is of type $c_2$ (and so is the corresponding singularity in $X$). The case of $c_2$ singularity has already been treated in Section \ref{SS_c2} and presents no obstruction to quantization. Therefore we only need to analyze the behaviour of the $K$-orbits $\Orb_K$ and relevant vector bundles near $\Orb'$. 

In Table \ref{Table:Sing_m}, we list all pairs $(\Orb, \Orb')$ together with $K$-orbits $\Orb_K$ in $\Orb \cap \kf^\perp$ and $\Orb'_K = \overline{\Orb}_K \cap \Orb'$. Note that all the real semsimple groups relevant here have $K$ semisimple. 
We mark a $K$-orbit by ``$*$" if it is nonlinearly admissible, and by ``$\times$" if it is not admissible. Otherwise the orbit is linearly admissible.

\begin{table}[h!]
	\caption{Orbits with singularity $m$}
	\label{Table:Sing_m}\vskip 1em
	\begin{tabular}{|c|c|c|c|c|}
		\hline
		$\g_\R$ &  $\Orb$ & $\Orb'$ & $\Orb_K$ & $\Orb_K'$  \\ \hline
		{\bf G\, I$=G_{2(2)}$} & $\tilde{A}_1$ & $A_1$ & $2^*$ & $1$ \\ \hline
		{\bf F \,I $=F_{4(4)}$} & $\tilde{A}_2 + A_1$ & $A_2 + \tilde{A}_1$ & $13$  & $10^\times$  \\\hline
		{\bf E I = $E_{6(6)}$}  & $A_3+A_1$ & $2A_2+A_1$ & $15^*$ & $11$  \\ \hline
		{\bf E II = $E_{6(2)}$}  & $A_3+A_1$ & $2A_2+A_1$ & \makecell{$18^*$ \\ $19^*$} & \makecell{$17$  \\ none} \\ \hline
		{\bf E V= $E_{7(7)}$} & $(A_3+A_1)'$ & $2A_2+A_1$ & $25^\times$  & $24$   \\ \hline
		{\bf E VI = $E_{7(-5)}$} & $(A_3+A_1)'$ & $2A_2+A_1$ & \makecell{$17^*$ \\ $18^*$} & \makecell{$16$  \\ none} \\ \hline
		{\bf E VIII $=E_{8(8)}$}  & $A_5+A_1$ & $A_4+A_3$ & $58^\times$ & $57$ \\ \hline
		{\bf E VIII $=E_{8(8)}$}  & $D_5(a_1)+A_2$ & $A_4+A_3$ & $59^\times$ & $57$ \\ \hline
		{\bf E VIII $=E_{8(8)}$}  & $A_3+A_1$ & $2A_2+A_1$ & $18^\times$ & $17$  \\ \hline
		{\bf E IX $=E_{8(-24)}$} & $A_3+A_1$ & $2A_2+A_1$ & \makecell{$16^*$ \\ $17^*$} & \makecell{none \\ $15$ }   \\ \hline
	\end{tabular}
\end{table}

For all orbits in Table \ref{Table:Sing_m}, the methods in the main text does not apply since $\operatorname{codim}_{\overline{\Orb}_K}\partial \Orb_K =1$. In Proposition \ref{prop:extension} below, however, we will show that any admissible vector bundle over $\Orb_K$ can be extended uniquely to a certain twisted $\mathscr{D}$-module over a $K$-stable closed smooth Lagrangian subvariety $Y$ in $X^{reg}$, which contains $\Orb_K$ as an open dense subvariety and satisfies $\operatorname{codim}_{\overline{Y}}(\overline{Y}\setminus Y)\geqslant 3$. Then in Theorem \ref{thm:non-normal} we will apply the results from \cite{LY} to construct quantization of this twisted $\mathscr{D}$-module to obtain a HC $(\g,K)$-module.

Note that all orbits above except for $A_3+A_1$ in $E_6$ are birationally rigid. Let $\A_0$ be the canonical quantization of $X$ and let 
$\mathcal{J}_0$ denote the kernel of the quantum comoment map $\mathcal{U}\rightarrow \A_0$. By \cite[Theorem 1.0.2]{MM}, $\mathcal{J}_0$
is a maximal ideal in $\U$. Combining \cite[Proposition 6.5.4]{LMM} and \cite[Lemma 3.5.2]{LMM}, we see that $\U/\mathcal{J}_0\xrightarrow{\sim}\A_0$. We emphasize, however, that this isomorphism does not intertwine the filtrations on these 
algebras: the quantization filtration on $\A_0$ and the filtration on $\U/\J_0$ induced by the PBW filtration on $\U$.
Indeed, the quotient of $\gr \U/\J_0$ by its radical is $\C[\overline{\Orb}]$, while
$\gr\A_0=\C[\Orb]$. Since $\overline{\Orb}$ is not normal, the graded algebras  $\gr \U/\J_0$ and $\gr\A_0$ are not isomorphic.
In what follows we always consider $\A_0$ with its quantization filtration.

We have the following analog of Proposition \ref{Prop:HC_bijection}, but the proof is more involved.

\begin{Prop}\label{Prop:HC_bijection_nonnormal}
Let $\Orb$ be any nilpotent orbit in $\g^*$ (not necessarily satisfying \ref{eq:codim_condition}).  Then there is a natural bijection between
\begin{itemize}
\item irreducible $K$-equivariant HC $(\A_0,\theta)$-modules of GK dimension
$\frac{1}{2}\dim \Orb$,
\item and irreducible $K$-equivariant $\U/\J_0$-modules
of GK dimension $\frac{1}{2}\dim \Orb$.
\end{itemize}
\end{Prop}
\begin{proof}
As with the proof of Proposition \ref{Prop:HC_bijection}, we need to establish an analog of Lemma \ref{Lem:HC_bijection2}: 
\begin{itemize}
\item[(*)] every irreducible HC $\U/\J_0$-module $M$ (whose GK dimension is automatically $\frac{1}{2}\dim \Orb$ because $\J_0$
is maximal) admits a good filtration making it a HC $(\A_0,\theta)$-module. 
\end{itemize}
Note the claim of Lemma \ref{Lem:HC_bijection1} is vacuous. Set $X:=\operatorname{Spec}(\C[\Orb])$.

In \cite{filtration}, the first named author has shown that every irreducible strongly $K$-equivariant $\A_0$-module 
admits a so called {\it canonical filtration}, its existence follows from \cite[Theorem 3.12]{filtration}.
In particular, by \cite[Theorem 3.13]{filtration}, the annihilator of $\gr M$ with respect to the canonical filtration is 
a radical ideal. Let $Y\subset X$ be the closed subvariety associated to the annihilator. It remains to show that $Y\subset X^\theta$.

By a theorem of Gabber, all irreducible components of $V(M)$ have the same dimension. A more general version can be found
in \cite[Theorem 0.4]{YZ}.
Arguing as in Step 2 of the proof of Lemma \ref{Lem:HC_bijection1}, we see that this dimension is 
$\frac{1}{2}\dim \Orb$ and all components intersect $\Orb$. It follows that 
$Y\cap \Orb$ is Zariski dense in $Y$. And it is easy to see that 
$Y\cap \Orb\subset X^\theta$. This shows $Y\subset X^\theta$ and finishes the proof.     
\end{proof}


\begin{Cor}\label{Cor:support_nonnormal}
	Let $\Orb$ be as listed in Table \ref{Table:Sing_m}. Let $M$ be an  irreducible $K$-equivariant $\U/\J_0$-module (of GK dimension $\frac{1}{2}\dim \Orb$).  Then the  support of $M$ in $X = \spec(\C[\Orb])$ (w.r.t. the filtration on $\A_0$ such that $\gr \A_0 \simeq \C[\Orb]$) is the closure in $X$ of a single (connected) $K$-orbit $\Orb_K \subset \Orb^\theta \subset X$.
\end{Cor}

\begin{proof}
	Observe from Table \ref{Table:Sing_m} that, when there are two different $K$-orbits $\Orb_K$ in $\Orb$, their closures only intersect in locus of codimension greater than $2$ (clear from the $\Orb'_K$ column). By \cite{filtration} again, $V(M) \subset \overline{\Orb}$ is connected in codimension $1$, so it can only be the closure of one single $\Orb_K$. Therefore the support of $M$ in $X$ is also the closure of $\Orb_K$.
\end{proof}

\subsection{Extension of twisted local systems} \label{appendix_A_extension}
Let $Q'^\circ$ be the identity component of the reductive stablizer $Q'$ of $\Orb'$. We can find all $Q'^\circ$ in the tables of \cite{Alexeevskii}. Note that in all cases but one, $Q'^\circ$ is isomorphic to $\operatorname{SL}_2$, 
no matter $G$ is adjoint or simply connected. The exceptional case is $(\Orb,\Orb')=((A_3+A_1)', 2A_2+A_1)$ in $E_7$. In the this case, $Q'^\circ \simeq \operatorname{SL}_2 \times \operatorname{SO}(3)$ for the adjoint $E_7$ and $Q'^\circ \simeq \operatorname{SL}_2 \times \operatorname{SL}_2$ for the simply connected $E_7$. From now on we take $G$ to be simply connected and $K=G^\sigma$, so that $Q'^\circ$ is always a product of several copies of $\operatorname{SL}_2$. Suppose $\Orb$, $\Orb'$, $\Orb_K$ and $\Orb'_K$ are taken from one of the rows in Table \ref{Table:Sing_m}. If $\Orb'_K$ is empty then there is nothing to do. So we can assume $\Orb'_K$ is nonempty.  Note that all real simple groups $\GR = G(\R)$ with $\g_\R$ from \ref{Table:Sing_m} satisfy $\pi_1(\GR) \simeq \pi_1(K) \simeq \Z_2$ (cf. Section \ref{subsec:cover}).

We have $X^\circ:=X^{reg}=\Orb \sqcup \Orb'$, which is a smooth graded symplectic variety. Set $Y = \Orb_K \sqcup \Orb'_K$, then $Y=(X^{\circ})^\theta$ is a smooth Lagrangian subvariety. Since $X^{sing}$ has codimension strictly greater than $4$, 
we have $\operatorname{codim}_{\overline{Y}}(\overline{Y}\setminus Y)\geqslant 3$
Note that  $\Orb'_K$ is a subvariety in $X^{reg}$. Fix a point $\chi' \in \Orb'_K \subset X^{reg}$ and let $\underline{\chi}$
denote its image in $\kf^\perp$. Consider a normal $\slf_2$-triple $(e',h',f')$
corresponding to $\underline{\chi}$ and the $\theta$-stable Slodowy slice $\underline{S}\subset \g^*$
associated to this triple. Let $\check{X}$ denote the preimage of $\underline{S}$
in $X=\operatorname{Spec}(\C[\Orb])$, this is a transverse slice to $\Orb'$ in $X$.

By \cite[Proposition 3.3]{FJLS1}, there always exists a simple factor $R \simeq \operatorname{SL}_2$ of $Q'^\circ$, so that there is a graded Possion isomorphism between $\check{X}$ with the Kazhdan $\cm$-action and the standard symplectic vector $\C^2$ with the usual dilation $\cm$-action, such that the $R$-action on $\check{X}$ can be identified with the standard $\SL_2$-action on $\C^2$ (since any automorphism of $\SL_2$ is inner). Also note that $\check{X}\subset X$ is $\theta$-stable and after possible conjugation, one can assume that $\theta$ acts on $\C^2 \simeq \check{X}$ by $s: (z_1, z_2) \mapsto (z_1, -z_2)$, and $\theta$ acts on $R \simeq \operatorname{SL}_2$ by $A \mapsto sAs$, so that we can identify $R^\sigma$ with the multiplicative group $\cm$ by $\operatorname{diag}(a, a^{-1}) \mapsto a$. Therefore $\check{Y} := \check{X}^\theta \simeq (z_2=0) \subset \C^2$ and we have an isomorphism $\check{Y} \simeq \C, (z, 0) \mapsto z$, so that we have the usual linear coordinate $z$ on $\check{Y}$. Hence the action of $R^\sigma \simeq \cm$ on $\check{Y}$ can be identified with the usual scalar multiplication $(a,z) \mapsto az$ on $\C$. Let $T$ be the connected closed subrgoup of $K'_Q$ with Lie algebra $\rf^\sigma$, so that the restriction of the group homomorphism $K \to G$ to $T$ gives a $n$-fold covering map $T \to R^\sigma$ and the action of $T \simeq \cm$ on $\check{Y} \simeq \C$ can be then identified with the action $(t,z) \mapsto t^n z$. In fact, $n$ can only be $1$ or $2$ in our setting.

Let $\omega_Y$ be the canonical line bundle of $Y$. Since we consider the canonical quantizations of $X$ and $X^\circ$, we have an isomorphim $\TT^+(\hat{\A}_\hbar^{X^\circ}, Y) \simeq \TT^+(\frac{1}{2}\omega_Y)$ of graded Picard algebroids with strong $K$-actions (see \cite[Lemma 6.2.4]{LY} and the paragraph below it). Let $\mathscr{D}_{\TT^+}$ denote the sheaf of twisted differential operators associated to $\TT^+(\hat{\A}_\hbar^{X^\circ}, Y) \simeq \TT^+(\frac{1}{2}\omega_Y)$ with the induced strong $K$-action as in Section \ref{subsec:pic_defn}. Suppose the $R^\sigma$-action on the fiber of $\omega_Y$ at $\chi' \in \Orb'_K$ is given by a character $z \mapsto z^l$ of $R^\sigma \simeq \cm$ with $l=t+1$, where $t$ corresponds to the character $\rho'_\omega$ for $\Orb'_K$, which is even if and only if $\Orb'_K$ is admissible for $K$. Now consider the pullback $\mathscr{D}_{\TT^+}|_{\check{Y}}$ of $\mathscr{D}_{\TT^+}$ to the slice $\check{Y}$. By Proposition \ref{Prop:slice_strongness}, (i), we see that $\mathscr{D}_{\TT^+}|_{\check{Y}}$ can be trivialized so that the strong $\rf^\sigma$-action $\rf^\sigma \to \mathscr{D}_{\TT^+}|_{\check{Y}} $ is given by the differential operator $z\partial + \frac{l}{2}$. 

Restricting further to $\check{Y}^\times$ gives a second trivialization of $\mathscr{D}_{\TT^+}|_{\check{Y}^\times}$. The first and second trivialization are related by the automorphism of $\mathscr{D}_{\cm}$ induced by
\[ \partial \mapsto \partial + \frac{l}{2z} = z^{-\frac{l}{2}} \partial  z^{\frac{l}{2}}. \]

The $T$-stablizer of any point in $\check{Y} \simeq \C$ is the group $M \subset T \simeq \cm$ of $n$-th root of $1$. Let $\mu_0$ and $\mu_1$ denote the trivial representation and the identity representation of $M$ respectively (if $T=R^\sigma$ so that $M$ is trivial, then $\mu_1$ is identical to $\mu_0$). Let $\tau_k$ denote the $T$-equivariant flat connection on $\check{Y}^\times \simeq \cm$ correponding to the representation $\mu_k$ of $M$, $0 \leq k \leq n-1$. The global sections of $\tau_k$ on $\check{Y}^\times \simeq \cm$ form the vector space spanned by $z^{p+\frac{k}{n}}$, $p \in \Z$. Let $i: \check{Y}^\times \hookrightarrow \check{Y}$ denote the open inclusion. If we use the second trivialization of $\mathscr{D}_{\TT^+}|_{\check{Y}}$ on $\cm$, the direct image $\mathscr{D}_{\TT^+}|_{\check{Y}}$-module $i_* \tau_k$ is isomorphic to the $\mathscr{D}_\C$-module which is the direct image of the $\mathscr{D}_{\cm}$-module generated by $z^{\frac{k}{n}+\frac{l}{2}}$. This module is reducible if and only if it contains constant functions, i.e., if and only if $\frac{k}{n}+\frac{l}{2}$ is an integer, if and only if $ \tau_k$ can be extended to a $\OO_{\check{Y}}$-coherent $\mathscr{D}_{\TT^+}|_{\check{Y}}$-module on $\check{Y}$, if and only if $\tau_k$ has trivial monodromy around $0$. In this case, $i_* \tau_k$ is isomorphic to the $\mathscr{D}_\C$-module which is the direct image of  $\OO_{\cm}$ with the trivial connection.

\begin{Prop}\label{prop:extension}
	A strongly $K$-equivariant $\TT^+(\hat{\A}_\hbar^{X^\circ}, \Orb_K)$-module $\EE$ on $\Orb_K$ can be extended to an $\OO_X$-coherent $\mathscr{D}_{\TT^+}$-module over $Y$ if and only if the restriction of $\EE$ to $\check{Y}^\times$ is a direct sum of $\tau_k$ as $T$-equivariant connection such that $\frac{k}{n}+\frac{l}{2} \in \Z$. In this case, the extension is unique and is the minimal extension $\EE_{min}$ of $\EE$. Moreover, the strongly $K$-equivariant structure and weakly $\cm$-equivariant structure on $\EE$ extend uniquely to $\EE_{min}$.
\end{Prop}

\begin{proof}
	Clearly, if the desired extension $\EE'$ exists, then it is a holonomic $\mathscr{D}_{\TT^+}$-module that has no nonzero $\mathscr{D}_{\TT^+}$-submodules or quotients with support on $\Orb_K'$, therefore must be the minimal extension $\EE_{min}$ of $\EE$. The statement about strong $K$-equivariant and weakly $\cm$-equivariant structures  on $\EE_{min}$ is clear.

	First assume $\EE$ extends to an $\OO_X$-coherent $\mathscr{D}_{\TT^+}$-module $\EE'$ over $X$. Then the restriction $\EE'|_{\check{Y}}$ of $\EE'$ to $\check{Y}$ is also an $\OO_{\check{Y}}$-coherent extension of $\EE|_{\check{Y}^\times}$ to $\check{Y}$. By the discussion above, this implies that $\frac{k}{n}+\frac{l}{2} \in \Z$ for any direct summand $\tau_k$ of $\EE|_{\check{Y}^\times}$ of $\EE$.

	Conversely, assume that restriction of $\EE$ to $\check{Y}^\times$ is a direct sum of $\tau_k$ as $T$-equivariant connection such that $\frac{k}{n}+\frac{l}{2} \in \Z$.  Let us consider the analytification $\EE^{an}$ of $\EE$. As a twisted local system, $\EE^{an}$ has trivial monodromy around $\Orb_K'$ by the discussion above, therefore it can be uniquely extended to holormorphic vector bundle $\wt{\EE}^{an}$ over $X$ which is also a $\mathscr{D}_{\TT^+}^{an}$-module (see \cite[II.5]{Deligne} and \cite[Theorem 5.2.17]{Hotta}). By the twisted version of  Riemann-Hilbert correspondence (see Section \ref{subsec:RH}; also see \cite[Theorem 5.3.8]{Hotta}), $\wt{\EE}^{an}$ is the analytification of an algebraic $\OO_X$-coherent $\mathscr{D}_{\TT^+}$-module $\wt{\EE}$, which we have seen must be the minimal extension $\EE_{min}$ of $\EE$.
\end{proof}

\subsection{Classification of HC modules} \label{appendix_A_classification}

Consider the category $\HC^{K,\kappa}_{Y}(\A_0,\theta)$ of all $(K,\kappa)$-equivariant HC $(\A_0,\theta)$ modules supported 
on $\overline{Y}$. Pick a point $\chi\in \Orb_K$ and consider the restriction functor 
$\bullet_{\dagger,\chi}:\HC^{K,\kappa}_Y(\A_0,\theta)\rightarrow \underline{\A}_0\operatorname{-mod}^{K_Q,\kappa_Q}$ from Section \ref{subsec:W-algebra_nilpotent}
(since $\A_0$ quantizes the normalization of an orbit, $d=1$; also recall that $\underline{\A}_0=\C$).  
   
\begin{Prop}\label{Prop:fully_faithful}
The functor $\bullet_{\dagger,\chi}:\HC^{K,\kappa}_Y(\A_0,\theta)\rightarrow \underline{\A}_0\operatorname{-mod}^{K_Q,\kappa_Q}$
is fully faithful.
\end{Prop}
\begin{proof}
The functor $\bullet_{\dagger,\chi}$ is faithful: a module killed by $\bullet_{\dagger,\chi}$ must be supported on
on $Y\setminus \Orb_K$. It is holonomic, so its the associated variety of the annihilator is  a proper subvariety of $X$, and, by \cite{B_ineq}, this contradicts the simplicity of $\A_0$.
 
It remains to show the fullness. 
We claim that it is enough to show that the functor $\bullet_{\dagger,\chi}:\wHC^{gr}_Y(\A_\hbar,\theta)^{K,\kappa}\rightarrow \underline{\A}_\hbar\operatorname{-mod}^{gr,K_Q,\kappa_Q}$ is full.  Let $M,N\in \HC^{K,\kappa}_Y(\A,\theta)$, and let $\psi$ be a homomorphism 
$M_{\dagger,\chi}\rightarrow N_{\dagger,\chi}$. Shifting a good filtration on $N_\hbar$ we can assume that 
$\psi$ is the specialization at $\hbar=1$ of a homomorphism $M_{\hbar,\dagger,\chi}\rightarrow N_{\hbar,\dagger,\chi}$. 

It remains to show that $\bullet_{\dagger,\chi}:\wHC^{gr}_Y(\A_\hbar,\theta)^{K,\kappa}\rightarrow \underline{\A}_\hbar\operatorname{-mod}^{gr,K_Q,\kappa_Q}$ is full. Let $\A_\hbar^{reg},\A_\hbar^{\Orb}$
denote the microlocalizations of $\A_\hbar$ to $X^{reg}$ and $\Orb$. Recall, Section \ref{SSS_conclusion_microloc},
that $\bullet_{\dagger,\chi}$ is the composition of, $\bullet|_{\Orb}$, the restriction to $\Orb$ and a category equivalence. 
Since $\operatorname{codim}_Y(Y\cap X^{sing})\geqslant 2$, the global section functor 
$\wHC^{gr}_{Y\cap X^{reg}}(\A_\hbar^{reg},\theta)^{K,\kappa}\rightarrow \wHC^{gr}(\A_\hbar,\theta)^{K,\kappa}$ is right inverse to
the microlocalization functor, compare to the proof of Lemma \ref{Lem:HC_bijection2}. 

It remains to show that the microlocalization functor 
$$\wHC^{gr}_{Y\cap X^{reg}}(\A_\hbar^{reg},\theta)^{K,\kappa}\rightarrow \wHC^{gr}_{\Orb_K}(\A_\hbar^\Orb,\theta)^{K,\kappa}$$
is fully faithful. On the full subcategories of objects annihilated by $\hbar$, the functor becomes the restriction of twisted local systems
to a dense open subset, it is fully faithful. From here we deduce that the functor is fully faithful on the full subcategories of modules annihilated by $\hbar^k$ for any $k$. Now we note that 
$$\Hom(M_\hbar,N_\hbar)\xrightarrow{\sim}\varprojlim_n \Hom(M_\hbar/(\hbar^k),N_\hbar/(\hbar^k)),$$
where $\Hom$ is taken in  $\wHC^{gr}_{Y\cap X^{reg}}(\A^{?}_\hbar,\theta)^{K,\kappa}$, where $?$ is $reg$ or $\Orb$. 
This implies that the microlocalization functor of interest is fully faithful. 
%
\end{proof}

Since $\underline{\A}_0\operatorname{-mod}^{K_Q,\kappa_Q}$ is a semisimple category, Proposition \ref{Prop:fully_faithful} has the following immediate consequence.

\begin{Cor}
	The functor $\bullet_{\dagger,\chi}$ gives rise to an exact equivalence between $\HC^{K,\kappa}_Y(\A_0,\theta)$ and some full subcategory in $\underline{\A}_0\operatorname{-mod}^{K_Q,\kappa_Q}$ that is closed under taking direct summands. In particular, $\bullet_{\dagger,\chi}$ sends irreducibles to irreducibles.
\end{Cor}

\begin{Thm}\label{thm:non-normal}
	For the groups and the orbits in Table \ref{Table:Sing_m},  we have the following conclusions: 
	\begin{enumerate}
		\item
		  Let $\Orb$ be the $G$-orbit $\tilde{A}_2 + A_1$ in $F_4$. For the connected real simple adjoint group $G^{ad}_\R$ of type $F_{4(4)}$, there is a unique irreducible HC $(\A_0,K)$-module. It is supported on the closure of the $K$-orbit $\# 13$. There is no genuine HC $(\A_0,\tilde{K})$-module for the non-linear universal ($2$-fold) covering group of $F_{4(4)}$. 
		\item 
		  For any other $G$-orbit $\Orb$ from Table \ref{Table:Sing_m}, there are no irreducible HC $(\A_0,K)$-modules for the relevant linear real simple groups. For the non-linear universal ($2$-fold) covering groups corresponding to the pairs $(\g, \tilde{K})$, any irreducible HC $(\A_0, \tilde{K})$-module
		  has associated variety equal to $\overline{\Orb}_K$ for a single $K$-orbit $\Orb_K$. Moreover, the conclusion of Theorem \ref{Thm:omnibus} also holds, so that for any $\Orb_K$, $\operatorname{Irr}(\HC_{\overline{\Orb}_K}(\A_0, \tilde{K}))$ is nonempty if and only if $\Orb_K$ is nonlinearly admissible. If this is the case, then $\operatorname{Irr}(\HC_{\overline{\Orb}_K}(\A_0, \tilde{K}))$ is in bijection with the set of all irreducible equivariant (for $\tilde{K}$) twisted (by half-canonical twist) local systems
		  on $\Orb_K$. 
	\end{enumerate}
\end{Thm}

\begin{proof}
	We shall apply Proposition \ref{prop:extension} below case by case to check that, for appropriate $K$ in question, the irreducible $K$-equivariant admissible vector bundle corresponding to every irreducible object $V$ in  $\underline{\A}_0\operatorname{-mod}^{K_Q,\kappa_Q}$ can be extended uniquely to an $\OO_X$-coherent strongly $K$-equivariant and weakly $\cm$-equivariant  $\mathscr{D}_{\TT^+}$-module $\EE$ over $Y$. Then the $\cm$-action on $\EE$ is strong by Proposition \ref{prop:irred_strongness}. Once this is established, we can either apply Theorem \ref{thm:quan_lag_equiv} with $\kappa=0$ or directly invoke \cite[Theorem 6.2.5]{LY} to obtain a $K$-equivariant graded Hamiltonian quantization $\EE_\hbar$ of $\EE$ over $Y$. Now we follow the same proof of \cite[Theorem 7.2.3]{LY} and take $M_\hbar$ to be the $\cm$-finite part of $\Gamma(Y, \EE_\hbar)$ and $M := M_\hbar|_{\hbar=1}$, so that $M$ is an $K$-equivariant HC $(\A_0,\theta)$ module supported on $\overline{Y}$. It is irreducible by the maximality of $\J_0$. Since $\operatorname{codim}_{\overline{Y}} \partial Y \geq 3$, the same proof of \cite[Proposition 7.3.1]{LY} applies and we have $M_{\dagger,\chi}=V$.

	For case (1), Table \ref{Table:Sing_m} indicates that the $K$-orbit $\Orb_K$ is admissible  for the adjoint $F_{4(4)}$ (note that the adjoint complex group of type $F_4$ is simply connected), but $\Orb'_K$ is not. In this case $Q'^\circ = R \simeq \SL_2$ as mentioned at the beginning of \ref{appendix_A_extension}. This implies that $l$ is even, $n=1$ and $k=0$ in Proposition \ref{prop:extension}, hence the unique $K$-equivariant admissible vector bundle $\EE$ over $\Orb_K$ extends uniquely to a twisted local system over $Y$. By the same type of computation as in Section \ref{SS_exceptional}  using the \cite[Table I]{Noel1}, we find that the preimage of $K^\circ_Q$ in $\tilde{K}_Q$ is a connected double cover, hence any admissible vector bundle $\EE$ for $\tilde{K}$ is actually $K$-equivariant. 
	
	For case (2), note that if there are two different $K$-orbits in the same $\Orb$, the intersection of their closures are of codimenision greater than $1$ by Table \ref{Table:Sing_m}, therefore any irreducible $(\A_0,K)$-module can only supported on the closure of a single $K$-orbit. For any nonlinearly admissible $\Orb_K$, if $\Orb'_K$ is empty, then we are done. If $\Orb'_K$ is nonempty, then Table \ref{Table:Sing_m} shows it is always linearly admissible. Moreover, computation shows that $T \subset \tilde{K}$ is always a connected double cover of $R^\sigma$. These imply  that $l$ is odd, $n=2$ and $k=1$ in Proposition \ref{prop:extension}. Therefore again any genuine $\tilde{K}$-equivariant admissible vector bundle $\EE$ over $\Orb_K$ extends uniquely to a twisted local system over $Y$. 
\end{proof}

Combining Theorem \ref{thm:non-normal} and Corollary \ref{Cor:support_nonnormal}, we complete the classification of all irreducible $(\U/\J_0, \tilde{K})$-modules of GK dimension $\frac{1}{2}\dim \Orb$.

\end{document}